\documentclass{amsart}

\usepackage{amssymb}
\usepackage{amsthm}
\usepackage[brazil,spanish,english]{babel}
\usepackage{graphicx}
\usepackage{hyperref}
\usepackage[latin1]{inputenc}
\usepackage{epstopdf}
\usepackage{srcltx}
\usepackage[normalem]{ulem}
\usepackage{vruler}
\usepackage{color}

\usepackage[hmargin=30mm,top=30mm,bottom=35mm]{geometry}

\theoremstyle{plain}
\newtheorem{theorem}{Theorem}
\newtheorem{lemma}[theorem]{Lemma}
\newtheorem{sub-lemma}[theorem]{Sub-lemma}
\newtheorem{corollary}[theorem]{Corollary}

\newtheorem{proposition}[theorem]{Proposition}
\newtheorem*{proposition*}{Proposition}
\newtheorem*{theorem*}{Theorem}

\newtheorem{theoremain}{Theorem}

\newtheorem{propositionmain}[theoremain]{Proposition}
\newtheorem{corollarymain}[theoremain]{Corollary}

\theoremstyle{definition}

\newtheorem*{definition*}{Definition}
\newtheorem*{definitions*}{Definitions}
\theoremstyle{remark}

\newtheorem*{claim*}{Claim}

\newtheorem*{example*}{Example}
\newtheorem{remark}[theorem]{Remark}
\newtheorem*{remark*}{Remark}
\newtheorem*{remarks*}{Remarks}

\newtheorem*{question*}{Question}

\newcommand{\N}{\mathbb{N}}

\renewcommand{\S}{\mathbb{S}}
\newcommand{\T}{\mathbb{T}}
\newcommand{\A}{\mathbb{A}}
\newcommand{\R}{\mathbb{R}}
\newcommand{\Z}{\mathbb{Z}}
\newcommand{\Q}{\mathbb{Q}}

\renewcommand{\ge}{\geqslant}
\renewcommand{\le}{\leqslant}
\renewcommand{\geq}{\geqslant}
\renewcommand{\leq}{\leqslant}

\begin{document}


\title{Topological horseshoes for surface homeomorphisms}

\author{Patrice Le Calvez}

\address{Sorbonne Universit\'e, CNRS, Institut de Math\'ematiques de Jussieu-Paris Rive Gauche, 
IMJ-PRG, F-75005 Paris, France, et Institut Universitaire de France}

\email{patrice.le-calvez@imj-prg.fr}

\author{Fabio Tal}

\address{University of S\~{a}o Paulo, Instituto de Matem\'atica e Estat\'{\i}stica, Rua do Mat\~ao 1010, 05508-090, S\~ao Paulo, Brasil \& Friedrich-Schiller-Universit\''at, Institute of Mathematics, Ernst-Abbe-Platz 2, 07743, Jena, Germany}

\email{fabiotal@ime.usp.br }
\thanks{F. A. Tal was partially supported by the Alexander von Humboldt foundation as well as by CAPES, FAPESP and CNPq-Brasil.}

 \date{}

\begin{abstract}
In this work we develop a new criterion for the existence of topological horseshoes for surface homeomorphisms in the isotopy class of the identity. Based on our previous work on forcing theory, this new criterion is purely topological and can be expressed in terms of equivariant Brouwer foliations and transverse trajectories. We then apply this new tool in the study of the dynamics of homeomorphisms of surfaces with zero genus and null topological entropy and we obtain several applications. For homeomorphisms of the open annulus $\A$ with zero topological entropy, we show that rotation numbers \textcolor{black}{exist} for all points with nonempty omega \textcolor{black}{ limit sets}, and that if $\A$ is a generalized region of instability then it admits a single rotation \textcolor{black}{number}. We also offer a new proof of a recent result of Passegi, Potrie and Sambarino, showing that zero entropy dissipative homeomorphisms of the annulus having as an attractor a circloid have a single rotation number. 

Our work also studies homeomorphisms of the sphere without horseshoes. For these maps we present a structure theorem in terms of fixed point free invariant sub-annuli, as well as a very restricted description of all possible dynamical behavior in the  transitive subsets. This description ensures, for instance, that transitive sets can contain at most $2$ distinct periodic orbits and that, in many cases, the restriction of the homeomorphism to the transitive set must be an extension of an odometer. In particular, we show that any nontrivial and stable transitive subset of a dissipative diffeomorphism of the plane is always infinitely renormalizable in the sense of Bonatti-Gambaudo-Lion-Tresser.

 \end{abstract}

\maketitle



















































\bigskip
\bigskip
\section{Introduction}

In the whole text, we will define the entropy of a homeomorphism $f$ of a Hausdorff locally compact topological space $X$ as being the topological entropy of the extension of $f$ to the Alexandrov compactification of $X$, that fixes the point at infinity. We will say that $f$ is {\it topologically chaotic} if its entropy is positive and if the number of periodic points of period $n$ for some iterate of $f$ grows exponentially in $n$. An important class of topologically chaotic homeomorphisms is the class of homeomorphisms that possesses a {\it topological horseshoe}. Let us give a precise definition of what we have in mind: a compact subset $Y$ of $X$ is a topological horseshoe if it is invariant by a power $f^r$ of $f$ and if $f^r\vert_Y$ admits a finite extension $g:Z\to Z$ on a Hausdorff compact space such that:
\begin{itemize}
\item $g$ is an extension of the Bernouilli shift $\sigma :\{1,\dots, q\}^{\Z}\to \{1,\dots, q\}^{\Z}$, where $q\geq 2$;
\item The preimage of every $s$-periodic sequence of $\{1,\dots,q\}^{\Z}$  by the factor map $u:Z\to \{1,\dots, q\}^{\Z}$ contains  a $s$-periodic point of $g$.
\end{itemize}
By an extension we mean a  homeomorphism semi-conjugated to a given homeomorphism, by a finite extension we mean an extension such that the fibers of the factor map are all finite with an uniform bound in their cardinality. Note that, if $h(f)$ denotes the topological entropy of $f$, then
$$ rh(f)=h(f^r)\geq  h(f^r\vert_Y)=h(g)\geq h(\sigma)=\log q,$$and that $f^{rn}$ has at least $q^n/m$ fixed points for every $n\geq 1$, if the cardinality of the fibers of the factor map $v:Z\to Y$ are uniformly bounded by $m$.\footnote{The reader will notice that the topological horseshoes that will be constructed in this article possess the following additional stability property: if $U$ is neighborhood of $Y$ and $\mathcal U$ a neighborhood of $f$ for the compact open topology, then every map $f'\in\mathcal U$ admits a topological horseshoe $Y'\subset U$ satisfying the same properties as $Y$.

\textcolor{black}{We can add that the requirement to consider finite extensions comes from the following fact: extensions $g$ of Bernouilli shifts will be
constructed for lifts of $f^r\vert _U$ to some covering spaces of $U$,
where $U$ is an open subset of a surface $M$ invariant by the power
$f^r$ of a homeomorphism $f$ of $M$. It is not clear that an extension
of a Bernouilli shift can be constructed directly on the surface.  More
generally, it is a natural question to determine if, whenever $Y$ is a topological horseshoe, then one may find a subset $Y'$ invariant by a power $f^{r'}$ of $f$ such that $f^{r'}\vert_{Y'}$ satisfies the same properties as $g$.}}

Since the groundbreaking work of Smale \cite{smale}, the presence of horseshoes have been a paradigmatic feature of dynamical systems, and its prevalence as a phenomenon is well known. Their role in the particular case of surface dynamics is even larger, as demonstrated by the results of Katok in \cite{Katok}, showing by Pesin's Theory that any sufficiently smooth diffeomorphism of a surface with strictly positive topological entropy has a topological horseshoe. Nonetheless, this direct relationship between existence of horseshoes and topological entropy is not valid in the context of homeomorphisms (see \cite{Rees}, \cite{BeguinCrovisierLeRoux1}).

This article is a natural continuation of \cite{LeCalvezTal}, where the authors gave a new orbit forcing theory for surface homeomorphisms related to the notion of {\it maximal isotopies}, {\it transverse foliations} and {\it transverse trajectories}. A classical example of forcing result is Sharkovski's theorem: there exists an explicit total order  $\preceq$ on the set of natural integers such that every continuous transformation $f$ on $[0,1]$ that contains a periodic orbit of period $m$ contains a periodic orbit of period $n$ if $n\preceq m$. Recall that if $f$ admits a periodic orbit whose period is not a power of $2$, one can construct a Markov partition and \textcolor{black}{code} orbits with a sub-shift of finite type: more precisely there exists a topological horseshoe. The simplest situation from which the existence of a topological horseshoe can be deduced is the existence of a point $x$ such that $f^3(x)\leq x<f(x)<f^2(x)$. The fundamental theorem of this paper, which will be announced at the end of this introduction, is a simple criterion of existence of a topological horseshoe stated in terms of transverse trajectories. We already gave a partial result in this direction in \cite{LeCalvezTal} but it required stronger hypothesis on the transverse trajectories. Moreover the conclusion was weaker: we proved that the map was topologically chaotic but did not get a topological horseshoe. 

The main result can be used as an interesting tool in the study of surface homemorphisms, with relations to horseshoes or to entropy. It has the advantage of being a purely topological technique, and as so can be applied with no restrictions to homeomorphisms, but the theorem has also new and relevant  consequences even in the presence of additional smoothness.  
We apply our main result do deduce many applications to the study of homeomorphisms of a surface of genus zero isotopic to the identity. This will be done usually by studying the complementary problem of describing the behaviour of homeomorphisms with no topological horseshoes, and consequently of $\mathcal{C}^{1+\epsilon}$ diffeomorphisms with zero topological entropy. This subject has been previously investigated for instance by Handel in \cite{Handel}, who studied among other things the existence of rotation numbers for maps of the annulus, and by Franks and Handel in \cite{FranksHandel}, who presented a structure theorem for $\mathcal{C}^{\infty}$ diffeomorphisms of the sphere preserving the volume measure, a result we extended to the $\mathcal{C}^{0}$ context in \cite{LeCalvezTal}. Another related topic that has received considerable attention is the study of dissipative diffeomorphisms of the plane with no entropy, in particular in the Henon family, with an emphasis in describing the dynamics of diffeomorphisms that can be approximated by maps with strictly positive entropy (see, for instance, \cite{CLM, LM, HMT}). Recent progress, which includes a description of dynamical behavior of transitive sets for the subset of {\it strongly dissipative diffeomorphisms}, was obtained by Crovisier, Kocsard, Koropecki and Pujals (see \cite{CKKP}), which enabled a breakthrough in the understanding of the boundary of entropy zero for a region of the parameter plane for the Henon family (see \cite{CPT}).

\subsection{Applications}
Our first application, one that is a main tool in obtaining the other results,  is about rotation numbers for annular homeomorphisms with no topological horseshoes. We will use the definition introduced by Le Roux \cite{LeRoux} and developed by Conejeros \cite{Conejeros}. Write $\T^1=\R/\Z$, denote by $\check\pi: \R^2\to \A$ the universal covering projection of the open annulus $\A=\T^1\times\R$ and by $\pi_1:\R^2 \to \R$ the projection in the first coordinate. Write $T:(x,y)\mapsto (x+1,y)$ for the fundamental covering automorphism. Let $f$ be a homeomorphism  of $\A$ isotopic to identity and $\check f$ a lift to $\R^2$.  Say that a point $z$ is an {\it escaping} point of $f$, if the sequence $(f^n(z))_{n\geq0}$ converges to an end of $\A$ or equivalently if its $\omega$-limit set $\omega(z)$ is empty. We will denote $\mathrm{ne}^+(f)$ the complement of the set of escaping points. Denote by $\mathrm{ne}^-(f)=\mathrm{ne}^{+}(f^{-1})$ the set of points $z$ such that the $\alpha$-limit set $\alpha(z)$ is non empty. We will define  $ \mathrm{ne}(f)=\mathrm{ne}^+(f)\cup \mathrm{ne}^-(f^{-})$ and denote $\Omega(f)$ the set of non-wandering points. We say that $z\in\mathrm{ne}^+(f)$ has a {\it rotation number $\mathrm{rot}_{\check f}(z)$}  if, for every compact set $K\subset\A$ and every increasing sequence of integers $(n_k)_{k\geq 0}$ such that $f^{n_k}(z)\in K$, we have  
$$\lim_{k\to+\infty}\frac{1}{n_k}\left( \pi_1(\check f^{n_k}(\check z))- \pi_1(\check z)\right)=\mathrm{rot}_{\check f}(z),$$ 
if $\check z$ is a lift of $z$. \footnote{In the article, when writing $\mathrm{rot}_{\check f}(z)=\rho$, we implicitly suppose that $z$ belongs to $\mathrm{ne}^+(f)$ and has a rotation number equal to $\rho$.}

It is straightforward that if $h$ is a homeomorphism of $\A$ then $z$ belongs to $\mathrm{ne}^+(f)$ is and only if $h(z)$ belongs to $\mathrm{ne}^+(h\circ f\circ h^{-1})$. Moreover, if $\check h$ is a lift of $h$ to $\R^2$, then $\mathrm{rot}_{\check h\circ\check f\circ\check h^{-1}}(h(z))$ exists if and only if $\mathrm{rot}_{\check f}(z)$ exists and we have $$ \mathrm{rot}_{\check h\circ\check f\circ\check h^{-1}}(h(z)) = \begin{cases} 
\mathrm{rot}_{\check f}(z) &\mathrm{if }\,\, h\mathrm{\,\,  induces \,\,  the \,\, identity \,\, on\,\,  }H_1(\A,\Z),\\
-\mathrm{rot}_{\check f}(z) &\mathrm{otherwise.}\\
\end{cases}$$ 
Consequently, one can naturally define rotation numbers for a homeomorphism of an abstract annulus $A$ and a given lift to the universal covering space, as soon as a generator of $H_1(A,\Z)$ is chosen. More precisely, let $f$ be homeomorphism of $A$ isotopic to the identity and $\check f$ a lift of $f\vert_A$ to the universal covering space $\check A$ of $A$. Let $\kappa$ be a generator of $H_1(A,\Z)$ and $h:A\to \A$ a homeomorphism such that $h_*(\kappa)=\kappa_*$, where $\kappa_*$ is the generator of $H_1(\A,\Z)$ induced by the loop $\Gamma_* :t\mapsto (t,0)$. If $\check h:\check A\to \R^2$ is a lift of $h$, then $\check h\circ \check f\circ\textcolor{black}{\check{h}}^{-1}$ is a lift of  $h\circ f\vert_A\circ h^{-1}$, independent of the choice of $\check h$. If $z$ belongs to $\mathrm{ne}^+(f\vert_A)$, then $h(z)$ belongs to $\mathrm{ne}^+(h\circ f\vert_A\circ h^{-1})$. The existence of \textcolor{black}{$\mathrm{rot}_{\check h\circ \check f\vert_A\circ \check h^{-1}}(h(z))$} does not depend on the choice of $h$. In case such a rotation number is well defined, we will denote it $\mathrm{rot}_{\check f, \kappa}(z)$ because it does not depend on $h$. Note that $$\mathrm{rot}_{\check f, -\kappa}(z)=-\mathrm{rot}_{\check f, \kappa}(z).$$
We will also define
$$\mathrm{rot}_{f, \kappa}(z)=\mathrm{rot}_{\check f, \kappa}(z)+\Z\in \R/\Z.$$

\begin{theoremain}
\label{thmain:rotation-number}
Let $f$ be a homeomorphism of $\A$ isotopic to the identity and $\check f$ a lift of $f$ to $\R^2$. We suppose that $f$ has no topological horseshoe. Then
\begin{enumerate}
\item each point $z\in\mathrm{ne}^+(f)$  has a \textcolor{black}{well defined} rotation number $\mathrm{rot}_{\check f}(z)$;
\item for every $z,z'\in\mathrm{ne}^+(f)$ such that $z'\in\omega(z)$, we have $\mathrm{rot}_{\check f}(z')=\mathrm{rot}_{\check f}(z)$;
\item if $z\in \mathrm{ne}^+(f)\cap \mathrm{ne}^-(f)$ is non-wandering, then $\mathrm{rot}_{\check f^{-1}}(z)=-\mathrm{rot}_{\check f}(z)$;
\item the map $\mathrm{rot}_{\check f^{\pm}} : \Omega(f) \cap \mathrm{ne}(f)\to\R$ is continuous, where
$$ \mathrm{rot}_{\check f^{\pm}}(z) = \begin{cases}\mathrm{rot}_{\check f}(z) &\mathrm{if}\enskip  z\in \Omega(f) \cap  \mathrm{ne}^+(f),\\
-\mathrm{rot}_{\check f^{-1}}(z) &\mathrm{if}\enskip  z\in \Omega(f) \cap  \mathrm{ne}^-(f).
\end{cases}$$

\end{enumerate} 
\end{theoremain}

\begin{remarks*} In \cite{LeCalvezTal}, we proved the existence of the rotation number $\mathrm{rot}_{\check f}(z)$ for every bi-recurrent point $z$, assuming that $f$ is not topologically chaotic (which is a stronger hypothesis).  Similar results were previously proved by Handel \cite {Handel}. 

\end{remarks*}

\begin{example*} The map $\check f :(x,y)\mapsto (x+y, \,y\vert y\vert)$ lifts a homeomorphism $f$ of $\A$ isotopic to the identity such that
$$\mathrm{ne}^+(f)=\T^1\times[-1,1],\enskip \mathrm{ne}^-(f)=\A, \enskip \Omega(f)=\T^1\times\{-1,0,1\}.$$ Note that
$$ \mathrm{rot}_{\check f}(z) = \begin{cases} 
-1 &\mathrm{if}\enskip  z\in \T^1\times\{-1\},\\
0 &\mathrm{if}\enskip  z\in \T^1\times(-1,1),\\
1 &\mathrm{if}\enskip  z\in \textcolor{black}{\T^1\times\{1\}},\\
\end{cases}$$ and
$$ \mathrm{rot}_{\check f^{-1}}(z) =  \begin{cases} 
1 &\mathrm{if}\enskip  z\in \T^1\times(-\infty,0),\\
0 &\mathrm{if}\enskip  z\in \T^1\times\{0\},\\
-1 &\mathrm{if}\enskip  z\in \T^1\times(0,+\infty).\\
\end{cases}$$
The maps $ \mathrm{rot}_{\check f}$ and $\mathrm{rot}_{\check f^{-1}}$ are non continuous on their domains of definition and the equality $ \mathrm{rot}_{\check f}(z)=-\mathrm{rot}_{\check f^{-1}}(z)$ is true only if $z\in\Omega(f)$.
\end{example*}

The next result is a strong improvement of the assertion (3) of Theorem \ref{thmain:rotation-number}, it will be expressed in terms of {\it Birkhoff cycles} and {\it Birkhoff recurrence classes}. Let $X$ be a metric space, $f$ a homeomorphism of  $X$, and  $z_1$, $z_2$ two points of $X$. We say  that there exists a {\it Birkhoff connection} from $z_1$ to $z_2$ if for every neighborhood $W_1$ of $z_1$ and every neighborhood $W_2$ of $z_2$, there exists $n\geq 1$ such that $W_1\cap f^{-n}(W_2)\not=\emptyset$. A {\it Birkhoff cycle} is a finite sequence $(z _i)_{i\in\Z/p\Z}$  such that for every $i\in\Z/p\Z$, there exists a Birkhoff connection from $z_i$ to $z_{i+1}$. A point $z$ is said to be {\it Birkhoff recurrent} for $f$ if there exists a Birkhoff cycle containing $z$. Note that a point $z$ is non-wandering if and only if it defines a Birkhoff cycle \textcolor{black}{with $p=1$}, so every non wandering point is Birkhoff recurrent. One can define an equivalence relation in the set of Birkhoff recurrent points. Say that $z_1$ is {\it Birkhoff equivalent} to $z_2$ if there exists a Birkhoff cycle containing both points. The equivalence classes will be called Birkhoff recurrence classes. Note that every $\omega$-limit set or $\alpha$-limit set is contained in a single Birkhoff recurrence class. In particular, every transitive set, which means every set $\Lambda$ such that there exists $z\in\Lambda$ satisfying $\omega(z)=\Lambda$,  is contained in a single Birkhoff recurrence class.  On the other hand, it is easily seen that any Birkhoff recurrent class is contained in a chain recurrent class, but the converse does not need to hold, for instance when $f$ is the identity.

 If $f$ is a homeomorphism of $\A$ isotopic to the identity, we denote $f_{\mathrm{sp}}$ the continuous extension of $f$ to the sphere obtained by adding the two ends $N$ and $S$ of $\A$.

\begin{theoremain}\label{thmain:birkhoffcycles}
Let $f$ be a homeomorphism of $\A$ isotopic to the identity and $\check f$ a lift of $f$ to $\R^2$.  We suppose that $f$ has no topological horseshoe. If $\mathcal B$ is a Birkhoff recurrence class of $f_{\mathrm{sp}}$, there exists   $\rho\in\R$ such that,
$$\begin{cases}  \mathrm{rot}_{\check f}(z) = \rho&\mathrm{if}\enskip  z\in {\mathcal B}\cap\mathrm{ne}^+(f),\\
 \mathrm{rot}_{\check f^{-1}}(z) = -\rho&\mathrm{if}\enskip  z\in {\mathcal B}\cap\mathrm{ne}^-(f).
\end{cases}$$
\end{theoremain}

We immediately deduce:

 \begin{corollarymain}\label{corollarymain:birkhoffclasses}
Let $f$ be a homeomorphism of $\A$ isotopic to the identity and $\check f$ a lift of $f$ to $\R^2$. We suppose that $f$ has no topological horseshoe. Then for every Birkhoff cycle $(z _i)_{i\in\Z/p\Z}$ of $f$ in $\mathrm{ne}(f)$, there exists $\rho\in\R$ such that, for every $i\in\Z/p\Z$:
$$\begin{cases}  \mathrm{rot}_{\check f}(z_i) = \rho&\mathrm{if}\enskip  z_i\in \mathrm{ne}^+(f),\\
 \mathrm{rot}_{\check f^{-1}}(z_i) = -\rho&\mathrm{if}\enskip  z_i\in \mathrm{ne}^-(f).
\end{cases}$$
\end{corollarymain}
If $f$ is a homeomorphism of $\A$ isotopic to the identity and $\check f$ a lift of $f$ to $\R^2$, we will say that $\rho\in\R$ is a rotation number of $\check f$ if either there exists $z\in \mathrm{ne}^+(f)$ such that $\mathrm{rot}_{\check f}(z) =\rho$ or if there exists $z\in \mathrm{ne}^-(f)$ such that $\mathrm{rot}_{\check f^{-1}}(z) =-\rho$. We can prove another result.

\begin{propositionmain}\label{propositionmain_regionofinstability} Let $f$ be a homeomorphism of $\A$ isotopic to the identity, $\check f$ a lift of $f$ to $\R^2$. We suppose that:
\begin{enumerate}
\item $\check f$ has at least two rotations numbers;
\item $N$ and $S$ belong to the same Birkhoff recurrence class of $f_{\mathrm{sp}}$ .
\end{enumerate}
\noindent Then $f$ has a topological horseshoe.
\end{propositionmain} 

\begin{remark*} It was already known by Birkhoff \cite{Birkhoff} that the hypothesis of the theorem are satisfied for an area preserving positive twist map of the annulus, restricted to a Birkhoff region of instability. The positiveness of the entropy was already stated in this case (see  Boyland-Hall \cite{BoylandHall}  and Boyland \cite{Boyland} for a topological proof, see Angenent \cite{Angenent} for a more combinatorial proof). The previous corollary extends the result for situations including a generalized region of instability: the area preserving hypothesis can be replaced by the existence of a cycle of Birkhoff connections containing the two ends, the twist condition by the existence of two different rotation numbers. 
\end{remark*}









The third application is a new proof of a recent result of Passegi, Potrie and Sambarino \cite{PassegiPotrieSambarino}. Say that a homeomorphism $f$ of $\A$ is {\it dissipative} if for every non empty open set $U$ of finite area, the area of $f(U)$ is smaller than the area of $U$. Say that an invariant compact set $X$ is {\it locally stable} if $X$ admits a fundamental system of forward invariant neighborhoods. Say that a compact set $X\subset\A$ is a {\it circloid} if 
\begin{itemize}
\item it is the intersection of a nested sequence of sets homeomorphic to $\T^1\times[0,1]$ and homotopic to $\T^1\times\{0\}$; 
\item it is minimal for the inclusion among sets satisfying this property.
\end{itemize}
Its complement has two connected components that are non relatively compact. A classical example is a {\it cofrontier}\,: $X$ is a compact set whose complement has exactly two connected components, they are not relatively compact and $X$ is their common boundary.

\begin{theoremain}\label{thmain:circloids}
 Let $f$ be a homeomorphism of $\A$ isotopic to the identity and $\check f$ a lift of $f$ to $\R^2$. We suppose that $f$ has no topological horseshoe. Let $X$ be an invariant circloid. If $f$ is dissipative or if $X$ is locally stable, then the function $\mathrm{rot}_{\widetilde f}$, which is well defined on $X$, is constant on this set. More precisely, the
sequence of maps
$$\check z\mapsto \frac{ \pi_1(\check f^{n}(\check z))- \pi_1(\check z)}{n}$$ 
converges uniformly to this constant on $\check\pi^{-1}(X)$.
\end{theoremain}

\begin{remark*} { In our proof we need either dissipativeness or local stability.} Nevertheless,  is not clear that they are necessary to get the conclusion. So, it is natural to ask whether the conclusion holds without supposing $f$ dissipative or $X$ locally stable.
\end{remark*}

We can conclude this description of results related to \textcolor{black}{rotation numbers} with the following result which has its own interest:

\begin{propositionmain}
\label{prop:zero-rotation}Let $f$ be a homeomorphism isotopic to the identity of $\A$ and $\check f$ a lift of $f$ to $\R^2$.  We suppose that $\check f$ is fixed point free and that there exists  $z_*\in\Omega(f)\cap \mathrm{ne}^+(f)$ such that  $\mathrm{rot}_{\check f}(z_*)$ is well defined and equal to $0$.

\begin{itemize}
\item Then, at least one of the following situation occurs:
\begin{enumerate}
\item  there exists $q$ arbitrarily large such that $\check f^q\circ T^{-1}$ has a fixed point;  
\item there exists $q$ arbitrarily large such that $\check f^q\circ T$ has a fixed point. \end{enumerate}
\item Moreover, if $z_*$ is positively recurrent then $f$ has a topological horseshoe.
\end{itemize}
\end{propositionmain}

\medskip

An important section of this article is devoted to a structural study of the homeomorphisms of the $2$-sphere $\S^2$ with no horseshoe. For instance we will  get a very precise description of the transitive
sets (see Proposition K). Let us state the fundamental result.
\begin{theoremain}\label{thmain:global-structure}
Let $f:\S^2\to\S^2$ be an orientation preserving homeomorphism that has no topological horseshoe. Then, the set
$$\Omega'(f)=\{z\in\Omega(f)\,\vert\, \alpha(z)\cup\omega(z)\not\subset\mathrm{fix}(f)\}.$$
can be covered by a family $(A_{\alpha})_{\alpha\in\mathcal A}$ of invariant sets such that:

\begin{enumerate}
\item $A_{\alpha}$ is a fixed point free topological open annulus and $f\vert_{A_{\alpha}}$ is isotopic to the identity;
\item if $\kappa$ is a generator of $H_1(A_{\alpha},\Z)$, there exists a lift of $f\vert_{A_{\alpha}}$ to the universal covering space of $A_{\alpha}$ whose rotation numbers are all included in $[0,1]$;
\item $A_{\alpha}$ is maximal for the two previous properties. 
\end {enumerate}
\end{theoremain}

 Such a result was already proved by Franks and Handel in the case of smooth area preserving diffeomorphisms. In that case, the first condition implies the second one. Moreover, the family 
$(A_{\alpha})_{\alpha\in\mathcal A}$ is explicitly defined in terms of {\it disk recurrent points} (see \cite{FranksHandel}) and the annuli are pairwise disjoint. This result is the main building block in a structure theorem for area preserving diffeomorphisms given by the two authors (one should apply the previous result to all iterates of $f$). 
We extended Franks-Handel results in the case of homeomorphisms with no wandering points in \cite{LeCalvezTal}. In the present article, there is an attempt to give a similar structure theorem with no assumption about wandering points. It appears that some of the results remain (usually with proofs technically tougher) and some not (for example, the annuli appearing in the previous theorem are not necessarily disjoint). There is probably some progress to do to obtain a complete description of homeomorphisms with no topological horseshoes. Nevertheless we were able to state some results on the structure of Birkhoff recurrence classes and of transitive sets.
{A property of Birkhoff recurrence classes, related to rotation numbers, has already been stated in Theorem \ref{thmain:birkhoffcycles}. Let us continue to describe the behavior of Birkhoff recurrence classes. Say a set $X\subset \mathrm{fix}(f)$ is {\it unlinked} if the restriction of $f$ to $\S^2\setminus X$ is still homotopic to the identity. Let us also say a homeomorphism $f$ of $\S^2$ is an {\it irrational pseudo-rotation} if:
\begin{itemize}
\item  $f$ has exactly two periodic points, $z_0$ and $z_1$,  which are both fixed;
\item $\mathrm{ne}(f\vert _{\S^2\setminus\{z_0, z_1\}})$ is not empty;
\item if $\check f$ is a lift of $f\vert _{\S^2\setminus\{z_0, z_1\}}$ to the universal covering space of $\S^2\setminus\{z_0, z_1\}$, its unique rotation number is an irrational number $\rho$.
\end{itemize}  

\begin{propositionmain}\label{propmain:birkoffclasses3fixedpoints}
Suppose that $f$ is an orientation preserving homeomorphism of $\S^2$ with no topological horseshoe. \textcolor{black}{Assume $f$ has at least three distinct recurrent points.} If $\mathcal{B}$ is a Birkhoff recurrence class containing two fixed points $z_0$ and $z_1$, then 
\begin{itemize}
\item[(1)] either there exists $q\geq1$ such that \textcolor{black}{every recurrent point in $\mathcal B$ is a periodic point of period $q$ and} $\mathrm{fix}(f^q)\cap\mathcal{B}$ is an unlinked set of $f^q$;
\item[(2)] or $f$ is an irrational pseudo-rotation.
\end{itemize}
\end{propositionmain}

\begin{corollarymain}\label{crmain:periodsofbirkhoffclasses} Let $f$ be an orientation preserving homeomorphism of $\S^2$ with  no topological horseshoe. Let $\mathcal{B}$ be a Birkhoff recurrence class containing periodic points of different periods, then:
\begin{itemize}
\item there exist integers $q_1$ and $q_2$, with $q_1$ dividing $q_2$, such that every periodic point in $\mathcal{B}$ has a period either equal to $q_1$ or to $q_2$
\item if $q_1\geq 2$, there exists a unique periodic orbit of period $q_1$ in $\mathcal B$;
\item if $q_1=1$, there exist at most two fixed points in $\mathcal B$.
\end{itemize}
\end{corollarymain}
As an illustration, note that if the restriction of $f$ to a transitive set is topologically chaotic, then $f$ has a topological horseshoe  \textcolor{black}{because there are infinitely many periodic points in the set.} 

Let us now enumerate the possible dynamics of transitive sets for homeomorphisms with no topological horseshoes.
The following corollary shows that, if there is no topological horseshoe, transitive sets almost always must contain at most a single periodic orbit.
\begin{corollarymain}\label{crmain:transitivewithperiodicorbit}
Let $f$ be an orientation preserving homeomorphism of $\S^2$ with no topological horseshoe. Assume that there exists a transitive set containing two distinct periodic orbits. Then $f$ is an irrational pseudo-rotation.
\end{corollarymain}

Before providing a more precise description of the dynamics of transitive sets for homeomorphisms with no topological horseshoes, let us say, following \cite{BGLT}, that $f:\S^2\to\S^2$ is {\it topologically infinitely renormalizable} over a closed, nonempty invariant set $\Lambda$, if there exists an increasing sequence $(q_n)_{n\geq 0}$ of positive integers and a decreasing sequence  $(D_n)_{n\geq 0}$ of open disks such that:
\begin{itemize}
\item $q_n$ divides $q_{n+1}$; 
\item $D_n$ is $f^{q_n}$ periodic;
\item the disks $f^k(D_n)$, $0\leq k<q_n$ are pairwise disjoint;
\item $\Lambda\subset \bigcup_{0\leq k<q_n}f^{k}(D_n)$.
\end{itemize}
We note that, if $f$ is topologically infinitely renormalizable over $\Lambda$, then the restriction of $f$ to $\Lambda$ is semi-conjugated to an odometer transformation and one can show, by standard Brouwer Theory arguments, that $f$ has periodic points of period $q_n$ for each $n\in \N$. 

\begin{propositionmain}\label{prmain:transitivesetsgeneralcase} Let $f:\S^2\to\S^2$ be an orientation preserving homeomorphism with no horseshoe and $\Lambda$ a  transitive set. Then:
\begin{enumerate}
\item either $\Lambda$ is a periodic orbit;
\item or $f$ is topologically infinitely renormalizable over $\Lambda$;
\item or $\Lambda$ is  of irrational type.
\end{enumerate}
\end{propositionmain} 
Let us precise the last case. It means that $\Lambda$ is the disjoint union of a finite number of closed sets $\{\Lambda_1, ..., \Lambda_q\}$ that are cyclically permuted by $f$, and that there exists an irrational number $\rho\in\R$ such that, \textcolor{black}{for all $n\geq 1$,} if $z_0$, $z_1$ are different fixed points of $f^{qn}$ \textcolor{black}{(possibly in $\Lambda$)}, there exists a lift $\check g$ of $f^{qn}\vert_{\S^2\setminus\{z_0,z_1\}}$ and a generator $\kappa$ of $H_1(\S^2\setminus\{z_0,z_1\},\Z)$ such that: 
\begin{itemize}
\item either the unique rotation number of \textcolor{black}{$\Lambda{_1}\setminus\{z_0, z_1\}$} for $\check g$ is {$n\rho$;}
\item or the unique rotation number of \textcolor{black}{$\Lambda{_1}\setminus\{z_0, z_1\}$} for $\check g$ is \textcolor{black}{$0$;}
\end{itemize}
{furthermore, there exist some $n\geq1$ and some pair of different fixed points of $f^{qn}$ such that the first situation holds.}

Among classical examples of transitive sets of irrational type one can mention: periodic closed curves with no periodic orbits, Aubry-Mather sets, the sphere itself for a transitive irrational pseudo-rotation (\textcolor{black}{for instance, Besicovitch's example in \cite{Besicovitch}}), and Perez-Marco hedgehogs \textcolor{black}{(see \cite{Perez-Marco})}. 

Finally, we can improve the previous proposition when considering $\mathcal{C}^{1}$ dissipative diffeomorphisms, a class that has been extensively studied in the context of the H\'enon family and the transition to positive entropy. 

\begin{propositionmain}\label{prmain:dissipativediffeomorfisms}
Let $f:\R^2\to\R^2$ be an orientation preserving $\textcolor{black}{\mathcal{C}^1}$ diffeomorphism with no horseshoe, such that $\vert \textcolor{black}{\mathrm{det}}Df(x)\vert<1$ for all $x\in\R^2$. Let $\Lambda$ be a compact transitive set that is locally stable. Then either $\Lambda$ is a periodic orbit, or $f$ is topologically infinitely renormalizable over $\Lambda$.
 \end{propositionmain}

\subsection{Forcing results}
Our two fundamental results deal with mathematical objects (maximal isotopies, transverse foliations, transverse trajectories, $\mathcal F$-transverse intersection) that will be reminded in the next section. We will detail the theorems in the same section. The first result is a realization theorem improving results stated in \cite{LeCalvezTal}. The second theorem is the key point in the proofs of all statements above. 

\begin{theoremain}\label{th: realization}
Let $M$ be an oriented surface, $f$ a homeomorphism of $M$ isotopic to the identity, $I$ a maximal isotopy of $f$ and $\mathcal F$ a foliation transverse to $I$. Suppose that  $\gamma:[a,b]\to \mathrm{dom}(I)$ is an admissible path of order $r$ with a $\mathcal{F}$-transverse self intersection at $\gamma(s)=\gamma(t)$, where $s<t$. Let $\widetilde \gamma$ be a lift of $\gamma$ to the universal covering space $\widetilde{\mathrm{dom}}(I)$ of $\mathrm{dom}(I)$ and $T$ the covering automorphism such that $\widetilde\gamma$ and $T(\widetilde\gamma)$ have a $\widetilde{\mathcal F}$-transverse intersection at  $\widetilde\gamma(t)=T(\widetilde\gamma)(s)$. Let $\widetilde f$ be the lift of $f\vert_{\mathrm{dom}(I)}$ to $\widetilde{\mathrm{dom}}(I)$, time-one map of the identity isotopy that lifts $I$, and $\widehat f$ the homeomorphism of the annular covering space $ \widehat{\mathrm{dom}}(I)=\widetilde {\mathrm{dom}}(I)/T$ lifted by $\widetilde f$. Then we have the following:
 \begin{enumerate}
 \item For every rational number $p/q\in(0,1]$, written in an irreducible way, there exists a point $\widetilde z\in \widetilde{\mathrm{dom}}(I)$ such that $\widetilde f^{qr}(\widetilde z)=T^p(\widetilde z)$ and such that $\widetilde I_{\widetilde{\mathcal F}}^{\Z}(\widetilde z)$ is equivalent to $\prod_{k\in\Z} T^k(\widetilde\gamma_{[s,t]})$.
  \item For every irrational number $\rho\in[0,1/r]$, there exists a compact set $\widehat Z_{\rho}\subset\widehat{\mathrm{dom}}(I)$ invariant by $\widehat f$, such that every point $\widehat z\in \widehat Z_{\rho}$ has a rotation number $\mathrm{rot}_{\widetilde f}(\widehat z)$ equal to $\rho$. Moreover if $\widetilde z\in\widetilde{\mathrm{dom}}(I)$ is a lift of $\widehat z$, then $\widetilde I_{\widetilde{\mathcal F}}^{\Z}(\widetilde z)$ is equivalent to $\prod_{k\in\Z} T^k(\widetilde\gamma_{[s,t]})$.
 \end{enumerate}
\end{theoremain}

\begin{theoremain} \label{th: horseshoe}Let $M$ be an oriented surface, $f$ a homeomorphism of $M$ isotopic to the identity, $I$ a maximal  isotopy of $f$ and $\mathcal F$ a foliation transverse to $I$. If there exists a point $z$ in the domain of $I$ and an integer $q\geq 1$ such that the transverse trajectory $I_{\mathcal F}^{q}(z)$ has a $\mathcal F$-transverse self-intersection, then $f$ has a topological horseshoe. Moreover, the entropy of $f$ is  at least equal to $\log 4/3q$.
\end{theoremain}

\begin{remark*} In \cite{LeCalvezTal} we stated a weaker version of this result. Supposing that $z$ was a recurrent point of $f$ and $f^{-1}$, we proved that $f$ was topologically chaotic, but we did not show the existence of a horseshoe.
The proof there was obtained through the coding of orbits by infinite sequences on a finite alphabet and applying a "forcing lemma" but it relied on some artificial preliminary results. The proof that is given here is much more direct and natural. It uses the orbit of a Brouwer line for a fixed point free $G$-equivariant homeomorphim of the plane, the group $G$ acting freely and properly. In addition, the  arguments are more geometrical; we construct a topological  horseshoe related to the Conley index theory. The conclusion is much stronger and the assumptions are weaker, no recurrence hypothesis is needed. The fact that the recurrence is no longer required is what allow us to get the applications announced in this introduction.
\end{remark*}

\subsection{Notations and organization of the paper}
We will use the following notations:
\begin{itemize}
\item If $Y$ is a subset of a topological space $X$, we will write $\overline Y$, $\mathrm{Int}(Y)$, $\mathrm{Fr}(Y)$ for its closure, its interior and its frontier, respectively.
 
\item We will denote $\mathrm{fix}(f)$ the set of fixed points of a homeomorphism defined on a surface $M$. 

\item A set $X\subset M$ will be called $f$-free if $f(X)\cap X=\emptyset$. 

\item A line of $M$ is a proper topological embedding  $\lambda:\R\to M$ (or the image of this embedding). If the complement of $\lambda$ has two connected components, we will denote $R(\lambda)$ the component lying on the right of $\lambda$ and $L(\lambda)$ the component lying on its left. 

\item If $\mathcal F$ is a topological foliation defined on $M$, we will denote $\phi_z$ the leaf passing through a given point $z\in M$.
\end{itemize}

The paper is organized as follows: In Section 2 we recall the basic features of Brouwer equivariant theory and of the forcing theory, as well as some basic results on rotation numbers and Birkhoff recurrence classes. In Section 3 we construct the topological horseshoes and prove Theorems \ref{th: realization} and \ref{th: horseshoe}. Section 4 describes properties of transverse trajectories with no transverse intersection in the sphere, Section 5 is dedicated to proving the applications concerning rotation numbers, including Theorems \ref{thmain:rotation-number}, \ref{thmain:birkhoffcycles} and \ref{thmain:circloids}, as well as Propositions \ref{propositionmain_regionofinstability} and \ref{prop:zero-rotation}. In Section 6 we derive our structure theorem for homeomorphisms of the sphere with no topological horseshoes and prove Theorem \ref{thmain:global-structure}. Lastly, Section 7 deals with the final applications for Birkhoff recurrence classes and transitive sets, proving Propositions \ref{propmain:birkoffclasses3fixedpoints}, \ref{prmain:transitivesetsgeneralcase} and \ref{prmain:dissipativediffeomorfisms} and related results. 

We would like to thank A. Koropecki for several discussions and suggestions regarding this work. \textcolor{black}{We are also very grateful for the referees'  impressive work. The suggestions greatly improved the presentation of this work.}

\section{Preliminaries}

\subsection{Maximal isotopies, transverse foliations and transverse trajectories}

\subsubsection{Maximal isotopies} Let $M$ be an oriented surface (not necessarily closed, not necessarily of finite type) and $f$ a homeomorphism of $M$ that is isotopic to the identity. We denote by $\mathcal I$ the space of {\it identity isotopies of $f$} which means the set of continuous paths defined on $[0,1]$ that join the identity to $f$ in the space of homeomorphisms of $M$, furnished with the $C^0$ topology (defined by the uniform convergence of maps and their inverse on compact sets). If $I=(f_t)_{t\in[0,1]}$ is such an isotopy, we define the {\it trajectory} $I(z)$ of a point $z\in M$ to be the path $t\mapsto f_t(z)$.  For every integer $n\geq 1$, we can define by concatenation 
$$I^n(z)=\prod_{0\leq k<n} I(f^k(z)).$$
Furthermore we can define
$$ I^{\N}(z)=\prod_{k\geq0} I(f^k(z)),
 \enskip I^{-\N}(z)=\prod_{k<0} I(f^k(z)),
\enskip I^{\Z}(z)=\prod_{k\in\Z} I(f^k(z)),$$
the last path being called the {\it whole trajectory} of $z$.

The {\it fixed point set of $I$} is the set $\mathrm{fix}(I)=\bigcap_{t\in[0,1]} \mathrm{fix}(f_t)$
and the {\it domain  of $I$} its complement, that we will denote $\mathrm{dom}(I)$.
We have a preorder on $\mathcal I$ defined as follows: say that $I\preceq I'$ if
\begin{itemize}
\item $\mathrm{fix}(I)\subset\mathrm{fix}(I')$;
\item $I'$ is homotopic to $I$ relative to $\mathrm{fix}(I)$.
\end{itemize}

We have the recent following result, due to B\'eguin-Crovisier-Le Roux \cite{BeguinCrovisierLeRoux2}: 
\begin{theorem}For every $I\in\mathcal I$, there exists $I'\in\mathcal I$ such that $I\preceq I'$ and such that $I'$ is maximal for the preorder.
\end{theorem}

\begin{remarks*} {}

 A weaker version of the theorem was previously proved by Jaulent \cite{Jaulent} and stated in terms of singular isotopies. While sufficient for the applications we are looking for, it is less easy to handle. Note that we used Jaulent's formalism in \cite{LeCalvezTal}.

An isotopy $I$ is maximal if and only if, for every $z\in\mathrm{fix}(f)\setminus\mathrm{fix}(I)$, the closed curve $I(z)$ is not contractible in $\mathrm{dom}(I)$ (see \cite{Jaulent}). Equivalently, if we lift the isotopy $I\vert_{\mathrm{dom}(I)}$ to an identity isotopy $\widetilde I=(\widetilde f_t)_{t\in[0,1]}$ on the universal covering space $\widetilde{\mathrm{dom}}(I)$ of $\mathrm{dom}(I)$, the maximality of $I$ means that $\widetilde f_1$ is fixed point free. Note that every connected component of $\widetilde{\mathrm{dom}}(I)$ must be a topological plane. 

\end{remarks*}

\subsubsection{Transverse foliation}

We keep the same notations as above. We have the following result (see \cite{LeCalvez1}):

\begin{theorem}
\label{th.transversefoliation}If $I\in\mathcal I$ is maximal, there exists a topological oriented singular foliation $\mathcal F$ on $M$ such that 

\begin{itemize}
\item the singular set $\mathrm{sing}(\mathcal F)$ coincides with $\mathrm{fix}(I)$;
\item for every $z\in \mathrm{dom}(I)$, the trajectory $I(z)$ is homotopic in  $\mathrm{dom}(I)$, relative to the ends, to a path $\gamma$ {\rm positively transverse \footnote{in the whole text, ``transverse'' will mean ``positively transverse''}} to $\mathcal F$, which means locally crossing each leaf from the right to the left.
\end{itemize}
\end{theorem}

We will say that $\mathcal F$ is {\it transverse to $I$}. It can be lifted to a non singular foliation $\widetilde{\mathcal F}$ on $\widetilde{\mathrm{dom}}(I)$ which is transverse to $\widetilde I$. This last property is equivalent to saying that every leaf $\widetilde\phi$ of $\widetilde{\mathcal F}$, restricted to the connected component of $\widetilde{\mathrm{dom}}(I)$ that contains it, is a {\it Brouwer line} of $\widetilde f_1$: its image is contained in the connected component $L(\widetilde\phi)$ of the complement of $\widetilde\phi$ lying on the left of $\widetilde\phi$ and its inverse image in the connected component $R(\widetilde\phi)$ lying on the right (see \cite{LeCalvez1}). The path $\gamma$ is not uniquely defined. When lifted to $\widetilde{\mathrm{dom}}(I)$, one gets a path $\widetilde\gamma$ from a lift $\widetilde z$ of $z$ to $\widetilde f_1(\widetilde z)$:

\begin{itemize}
\item that meets
once every leaf that is on the left of the leaf $\phi_{\widetilde z}$ containing $\widetilde z$ and on the right of the leaf $\phi_{\widetilde f(\widetilde z)}$ containing $\widetilde f(\widetilde z)$, 
\item that does not meet any other leaf.
\end{itemize}

 Every other path $\widetilde\gamma'$ satisfying these properties projects onto another possible choice $\gamma'$. We will say that two positively transverse paths $\gamma$ and $\gamma'$ are {\it equivalent} if they can be lifted to $\widetilde{\mathrm{dom}}(I)$ into paths that meet exactly the same leaves. We will write $\gamma=I_{\mathcal F}(z)$ and call this path the {\it transverse trajectory of $z$}  (it is defined up to equivalence). For every integer $n\geq 1$, we can define by concatenation 
$$I_{\mathcal F}^n(z)=\prod_{0\leq k<n} I_{\mathcal F}(f^k(z)).$$
Furthermore we can define
$$ I_{\mathcal F}^{\N}(z)=\prod_{k\geq0} I_{\mathcal F}(f^k(z)),
 \enskip I_{\mathcal F}^{-\N}(z)=\prod_{k<0} I_{\mathcal F}(f^k(z)),
\enskip I_{\mathcal F}^{\Z}(z)=\prod_{k\in\Z} I_{\mathcal F}(f^k(z)),$$
the last path being called the {\it whole transverse trajectory} of $z$.

The following proposition, that will be useful in the article,  is easy to prove (see \cite{LeCalvezTal}):

\begin{proposition}
\label{prop:stability} We have the following:
\begin{itemize}
\item for every $z\in\mathrm{dom}(I)$ and every $n\geq 1$, there exists a neighborhood $U$ of $z$ such that $I_{\mathcal F}^n(z)$ is a subpath (up to equivalence) of $I_{\mathcal F}^{n+2}(f^{-1}(z'))$, if $z'\in U$;

\item for every $z\in\mathrm{dom}(I)$, every $z'\in\omega(z)\cap \mathrm{dom}(I) $, every $m\geq 0$ and every $n\geq 1$, $I_{\mathcal F}^n(z')$ is a subpath (up to equivalence) of $I_{\mathcal F}^{\N}(f^m(z))$.
\end{itemize}
\end{proposition}

We will say that a path $\gamma:[a,b]\to {\mathrm{dom}}(I)$, positively transverse to $\mathcal F$, is {\it admissible of order $n$} if it is equivalent to a path $I_{\mathcal{F}}^n
 (z)$, $z\in\mathrm{dom}(I)$. It means that if $\widetilde \gamma:[a,b]\to \widetilde {\mathrm{dom}}(I)$ is a lift of $\gamma$,  there exists a point $\widetilde z\in\widetilde{\mathrm{dom}}(I)$ such that $\widetilde z\in\phi_{\widetilde \gamma (a)}$ and $\widetilde f^n(\widetilde z)\in\phi_{\widetilde \gamma (b)}$, or equivalently, that 
 $\widetilde f^n(\phi_{\widetilde \gamma (a)})\cap \phi_{\widetilde \gamma(b)}\not=\emptyset$.

\subsection {$\mathcal F$-transverse intersection}

Let us begin with a definition. Given three disjoint lines $\lambda_0,\lambda_1, \lambda_2:\R\to \widetilde{\mathrm{dom}}(I)$ contained in the same connected component of $\widetilde{\mathrm{dom}}(I)$, we will say that that $\lambda_2$ is {\it above $\lambda_1$ relative to $\lambda_0$} (and that $\lambda_1$ is {\it below $\lambda_2$ relative to $\lambda_0$}) if none of the lines separates the two others and if, for every pair of disjoint paths $\gamma_1, \gamma_2$ joining $z_1=\lambda_0(t_1)$ to $z_1'\in \lambda_1$ and  $z_2=\lambda_0(t_2)$ to $ z_2'\in\lambda_2$ respectively, such that the paths do not meet the lines but at their endpoints, one has that $t_1< t_2$.
 
 \begin{figure}[ht!]
\hfill
\includegraphics [height=48mm]{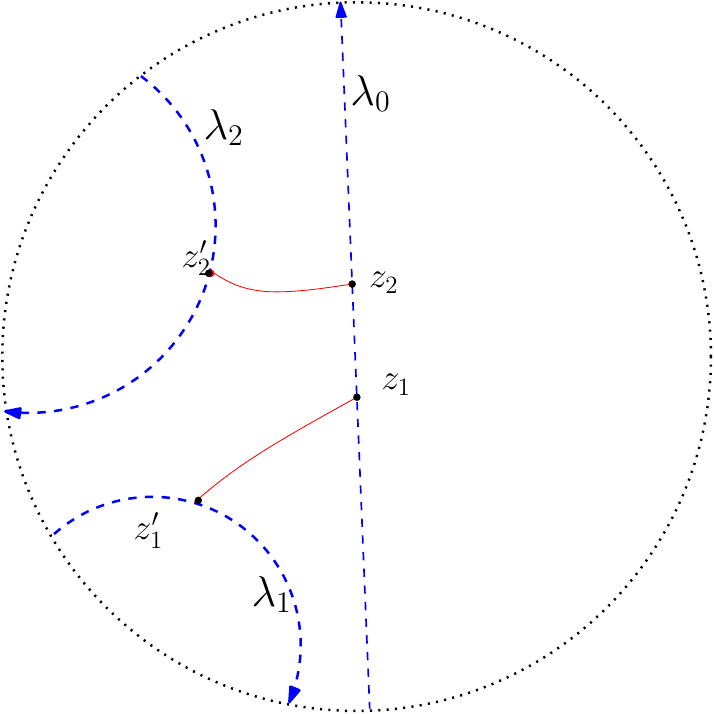}
\hfill{}
\caption{\small Order of lines relative to $\lambda_0$.}
\label{figure_order_lines}
\end{figure}
 
\medskip
Let $\gamma_1:J_1\to \mathrm{dom}(I)$ and $\gamma_2:J_2\to \mathrm{dom}(I)$ be two transverse paths defined on intervals $J_1$ and $J_2$ respectively. Suppose that there exist $t_1\in J_1$ and $t_2\in J_2$ such that $\gamma_1(t_1)=\gamma_2(t_2)$. We will say that $\gamma_1$ and $\gamma_2$ have a {\it $\mathcal F$-transverse intersection} at  $\gamma_1(t_1)=\gamma_2(t_2)$ if there exist $a_1, b_1\in J_1$ satisfying $a_1<t_1<b_1$  and $a_2, b_2\in J_2$ satisfying $a_2<t_2<b_2$ such that if $\widetilde\gamma_1,\widetilde\gamma_2$ are lifts of $\gamma_1, \gamma_2$ to $\widetilde{\mathrm{dom}}(I)$ respectively, verifying $\widetilde\gamma_1(t_1)=\widetilde\gamma_2(t_2)$, then

\begin{itemize}

\item $\phi_{\widetilde\gamma_1(a_1)}\subset L(\phi_{\widetilde\gamma_2(a_2)}),\enskip \phi_{\widetilde\gamma_2(a_2)}\subset L(\phi_{\widetilde\gamma_1(a_1)})$;

\item $\phi_{\widetilde\gamma_1(b_1)}\subset R(\phi_{\widetilde\gamma_2(b_2)}),\enskip \phi_{\widetilde\gamma_2(b_2)}\subset R(\phi_{\widetilde\gamma_1(b_1)})$

\item $\phi_{\widetilde\gamma_2(b_2)}$ is below $\phi_{\widetilde\gamma_1(b_1)}$ relative to $\phi_{\widetilde\gamma_1(t_1)}$ if $\phi_{\widetilde\gamma_2(a_2)}$ is above $\phi_{\widetilde\gamma_1(a_1)}$ relative to $\phi_{\widetilde\gamma_1(t_1)}$ and above $\phi_{\widetilde\gamma_1(b_1)}$ if $\phi_{\widetilde\gamma_2(a_2)}$ is below $\phi_{\widetilde\gamma_1(a_1)}$.
\end{itemize}

\begin{figure}[ht!]
\hfill
\includegraphics [height=48mm]{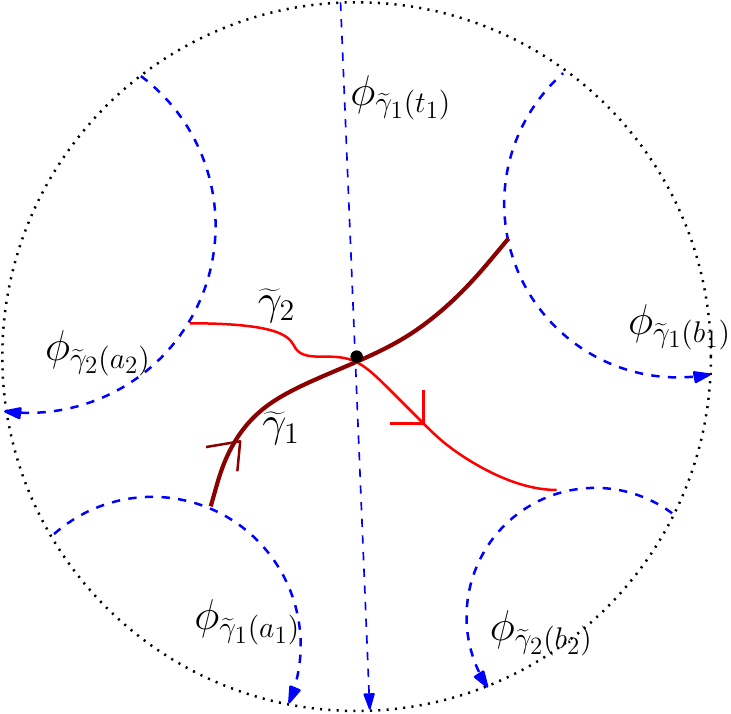}
\hfill{}
\caption{\small $\mathcal{F}$-transverse intersection. }
\label{figure1}
\end{figure}


 Roughly speaking, it means that there is a ``crossing'' in the space of leaves of $\widetilde{\mathcal F}$ (which is a one-dimensional non Hausdorff manifold). \textcolor{black}{More precisely, if $\widetilde{\beta_1}$ is a path joining $\phi_{\widetilde\gamma_1(a_1)}$ to $\phi_{\widetilde\gamma_1(b_1)}$ and $\widetilde{\beta_2}$ is a path joining $\phi_{\widetilde\gamma_2(a_2)}$ to $\phi_{\widetilde\gamma_2(b_2)}$, then they must intersect}.  If $\gamma_1=\gamma_2$ one speaks of a {\it $\mathcal F$-transverse self-intersection}. This means that if $\widetilde \gamma_1$ is a lift of $\gamma_1$, there exist a covering automorphism $T$ such that $\widetilde\gamma_1$ and $T\widetilde\gamma_1$ have a $\widetilde{\mathcal F}$-transverse intersection at  $\widetilde\gamma_1(t_1)=T\widetilde\gamma_1(t_2)$. 

The next proposition is the key result of  \cite{LeCalvezTal}. It permits to construct new admissible paths from a pair of admissible paths.

\begin{proposition}\label{pr: fundamental}
 Suppose that  $\gamma_1: [a_1,b_1]\to M$ and $\gamma_2: [a_2,b_2]\to M$ are transverse paths that intersect $\mathcal{F}$-transversally at $\gamma_1(t_1)=\gamma_2(t_2)$. If $\gamma_1$ is admissible of order $n_1$ and $\gamma_2$ is admissible of order $n_2$, then $\gamma_1\vert_{[a_1,t_1]}\gamma_2\vert_{[t_2,b_2]}$ and $\gamma_2\vert_{[a_2,t_2]}\gamma_1\vert_{[t_1,b_1]}$ are admissible of order $n_1+n_2$. Furthermore, either one of these paths is admissible of order $\min(n_1, n_2)$ or both paths are admissible of order $\max(n_1, n_2)$.
\end{proposition}

So, the meaning of Theorem \ref{th: horseshoe} is that the existence of an admissible path of order $q$ with a self intersection implies the existence of a horseshoe. Moreover it gives a lower bound $\log 4/3q$ of the entropy.

\subsection{Rotation numbers}\label{subsect:rotnumb}
As explained in the introduction, rotation numbers can be naturally defined in an abstract annulus $A$ for a homeomorphism $f$ isotopic to the identity and a given lift to the universal covering space, as soon as a generator of $H_1(A,\Z)$ is chosen. We will state two useful propositions. In this article a {\it loop} on the sphere is a continuous map $\Gamma:\T^1\to\S^2$ \textcolor{black}{and the {\it natural lift} of a  $\Gamma$ is the path $\gamma:\R\to\S^2$ such that $\gamma(t)=\Gamma(t+\Z)$. }

\textcolor{black}{First let us define the notion of  {\it dual function} and its properties. If $\Gamma$ is a loop, one can introduce a {\it function} $\delta$, defined up to an additive constant on $\S^2\setminus\Gamma$ as follows: for every $z$ and $z'$ in $\R^2\setminus\Gamma$, the difference $\delta(z')-\delta(z)$ is the algebraic intersection number $\Gamma\wedge \gamma'$ where $\gamma'$ is any path from $z$ to $z'$.  Now, suppose that $\mathcal F$ is a singular oriented foliation on the sphere $\S^2$ and that $\Gamma$ is {\it transverse} to $\mathcal F$ (which means that its natural lift $\gamma$ is transverse to $\mathcal F$). In that case $\delta$ decreases along each leaf with  a jump at every intersection point. One proves easily that $\delta$ is bounded and that the space of leaves that meet $\Gamma$, furnished with the quotient topology is a (possibly non Hausdorff) one dimensional manifold (see the end of section 3.1 in \cite{LeCalvezTal}). In particular, there exists $n_0$ such that $\Gamma$ meets each leaf at most $n_0$ times. If $\Gamma$ is a simple loop (which means that it is injective), $\delta$ takes exactly two values and $\Gamma$ meets each leaf at most once. In that case we will write $U_{\Gamma}$ for the union of leaves that meet $\Gamma$. It is an open annulus. We can define the dual function of a multi-loop $\Gamma=\sum_{1\leq i\leq p}\Gamma_i$  and have similar results in case each $\Gamma_i$ is transverse to $\mathcal F$. } 
\medskip

\begin{proposition} \label{prop: rotation numbers}
Let $f$ be an orientation preserving homeomorphism of the sphere $\S^2$ and $(z_i)_{i\in \Z/3\Z}$ three distinct fixed points. For every $i\in \Z/3\Z$, set $A_i=\S^2\setminus\{z_{i+1},z_{i+2}\}$ and \textcolor{black}{fix  a generator $\textcolor{black}{\kappa_i}$ of $H_1(A_i,\Z)$} by taking the oriented boundary of a small closed disk containing $z_{i+1}$ in its interior. Denote $f_i$ the restriction of $f$ to $A_i$ and write $\check f_i$ for the lift of $f_i$ to the universal covering space of $A_i$ that fixes the preimages of $z_i$. Let $z\in \S^2$ be a point such that the sequence $(f^n(z))_{n\geq0}$ does not converge to $z_i$, for every $i\in\Z/3\Z$. If the three rotations numbers $\mathrm {rot}_{\check f_i}(z)$, $i\in\Z/3\Z$, are well defined then $\sum_{i\in\Z/3\Z}\mathrm {rot}_{\check f_i}(z) =0$.
\end{proposition}

\begin{proof}\textcolor{black}{ For every $i\in\Z/3\Z$ we will denote $\check A_i$ the universal covering space of $A_i$ and $\check \pi_i:\check A_i\to A_i$ the covering projection.  We will denote $T_i$ the generator of the group of covering automorphisms, that is sent onto $\kappa_i$ by the Hurewicz morphism. We will fix a homeomorphism $h_i:A_i\to\A$ such that $(h_i)_*(\kappa_i)=\kappa_*$ and choose a lift $\check h_i:\check A_i\to\R^2$ of $h_i$.}  Fix an open topological disk $D$ that contains a point $z'\in\omega(z)$ and \textcolor{black}{whose closure does not contain any $z_i$}. Then fix a path {$\beta$} joining $z'$ to $z$ that does not contain any $z_i$. It is a classical fact that there exists an identity isotopy $I$ that fixes each $z_i$. Moreover $I\vert_{A_i}$ is lifted to an identity isotopy \textcolor{black}{$\check I_i$} of $\check f_i$.  Let $(n_k)_{k\geq 0}$ be an increasing sequence such that $f^{n_k}(z)\in D$, for every $k\geq 0$. Consider the loop $\Gamma_k$ defined by \textcolor{black}{$I^{n_k}(z)\alpha_k\beta$, where $\alpha_k$} is a path included in $D$ that joins $f^{n_k}(z)$ to $z'$. Fix a dual function $\delta_k$ of $\Gamma_k$. It is defined up to an additive constant, and its values on each $z_i$ are independent of the choice of  \textcolor{black}{$\alpha_k$. The integer $m_k = \delta_k(z_{i+1})- \delta_k(z_{i+2})$ is the algebraic intersection number between a given  path joining $z_{i+1}$ to $z_{i+2}$ and $\Gamma_k$. Consequently, if $\check z\in\check A_i$ is a given lift of $z$, then $\Gamma_k$ can be lifted to a path $\check \gamma_k:\R\to\check A_i$ such that $\check\gamma_k(0)=\check z$ and $\check \gamma_k(t+1)=T_i^{m_k}(\check\gamma_k(t))$ for every $t\in\R$.  One can write $\check \gamma_k{}\vert_{[0,1]}= \check I_i^{n_k}(\check z)\check\alpha_k\check\beta_k,$ where $\check \alpha_k$ is a lift of $\alpha_k$ and $ \check\beta_k$ a lift of $\beta$. Consequently, there exists $M>0$ such that for every $k$, it holds that $$\vert \pi_1(\check h_i\circ \check f_i^{n_k}(\check z))-\pi_1(\check h_i(\check z) )-m_k\vert \leq M.$$ }By definition of $\mathrm {rot}_{\check f_i}(z)$, one has
$$\mathrm {rot}_{\check f_i}(z) =\lim_{k\to+\infty} {m_k\over n_k} =\lim_{k\to+\infty}  {\delta_k(z_{i+1})- \delta_k(z_{i+2})\over n_k}.$$
The conclusion is immediate.
\end{proof}

The proof of the next result can be found in \cite{LeRoux}.

\begin{proposition} \label{prop: rotation numbers powers}
Let $f$ be a homeomorphism of the annulus $\A$ and $\check f$ a lift to the universal covering space. Fix two integer numbers $p\in\Z$ and $q\geq 1$.  Then, one has $\mathrm{ne}(f)= \mathrm{ne}(f^q)$. Moreover $\mathrm{rot}_{\check f^q\circ T^{p}}(z)$ is defined if and only if it is the case for $\mathrm{rot}_{\check f}(z)$ and one has $\mathrm{rot}_{\check f^q\circ T^{p}}(z)=q\mathrm{rot}_{\check f}(z)+p$.
\end{proposition}

\subsection{Birkhoff cycles and Birkhoff recurrence classes}

 Let $X$ be a Hausdorff topological space and $f:X\to X$ a homeomorphism. We recall from the introduction that, given points $z_1$ and $z_2$ in $X$, we say that there exists a Birkhoff connection from $z_1$ to $z_2$ if, for every neighborhood $W_1$ of $z_1$ in $X$ and every neighborhood $W_2$ of $z_2$, there exists $n>0$ such that $W_1\cap f^{-n}(W_2)$ is not empty. We will write $z_1\underset{f}\rightarrow z_2$ if there is a Birkhoff connection from $z_1$ to $z_2$ and will write $z\underset{f}\preceq z'$ if there exists a family $(z_i)_{0\leq i\leq p}$ such that
$$z=z_0\underset{f}\rightarrow z_1\underset{f}\rightarrow \dots \underset{f}\rightarrow z_{p-1}\underset{f}\rightarrow z_p=z'.$$

This same concept appears for instance in \cite{Arnaud} or in \cite{crovisier_birkhoff}, where it is said that $z_2$ is a weak iterate of $z_1$. We have defined in the introduction the concepts of Birkhoff cycles, Birkhoff recurrent points and Birkhoff recurrent classes for $f$. Note that in \cite{Arnaud} a Birkhoff cycle was named an $(\alpha,\omega)$ closed chain, and Birkhoff recurrent points were named $(\alpha, \omega)$-recurrent. Note that the set of Birkhoff recurrent points of $f$ coincides with the set of Birkhoff recurrent points of $f^{-1}$ and that the Birkhoff recurrent classes are the same for $f$ and $f^{-1}$. Let us state some other simple additional properties that will be used frequently in the remainder of the article.   
\begin{proposition} \label{prop: birkhoff connexions} Let $X$ be a Hausdorff topological space and $f$ a homeomorphism of $X$.
\begin{enumerate}
\item If there exists a Birkhoff connection from $z_1$ to $z_2$, then there exists a Birkhoff connection from $z_1$ to $f(z_2)$ and, if $z_2$ is not $f(z_1)$, then there exists a Birkhoff connection from $f(z_1)$ to $z_2$.  Furthermore, if $z_2\notin \{ f^n(z_1), n\geq1\}$ , then there exists a Birkhoff connection from any point $z\in\omega(z_1)$ to $z_2$.

\item
Let $\mathcal{B}$ be a nonempty Birkhoff recurrence class for $f$. Then $\mathcal{B}$ is invariant and, for all $z\in\mathcal{B}$ the closure of the orbit of $z$ is included in $\mathcal{B}$. 

\item If $f$ has a Birkhoff connection from $z_1$ to $z_2$, then given $q\ge 1$, there exists some $0\le j < q$ such that $f^q$ has a Birkhoff connection from $z_1$ to $f^j(z_2)$. 

\item The set of Birkhoff recurrent points of $f$ is equal to the set of Birkhoff recurrent points  of $f^q$ for every $q\geq 2$. Moreover, if $z_1$ and $z_2$ are in the same Birkhoff recurrence class of $f$, then given $q\ge 2$, there exists some $0\le j < q$ such that $z_1$ and $f^j(z_2)$ are in the same Birkhoff recurrence class of $f^q$. More precisely, if $(z_i)_{i\in \Z/r\Z}$ is a Birkhoff circle of $f$, there exist $p\geq 1$ and a Birkhoff circle $(z'_j)_{j\in \Z/pr\Z}$ of $f^q$ such that $z'_j$ is in the $f$-orbit of $z_i$ if $j\equiv i \pmod r$.

\end{enumerate}
\end{proposition}
\begin{proof} Suppose that $f$ has a Birkhoff connection from $z_1$ to $z_2$. If $W_1$ is a neighborhood of $z_1$ and $W_2$ is a neighborhood of $f(z_2)$, then $f^{-1}(W_2)$ is a neighborhood of $z_2$ and so, there exists $n>0$ such that $W_1\cap f^{-n}(f^{-1}(W_2))=W_1\cap f^{-n-1}(W_2)$ is not empty. So, there is a Birkhoff connection from $z_1$ to \textcolor{black}{$f(z_2)$}. Suppose now that $f(z_1)\not=z_2$.  If $W_1$ is a neighborhood of $f(z_1)$, $W_2$ is a neighborhood of $z_2$, and $W_1\cap W_2=\emptyset$, then there exists $n>0$ such that $f^{-1}(W_1)\cap f^{-n}(W_2)\not=\emptyset$ because $f^{-1}(W_1)$ is a neighborhood of $z_1$, and so $W_1\cap f^{-n+1}(W_2)\not=\emptyset$. Note that $n\not=1$, which implies that there is a Birkhoff connection from $f(z_1)$ to $z_2$.  Finally, suppose that $z_2\notin \{ f^n(z_1), n\geq1\}$ and that $z$ belongs to $\omega(z_1)$. Let $W_1$ be a neighborhood of $z$ and $W_2$ a neighborhood of $z_2$. There exists $n_1\geq 0$ such that $f^{n_1}(z)\in W_1$. As explained above, there  is a Birkhoff connection from $f^{n_1}(z_1)$ to $z_2$, and so, there exists $n>0$ such that $W_1\cap f^{-n}(W_2) \not=\emptyset$. The assertion (1) is proven.
 
Let us prove (2). Of course, the set of Birkhoff recurrent points is invariant by $f$. Moreover (2) is obviously true if $z$ is a periodic point. Suppose now that $z$ is a Birkhoff recurrent point that is not periodic and let us prove that $z$ and $f(z)$  are in the same Birkhoff recurrence class. There exists a sequence $(z_i)_{0\leq i\leq p}$ such that 
  $$z=z_0\underset{f}\rightarrow z_1\underset{f}\rightarrow \dots \underset{f}\rightarrow z_{p-1}\underset{f}\rightarrow z_p=z.$$ As $z$ is not periodic, there exists $i_0<p$ such that $z_{i_0+1}\not=f(z_{i_0})$. So, using (1), one deduces that
  $$z\underset{f}\rightarrow f(z)\underset{f}\rightarrow f(z_1)\underset{f}\rightarrow \dots \underset{f}\rightarrow f(z_{i_0})\underset{f}\rightarrow z_{i_0+1}\underset{f}\rightarrow \dots \underset{f}\rightarrow z_{p-1}\underset{f}\rightarrow z_p=z,$$ 
which implies that $z$ and $f(z)$ are in the same Birkhoff recurrence class.
Let us prove now that $\omega(z)$ is included in the same class by fixing $z'\in\omega(z)$. Let $i_0$ be the smallest integer such that $z_{i_0+1}$ is not in the strict positive orbit of $z_{i_0}$. Using (1), one deduces that
  $$z=z_0\underset{f}\rightarrow z_1\underset{f}\rightarrow\dots \underset{f}\rightarrow z_{i_0} \underset{f}\rightarrow z'\underset{f}\rightarrow z_{i_0+1}\underset{f}\rightarrow \dots \underset{f}\rightarrow z_{p-1}\underset{f}\rightarrow z_{p}=z,$$ 
which implies that $z$ and $z'$ are in the same Birkhoff recurrence class. We prove in a similar way that $\alpha(z)$ is included in the Birkhoff recurrence class of $z$.

 Let us prove (3) by contraposition. Suppose that for every $0\leq j<q$, there exists a neighborhood $W_1^j$ of $z_1$ and a neighborhood $W_2^j$ of $f^j(z_2)$ such that $W_1^j\cap f^{-qs}(W_2^{j})=\emptyset$ for every $s>0$. The set $W_1=\bigcap_{0\leq j<q} W_1^j$ is a neighborhood of $z_1$ and the set $W_2=\bigcap_{0\leq j<q} f^{-j}(W_2^j)$ a neighborhood of $z_2$. Note that 
 $$W_1\cap f^{-qs}(W_2)\subset W^0_1\cap f^{-qs}(W^0_2)=\emptyset,$$ for every $s>0$ and that 
 $$W_1\cap f^{-qs-r}(W_2)\subset W_1^{q-r}\cap f^{-q(s+1)} (W_2^{q-r})=\emptyset,$$ for every $s\geq 0$ and every $0<r<q$ .  Consequently, $W_1\cap f^{-n}(W_2)=\emptyset$ for every $n>0$.
 
Let us prove (4). The fact that a Birkhoff recurrent point of $f^q$ is a Birkhoff recurrent point of $f$ is obvious. Let us prove the converse. Suppose that $z$ is a Birkhoff recurrent point of $f$ and fix $q\geq 2$. If $z$ is a periodic point of $f$, of course it is a Birkhoff recurrent point of $f^q$. Suppose now that $z$ is not periodic and consider a sequence $(z_i)_{0\leq i\leq p}$ such that 
  $$z=z_0\underset{f}\rightarrow z_1\underset{f}\rightarrow \dots \underset{f}\rightarrow z_{p-1}\underset{f}\rightarrow z_p=z.$$
  Applying (3), one deduces that there exists a sequence $(s_i)_{1\leq i\leq p}$, such that $0\leq s_{i+1}-s_i<q$ if $i<p$ and 
  $$z=z_0\underset{f^q}\rightarrow f^{s_1}(z_1)\underset{f^q}\rightarrow \dots \underset{f^q}\rightarrow f^{s_{p-1}}(z_{p-1})\underset{f^q}\rightarrow f^{s_{p}}(z_p)=f^{s_{p}}(z),$$which implies that $z\underset{f^q}\preceq f^{s_{p}}(z)$, where $0\leq s_q\leq p(q-1)$.
Let $i_0<p$ be an integer such that $z_{i_0+1}$ is not in the strict positive orbit of $z_{i_0}$. Using (1), one knows that
  $$z=z_0\underset{f}\rightarrow z_1\underset{f}\rightarrow\dots \underset{f}\rightarrow z_{i_0} \underset{f}\rightarrow f^{-pq}(z_{i_0+1})\underset{f}\rightarrow \dots \underset{f}\rightarrow f^{-pq}(z_{p-1})\underset{f}\rightarrow f^{-pq}(z_{p})=f^{-pq}(z).$$ Applying (3) again as above, one deduces that there exists $0\leq s'_q\leq p(q-1)$ such that
  $z\underset{f^q}\preceq f^{s'_{p}-pq}(z)$.
  Note know that 
 \begin{equation}\label{eq:birkhoffcycleiterate}
 z\underset{f^q}\preceq f^{s_{p}}(z)\underset{f^q}\preceq f^{2s_{p}}(z)\underset{f^q}\preceq \dots \underset{f^q}\preceq f^{(pq-s'_q)s_{p}}(z)\underset{f^q}\preceq f^{(pq-s'_q)(s_{p}-1)}(z)\underset{f^q}\preceq\dots \underset{f^q}\preceq z,\tag{*}
 \end{equation}
  which implies that $z$ is a Birkhoff recurrent point of $f^q$. 
  
Suppose now that $z$ and $z'$ are two Birkhoff recurrent points of $f$ in the same Birkhoff recurrence class. We want to prove that $z$ and $f^j(z')$ are in the same Birkhoff recurrence class of $f^q$ for some $0\le j < q$. Using the fact that the Birkhoff recurrence class of $f^q$ are invariant by $f^q$, it is sufficient to prove that $z$ and $f^j(z')$ are in the same Birkhoff recurrence class of $f^q$ for some $j\in\Z$. Of course, one can suppose that $z$ and $z'$ are not on the same orbit of $f$. There exists a sequence $(z_i)_{0\leq i\leq p}$ such that 
  $$z=z_0\underset{f}\rightarrow z_1\underset{f}\rightarrow \dots \underset{f}\rightarrow z_{p-1}\underset{f}\rightarrow z_p=z$$
  and such that $z'$ is equal to a certain $z_i$. We have seen that there exists a sequence $(s_i)_{1\leq i\leq p}$, such that $0\leq s_{i+1}-s_i<q$ if $i<p$ and 
  $$z=z_0\underset{f^q}\rightarrow f^{s_1}(z_1)\underset{f^q}\rightarrow \dots \underset{f^q}\rightarrow f^{s_{p-1}}(z_{p-1})\underset{f^q}\rightarrow f^{s_{p}}(z_p)=f^{s_{p}}(z).$$ Moreover we have seen that $z$ and $f^{s_{p}}(z)$ are in the same Birkhoff recurrence class of $f^q$ (the argument above works even if $z$ is periodic). So $z$ is in the same Birkhoff recurrence class of $f^q$ that a point in the orbit of $z'$. Finally, one observes that from (\ref{eq:birkhoffcycleiterate}) one obtains a Birkhoff cycle for $f^q$ with the properties stated in the last assertion of (4).  
  \end{proof}

We can note that, in general,  the class itself does not need to be closed. For instance, let $X$ be the space obtained by identifying, in $[0,1]\times\{0,1\}$, the points $(x,0)$ and $(x,1)$ if $x\in F=\{0\}\cup\{1/n, n\geq 1\}$. Let $T_1:[0,1]\to [0,1]$ be a homeomorphism such that $T_1(x)=x$  if $x\in F$ and $T_1(x)>x$ otherwise. Let $T:X\to X$ be such that $T(x,0)=(T_1(x),0)$ and $T(x,1)= (T_1^{-1}(x), 1)$. Note that, for every $n\geq 1$, $(1/n,0)=(1/n,1)$ and $(1/(n+1),0)=(1/(n+1),1)$ belong to the same Birkhoff class. Note also that the Birkhoff class of $(0,0)$ is reduced to this point.

The set of Birkhoff recurrent points is not necessarily closed. For instance, let $(h_n)_{n\geq 1}$ be a sequence of orientation preserving homeomorphisms of $\T^1=\R/\Z$ such that $h_n$ has exactly $n$ fixed points and has a lift $\check h_n:\R\to\R$ having fixed points and satisfying $\check h_n(x)\ge x$ for all $x\in\R$. Assume furthermore that $h_n$ converges uniformly to a homeomorphism $h_0$ that is not the identity, but has infinitely many fixed points. Consider the homeomorphism $T$ of  $X=\T^1\times F$  defined as a $T(x, 1/n)= (h_n(x), 1/n)$ in case $n>0$ and $T(x,0)=(h_0(x), 0)$. Every point $(x, y)$ with $y\not=0$,  is a Birkhoff recurrent point, but 
$(x,0)$ is a Birkhoff recurrent point if and only if $x$ is fixed by $h_0$.

We will conclude this section by studying what happens, when looking at the iterates.
\begin{proposition} \label{prop: birkhoff classes iterate} Let $X$ be a Hausdorff topological space and $f$ a homeomorphism of $X$. We suppose that $\mathcal B$ is a Birkhoff recurrence class of $f$ and fix $q\geq 2$. Then, there exist $q'\in\{1,\dots, q\}$ dividing $q$ and a family of pairwise disjoint sets $({\mathcal B}'_i)_{i\in\Z/q'\Z}$ such that:
\begin{itemize}
\item ${\mathcal B}'_i$ is a Birkhoff recurrence class of $f^q$ for every $i\in\Z/q'\Z$;
\item $\mathcal B=\bigsqcup_{i\in\Z/q'\Z} {\mathcal B}'_i$;
\item $f({\mathcal B}'_i)={\mathcal B}'_{i+1}$, for every $i\in\Z/q'\Z$.
\end{itemize}
Moreover, if the closure of the $f$-orbit of $z_*\in{\mathcal B}$ is compact, each set $\overline{ O_f(z_*)}\cap \mathcal B'_i$ is compact and is the closure of the $f^{q'}$-orbit of a point of the orbit of $z_*$. In addition, there exists a family of pairwise disjoint open sets $(U_i)_{i\in\Z/q'\Z}$ such that:
\begin{itemize}
\item  $\overline{ O_f(z_*)}\cap \mathcal B'_i\subset U_i$ for every $i\in\Z/q'\Z$;
\item $f(U_i)=U_{i+1}$, for every $i\in\Z/q'\Z$.
\end{itemize}
\end{proposition}
\begin{proof} We have seen that every Birkhoff recurrent point of $f$ is a Birkhoff recurrent point of $f^q$ and that every Birkhoff recurrence class of $f$ is invariant by $f$. Fix $z\in \mathcal B$ and denote ${\mathcal B}'_k$, the Birkhoff recurrence class of $f^k(z)$ for $f^q$. Of course $f( {\mathcal B}'_k)= {\mathcal B}'_{k+1}$ for every $k\in\Z$. Moreover one has $ {\mathcal B}'_q= {\mathcal B}'_{0}$. We deduce that there exists a divisor $q'$ of $q$ such that the sequence $({\mathcal B}'_k)_{k\in\Z}$ is $q'$-periodic. We can index it with $i\in\Z/q'\Z$ to get our family $({\mathcal B}'_i)_{i\in\Z/q'\Z}$. The three first items are satisfied. 
Note also that we cannot have $z\underset{f^q}\preceq f^r(z')$ if $z$ and $z'$ belong to  ${\mathcal B}'_{0}$ and $r\not\in q'\Z$. Indeed, we would have
$$z\underset{f^q}\preceq f^r(z')\underset{f^q}\preceq f^{r}(z)\underset{f^q}\preceq f^{2r}(z')\underset{f^q}\preceq f^{2r}(z)\underset{f^q}\preceq\dots \underset{f^q}\preceq f^{q'r}(z')\underset{f^q}\preceq z,$$
contradicting the fact that $z$ and $f^r(z')$ are not in the same Birkhoff recurrence class of $f^q$. In particular there is no Birkhoff connection of $f^q$ from $z$ to $f^r(z')$ and no Birkhoff connection of $f^q$ from \textcolor{black}{$f^r(z')$ to $z$}.  One deduces that there exist an open neighborhood $V^r_{z,z'}$ of $z$ and an open neighborhood $W^r_{z,z'}$ of $z'$ such that $f^{qk}(V^r_{z,z'})\cap f^r(W^r_{z,z'})=\emptyset$, for every $k\in\Z$. Setting $$V_{z,z'}=\bigcap_{0<r<q, q'\nmid r}V^r_{z,z'}\,,\enskip W_{z,z'}=\bigcap_{0<r<q, q'\nmid r}W^r_{z,z'},$$ one gets an open neighborhood 
$V_{z,z'}$ of $z$ and an open neighborhood $W_{z,z'}$ of $z'$ such that $$k\not\in q'\Z \rightarrow f^{k}(V_{z,z'})\cap W_{z,z'}=\emptyset.$$ 
Suppose that the closure of the $f$-orbit of $z_*\in \mathcal B$ is compact. One can suppose that $z_*$ belongs to ${\mathcal B}'_0$. Write $O_f(z_*)=\bigcup_{k\in\Z} O_{f^q} (f^k(z_*))$. The family $\left(O_{f^q} (f^k(z_*))\right)_{k\in\Z}$ is $q$-periodic and $O_{f^q} (f^k(z_*))\subset {\mathcal B}'_i$, if $k+q'\Z=i$. 
As explained in Proposition \ref{prop: birkhoff connexions}, the closure of $O_{f^q} (f^k(z_*))$ is included in the Birkhoff recurrence class of $f^k(z_*)$, which means in ${\mathcal B}'_i$, if $k+q'\Z=i$.  Moreover $\overline {O_{f^q}(f^k(z_*))}$ is compact, being a closed subset of $\overline{O_f(z_*)}$. So, one gets the following decomposition in disjoint compact sets:
$$\overline{O_f(z_*)}=\bigsqcup_{i\in \Z/q'\Z} \overline{O_f(z_*)}\cap {\mathcal B}'_i=\bigsqcup_{i\in \Z/q'\Z}\left( \bigcup_{k\in\Z,\, k+q'\Z=i} \overline {O_{f^q}(f^k(z_*))}\right).$$
By compactness of $ \overline{O_f(z_*)}\cap {\mathcal B}'_0$, for every point $z\in  \overline{O_f(z_*)}\cap {\mathcal B}'_0$ one can find a finite subset $I'_z$ of $\overline{O_f(z_*)}\cap {\mathcal B}'_0$ such that $ \overline{O_f(z_*)}\cap {\mathcal B}'_0\subset\bigcup_{z'\in I'_z} W_{z,z'}$. The set $V_z=\bigcap _{z'\in I'_z} V_{z,z'}$ is an open neighborhood of $z$, the set $W_z=\bigcup _{z'\in I'_z} W_{z,z'}$ an open neighborhood of $\overline{O_f(z_*)}\cap {\mathcal B}'_0$ and one knows that  
 $$k\not\in q'\Z \rightarrow f^{k}(V_{z})\cap W_{z}=\emptyset.$$
 Using again the compactness of $ \overline{O_f(z_*)}\cap {\mathcal B}'_0$, one can find a finite subset $I$ of $\overline{O_f(z_*)}\cap {\mathcal B}'_0$ such that $ \overline{O_f(z_*)}\cap {\mathcal B}'_i\subset\bigcup_{z\in I} V_{z}$. The sets $V=\bigcup _{z\in I} V_{z}$ and $W=\bigcap _{z\in I} W_{z}$ are open neighborhoods of $\overline{O_f(z_*)}\cap {\mathcal B}'_0$ and one knows that  
 $$k\not\in q'\Z \rightarrow f^{k}(V)\cap W=\emptyset.$$ 
For every $i\in\Z/q'\Z$, define $$U_i=\bigcup_{k\in\Z, \, k+q'\Z=i} f^k(V\cap W).$$The family $(U_i)_{i\in\Z/q'\Z}$ satisfies the two last items.
\end{proof}

\begin{remarks*} Suppose that the \textcolor{black}{closure of the} orbit of $z_*\in \mathcal B$ is compact. Set $U=\bigcup_{i\in\Z/q'\Z} U_i$. The map $f$ acts naturally as a permutation on the set of connected components of $U$ that meets $\overline {O_f(z_*)}$ and this action is transitive. Moreover this set is finite because $\overline {O_f(z_*)}$ is compact. So, this permutation has a unique cycle whose length is a multiple of $q'$ because the $U_i$ are pairwise disjoint and $f(U_i)=U_{i+1}$. In other terms, there exists an integer $r\geq 1$ and a family $(U'_j)_{j\in \Z/ rq'\Z}$ of distinct connected components of $U$ that cover  $\overline {O_f(z_*)}$ and satisfy $f(U'_j)=U'_{j+1}$.

There are two kinds of Birkhoff recurrent points: the points such that there exists an integer $M$ that bounds, for every $q$, the integer $q'$ appearing in the proof of Proposition \ref{prop: birkhoff classes iterate} applied to the Birkhoff recurrence class of $z$; the points where such an integer $M$ does not exist. For points of the latter type, $f$ is ``abstractly'' topologically infinitely renormalizable. This dichotomy will appear in the final section of our article. 

\end{remarks*}

\section{Construction of horseshoes}

 We suppose \textcolor{black}{from this point on, and until the beginning of Section \ref{subsec:TheoMandN},}   that  
\begin{itemize}
\item$M$ is an oriented surface;
\item $f$ is a homeomorphism of $M$ isotopic to the identity;
\item $I=(f_t)_{t\in[0,1]}$ is a maximal isotopy of $f$;
\item $\mathcal F$ is a foliation transverse to $I$;
\item $\gamma=[a,b]\to \mathrm{dom}(I)$ is an admissible transverse path  or order $1$;
\item $\gamma$ has a $\mathcal F$-transverse self intersection at $\gamma(s)=\gamma(t)$, where $a<s<t<b$.
\end{itemize}
We want to show that $f$ has a topological horseshoe, as defined in the introduction. We start by describing the dynamics that can be deduced from our assumptions, stating some useful lemmas. We will mainly work in the universal covering space $\widetilde{\mathrm{dom}}(I)$  of the domain of $I$ and occasionally in other covering spaces and will construct there topological horseshoes with a precise geometric description. We denote $\widetilde I=(\widetilde f_t)_{t\in[0,1]}$ the identity isotopy of  $\widetilde{\mathrm{dom}}(I)$  that lifts $I\vert_{ \mathrm{dom}(I)}$ and set $\widetilde f=\widetilde f_1$. We denote $\widetilde{\mathcal F}$ the non singular foliation of $\widetilde{\mathrm{dom}}(I)$  that lifts $\mathcal F\vert_{ \mathrm{dom}(I)}$. Fix a lift $\widetilde\gamma$ of $\gamma$.  By assumptions, there exists a non trivial covering transformation $T$ such that $T(\widetilde\gamma)$ and $\widetilde\gamma$ have a $\widetilde{\mathcal F}$-transverse intersection at $T(\widetilde\gamma(s))=\widetilde\gamma(t)$, where $s<t$. Since $\widetilde \gamma$ and $T(\widetilde \gamma)$ have a $\widetilde{\mathcal{F}}$-transverse intersection, the sets $R(\widetilde\phi_{\widetilde\gamma(a)}), R(T(\widetilde\phi_{\widetilde\gamma(a)})), L(\widetilde\phi_{\widetilde\gamma(b)})$ and $L(T(\widetilde\phi_{\widetilde\gamma(b)}))$ are all disjoint. Consider the annular covering space $\widehat{\mathrm{dom}}(I)=\widetilde{\mathrm{dom}}(I)/T$. Denote by $\pi: \widetilde{\mathrm{dom}}(I)\to \widehat{\mathrm{dom}}(I)$ the covering projection, by $\widehat I$ the induced identity isotopy, by $\widehat f$ the induced lift of $f$, by $\widehat{\mathcal{F}}$ the induced foliation and by $\widehat \gamma$ the projection of $\widetilde \gamma$. The fact that $R(\widetilde\phi_{\widetilde\gamma(a)})$ and  $R(T(\widetilde\phi_{\widetilde\gamma(a)}))$ are disjoint implies that $\widehat\phi_{\widehat\gamma(a)}$, the projection of $ \widetilde\phi_{\widetilde\gamma(a)}$, is homoclinic to an end of $\widehat{\mathrm{dom}}(I)$ and that the sets $\overline{R(T^k(\widetilde\phi_a))}$, $k\in\Z$, are pairwise disjoint. The fact that $L(\widetilde\phi_{\widetilde\gamma(b)})$ and  $L(T(\widetilde\phi_{\widetilde\gamma(b)}))$ are disjoint implies that $\widehat\phi_{\widehat\gamma(b)}$, the projection of $ \widetilde\phi_{\widetilde\gamma(b)}$, is homoclinic to an end of $\widehat{\mathrm{dom}}(I)$ and that the sets $\overline{L(T^k(\widetilde\phi_{\widetilde\gamma(b)}))}$, $k\in\Z$ are pairwise disjoint.  The fact that $T(\widetilde\gamma)$ and $\widetilde\gamma$ have a $\widetilde{\mathcal F}$-transverse intersection implies that the two previously mentioned ends are equal. Otherwise, one could find a path $\widetilde\delta$ joining $\widetilde\phi_{\widetilde\gamma(a)}$ to $\widetilde\phi_{\widetilde\gamma(b)}$ and projecting onto a simple path of $\widehat{\mathrm{dom}}(I)$ and consequently such that $\widetilde\delta\cap T(\widetilde\delta)=\emptyset$.  We will denote $S$ this common end and $N$ the other one. Note that $\overline{R(T^k(\widetilde\phi_a))}\cap\overline{ L(T^{k'}(\widetilde\phi_b))}=\emptyset$ for all integers $k$ and $k'$. The complement of $\overline{R(\widehat\phi_{\widehat\gamma(a)})}\cup\overline{L(\widehat\phi_{\widehat\gamma(b)})}$ is an essential open set of $\widehat{\mathrm{dom}}(I) $. If $\Gamma^*$ is a simple essential loop in this complement that is lifted to a line  ${\widetilde\gamma}^*:\R\to\widetilde{\mathrm{dom}}(I)$ satisfying ${\widetilde\gamma}^*(t+1)=T({\widetilde\gamma}^*(t))$ for every $t\in\R$, there exists a unique $k_0\ge 1$ such that $T^{-k_0}(\widetilde\phi_{\widetilde{\gamma}(b)})$ is above $\widetilde\phi_{\widetilde{\gamma}(a)}$ but below $T(\widetilde\phi_{\widetilde{\gamma}(a)})$ relative to $\widetilde\gamma^*$.
The integer $k_0$ can be defined by the fact that $\gamma$ and $T^k(\gamma)$ have a $\widetilde{\mathcal{F}}$-transverse intersection if $1\leq k\leq k_0$ and no transverse intersection if $k>k_0$.
We will suppose by now and until giving the proofs of theorems \textcolor{black}{M and N in Section \ref{subsec:TheoMandN}} that $k_0=1$. For convenience, we will denote $\widetilde\phi_a=\widetilde\phi_{\widetilde\gamma(a)}$ and $\widetilde\phi_b=T^{-1}(\widetilde\phi_{\widetilde\gamma(b)})$. We will often use the following result:  if $\sigma$ is a path connecting $\widetilde \phi_a$ to $T(\widetilde \phi_a)$ and contained in $\bigcap_{k\in\Z} T^{k}(R(\widetilde \phi_b))$ and if $j\not=0$, then $\widetilde \phi_b$ and $T^{j}(\widetilde \phi_b)$ lie in different connected components of the complement of $\widetilde \phi_a\cup\sigma\cup T(\widetilde \phi_a)$.

\begin{figure}[ht!]
\hfill
\includegraphics [height=48mm]{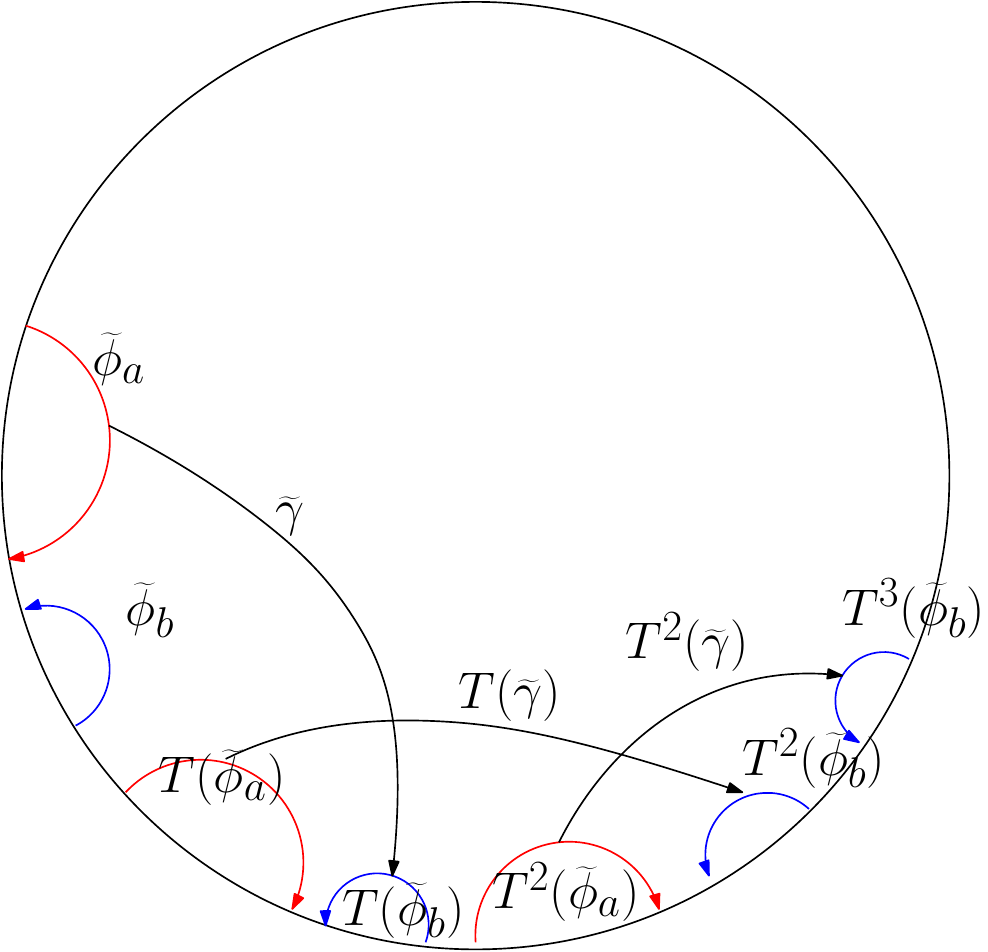}
\hfill{}
\caption{\small Lemma \ref{lm:pqintersections} }
\label{figureintersectionlemma}
\end{figure}

\subsection{\textcolor{black}{Preliminary results}}

\begin{lemma} \label{lm:pqintersections}
\label{lemma:accessibility}For every $q\geq 2$ and every $p\in\{0,\dots, q\}$,  one has   $\widetilde f^q(\widetilde\phi_a)\cap T^p(\widetilde\phi_b)\not=\emptyset$.
\end{lemma} 

\begin{proof}

Let us explain first why $$\widetilde f^q(\widetilde\phi_a)\cap T^p(\widetilde\phi_b)\not=\emptyset\enskip \mathrm{and} \enskip q'>q \Rightarrow\enskip\widetilde f^{q'}(\widetilde\phi_a)\cap T^p(\widetilde\phi_b)\not=\emptyset.$$
Recall that, for every leaf $\widetilde\phi$ of $\widetilde{\mathcal F}$, one has
 
$$\overline{R(\widetilde\phi)}\subset \widetilde f(R(\widetilde \phi)), \enskip \widetilde f(\overline{L(\widetilde\phi)})\subset L(\widetilde \phi),$$ which implies that
$$\widetilde{f}^{q}(\overline{R(\widetilde\phi)})\subset \widetilde f^{q'}(R(\widetilde \phi)), \enskip \widetilde f^{q'}(\overline{L(\widetilde\phi)})\subset \widetilde{f}^{q}(L(\widetilde \phi)),$$if $q'> q\geq 0$. We deduce that
$$\widetilde f^q(\widetilde\phi_a)\cap T^p(\widetilde \phi_b)\not=\emptyset \Rightarrow\widetilde f^q\left(\overline{R(\widetilde\phi_a)}\right)\cap \overline{L(T^p(\widetilde \phi_b)})\not=\emptyset \Rightarrow\widetilde f^{q'}\left(\overline{R(\widetilde\phi_a)}\right)\cap \overline{L(T^p(\widetilde \phi_b)})\not=\emptyset.$$ 
It means that there exists a point on the right of $\widetilde\phi_a$ that is sent by $\widetilde f^{q'}$ on the left of $T^p(\widetilde\phi_b)$. But this implies that there is a point on $\widetilde\phi_a$ that is sent by $\widetilde f^{q'}$ on $T^p(\widetilde\phi_b)$. The proof is easy and can be found in \cite{LeCalvezTal}. The reason is the following: in case $\widetilde f^{q'}(\widetilde\phi_a)\cap T^p(\widetilde\phi_b)=\emptyset$, one should have $$\widetilde f^{q'}(R(\widetilde\phi_a))\cup L(T^p(\widetilde \phi_b))=\widetilde{\mathrm{dom}}(I),$$ but this is prohibited by the the fact that $$T(R(\widetilde\phi_a))\cap\left(f^{q'}(R(\widetilde\phi_a))\cup L(T^p(\widetilde \phi_b))\right)=\emptyset.$$

The path $\widetilde\gamma$ is an admissible path of $\widetilde f$ of order $1$, which means that  $\widetilde f(\widetilde\phi_a)\cap T(\widetilde \phi_b)\not=\emptyset$. So the lemma is proved if $p=1$. 
Applying Proposition \ref {pr: fundamental}, one deduces that boths paths 
$$\widetilde\gamma\vert_{[a,t]} \,T(\widetilde\gamma)\vert_{[s,b]}, \enskip T(\widetilde\gamma)\vert_{[a,s]}\, \widetilde\gamma\vert_{[t,b]}$$ are admissible paths of $\widetilde f$ of order $2$, which means that $$\widetilde f^2(\widetilde\phi_a)\cap T^2(\widetilde\phi_b)\not=\emptyset, \enskip \widetilde f^2(T(\widetilde\phi_a))\cap T(\widetilde\phi_b)\not=\emptyset.$$ Consequently the lemma is proven if $p\in\{0,2\}$. 

Let us suppose now that $p\geq 3$. It is also proved in \textcolor{black}{Proposition 23 of} \cite{LeCalvezTal} that 
$$\widetilde\gamma_{[a,s]}\prod_{0\leq i<p} T^i(\widetilde\gamma_{[s,t]}) T^{p-1} (\widetilde\gamma_{[t,b]})$$ is admissible of order $p$, which means that  $\widetilde f^p(\widetilde\phi_a)\cap T^p(\widetilde\phi_b)\not=\emptyset$. Consequently the lemma is proved for every $p$.\end{proof}

Let us state another important lemma:

\begin{lemma} 
\label{lemma:following}Every simple path $\delta:[c,d]\to \widetilde{\mathrm{dom}}(I)$ that joins $ T^{-p_0}(\widetilde\phi_a)$ to $ T^{p_1}(\widetilde\phi_a)$, where $p_0$, $p_1$ are positive, and which is $T$-free, meets $R(\widetilde\phi_a)$.
\end{lemma}
\begin{proof}Recall first, by standard Brouwer theory arguments, that $\delta$ is $T^k$-free for every $k\geq1$, because $T$ is fixed point free. There exists $c'\in[c,d)$ and $p'_0>0$, uniquely defined, such that $\delta(c')$ belongs to $ T^{-p'_0}(\widetilde\phi_a)$ and such that $\delta_{(c',d]}$ is disjoint from every leaf $ T^{-p}(\widetilde\phi_a)$, $p>0$. Then there exists $d'\in(c',d]$, and $p'_1>0$ uniquely defined, such that $\delta(d')$ belongs to $T^{p'_1}(\widetilde\phi_a)$ and such that $\delta_{[c',d')}$ is disjoint from  every leaf $ T^{p}(\widetilde\phi_a)$, $p>0$. Note that $\delta_{(c',d')}$ is on the left of every leaf $T^{p}(\widetilde\phi_a)$, $p\not=0$. One can extend $\delta_{[c',d']}$ as a line $\lambda:\R\to \widetilde{\mathrm{dom}}(I)$ such that $\lambda\vert_{(-\infty, c')}$ is included in $R(T^{-p'_0}(\widetilde\phi_a))$ and $\lambda\vert_{(d', +\infty)}$ is included in $R(T^{p'_1}(\widetilde\phi_a))$. We will prove by contradiction that $\delta_{[c',d']}$ meets $R(\widetilde\phi_a)$. If it is not the case, then 
$$R(\widetilde\phi_a)\subset R(\lambda), \enskip R(T^{p'_0+p'_1}(\widetilde\phi_a))\subset L(\lambda).$$ The point $T^{p'_0}(\delta(c'))$ being on $\widetilde\phi_a$, belongs to the closure of $R(\widetilde\phi_a)$. It does not belong to $\delta_{[c',d']}$ (because $\delta$ is $T^{p'_0}$-free), neither to $R(T^{-p'_0}(\widetilde\phi_a))$, nor to  $R(T^{p'_1}(\widetilde\phi_a))$. So it does not belong to $\lambda$.  Consequently it belongs to $R(\lambda)$. Similarly, $T^{p'_0}(\delta(d'))$ does not belong to $\lambda$ and belongs to the closure of $R(T^{p'_0+p'_1}(\widetilde\phi_a))$, so its belongs to $L(\lambda)$. The path $T^{p'_0}(\delta_{[c',d']})$, being disjoint from $\delta_{[c',d']}$, from $R(T^{-p'_0}(\widetilde\phi_a))$ and from $R(T^{p_1}(\widetilde\phi_a))$,  is disjoint from $\lambda$. We have found a contradiction.\end{proof}

\begin{remark*} We can prove similarly that every $T$-free simple path that joins $ T^{-p_0}(\widetilde\phi_b)$ to $ T^{p_1}(\widetilde\phi_b)$, where $p_0$, $p_1$ are positive, meets $L(\widetilde\phi_b)$.
\end{remark*}

For every path $\delta: [c,d]\to \widetilde{\mathrm{dom}}(I)$, we define its {\it interior} as being the path $\dot\delta=\delta\vert_ {(c,d)}$. The following sets 
$$L_a=\bigcap_{k\in\Z} L(T^k(\widetilde\phi_a)), \enskip R_b=\bigcap_{k\in\Z} R(T^k(\widetilde\phi_b))$$ are open sets homeomorphic to the plane. For every $q\geq 1$, one can look at the relatively compact connected components of $\widetilde f^{-q}(\widetilde\phi_b)\cap L_a$. By Lemma \ref{lemma:following}, every such a component is the interior of path joining $T^{p}(\widetilde\phi_a)$ to $T^{p'}(\widetilde\phi_a)$, where $\vert p-p'\vert\leq 1$. For every $p\geq 0$ and $q\geq 1$, we will define the set ${\mathcal X}_{p,q}$ of paths joining $\widetilde\phi_a$ to $T(\widetilde\phi_a)$ whose interior is a connected component of $T^p\circ f^{-q}(\widetilde\phi_b)\cap L_a$.  We denote ${\mathcal X}=\bigcup_{p\geq 0, q\geq 1} {\mathcal X}_{p,q}$. Note that two different paths in ${\mathcal X}$ have disjoint interiors, and are disjoint  if they do not belong to the same set ${\mathcal X}_{p,q}$. Fix a real parametrization $t\mapsto \widetilde\phi_a(t)$ of $\widetilde\phi_a$. If $\delta_0\in\mathcal X$  joins $\widetilde\phi_a(v_0)$ to $T(\widetilde\phi_a)(v'_0)$, and if $\delta_1\in\mathcal X$ joins $\widetilde\phi_a(v_1)$ to $T(\widetilde\phi_a)(v'_1)$, then $(v_1-v_0)(v'_1-v'_0)\leq 0$. Moreover, if $\delta_0\not=\delta_1$ at least one of the numbers $v_1-v_0$, $v'_1-v'_0$ is not zero. So one gets a natural total order on ${\mathcal X}$ by setting 
$$\delta_0\leq \delta_1\enskip\mathrm{if}\enskip v_1\leq v_0\enskip\mathrm{and} \enskip v'_0\leq v'_1.$$
One can extend this order to every set $\mathcal Y$ of paths containing $\mathcal X$, provided the paths of $\mathcal Y$ join $\widetilde\phi_a$ to $T(\widetilde\phi_a)$, have their interior included in $L_a$, have pairwise disjoint interiors, \textcolor{black}{and provided that no two paths in $\mathcal Y$ share both endpoints}.

\begin{lemma}\label{lemma: paths}
We have the following:

\begin{enumerate}
\item The set $\mathcal X_{p,q}$ is not empty if $q\geq 2$ and $1\leq p\leq q$.
\item If $\delta_0$ and $\delta_1$ are two paths in $\mathcal X$, there are finitely many paths in $\mathcal X_{p,q}$ between $\delta_0$ and $\delta_1$.
\item Suppose that $q\geq 3$ and $1\leq p_0<p_2< p_1\leq q$. For every $\delta_0\in \mathcal X_{p_0,q}$ and every $\delta_1\in \mathcal X_{p_1,q}$, there exist at least two paths in $\mathcal X_{p_2,q}$ between $\delta_0$ and $\delta_1$.\end{enumerate}

 \end{lemma}
\begin{proof} To get $(1)$, apply Lemma \ref{lemma:accessibility} that asserts that  $\widetilde f^{-q}\circ T^p(\widetilde\phi_b)$ meets $\widetilde \phi_a$ and $T(\widetilde \phi_a)$ if $q\geq2$ and $1\leq p\leq 2$, and then apply Lemma \ref{lemma:following}.

Let us prove $(2)$. Suppose that $\delta_0$ and $\delta_1$ are two paths in $\mathcal X$, and write $(\delta_i)_{i\in I}$ the family of paths in $\mathcal X_{p,q}$ that are between $\delta_0$ and $\delta_1$. It consists of subpaths of the line $\widetilde f^{-q}\circ T^p(\widetilde\phi_b)$  with pairwise disjoint interiors, whose union is relatively compact. If this family is infinite, it contains a sequence that converges to a point in the Hausdorff topology. Such a point should belong both to $\widetilde \phi_a$ and $T(\widetilde \phi_a)$. We have a contradiction.

Let us prove $(3)$. Suppose that $q\geq 3$ and $1\leq p_0<p_2< p_1\leq q$ and fix $\delta_0\in \mathcal X_{p_0,q}$ and $\delta_1\in \mathcal X_{p_1,q}$. By $(2)$, one can suppose that there is no path in $\mathcal X_{p_0,q}\cup \mathcal X_{p_1,q}$ between $\delta_0$ and $\delta_1$. Suppose for example that $\delta_0\prec\delta_1$. The path $\delta_0$  joins $\widetilde\phi_a(v_0)$ to $T(\widetilde\phi_a)(v'_0)$, and $\delta_1$ joins $\widetilde\phi_a(v_1)$ to $T(\widetilde\phi_a)(v'_1)$. Denote $\beta=\widetilde\phi_a\vert_{[v_1,v_0]}$ and $\beta'=T(\widetilde\phi_a)\vert_{[v'_0,v'_1]}$. Consider the simple closed curve $C=\beta\delta_0\beta'\delta_1^{-1}$ and the relatively compact connected component $\Delta$ of its complement. 

Let us prove that $C$ is $T$-free. Note first that $T(C)\cap C\subset \beta'\cap T(\beta)$. So, one must prove that $\beta'\cap T(\beta)=\emptyset$. Each path $\delta_1^{-1}\beta\delta_0$ and $T(\delta_0)T(\beta')T(\delta_1^{-1})$ join $T(\widetilde\phi_a)$ to itself. Moreover, these two paths  are disjoint and their interior are included in $L(T(\widetilde\phi_a))$. So, either one of the paths $\beta'$  or $T(\beta)$ is included in the other one, or they are disjoint. But in the first case, $\overline\Delta$ has to be forward or backward invariant by $T$, which contradicts the fact that $T$ is fixed point free. One deduces that $\beta'\cap T(\beta)=\emptyset$ and consequently that $C$ is $T$-free.

 So, $\overline\Delta$ is also $T$-free and $\Delta$ is included in $L_a$.  Applying Lemma \ref{lemma:following} to $ \widetilde f^q(\beta'')$, one deduces that every path $\beta''\subset\overline\Delta$ that joins $\delta_0$ to $\delta_1$ meets $\widetilde f^{-q}\circ T^{p_2}(L(\widetilde\phi_b)))$. So, there exists a simple path $\sigma\subset\widetilde f^{-q}\circ T^{p_2}(L(\widetilde\phi_b)))$ that joins  $\widetilde\phi_a$ to $T(\widetilde\phi_a)$ and whose interior is contained in $\Delta$. Consequently, there exists at least a path $\delta_2$ of $\mathcal X_{p_2,q}$ between $\delta_0$ and $\sigma$ and another path $\delta_2'$ of $\mathcal X_{p_2,q}$ between $\sigma$ and $\delta_1$, see Figure \ref{figureinbetweenpaths}. \end{proof}

\begin{figure}[ht!]
\hfill
\includegraphics [height=48mm]{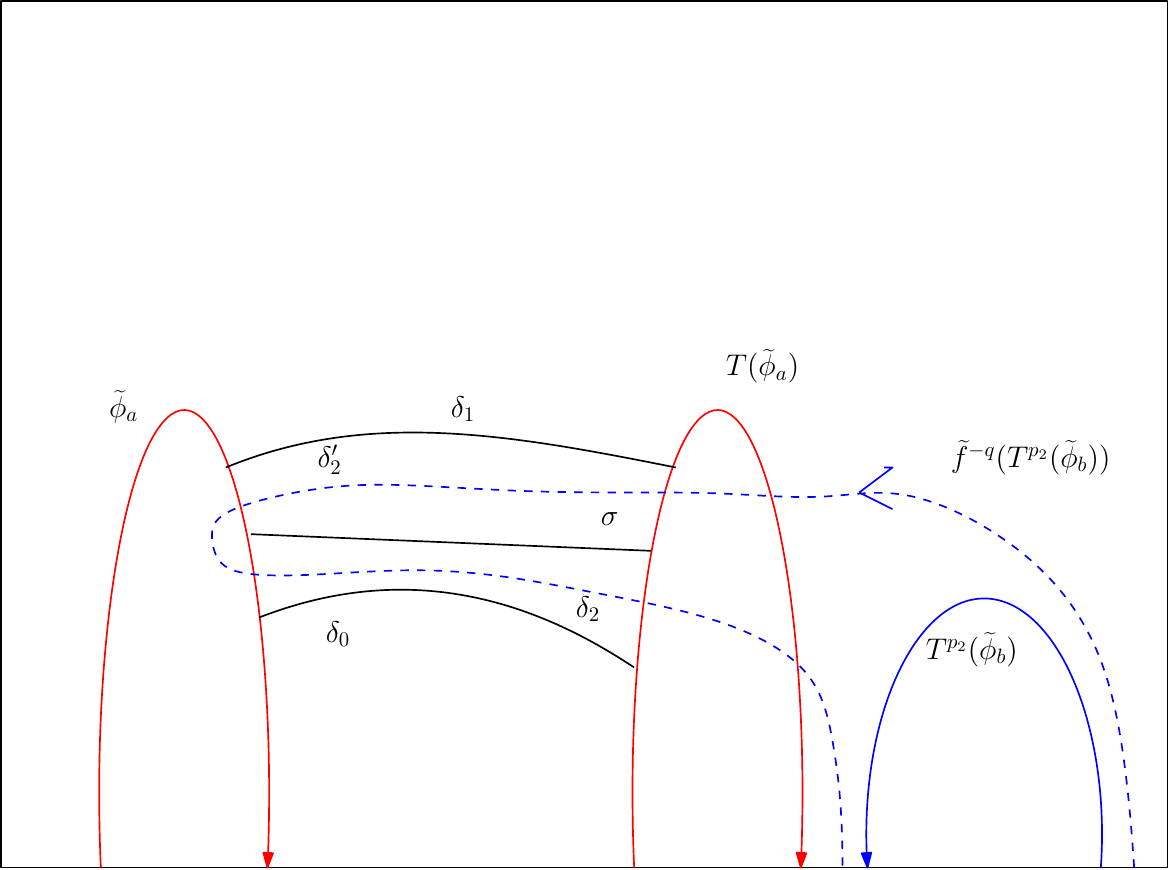}
\hfill{}
\caption{\small Lemma \ref{lemma: paths} }
\label{figureinbetweenpaths}
\end{figure}

\subsection{Weak version of the Realization Theorem (Theorem \ref{th: realization})}
Let us begin with the following preliminary version of Theorem \ref{th: realization}.

\begin{proposition}\label{prop: preliminary-realization}
 If $q\geq 2$ and $1\leq p\leq q$, then $\widetilde f^q\circ T^{-p}$ has a fixed point.
 
 \end{proposition}
\begin{proof} We set $\widetilde g=\widetilde f^q\circ T^{-p}$. The hypothesis implies that either $1\leq p+1\leq q$ or $1\leq p-1\leq q$. We will study the first case. The sets  $\mathcal X_{p,q}$ and $\mathcal X_{p+1,q}$ are non empty by assertion $(1)$ of Lemma \ref{lemma: paths}. Choose $\delta_0\in \mathcal X_{p,q}$ and $\delta_1\in \mathcal X_{p+1,q}$. Suppose that $\delta_0$ joins $\widetilde\phi_a(t_0)$ to $T(\widetilde\phi_a(t'_0))$ and that $\delta_1$ joins $\widetilde\phi_a(t_1)$ to $T(\widetilde\phi_a(t'_1))$. We remark that, as $\widetilde g(\delta_0)$ is contained in $\widetilde \phi_b$ and as $\widetilde g(\delta_1)$ is contained in $T(\widetilde \phi_b)$, then $\widetilde g(\delta_0)$ and $\widetilde g(\delta_1)$ lie in different connected components of the complement of $\widetilde \phi_a\cup \delta_0\cup T(\widetilde \phi_a)$ as well as in different connected components of the complement of $\widetilde \phi_a\cup \delta_1\cup T(\widetilde \phi_a)$.

Suppose first that $\delta_0\prec\delta_1$. Denote $\beta=\widetilde\phi_a\vert_{[t_1,t_0]}$ and $\beta'=T(\widetilde\phi_a)\vert_{[t'_0,t'_1]}$ and consider the simple closed curve $C=\beta\delta_0\beta'\delta_1^{-1}$. The index $i(\widetilde g, C)$ is well defined, we will prove that it is equal to {$-1$}. It will imply that $\widetilde g$ has a fixed point inside the topological disk bounded by $C$. Recall that to define $i(\widetilde g, C)$, one needs to choose a homeomorphism $h$ between $\widetilde{\mathrm{dom}}(I)$  and the Euclidean plane and consider the winding number of the vector field $\xi: z\mapsto  h\circ \widetilde g\circ h^{-1}(z)-z$ on $h(C)$, this integer being independent of the choice of $h$. To make the computation, one can use Homma's theorem \cite{Homma} that asserts that $h$ can be chosen such that 
$$h(\widetilde\phi _a)=\{0\}\times\R, \enskip h(T(\widetilde\phi _a))=\{1\}\times\R, \enskip h(\delta_0)=[0,1]\times\{0\}, \enskip h(\delta_1)=[0,1]\times\{1\}.$$
Note now that $\xi(z)$ is pointing: 
\begin{itemize}
\item on the right when $z$ belongs to $h(\beta)=\{0\}\times[0,1]$,
\item  on the left when $z$ belongs to $h(\beta')=\{1\}\times[0,1]$,
\item  below when $z$ belongs to $h(\delta_0)=\{0\}\times[0,1]$,
\item  above when $z$ belongs to $h(\delta_1)=\{1\}\times[0,1]$.
\end{itemize} This implies that the winding number is equal to $-1$.  

In case where $\delta_1\prec\delta_0$, define
$$\beta=\widetilde\phi_a\vert_{[t_0,t_1]}, \enskip \beta'=T(\widetilde\phi_a)\vert_{[t'_1,t'_0]}, \enskip C=\beta\delta_1\beta'\delta_0^{-1}$$ and choose $h$ such that
$$h(\widetilde\phi _a)=\{0\}\times\R, \enskip h(T(\widetilde\phi _a))=\{1\}\times\R, \enskip h(\delta_0)=[0,1]\times\{1\}, \enskip h(\delta_1)=[0,1]\times\{0\}.$$
Note now that $\xi(z)$ is pointing: 
\begin{itemize}
\item on the right when $z$ belongs to $h(\beta)=\{0\}\times[0,1]$,
\item  on the left when $z$ belongs to $h(\beta')=\{1\}\times[0,1]$,
\item  below when $z$ belongs to $h(\delta_0)=\{1\}\times[0,1]$,
\item  above when $z$ belongs to $h(\delta_1)=\{0\}\times[0,1]$.
\end{itemize} This implies that the winding number is equal to $+1$.  

In case $1\leq p-1\leq q$, we choose $\delta_0\in \mathcal X_{p,q}$ and $\delta_1\in \mathcal X_{p-1,q}$. The arguments are similar.\end{proof}

 \begin{figure}[ht!]
\hfill
\includegraphics [height=48mm]{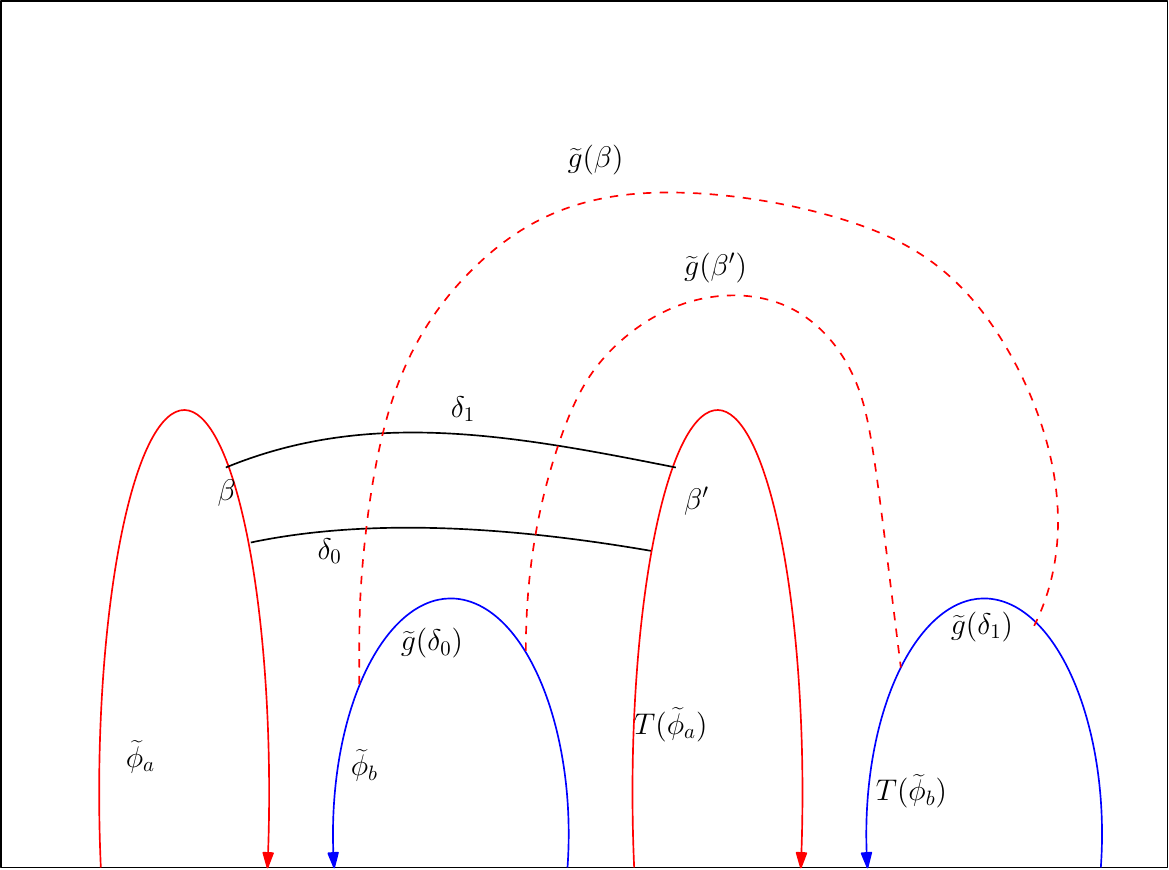}
\hfill{}
\caption{\small Proposition \ref{prop: preliminary-realization}. Relative position of $\widetilde g(\beta), \widetilde g(\delta_0), \widetilde g(\beta')$ and $\widetilde g(\delta_1)$. }
\label{figure_index_1}
\end{figure}
\begin{corollary}\label{cor: preliminary-realization}
 For every $p/q\in(0,1]$, such that $p$ and $q$ are positive integers, there exists a point $\widetilde z$ such that $\widetilde f^q(\widetilde z)=T^p(\widetilde z)$.
 \end{corollary}
 \begin{proof}  
 In case $p/q<1$ we just apply Proposition \ref{prop: preliminary-realization}. The case $p/q=1$ is particular, as not being an immediate consequence of Proposition \ref{prop: preliminary-realization}. Nevertheless, we know that $(\widetilde f\circ T^{-1})^2=\widetilde f^2\circ T^{-2}$ has a fixed point. Brouwer's theorem asserts that the existence of a periodic point implies the existence of a fixed point, for an orientation preserving plane homeomorphism. So $\widetilde f\circ T^{-1}$ has a fixed point.
 \end{proof}

\subsection{Existence of horseshoe in the universal covering space}  If $\delta_0$ and $\delta_1$ are two disjoint paths in $\mathcal X$, we denote $\overline{\Delta}_{\delta_0,\delta_1}$ the closed ``rectangle'' bounded by $\delta_0$, $\delta_1$, $\beta$ and $\beta'$, where $\beta$ is the subpath of $\widetilde \phi_a$ joining the ends of $\delta_0$, $\delta_1$ that lie on $\widetilde \phi_a$ and $\beta'$ is the subpath of $\widetilde T(\widetilde \phi_a)$ joining the two other ends. We will call {\it horizontal sides} the paths $\delta_0$, $\delta_1$, and vertical sides the paths $\beta$, $\beta'$. We will prove the following result. 

\begin{proposition} \label{prop: horseshoe-universal}
 If $q\geq 3$ and $1< p<q$, there exists a compact set $\widetilde Z_{p,q}\subset\widetilde{\mathrm{dom}}(I)$, invariant by $\widetilde f^q\circ T^{-p}$ such that
\begin{itemize}
\item the restriction of $\widetilde f^q\circ T^{-p}$ to $\widetilde Z_{p,q}$ is an extension of the Bernouilli shift $\sigma :\{1,2\}^{\Z}\to \{1,2\}^{\Z}$;
\item every $r$-periodic sequence of $\{1,2\}^{\Z}$ has a preimage by the factor map $\widetilde H_{p,q}:\widetilde Z_{p,q}\to \{1,2\}^{\Z}$ that is a $r$-periodic point of $\widetilde f^q\circ T^{-p}$.
\end{itemize}
\end{proposition}

\par\noindent\textit{Proof.} We suppose that $q\geq 3$ and fix $1<p<q$. Using Lemma \ref{lemma: paths} we can find four paths $\delta_1\in\mathcal X_{p-1,q}$, $\delta_2\in\mathcal X_{p,q}$,  $\delta_3\in\mathcal X_{p,q}$ and  $\delta_4\in\mathcal X_{p+1,q}$ all pairwise disjoint such that \textcolor{black}{$\delta_1<\delta_2<\delta_3<\delta_4$ or $\delta_4<\delta_3<\delta_2<\delta_1$. }

We set $$\overline\Delta_1=\overline\Delta_{\delta_1,\delta_2}, \enskip\overline\Delta_2=\overline\Delta_{\delta_3,\delta_4},\enskip \overline\Delta=\overline\Delta_{\delta_1,\delta_4}$$ and $\widetilde g=\widetilde f^q\circ T^{-p}$. Note again that $\widetilde g(\delta_1)\subset T^{-1}(\widetilde \phi_b)$, that both $\widetilde g(\delta_2)$ and $\widetilde g(\delta_3)$ are contained in $\widetilde \phi_b$ and that $\widetilde g(\delta_4)\subset T(\widetilde \phi_b)$.

Let us explain first why $\overline\Delta$ is a topological horseshoe in the sense of Kennedy-Yorke (see \cite{KennedyYorke}). Every continuum $K\subset \overline\Delta$ that meets $\delta_{1}$ and $\delta_{4}$ contains a continuum $K_1\subset \overline\Delta_1$ that meets $\delta_{1}$ and $\delta_{2}$  and a continuum $K_2\subset \overline\Delta_2$ that meets $\delta_{3}$ and $\delta_4$. The set $\widetilde g(K_1)$ is a continuum that meets $\widetilde\phi_b$ and $T^{-1}(\widetilde\phi_b)$ and that does not meet neither $R(\widetilde\phi_a)$ nor $R(T(\widetilde\phi_a))$. This implies that it contains a continuum included in $\overline\Delta$ that meets $\delta_{1}$ and $\delta_{4}$. Similarly,  $\widetilde g(K_2)$ is a continuum that meets $\widetilde\phi_b$ and $T(\widetilde\phi_b)$ and that does not meet neither $R(\widetilde\phi_a)$ nor $R(T(\widetilde\phi_a))$. So, it contains a continuum included in $\overline\Delta$ that meets $\delta_{1}$ and $\delta_{4}$. Summarizing,  every continuum $K\subset \overline\Delta$ that meets $\delta_{1}$ and $\delta_{4}$ contains two disjoint sub-continua whose images by $\widetilde g$ are included in $\overline\Delta$ and meet $\delta_{1}$ and $\delta_{4})$. This is the definition of a topological horseshoe in the sense of Kennedy-Yorke.  According to \cite{KennedyYorke}, there exists a compact set $\widetilde Z$ invariant by $\widetilde g$ such that $\widetilde g\vert_{\widetilde Z}$ is an extension of the two sided Bernouilli shift on $\{1,2\}^{\Z}$, see Figure \ref{figure_horseshoe_covering}.

 \begin{figure}[ht!]
\hfill
\includegraphics [height=48mm]{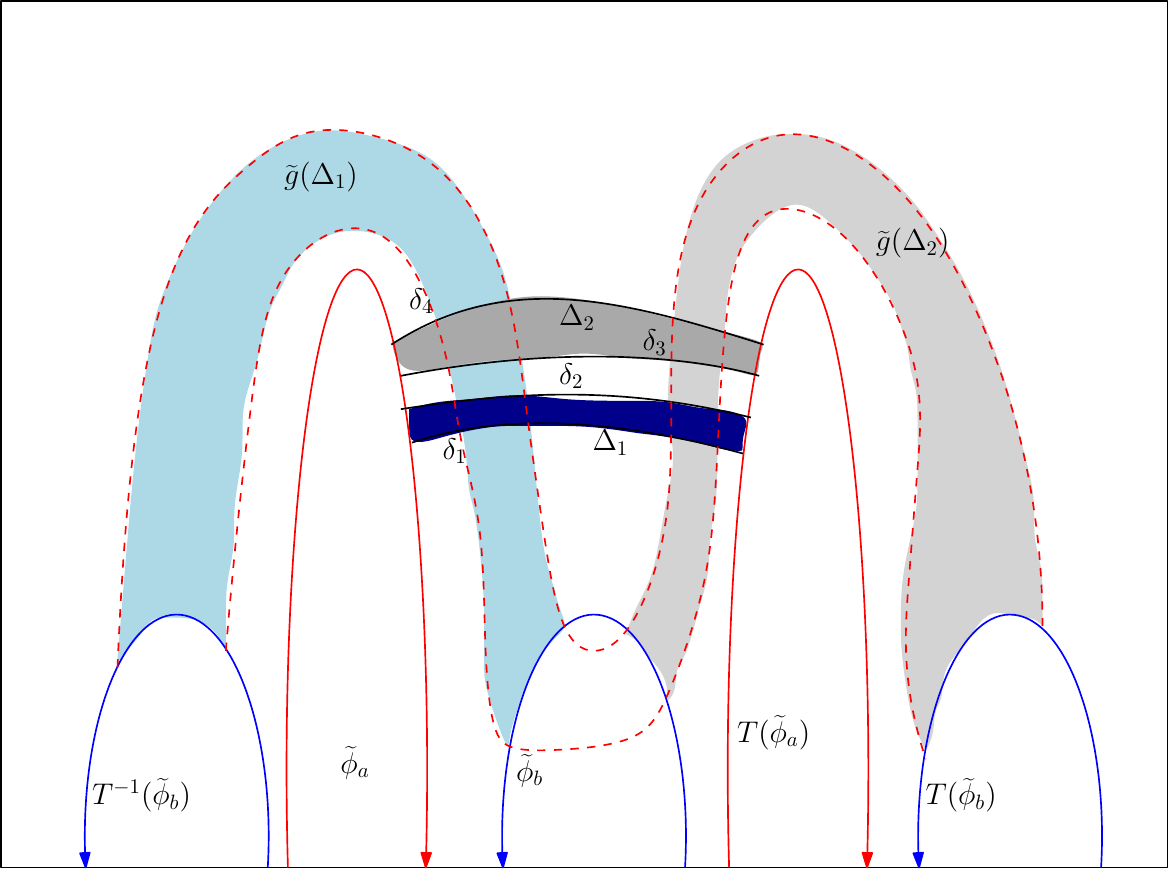}
\hfill{}
\caption{\small Horseshoe in the universal covering space}
\label{figure_horseshoe_covering}
\end{figure}

To get our proposition \textcolor{black}{we still must} prove that to every periodic sequence corresponds at least one periodic point in $\widetilde Z$. If $\widetilde z$ belongs to $\partial(\overline\Delta_1\cup \overline\Delta_2)$, then either its image by $\widetilde g$ or its inverse image is outside $\overline\Delta_1\cup \overline\Delta_2$. The first situation occurs when $z$ belongs to an horizontal side, the second one when $z$ belongs to a vertical side. This means that the set $\widetilde Z_{p,q}=\bigcap_{k\in\Z} \widetilde g^{-k}(\overline\Delta_1\cup \overline\Delta_2)$ is a compact set included in $\Delta_1\cup \Delta_2$, {it} is a {\it locally maximal invariant set} and $\overline\Delta_1\cup \overline\Delta_2$ an {\it isolating block} of $\widetilde Z_{p,q}$. Note that there exists a map $\widetilde H_{p,q}: \widetilde Z_{p,q}\to \{1,2\}^{\Z}$ that associates to a point $\widetilde z\in\widetilde Z_{p,q}$ the sequence $(\varepsilon_k)_{k\in\Z}$ such that $\widetilde g^{i}(\widetilde z)\in \Delta_{\varepsilon_k}$. The map $\widetilde H_{p,q}$ is continuous and satisfies $\widetilde H_{p,q}\circ\widetilde g=\sigma\circ \widetilde H_{p,q}$.  We will prove that every periodic sequence of period $r$ is the image by $\widetilde H_{p,q}$ of a periodic point of $\widetilde g$ of period $r$. 
This will imply that $\widetilde H_{p,q}$ is onto, because its image is compact, and we will get a complete proof of Proposition \ref{prop: horseshoe-universal}.

\textcolor{black}{We will implicitly use in the proof a result of Ker\'ekj\'art\'o about intersection of Jordan domains. Say that an open set $U$ of the plane is a {\it Jordan domain} if it is the bounded connected component of the complement of a simple loop. The boundary $\partial U$ of $U$ inherits a natural orientation and consequently a natural cyclic order. The following result gives a description of the intersection of Jordan domains in the non trivial case (see \cite{LeCalvezYoccoz1}).
\begin{lemma}\label{lemma:JordanDomains} Let $U$ and $U'$ be two Jordan domains such that
$$U\cap U'\not=\emptyset, \enskip U\not\subset U', \enskip U'\not\subset U.$$
Then, every connected component $U''$ of $U\cap U'$ is a Jordan domain. Moreover, one can write
$$\partial U''\setminus (\partial U\cap\partial U')=\bigsqcup_{i\in I} (a_i,b_i) \sqcup  \bigsqcup_{i'\in I'} (a_{i'},b_{i'}),$$
where  $(a_i,b_i)$ is a connected component of $\partial U\cap U'$ and  $(a_{i'},b_{i'})$ is a connected component of $\partial U'\cap U$, and it holds that
\begin{itemize}
\item the orientations induced on each interval $(a_i,b_i)$ by $\partial U$ and $\partial U''$ coincide;
 \item the orientations induced on each interval $(a_{i'},b_{i'})$ by $\partial U'$ and $\partial U''$ coincide;
 \item the cyclic orders induced on $\partial U\cap \partial U''$ by $\partial U$ and $\partial U''$ coincide;
  \item the cyclic orders induced on $\partial U'\cap \partial U''$ by $\partial U'$ and $\partial U''$ coincide.
  \end{itemize}
\end{lemma}}
\textcolor{black}{
\begin{lemma} \label{lemma:Markov} Fix ${\bf e}=(\varepsilon_{k})_{0\leq k<K}\in\{1,2\}^K$, where $K>1$, then denote $\beta$, $\beta'$ the vertical sides of $\Delta_{\varepsilon_0}$ and $\delta$, $\delta'$ the horizontal sides of  $\Delta_{\varepsilon_{K-1}}$.There exists a Jordan domain $U_{\bf e}$ that satisfies the following properties:
\begin{itemize}
\item $U_{\bf e}$ is a connected component of $\bigcap_{ 0\leq k<K} \widetilde g^{-k}( \Delta_{\varepsilon_k})$;
\item $U_{\bf e}$ is a connected component of $ \Delta_{\varepsilon_0}\cap \widetilde g^{-K+1}( \Delta_{\varepsilon_{K-1}})$;
\item $\partial U_{\bf e}\subset \beta\cup\beta'\cup \widetilde g^{-K+1}(\delta)\cup \widetilde g^{-K+1}(\delta')$;
\item there exists a unique connected component $\alpha$ of $\partial U_{\bf e}\cap \widetilde g^{-K+1}(\Delta_{\varepsilon_{K-1}})$ that is contained in $\beta$ and such that $g^{K-1}(\overline\alpha) $ joins $\delta$ and $\delta'$;
\item there exists a unique connected component $\alpha'$ of $\partial U_{\bf e}\cap \widetilde g^{-K+1}( \Delta_{\varepsilon_{K-1}})$ that is contained in $\beta'$ and such that $g^{K-1}(\overline\alpha') $ joins $\delta$ and $\delta'$;
\item there exists a unique connected component $\sigma$ of $\partial U_{\bf e}\cap  \Delta_{\varepsilon_{0}}$ that is contained in $\widetilde g^{-K+1}(\delta)$ and such that $\overline\sigma$ joins $\beta$ and $\beta'$;
\item there exists a unique connected component $\sigma'$ of $\partial U_{\bf e}\setminus  \partial \Delta_{\varepsilon_{0}}$ that is contained in $\widetilde g^{-K+1}(\delta')$ and such that $\overline\sigma'$ joins $\beta$ and $\beta'$.
\end{itemize} 
\end{lemma}}
\begin{proof} \textcolor{black}{Let us prove first that every connected component $U$ of $\bigcap_{ 0\leq k<K} \widetilde g^{-k}( \Delta_{\varepsilon_k})$ is a connected component of $ \Delta_{\varepsilon_0}\cap \widetilde g^{-K+1}( \Delta_{\varepsilon_{K-1}})$. Let $U'$ be the connected component of  $ \Delta_{\varepsilon_0}\cap \widetilde g^{-K+1}( \Delta_{\varepsilon_{K-1}})$ that contains $U$. The boundary of $U$ is included in $\bigcup_{ 0\leq k<K} \widetilde g^{-k}( \partial \Delta_{\varepsilon_k})$ and we have $$\widetilde g^{-k+1}(\overline\Delta_{\varepsilon_{k-1}})\cap  \widetilde g^{-k}( \partial \Delta_{\varepsilon_k}) \cap \widetilde g^{-k-1}(\overline\Delta_{\varepsilon_{k+1}})=\emptyset,$$ if $0<k<K-1$, because
$$\widetilde g(\overline \Delta_{\varepsilon_{k-1}})\cap   \partial \Delta_{\varepsilon_k} \cap \widetilde g^{-1}(\overline\Delta_{\varepsilon_{k+1}})=\emptyset.$$ }
\textcolor{black}{Indeed, the forward image of a horizontal side of $ \Delta_{\varepsilon_k}$ and the backward image of a vertical side of  $ \Delta_{\varepsilon_k}$ are disjoint from $\overline \Delta$. Consequently it holds that $\partial U\subset \partial \Delta_{\varepsilon_0} \cup \widetilde g^{-K+1}( \partial \Delta_{\varepsilon_{K-1}})$. This means that the frontier of $U$ in $U'$ is empty and so $U=U'$ by connectedness of $U'$.  The forward image of a horizontal side of $ \Delta_{\varepsilon_0}$ and the backward image of a vertical side of  $ \Delta_{\varepsilon_{K-1}}$ being disjoint from $\overline \Delta$, one deduces that  $\partial U\subset \beta\cup\beta'\cup \widetilde g^{-K+1}(\delta)\cup \widetilde g^{-K+1}(\delta')$.}

\textcolor{black}{We will argue by induction to prove  Lemma \ref{lemma:Markov} . Let us begin by studying the case where \textcolor{black}{$K=2$}.  Both paths $\widetilde g(\beta)$ and $\widetilde g(\beta')$ join $\widetilde \phi_b$ to $T(\widetilde \phi_b)$ or join $\widetilde \phi_b$ to $T^{-1}(\widetilde \phi_b)$ and are included in $L_a$. So, there exists a finite and odd number $m$ of connected components of $\widetilde g(\beta)\cap \Delta_{\varepsilon_1}$ whose closure join $\delta$ and $\delta'$ and a finite and odd number $m'$ of connected components of $\widetilde g(\beta')\cap \Delta_{\varepsilon_1}$ whose closure join $\delta$ and $\delta'$. Each of these $m+m'$ paths is included in the boundary of a connected component of $\widetilde g(\Delta_{\varepsilon_0})\cap \Delta_{\varepsilon_1}$, unique but depending on the path. Moreover, by Lemma  \ref{lemma:JordanDomains}, this boundary is a simple loop and it contains exactly two such paths. So there exists at least one component $V$ of $\widetilde g(\Delta_{\varepsilon_0})\cap \Delta_{\varepsilon_1}$ whose boundary contains a connected component $\widetilde g(\alpha)$ of $\widetilde g(\beta)\cap \Delta_{\varepsilon_1}$ whose closure joins $\delta$ and $\delta'$ and a connected component $\widetilde g(\alpha')$ of $\widetilde g(\beta')\cap \Delta_{\varepsilon_1}$ whose closure joins $\delta$ and $\delta'$. Using again Lemma  \ref{lemma:JordanDomains}, one deduces that the boundary of $U=\widetilde g^{-1}(V)$ contains a unique connected component $\sigma$ of $\widetilde g^{-1}(\delta)\cap \Delta_{\varepsilon_0}$ whose closure joins $\beta$ and $\beta'$ and a unique connected component $\sigma'$  of $\widetilde g^{-1}(\delta')\cap \Delta_{\varepsilon_0}$ whose closure joins $\beta$ and $\beta'$. Note that $U$ satisfies the conclusion of the lemma.}

\textcolor{black}{Now suppose that Lemma \ref{lemma:Markov} has been proven for $K$ and let us prove it for $K+1$. We denote $\beta$, $\beta'$ the vertical sides of $\Delta_{\varepsilon_0}$ and $\delta$, $\delta'$ the horizontal sides of  $\Delta_{\varepsilon_{K}}$. Consider a sequence ${\bf e}=(\varepsilon_{k})_{0\leq k\leq K}$,
 set ${\bf e'}=(\varepsilon_{k})_{0\leq k<K}$ and choose  $U_{\bf e}$ satisfying the conclusion of Lemma \ref{lemma:Markov}.  Every component $\alpha$ of $\partial U_{\bf e}\cap \widetilde g^{-K+1}(\Delta_{\varepsilon_{K-1}})$ is included in $\beta$ or in $\beta'$. Its image by $\widetilde g^{K-1}$ is included in $\Delta_{\varepsilon_{k-1}}$ and its closure joins a horizontal side of $\Delta_{\varepsilon_{k-1}}$ to a horizontal side of $\Delta_{\varepsilon_{k-1}}$. If both sides are equal, there exists a finite and even number of connected components of $\widetilde g^K(\alpha)\cap \Delta_{\varepsilon_K}$ whose closure joins $\delta$ to  $\delta'$. 
Furthermore there are finitely many components $\alpha$ such that this number is non zero (because there is finitely many components $\alpha$ whose diameter is larger than a given positive number). If both sides are different, there exists a finite and odd number of connected components of $\widetilde g^K(\alpha)\cap \Delta_{\varepsilon_K}$ whose closure joins $\delta$ to  $\delta'$. By assumptions, we know that there exists exactly one such component on $\beta$ and one on $\beta'$. Arguing like in the case $K=1$, we deduce that there exists at least one component $V_{\bf e}$ of $\widetilde g^K(U_{\bf e'})\cap \Delta_{\varepsilon_K}$ whose boundary contains a connected component of $\widetilde g^K(\beta)\cap \Delta_{\varepsilon_K}$ whose closure joins $\delta$ to  $\delta'$ and a connected component of $\widetilde g^K(\beta')\cap \Delta_{\varepsilon_K}$ whose closure joins $\delta$ to  $\delta'$. Here again, the boundary of $U_{\bf e} =\widetilde g^{-K}(V_{\bf e})$ contains a connected component of $\widetilde g^{-K}(\delta)\cap \Delta_{\varepsilon_0}$ whose closure joins $\beta$ to  $\beta'$ and a connected component of $\widetilde g^{-K}(\delta')\cap \Delta_{\varepsilon_0}$ whose closure joins $\beta$ to  $\beta'$. Note that $U_{\bf e} $ satisfies the conclusion of the lemma.}
 \end{proof}
 
 \textcolor{black}{The following result will immediately imply Proposition \ref{prop: horseshoe-universal}.
 \begin{lemma} \label{lemma:periodique} Let  ${\bf e}=(\varepsilon_{k})_{0\leq k<K}$ be a sequence in $\{1,2\}^K$, where $K>1$. If $U_{\bf e}$ is a Jordan domain that satisfies the conclusion of Lemma \ref{lemma:Markov}, then $U_{\bf e}$ contains a fixed point of $\widetilde g^K$.
\end{lemma}
\begin{proof} The proof is similar to the proof of Proposition \ref{prop: preliminary-realization}. We denote $\beta$ the vertical side of $\Delta_{\varepsilon_0}$ lying on $\widetilde \phi_a$ and $\beta'$ the vertical side of $\Delta_{\varepsilon_0}$ lying on $T(\widetilde \phi_a)$. We denote $\delta$ the horizontal side of $\Delta_{\varepsilon_{K-1}}$ whose image by $\widetilde g$ lies on $\widetilde\phi_b$ and $\delta'$ the horizontal side of $\Delta_{\varepsilon_{K-1}}$ whose  image by $\widetilde g$ lies on $T^{\pm 1} (\widetilde \phi_b)$ (meaning $T(\widetilde\phi_b)$ or $T^{-1}(\widetilde\phi_b)$. The paths $\alpha$, $\alpha'$, $\sigma$ and $\sigma'$ are defined in the statement of Lemma \ref{lemma:Markov}. To prove Lemma \ref{lemma:periodique}  we will show that the index $i(\widetilde g^K, \partial U_{\bf e})$ is equal to $-1$ if $\overline\sigma<\overline\sigma'$ and equal to $+1$ if $\overline\sigma'<\overline\sigma$. Note that 
$$\widetilde g^K(z)\in \begin{cases} &L_a \enskip\mathrm{if}  \enskip x\in \beta\cup\beta',\\
&\widetilde \phi_b \enskip\mathrm{if} \enskip z\in \widetilde g^{1-K}(\delta),\\
&T^{\pm 1}(\widetilde \phi_b)\enskip \mathrm{if} \enskip z\in \widetilde g^{1-K}(\delta').
\end{cases}
$$
Using Homma's theorem, consider an orientation preserving plane homeomorphism $h$ such that 
$$h(\widetilde\phi _a)=\{0\}\times\R, \enskip h(T(\widetilde\phi _a))=\{1\}\times\R, \enskip h(\sigma)=(0,1)\times\{0\}, \enskip h(\sigma')=(0,1)\times\{1\}.$$
Note now that $\xi(z)=h\circ\widetilde g^{K}\circ h^{-1}(z)-z$ points: 
\begin{itemize}
\item on the right when $z$ belongs to $h(\partial U_{\bf e}\cap \beta)$ 
\item  on the left when $z$ belongs to $h(\partial U_{\bf e}\cap \beta')$,
\item  below when $z$ belongs to $h(\partial U_{\bf e}\cap \widetilde g^{1-K}(\delta))$,
\item  above when $z$ belongs to $h(\partial U_{\bf e}\cap \widetilde g^{1-K}(\delta'))$
\end{itemize} The complement of $\alpha\cup\alpha'\cup\sigma\cup\sigma'$ in $\partial U_{\bf e}$ is the union of pairwise disjoint closed intervals $I_{i}$, $1\leq i\leq 4$, such that $I_1$ lies between $\delta'$ and $\alpha$, $I_2$ lies between $\alpha$ and $\delta$, $I_3$ lies between $\delta$ and $\beta'$, and $I_4$ lies between $\beta'$ and $\sigma'$. Note know that
\begin{itemize} 
\item $\xi(z)$ points on the right or above and is not a positive multiple of $(-1,-1)$ if $z\in h(I_1)$;
\item $\xi(z)$ points on the right or below and is not a positive multiple of $(-1,1)$ if $z\in h(I_2)$;
\item $\xi(z)$ points on the left or below and is not a positive multiple of $(1,1)$ if $z\in h(I_3)$;
\item $\xi(z)$ points on the left or above and is not a positive multiple of $(1,-1)$ if $z\in h(I_4)$.
\end{itemize}
It is sufficient to prove the first item, which is the consequence of the fact that every point in $I_1$ belongs either to $\beta$ or to $\widetilde g^{1-K}(\delta')$. One can construct a continuous function  $\xi': h(\partial U_{\bf e})\to \R^2\setminus\{0\}$  such that
\begin{itemize} 
\item $\xi(z)=(-1,-1)$ if $z\in h(I_1)$;
\item $\xi(z)=(-1,1)$ if $z\in h(I_2)$;
\item $\xi(z)=(1,1)$ if $z\in h(I_3)$;
\item $\xi(z)=(1,-1)$ if $z\in h(I_4)$;
\item $\xi(z)$ points on the left if $z\in h(\beta)$,
\item $\xi(z)$ points on the right if $z\in h(\beta')$,
\item $\xi(z)$ points above if  $z\in h(\sigma)$,
\item $\xi(z)$ points below if  $z\in h(\sigma')$.
\end{itemize}
The degree of $\xi'$ is equal to $-1$, moreover $\xi(z)$ is not a multiple of $\xi(z)$ for every $z\in h(\partial U_{\bf e})$. This implies that $i(\widetilde g^k, \partial U_{\bf e})=-1$.
In case $\delta'<\delta$, on can prove that $i(\widetilde g^k, \partial U_{\bf e})=1$ by an analogous method, like in the proof of Proposition \ref{prop: preliminary-realization}.
\end{proof}}

\begin{remark} \textcolor{black}{By looking more carefully at the proof of Lemma \ref{lemma:Markov}, we could be able to construct a Jordan domain $U_{\bf e}$ satisfying the conditions of Lemma  \ref{lemma:Markov} and such that $i(\widetilde g^k, \partial U_{\bf e})=\varepsilon$, where $\varepsilon\in\{-1,1\}$ depends {explicitly} on the sequence  ${\bf e}=(\varepsilon_{k})_{0\leq k<K}$ and on the relative position
of $\Delta_1$ and $\Delta_2$. Suppose for instance that $\delta_1<\delta_2<\delta_3<\delta_4$, meaning that $i(\widetilde g, \partial \Delta_1)=+1$ and $i(\widetilde g, \partial \Delta_2)=-1$. Set $K=2$ and consider the sequence $(1,2)$. The fact that $\widetilde g(\delta_2)$ is on $\widetilde\phi_b$ and $\widetilde g(\delta_4)$ is on $T^{-1}(\widetilde\phi_b)$  tells us that the sides $\delta$, $\delta'$ mentioned in the proof of Lemma \ref{lemma:periodique} are respectively $\delta_3$ and $\delta_4$. The fact that $\widetilde g(\delta_1)$ is on $T^{-1}(\widetilde\phi_b)$ and $\widetilde g(\delta_2)$ is  on $\widetilde\phi_b$ implies that $U_{\bf e}$ can be constructed in such a way that $\sigma'<\sigma$. Consequently we would have $i(\widetilde g^{2}, \partial U_{\bf e}) =1$. Similarly, in case the sequence is $(1,1)$ or $(2,2)$, then $U_{\bf e}$ can be constructed such that $i(\widetilde g^{2}, \partial U_{\bf e}) =-1$. More generally, in the case where $\delta_1<\delta_2<\delta_3<\delta_4$, for a given sequence  ${\bf e}=(\varepsilon_{k})_{0\leq k<K}$ one can find $U_{\bf e}$ such  $i(\widetilde g^{K}, \partial U_{\bf e}) =(-1)^r$, where $r$ is the {cardinality of the set of indices $0\le k <K$ such that $\varepsilon_{k}=1$.} What has been done in this section, in particular what is explained in this remark, is a very simple example of  {\it Conley index theory}, more precisely of {\it homological Conley index theory}.  This theory describes some aspects of the dynamics in a neighborhood of an isolated compact invariant set. What happens in case this invariant set can be partitioned into two compact isolated invariant sets (as it is the case here) is well described (in a general abstract situation) in \cite{Szymczak}. The proof given in this present article is a special case of what has been done in the unpublished article \cite{LeCalvezYoccoz2}.}
\end{remark}

\subsection{Existence of horseshoe in an annular covering space}
We will construct other ``geometric horseshoes'' in this subsection, not in the universal covering space of the domain of a maximal isotopy but in an annular covering space. The first interest of these constructions is that they will permit us to get the constant $\log 4/3q$ in the statement of Theorem  \ref{th: horseshoe}. The second interest is that the horsehoes that we will construct are {\it rotational horseshoes}, as defined for instance in \cite{PassegiPotrieSambarino}, and we will deduce some additional dynamical properties related to rotation numbers.

We consider in this subsection, the annular covering $\mathring{\mathrm{dom}}(I)=\widetilde{\mathrm{dom}}(I)/T^2$. Denote $\mathring I$ the induced identity isotopy, $\mathring f$ the induced lift of $f$, and $\mathring{\mathcal{F}}$ the induced foliation.

\begin{proposition} \label{prop: horseshoe-annular}
 If $q\geq 2$, there exists a compact set $\mathring Z_{q}\subset\mathring{\mathrm{dom}}(I)$, invariant by $\mathring f^q$ such that
\begin{itemize}
\item the restriction of $\mathring f^q$ to $\mathring Z_{q}$ is an extension of the Bernouilli shift $\sigma :\{1,\dots, 2q-2\}^{\Z}\to \{1,\dots, 2q-2\}^{\Z}$;
\item every $s$-periodic sequence of $\{1,\dots, 2q-2\}^{\Z}$ has a preimage by the factor map $\mathring H_{q}:\mathring Z_{q}\to \{1,\dots, 2q-2\}^{\Z}$ that is a $s$-periodic point of $\mathring f^q$.
\end{itemize}
\end{proposition}

\begin{proof} We suppose that $q\geq 2$. By Lemma \ref{lemma: paths}, we can find a decreasing or increasing sequence of pairwise disjoint paths in $\mathcal X$, denoted  $(\delta_i)_{1\leq i\leq 2q-2}$, such that, for every $l\in\{1,\dots, q-1\}$, one has
$$\delta_{2l-1}\in \mathcal X_{l,q},\enskip \delta_{2l}\in \mathcal X_{l+1,q}.$$
The rectangles of the family $\left(T^k\left(\overline\Delta_{\delta_{2l-1,2l}}\right)\right)_{1\leq l<q, k\in\Z}$ are pairwise disjoint and define by projection in $\mathring{\mathrm{dom}}(I)$  a family $(\mathring\Delta_i)_{1\leq i\leq 2q-2}$ of rectangles such that 
\begin{itemize}
\item $\mathring\Delta_i$ is the projection of $\overline\Delta_{\delta_{2i-1,2i}}$ if $1\leq i\leq q-1$,
\item $\mathring\Delta_i$ is the projection of $T(\overline\Delta_{\delta_{2(i-q+1)-1,2(i-q+1)}})$ if $q\leq i\leq 2q-2$.
\end{itemize}
See Figure \ref{figure_horseshoe_annular2}.

\begin{figure}[ht!]
\hfill
\includegraphics [height=48mm]{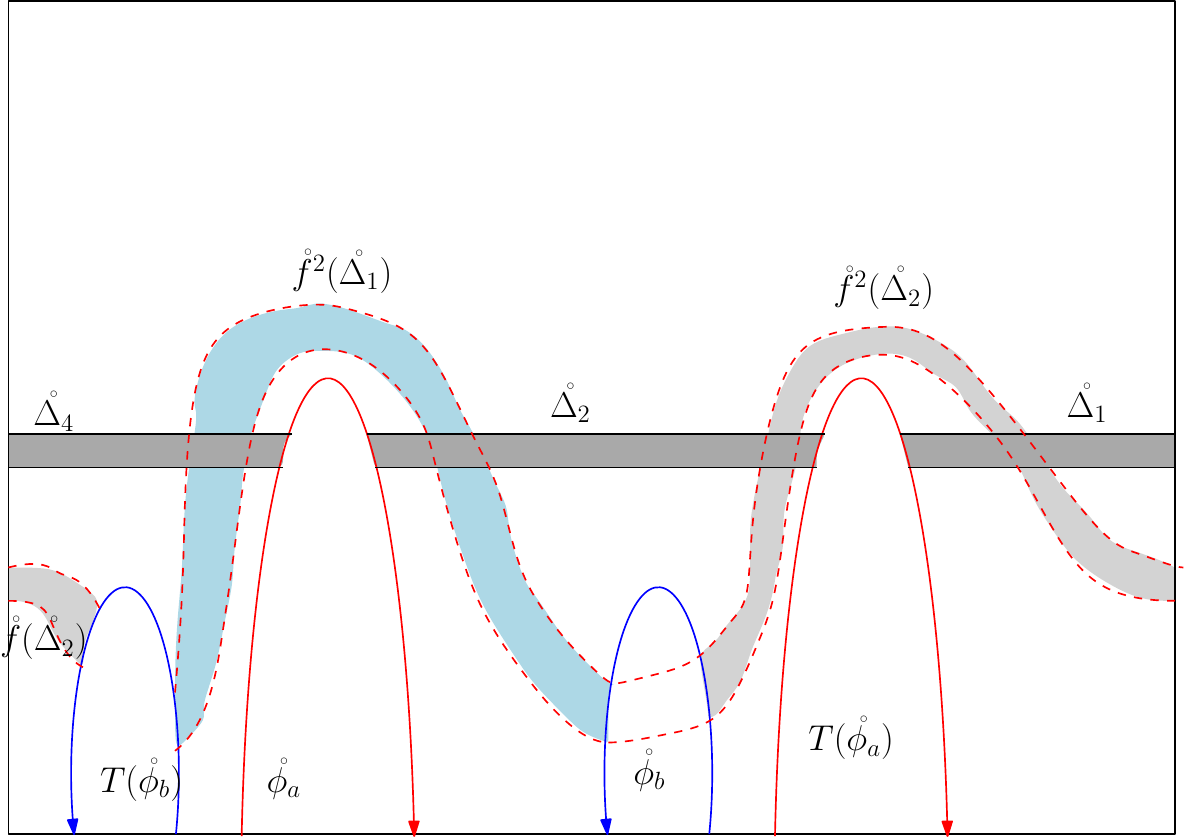}
\hfill{}
\caption{\small Horseshoe in the annular covering $\mathring{\mathrm{dom}}(I)$ - Case $q$=2.}
\label{figure_horseshoe_annular2}
\end{figure}

Here again $\bigcup_{1\leq i\leq 2q-2} \mathring\Delta_i$ is an isolating block and $\mathring Z_q=\bigcap_{k\in\Z} \mathring f^{-kq}(\bigcup_{1\leq i\leq 2q-2} \mathring\Delta_i)$ is a compact set included in the interior of $\bigcup_{1\leq i\leq 2q-2} \mathring\Delta_i$. Denote $\mathring H_q: \mathring Z_q\to \{1,\dots, 2q-2\}^{\Z}$ the map that associates to a point $\mathring z\in\mathring Z_q$ the sequence $(\varepsilon_k)_{k\in\Z}$ such that $\mathring f^{qk}(\mathring z)\in \Delta_{\varepsilon_k}$. It is continuous and satisfies $\mathring H_q\circ\mathring f^q=\sigma\circ \mathring H_q$. \textcolor{black}{ If we prove that every periodic sequence of period $K$ is the image by $\mathring H_q$ of a periodic point of $\mathring f^q$ of period $K$, it will implies that $\mathring H_q$ is onto, because its image is compact and we will get a complete proof of Proposition \ref{prop: horseshoe-annular}. In fact, adapting the proofs of Lemma \ref{lemma:Markov} and Lemma \ref{lemma:periodique} to this present situation, we can prove that  for every sequence ${\bf e}=(\varepsilon_{k})_{0\leq k<K}$ in $\{1,\dots, 2q-2\}^K$ there exists a Jordan domain $U_{\bf e}\subset\bigcap_{ 0\leq k<K} \mathring f^{-qk}( \Delta_{\varepsilon_k})$ such that $i(\mathring f^{qK}, \partial U_{\bf e})$ is equal to $+1$ or $-1$. Consequently $U_{\bf e}$ contains a fixed point of $\mathring f^{qK}$.}\end{proof}
 
 \begin{remark} The full topological horseshoe $\mathring Z_q$ constructed in Proposition \ref{prop: horseshoe-annular} on the double annular covering 
 $\mathring{\mathrm{dom}}(I)$ is relevant to obtain the bound on the topological entropy, but we are also interested in applications concerning rotation numbers. This is easier to do on the annular covering space $\widehat{\mathrm{dom}}(I)=\widetilde{\mathrm{dom}}(I)/T$ for which $\mathring{\mathrm{dom}}(I)$ is a $2$-fold covering. One can consider the projection $P:\mathring{\mathrm{dom}}(I)\to\widehat{\mathrm{dom}}(I)$, and the set $\widehat Z_q=P(\mathring Z_q)$, which is invariant by $\widehat f^q$, where $\widehat f$ is the natural lift of $f\vert_{\mathrm {dom}(I)}$ to $\widehat{\mathrm{dom}}(I)$. Furthermore, if we define $\widehat H_q:\widehat Z^q\to \{1, \dots, q-1\}^{\Z}$, so that the $l$-th letter of $\widehat H_q(\widehat z)$ is $i$ if  the $l$-th letter of $\mathring H_q(\mathring z)$ is either $i$ or $i+q-1$, where $\mathring{z}\in P^{-1}(\widehat{z})$, then the restriction of $\widehat f^q$ to $\widehat Z_q$ is semi-conjugated, by $\widehat H_q$, to the Bernouilli shift on $\{1, \dots, q-1\}^{\Z}$.
 
We will identify a compact invariant subset $\widehat Z_q'$ of $\widehat Z_q$, such that the restriction of $\widehat{f}^q$ to this set is still semi-conjugated by $\widehat{H_q}$ to the Bernouilli shift, and that is a {\it rotational horseshoe},  in the sense that the existence of a rotation number for a point $\widehat z\in\widehat Z_q'$ (and its value) can be determined by the knowledge of the sequence $\widehat H_q(\widehat z)$. 
  
Keeping the notations of Proposition \ref{prop: horseshoe-annular}, let $\overline \Delta_0\subset \widetilde{\mathrm{dom}}(I)$ be the set bounded by $\delta_1, \delta_{2q-2},\widetilde \phi_a$ and $T(\widetilde \phi_a)$, and note that, for $1\le i\le q-1$, $\widetilde f^{q}(\overline \Delta_{\delta_{2i-1},\delta_{2i}})$ intersects both $T^{i}(\Delta_0)$ and $T^{i+1}(\Delta_0)$. We set $\overline \Delta_i'= \overline \Delta_{\delta_{2i-1},\delta_{2i}}\cap \widetilde f^{-q}(T^{i}(\Delta_0))$, which projects, in $\widehat{\mathrm{dom}}(I)$ to a subset $\widehat \Delta_i'$ of $P(\mathring \Delta_i)$, see figure \ref{rotationalHorseshoe}. The set $\widehat Z_q'$ is defined as $\bigcap_{k\in\Z}\widehat f^{-kq}\left(\bigcup_{1\le i\le q-1}\widehat \Delta_i'\right)$.  Fix $\widehat z\in \widehat Z_q'$ and write  $(\varepsilon_k)_{k\in\Z}=\widehat H_q(\widehat z)$. Fix $k\geq 1$ and set $\alpha_n=\sum_{0\leq k<n} \varepsilon_k$. If $\widetilde z$ is the lift of $\widehat z$ that belongs to $\overline\Delta_{\delta_{2\varepsilon_0-1,2\varepsilon_0}}$, then $\widehat  f^{nq}(\widehat z)$ belongs to $T^{\alpha_n}(\overline\Delta_{\delta_{2\varepsilon_n-1,2\varepsilon_n-1}})$. In particular $\widehat z$ has a rotation number $\rho$ (for the lift $\widetilde f$) if and only if the sequence $(\alpha_n/n)_{n\geq 1}$ converges to $q\rho$. We deduce the following result:\end{remark}

\begin{figure}[ht!]
\hfill
\includegraphics [height=78mm]{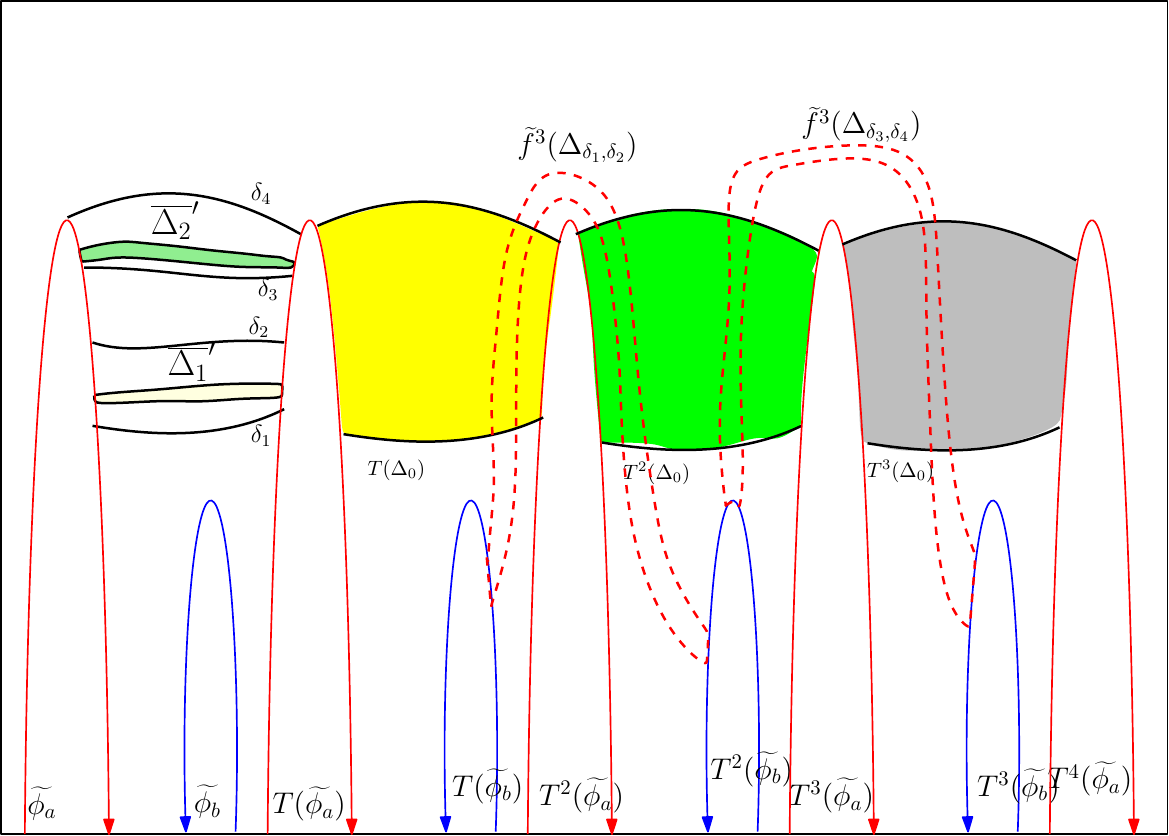}
\hfill{}
\caption{\small Rotational horseshoe for $q=3$. The sets $\overline{\Delta_1}'$ and $\overline{\Delta_2}'$ appear respectively as a light yellow shaded region and a light green shaded region.}
\label{rotationalHorseshoe}
\end{figure}

\begin{proposition}\label{prop: irrational compact set}
For any $0<\rho\le 1$, there exists a nonempty minimal set $\widehat N_{\rho}$ such that, for each $\widehat z\in \widehat N_{\rho}$, one has that $\mathrm{rot}_{\widetilde f}(\widehat z)=\rho.$
\end{proposition}

\begin{proof} If $\rho$ is rational, the result is an immediate consequence of Proposition \ref{prop: preliminary-realization} and $\widehat N_{\rho}$ can be chosen as a periodic orbit. We suppose by now that $\rho$ is irrational.
It suffices to show that there exists an integer $q\geq 3$ and a nonempty closed set $X_{\rho}$ of $\{1, \dots, q-1\}^{\Z}$, invariant by the Bernouilli shift and such that for every $(\varepsilon_k)_{k\in\Z}\in X_{\rho}$, one has 
$$\lim_{n\to+\infty} {1\over n} \sum_{k=0}^n \varepsilon _k=q\rho.$$
Indeed, every minimal subset of the compact set $\bigcup_{0\leq l<q}\widehat f^l(\widehat H_q^{-1}(X_{\rho}))$ will satisfy the prescribed properties.

The existence of $X_{\rho}$ is a well known fact concerning Sturmian sequences. Set 
$$\rho'= {q\over q-2}\,\rho-{1\over q-2},$$choosing $q$ large enough to insure that $\rho'\in(0,1)$.
 Consider the partition of $[0,1[=[0, 1-\rho'[\cup [1-\rho', 1[$ and define a sequence ${\bf e}'=(\varepsilon'_k)_{k\in\Z}$ in the following way
$$ \varepsilon'_k= \begin{cases}0 & \mathrm {if}\enskip k\rho'-[k\rho']\in [0,1-\rho'[,\\1 & \mathrm {if}\enskip k\rho'-[k\rho']\in [1-\rho',1[.\\
\end{cases}$$
Every sequence ${\bf e}=(\varepsilon_k)_{k\in\Z}$ in the closure $X'$ of the orbit of ${\bf e}'$ by the Bernouilli shift defined on $\{0,1\}^{\Z}$ satisfies
$$\lim_{n\to+\infty} {1\over n} \sum_{k=0}^n \varepsilon _k=\rho'.$$ 
So the set
$$X_{\rho}=\{(1+(q-2)\varepsilon_k)_{k\in\Z}\,\vert \, (\varepsilon_k)_{k\in\Z}\in X'\}$$ satisfies the prescribed properties.
\end{proof}

\subsection{Proofs of Theorem \ref{th: realization} and Theorem \ref{th: horseshoe}}\label{subsec:TheoMandN}

 Let us begin by proving the realization theorem (Theorem  \ref{th: realization}). Let us remind the statement:
 
\begin{theorem}
Let $M$ be an oriented surface, $f$ a homeomorphism of $M$ isotopic to the identity, $I$ a maximal identity isotopy of $f$ and $\mathcal F$ a foliation transverse to $I$. Suppose that  $\gamma:[a,b]\to \mathrm{dom}(I)$ is an admissible path of order $r$ with a $\mathcal{F}$-transverse self intersection at $\gamma(s)=\gamma(t)$, where $s<t$. Let $\widetilde \gamma$ be a lift of $\gamma$ to the universal covering space $\widetilde{\mathrm{dom}}(I)$ of $\mathrm{dom}(I)$ and $T$ the covering automorphism such that $\widetilde\gamma$ and $T(\widetilde\gamma)$ have a $\widetilde{\mathcal F}$-transverse intersection at  $\widetilde\gamma(t)=T(\widetilde\gamma)(s)$. Let $\widetilde f$ be the lift of $f\vert_{\mathrm{dom}(I)}$ to $\widetilde{\mathrm{dom}}(I)$, time-one map of the identity isotopy that lifts $I$, and $\widehat f$ the homeomorphism of the annular covering space $ \widehat{\mathrm{dom}}(I)=\widetilde {\mathrm{dom}}(I)/T$ lifted by $\widetilde f$. Then we have the following:
 \begin{enumerate}
 \item For every rational number $p/q\in(0,1]$, written in an irreducible way, there exists a point $\widetilde z\in \widetilde{\mathrm{dom}}(I)$ such that $\widetilde f^{qr}(\widetilde z)=T^p(\widetilde z)$ and such that $\widetilde I_{\widetilde{\mathcal F}}^{\Z}(\widetilde z)$ is equivalent to $\prod_{k\in\Z} T^k(\widetilde\gamma_{[s,t]})$.
  \item For every irrational number $\rho\in[0,1/r]$, there exists a compact set $\widehat Z_{\rho}\subset\widehat{\mathrm{dom}}(I)$ invariant by $\widehat f$, such that every point $\widehat z\in \widehat Z_{\rho}$ has a rotation number $\mathrm{rot}_{\widetilde f}(\widehat z)$ equal to $\rho$. Moreover if $\widetilde z\in\widetilde{\mathrm{dom}}(I)$ is a lift of $\widehat z$, then $\widetilde I_{\widetilde{\mathcal F}}^{\Z}(\widetilde z)$ is equivalent to $\prod_{k\in\Z} T^k(\widetilde\gamma_{[s,t]})$.
 \end{enumerate}
\end{theorem}

\begin{proof} We keep the notations introduced at the beginning of the section (but replacing $f$ with $f^r$). In particular we have a foliation $\widehat{\mathcal F}$ defined on the sphere $\widehat{\mathrm{dom}}(I)\cup\{S,N\}$ with two singular points $S$ and $N$. We will frequently use classical properties of foliations on a sphere in what follows. The path $\widetilde\gamma\vert_{[s,t]}$  projects onto a closed path of $\widehat{\mathrm{dom}}(I)$ that defines naturally an essential loop $\widehat{\Gamma}$. This loop is simple because it is transverse to $\widehat{\mathcal F}$. There is no loss of generality by supposing that that $S$ is on the right of $\widehat\Gamma$ and  $N$ on its left. Write $U_{\widehat\Gamma}$ for the union of leaves that meet $\widehat\Gamma$, it is an open annulus. We have the following results:
 \begin{itemize}
 \item Every leaf in $U_{\widehat\Gamma}$ has a $\omega$-limit set equal to $S$.
\item  \textcolor{black}{In the case where $\widehat{\mathcal F}$ has a closed leaf, the frontier of $U_{\widehat \Gamma}$ contains a closed leaf $\widehat \phi$ that is the $\alpha$-limit set of every leaf in $U_{\widehat\Gamma}$. Moreover every simple loop transverse to $\widehat{\mathcal F}$ that is not equivalent to $\widehat \Gamma$ is essential, disjoint from $U_{\widehat\Gamma}$ and separated from this set by $\widehat\phi$.}
\item  \textcolor{black}{ In the case where $\widehat{\mathcal F}$ has no closed leaf, then every leaf of $U_{\widehat\Gamma}$ is adherent to $N$.}
\end{itemize}
The first property can be deduced from the fact that $\widehat \gamma$ joins $\widehat\phi_{\widehat\gamma(a)}$, which is homoclinic to $S$, to $\widehat\phi_{\widehat\gamma(b)}$ which is also homoclinic to $S$.

Recall that there exists $k_0\geq 1$ such that $\widetilde\gamma$ and $T^k(\widetilde\gamma)$ have a $\widetilde{\mathcal{F}}$-transverse intersection if $1\leq k\leq k_0$ and no transverse intersection if $k>k_0$.  As explained above, one gets a simple loop  $\widehat\Gamma_{k_0}$ of $\widehat{\mathrm{dom}}(I)_{k_0}=\widetilde{\mathrm{dom}}(I)/T^{k_0}$ transverse to the foliation $\widehat{\mathcal F}_{k_0}$ naturally defined on $\widehat{\mathrm{dom}}(I)_{k_0}$. Of course $\widehat{\mathrm{dom}}(I)_{k_0}$ is the $k_0$-fold finite covering of  $\widehat{\mathrm{dom}}(I)$ and $\widehat{\mathcal F}_{k_0}$ the lifted foliation. Applying $(1)$ both to $\widehat\Gamma$ and $\widehat\Gamma_{k_0}$ and $(2)$ to $\widehat\Gamma_{k_0}$, one deduces that $\widehat\Gamma_{k_0}$ is equivalent to the lift of $\widehat\Gamma$ to  $\widehat{\mathrm{dom}}(I)_{k_0}$. This means that if $\widetilde\gamma$ and $T^k(\widetilde\gamma)$ have a $\widetilde{\mathcal{F}}$-transverse intersection  at $\widetilde \gamma(t') =T^{k_0}(\widetilde \gamma(s'))$, then $s'<t'$ and $\prod_{k\in\Z} T^k(\widetilde\gamma\vert_{[s,t]})$ and $\prod_{k\in\Z} T^{kk_0}(\widetilde\gamma\vert_{[s',t']})$ are equivalent.

One can apply what has been done in this section to $\widetilde f^r$ and $T^{k_0}$. Applying Corollary \ref{cor: preliminary-realization}, one can find for every $p/q\in(0,1]$ written in an irreducible way, a point $\widetilde z_{k_0}\in \widetilde{\mathrm{dom}}(I)$ such that $\widetilde f^{qk_0r}(\widetilde z_{k_0})=T^{pk_0}(\widetilde z_{k_0})$. Indeed $p/qq_0\in(0,1]$. This implies that the $k_0$-th iterate of $\widetilde f^{qr}\circ T^{-p}$ has a fixed point. By Brouwer theory, one deduces that $f^{qr}\circ T^{-p}$ itself has fixed point, which means a point $\widetilde z\in \widetilde{\mathrm{dom}}(I)$ such that $\widetilde f^{qr}(\widetilde z)=T^p(\widetilde z)$. The point projects onto a periodic point $\widehat z$ whose transverse trajectory defines an essential transverse loop. 

Suppose first that \textcolor{black}{we are in the case where $\widehat{\mathcal F}$ has no closed leaf and every leaf of $U_{\widehat\Gamma}$ is adherent to $N$}. In this case, there is no essential transverse loop but the powers of $\widehat\Gamma$, so $\widetilde I_{\widetilde{\mathcal F}}^{\Z}(\widetilde z)$ is equivalent to $\prod_{k\in\Z} T^k(\widetilde\gamma_{[s,t]})$. 

Suppose now that \textcolor{black}{we are in the case where the frontier of $U_{\widehat \Gamma}$ contains a closed leaf $\widehat \phi$.} One of the sets $\widehat f(\overline{U_{\widehat\Gamma}})$ or $\widehat f^{-1}(\overline{U_{\widehat\Gamma}})$ is included in $U_{\widehat \Gamma}$. In that case, $\widehat A=\bigcup_{k\in\Z} \widehat f^{r}(U_{\widehat\Gamma})$ is an essential invariant annulus and its preimage $\widetilde A=\widehat\pi^{-1}(\widehat A)$ is a connected and simply connected domain. But $\widetilde A$ contains the closed curve $C$ that appears in the proof of Proposition \ref{prop: preliminary-realization} and so the point $\widetilde z_{k_0}$ can be chosen to belong to $\widetilde A$. It is the same for $\widetilde z$ because one can apply Brouwer theory to the restriction of $\widetilde f^{qr}\circ T^{-p}$ to $\widetilde A$. But this implies that  the orbit of $\widetilde z$ projects onto a periodic orbit included in $U_{\widehat\Gamma}$. One deduces that $\widetilde I_{\widetilde{\mathcal F}}^{\Z}(\widetilde z)$ is equivalent to $\prod_{k\in\Z} T^k(\widetilde\gamma_{[s,t]})$. 

The item (2) can be proven similarly using Proposition \ref{prop: irrational compact set}.
\hfill$\Box$

We conclude by proving Theorem \ref{th: horseshoe}. Let us
remind its statement.

\begin{theorem}  \label{th: horseshoe2}Let $M$ be an oriented surface, $f$ a homeomorphism of $M$ isotopic to the identity, $I$ a maximal identity isotopy of $f$ and $\mathcal F$ a foliation transverse to $I$. If there exists a point $z$ in the domain of $I$ and an integer $r\geq 1$ such that the transverse trajectory $I_{\mathcal F}^{r}(z)$ has a $\mathcal F$-transverse self-intersection, then $f$ has a topological horseshoe. Moreover, the entropy of $f$ is  at least equal to $\log 4/3r$.
\end{theorem}

\begin{proof} Let us apply Proposition \ref{prop: horseshoe-universal}: there exist a covering transformation $T$ of the universal covering space $\widetilde{\mathrm{dom}}(I)$ and for every $q\geq 3$ and $2\leq p<q$ a compact set $\widetilde Z_{p,q}\subset\widetilde{\mathrm{dom}}(I)$, invariant by $\widetilde f^{qr}\circ T^{-p}$  such that
\begin{itemize}
\item the restriction of $\widetilde f^{qr}\circ T^{-p}$ to $\widetilde Z_{p,q}$ is an extension of the Bernouilli shift $\sigma :\{1,2\}^{\Z}\to \{1,2\}^{\Z}$;
\item every $s$-periodic sequence of $\{1,2\}^{\Z}$ has a preimage by the factor map $\widetilde H_{p,q}:\widetilde Z_{p,q}\to \{1,2\}^{\Z}$ that is a $s$-periodic point of $\widetilde f^{qr}\circ T^{-p}$.
\end{itemize}
The covering projection $\widetilde\pi: \widetilde{\mathrm{dom}}(I)\to \mathrm{dom}(I)$ induces a semi-conjugacy between $\widetilde f^{qr}\circ T^{-p}\vert _{\widetilde Z_{p,q}}$ and $f^{qr}\vert _{\widetilde\pi(\widetilde Z_{p,q})}$. The set $\widetilde Z_{p,q}$ being compact, projects onto a compact subset $Z_{p,q}$ of  $\mathrm{dom}(I)$ and every point in $Z_{p,q}$ has finitely many lifts in $\widetilde Z_{p,q}$, with an uniform upper bound. The set  $Z_{p,q}$ is a topological horseshoe, as defined in the introduction.

One could have applied Proposition \ref{prop: horseshoe-annular}: for every $q\geq 2$, there exists a topological horseshoe $Z_{q}$, image by the covering projection $\mathring \pi: \mathring{\mathrm{dom}}(I)\to \mathrm{dom}(I)$ of the set $\mathring Z_{q}$ defined by Proposition \ref{prop: horseshoe-annular}. One has
$$ h(f)\geq  {1\over qr} h(f^{qr}\vert_{Z_q})={1\over qr} h(\mathring f^{qr}\vert_{\mathring Z_q})\geq {1\over qr}\log (2q-2).$$
Noting that the function $$ q\mapsto {1\over q}\log (2q-2)$$ reaches its maximum in $\N$ for $q=3$, one concludes that $$ h(f)\geq  {\log 4\over 3r}.$$ \end{proof}

\section{Paths on the sphere with no $\mathcal F$ transverse self-intersection and applications}

 \subsection{Definitions} We suppose in this section that $\mathcal F$ is a singular oriented foliation on the sphere $\S^2$.

\begin{definition*}Let $J_1$ and $J_2$ be two real intervals. We will say that $J_1$ is {\it on the left} of $J_2$ (equivalently $J_2$ is {\it on the right} of $J_1$) if the lower end of $J_2$ is not smaller than the lower end of $J_1$ and the upper end of $J_2$ is not smaller that the upper end of $J_1$.  In particular $J_1$ is on the left and on the right of itself.
\end{definition*}

\begin{definitions*} Let $\gamma:J\to\mathrm {dom} ({\mathcal F})$ be a transverse path and  $\Gamma_0$ a transverse loop. 

\begin{itemize}
\item The path $\gamma$ {\it is drawn on} $\Gamma_0$ if it is equivalent to a sub-path of the natural lift $\gamma_0$ of $\Gamma_0$; 
\item The path $\gamma$ {\it draws} $\Gamma_0$ if there exist $a<b$ in $J$ and $t\in\R$ such that $\gamma\vert_{[a,b]}$ is equivalent to $\gamma_0\vert_{[t,t+1]}$.
\item The path $\gamma$ {\it draws infinitely} $\Gamma_0$ if there exist $c\in J$ and $t\in\R$ such that $\gamma\vert_{(\inf J,c]}$ is equivalent to $\gamma_0\vert_{(-\infty, t]}$,  or if $\gamma\vert_{[c,\sup J)}$ is equivalent to $\gamma_0\vert_{[t,+\infty)}$. 
\item The path $\gamma$ {\it exactly draws} $\Gamma_0$ if it draws $\Gamma_0$ and is drawn on $\Gamma_0$.
\item The path $\gamma$ {\it exactly draws infinitely } $\Gamma_0$ if it draws infinitely $\Gamma_0$ and is drawn on $\Gamma_0$.
\end{itemize}
\end{definitions*}

\begin{remark*}
In case $\Gamma_0$ is a transverse simple loop, $\gamma$ draws $\Gamma_0$ if there exist $a<b$ in $J$ such that $\gamma\vert_{[a,b]}$ is included in $U_{\Gamma_0}$ and meets every leaf at least once (or equivalently meets a leaf twice). One proves easily that a path that draws a transverse loop draws a transverse simple loop (see \cite{LeCalvezTal}).
\end{remark*}

 \begin{definitions*}Let $\gamma:J\to\mathrm {dom} ({\mathcal F})$ be a transverse path and  $\Gamma_0$ a transverse simple loop. We set $$J_{\Gamma_0}=\{t\in J\,\vert\, \gamma(t)\in U_{\Gamma_0}\}.$$ \begin{itemize}
 
 \item A  connected component $J_0$ of $J_{\Gamma_0}$ is a {\it drawing component} if $\gamma\vert_{J_0}$ draws $\Gamma_0$. 
\item A  connected component $J_1$ of $J_{\Gamma_0}$ is a {\it crossing component} and $\gamma\vert_{J_1}$ {\it crosses} $\Gamma_0$ if both ends $a$, $b$ of  $J_1$ are in $J$ and $\gamma(a)$ and $\gamma(b)$ belong to different components of $\S^2\setminus U_{\Gamma}$. 

\item The path $\gamma$ {\it crosses} $\Gamma_0$ if there is at least one crossing component.\end{itemize}
 \end{definitions*}

\subsection{Intersection of paths and loops}

\begin{proposition}
\label{prop:drawing-crossing}
Suppose that $\gamma:J\to\mathrm {dom} ({\mathcal F})$ is  a transverse path with no $\mathcal F$-transverse self-intersection. Then:

\begin{enumerate}\item if $\gamma$ draws a transverse simple loop $\Gamma_0$, there exists a unique drawing component of $J_{\Gamma_0}$;

\item if $\gamma$ draws and crosses a transverse simple loop $\Gamma_0$, there exists a unique crossing component of $J_{\Gamma_0}$ and it is the drawing component of $J_{\Gamma_0}$;

\item if $\gamma$ draws a transverse simple loop $\Gamma_0$ and does not cross it, then the drawing component \textcolor{black}{contains a neighborhood of at least one end of 
$J$};

\item if $\gamma$ draws two non-equivalent transverse simple loops $\Gamma_0$ and $\Gamma_1$, then the drawing component  of $J_{\Gamma_0}$ is on the right of the drawing component  of $J_{\Gamma_1}$ or on its left.
\end{enumerate} 
\end{proposition}

\begin{proof} 
\medskip
\textcolor{black}{We will begin by proving $(1)$ in  the case where $\gamma$ does not cross $\Gamma_0$.} Suppose that $J_0$ and $J_1$ are two drawing components of $J_{\Gamma_0}$ and that $J_1$ is on the right of $J_0$. The upper end $b_0$ of $J_0$ and the lower end $a_1$ of $J_1$ are finite, belong to $J$, and none of the points $\gamma(b_0)$, $\gamma(a_1)$ belong to $U_{\Gamma_0}$. The fact that $\gamma$ does not cross $\Gamma_0$ implies that $\gamma(b_0)$ and $\gamma(a_1)$ belong to the same connected component of $\S^2\setminus U_{\Gamma_0}$. \textcolor{black}{It follows that there exist, on the same connected component $X$ of $\S^2\setminus U_{\Gamma_0}$,} \textcolor{black}{two boundary leaves,  such that $U_{\Gamma_0}$ is  locally on the right of one of them and on the left of the other one, and so there is no spiraling phenomena: the leaves of $U_{\Gamma_0}$ cannot accumulate on the boundary leaves that are in $X$.  This implies that for every $c_0\in J_0$ and $c_1\in J_1$ the paths $\gamma\vert_{[c_0,b_0)}$ and $\gamma\vert_{(a_1, c_1]}$ meet every leaf of $U_{\Gamma_0}$ finitely many times and one can choose $c_0$ and $c_1$ such that every leaf is met exactly once by $\gamma_{[c_0,b_0)}$ and $\gamma\vert_{(a_1,c_1]}$. }   Consequently, there exist $t_0$, $t_1$, with $t_0\leq t_1<t_0+1$ such that $\gamma\vert_{[c_0,b_0)}$ is equivalent to $\gamma_0\vert_{[t_0,t_0+1)}$ and $\gamma\vert_{(a_1,c_1]}$ is equivalent to $\gamma_0\vert_{(t_1,t_1+1]}$. \textcolor{black}{Now, fix a lift $\widetilde\gamma_0$ of $\gamma_0$ to the universal  covering space $\widetilde{\mathrm {dom}}({\mathcal F})$ of $\mathrm {dom}({\mathcal F})$ and denote $U_{\widetilde \gamma_0}$ the union of the leaves of the lifted foliation $\widetilde{\mathcal F}$ that meet $\widetilde \gamma_0$. There exists a unique lift $\widetilde \gamma$ of $\gamma$ such that $\widetilde\gamma\vert_{[c_0,b_0)}$ is equivalent to $\widetilde\gamma_0\vert_{[t_0,t_0+1)}$ and a unique lift $\widetilde\gamma'$ of $\gamma$  such that $\widetilde\gamma'\vert_{(a_1,c_1]}$ is equivalent to $\widetilde\gamma_0\vert_{(t_1,t_1+1]}$. \textcolor{black} {For any given $t\in (t_1,t_0+1)$, there exists $d_0=d_0(t)\in (c_0,b_0)$ and $d_1=d_1(t)\in (a_1,c_1)$} such that $\phi_{\widetilde\gamma(d_0)}=\phi_{\widetilde\gamma'(d_1)}=\phi_{\widetilde\gamma_0(t)}$.  Suppose first that $\gamma(b_0)$ and $\gamma(a_1)$ are on the right of $\Gamma_0$. In that case, when $t$ is close to $t_1$, then the paths $\widetilde\gamma\vert_{[c_0,d_0]}$ and $\tilde\gamma\vert_{[a_1,d_1]}$ are disjoint and $\widetilde\gamma(d_0)$ is located before $\widetilde\gamma'(d_1)$ on the leaf  
$\phi_{\widetilde\gamma_0(t)}$. So $\phi_{\widetilde\gamma(c_0)}$  is below $\phi_{\widetilde\gamma'(a_1)}$ relative to $\phi_{\widetilde\gamma(t)}$, for every $t\in(t_1,t_0+1)$ because it is the case when $t$ is close to $t_1$.  When $t$ is close to $t_0+1$, then the paths $\tilde\gamma\vert_{[d_0,b_0]}$ and $\tilde\gamma\vert_{[d_1,c_1]}$ are disjoint and $\widetilde\gamma(d_0)$ is located after $\widetilde\gamma'(d_1)$ on 
$\phi_{\widetilde\gamma_0(t)}$. So  $\phi_{\widetilde\gamma(b_0)}$  is above $\phi_{\widetilde\gamma'(c_1)}$ relative to $\phi_{\widetilde\gamma(t)}$  for every $t\in(t_1,t_0+1)$ because it is the case when $t$ is close to $t_0+1$.  In this situation, there exists at least one value of $t$ such that $\widetilde\gamma(d_0)=\widetilde\gamma'(d_1)$ and $\widetilde\gamma$ and $\widetilde \gamma'$ have a $\widetilde{\mathcal F}$-transverse intersection at $\widetilde\gamma(d_0)=\widetilde\gamma'(d_1)$. Consequently, $\gamma$ has a ${\mathcal F}$-transverse intersection which contradicts the hypothesis. In the case where  $\gamma(b_0)$ and $\gamma(a_1)$ are  both on the \textcolor{black}{left} of $\Gamma_0$, we get the same contradiction by proving that $\phi_{\widetilde\gamma(c_0)}$  is above $\phi_{\widetilde\gamma'(a_1)}$ relative to $\phi_{\widetilde\gamma(t)}$ and $\phi_{\widetilde\gamma(b_0)}$  is below $\phi_{\widetilde\gamma'(c_1)}$ relative to $\phi_{\widetilde\gamma(t)}$ for every $t\in(t_1,t_0+1)$,} \textcolor{black}{ see Figure \ref{figure_noncrossingcomponent}.}

\begin{figure}[ht!]
\hfill
\includegraphics [height=48mm]{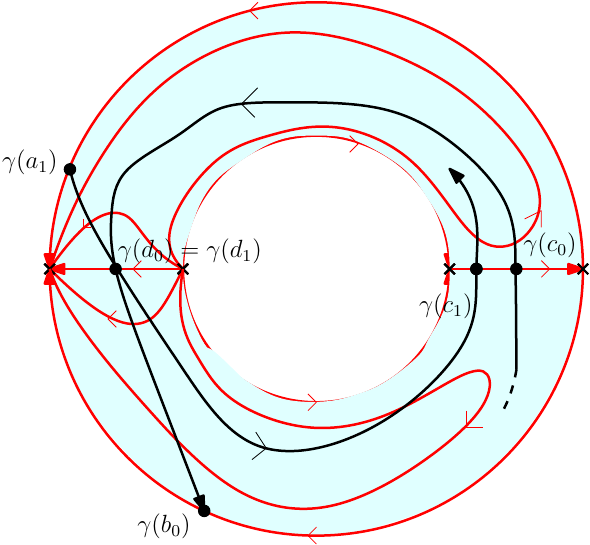}
\hfill{}
\caption{\small Proposition \ref{prop:drawing-crossing}. Depiction of why item (1) holds, when $\gamma$ does not cross $\Gamma_0$. Leafs of the foliation are drawn in red, transversal paths in black, and the annulus $U_{\Gamma_0}$ is the shaded region, both $\gamma(a_1)$ and $\gamma(b_0)$ are on the right of $\Gamma_0$.}
\label{figure_noncrossingcomponent}
\end{figure}

\textcolor{black}{ Now, let us suppose that $\gamma$ draws and crosses $\Gamma_0$. If we prove that every drawing component and every crossing component coincide, we will deduce that there is a unique drawing component, a unique  crossing component and that they coincide. We will obtain in \textcolor{black}{this} way the end of the proof of (1) and the proof of (2). We will argue by contradiction and suppose that it is not the case. Denote $\gamma_0$ the natural lift of $\Gamma_0$. Consider a drawing component $J_0$ and a crossing component $J_1$ and suppose that $J_0\not=J_1$. Like in the proof explained above, there are different cases to consider (in fact four). Nevertheless all situations are equivalent and we will restrict ourselves to a unique one.  First we will suppose  that $J_1=(a_1,b_1)$ is on the right of $J_0=(a_0,b_0)$. This implies that the upper end $b_0$ of $J_0$ is finite,  belongs to $J$, and that $\gamma(b_0)\not\in U_{\Gamma_0}$. One of the points $\gamma(a_1)$ or $\gamma(b_1)$ belongs to the same connected component of $\S^2\setminus U_{\Gamma_0}$ as $\gamma(b_0)$. We can suppose that it is $\gamma(a_1)$ \textcolor{black}{because otherwise, since $\gamma(b_0)$ and $\gamma(a_1)$ lie in different connected components of $\S^2\setminus U_{\Gamma_0}$ as $\gamma(b_0)$, there must exists another crossing component $J_2=(a_2, b_2)$ to the right of $J_0$ and to the left of $J_1$ such that $\gamma(b_0)$ and $\gamma(a_2)$ lie in the same connected components of $\S^2\setminus U_{\Gamma_0}$ }. Moreover we will suppose that $\gamma(b_0)$ and $\gamma(a_1)$ are on the right of $\Gamma_0$ (See Figure \ref{figure_crossingcomponent} for a depiction of construction below). Here again, for the reasons explained above in the proof of (1) for every $c_0\in J_0$ and $c_1\in J_1$, the paths $\gamma\vert_{[c_0,b_0)}$ and $\gamma\vert_{(a_1, c_1]}$ meet every leaf of $U_{\Gamma_0}$ finitely many times. Consequently one can choose $c_0$ such that every leaf is met exactly once by $\gamma\vert_{[c_0,b_0)}$.  Consequently, there exist $t_0$, $t_1$, $t'_1$, with $t_0\leq t_1<t_0+1$ such that $\gamma\vert_{[c_0,b_0)}$ is equivalent to $\gamma_0\vert_{[t_0,t_0+1)}$ and $\gamma\vert_{(a_1,b_1)}$ is equivalent to $\gamma_0\vert_{(t_1,t'_1)}$. There exists a unique lift $\widetilde \gamma$ of $\gamma$ such that $\widetilde\gamma\vert_{[c_0,b_0)}$ is equivalent to $\widetilde\gamma_0\vert_{[t_0,t_0+1)}$ and a unique lift $\widetilde\gamma'$ of $\gamma$  such that $\widetilde\gamma'\vert_{(a_1,b_1)}$ is equivalent to $\widetilde\gamma_0\vert_{(t_1,t'_1)}$. Set $t_2=\min(t'_1, t_0+1)$. \textcolor{black}{ As before, if $t\in (t_1,t_2)$ then there exist $d_0=d_0(t)\in (c_0,b_0)$ and $d_1=d_1(t)\in (a_1,b_1)$ such that $\phi_{\widetilde\gamma(d_0)}=\phi_{\widetilde\gamma'(d_1)}=\phi_{\widetilde\gamma_0(t)}$. By the same argument as before, one has that} $\phi_{\widetilde\gamma(c_0)}$  is below $\phi_{\widetilde\gamma'(a_1)}$ relative to $\phi_{\widetilde\gamma(t)}$, for every $t\in(t_1,t_2)$. \textcolor{black}{Also as before, when $t$} is close to $t_2$, \textcolor{black}{irregardless if $t_2$ is equal to $t'_1$ or to $t_0+1$,} then the paths $\tilde\gamma\vert_{[d_0,b_0]}$ and $\tilde\gamma\vert_{[d_1,b_1]}$ are disjoint and $\widetilde\gamma(d_0)$ is located after $\widetilde\gamma'(d_1)$ on 
$\phi_{\widetilde\gamma_0(t)}$ . So  $\phi_{\widetilde\gamma(b_0)}$  is above $\phi_{\widetilde\gamma'(b_1)}$ relative to $\phi_{\widetilde\gamma(t)}$  for every $t\in(t_1,t_2)$ because it is the case when $t$ is close to $t_2$.  In this situation, there exists at least one value of $t$ such that $\widetilde\gamma(d_0)=\widetilde\gamma'(d_1)$ and $\widetilde\gamma$ and $\widetilde \gamma'$ have a $\widetilde{\mathcal F}$-transverse intersection at $\widetilde\gamma(d_0)=\widetilde\gamma'(d_1)$. Consequently, $\gamma$ has a ${\mathcal F}$-transverse intersection which contradicts the hypothesis. }

\begin{figure}[ht!]
\hfill
\includegraphics [height=48mm]{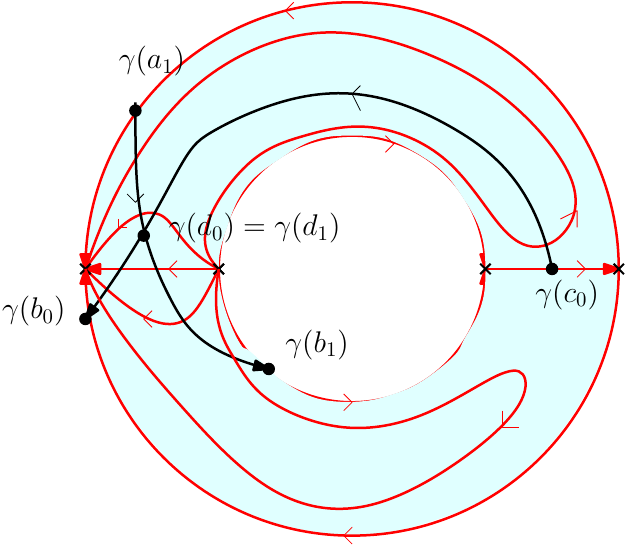}
\hfill{}
\caption{\small Proposition \ref{prop:drawing-crossing}. Depiction of why item (2) holds. Leafs of the foliation are drawn in red, transversal paths in black, and the annulus $U_{\Gamma_0}$ is the shaded region.}
\label{figure_crossingcomponent}
\end{figure}

\medskip
Let us prove $(3)$. Suppose that $\gamma$ draws a transverse simple loop $\Gamma_0$ and does not cross it. If the drawing component $J_0$ of $J_{\Gamma_0}$ does not \textcolor{black}{contain a neighborhood of at least one end of 
$J$,} then the two ends $a_0$ and $b_0$ of $J_0$ are finite and $\gamma(a_0)$ and $\gamma(b_0)$ do not belong to $U_{\Gamma}$. More precisely  $\gamma(a_0)$ and $\gamma(b_0)$ belong 
to the same connected component of $\S^2\setminus U_{\Gamma}$ because $\gamma$ does not cross $\Gamma_0$.  Here again, for every $c_0\in J_0$ the paths $\gamma\vert_{(a_0,c_0]}$ and $\gamma\vert_{[c_0, b_0)}$ meet every leaf of $U_{\Gamma_0}$ finitely many times and one can choose $c_0$ and $c'_0$ in $J_0$ such that every leaf is met exactly once by $\gamma\vert_{(a_0,c_0]}$ and $\gamma\vert_{[c'_0, b_0)}$.  Consequently, there exist $t_0$, $t'_0$, with $t'_0\leq t_0< t'_{0}+1$ such that $\gamma\vert_{(a_0,c_0]}$ is equivalent to $\gamma_0\vert_{(t_0,t_0+1]}$ and $\gamma\vert_{[c'_0,b_0)}$ is equivalent to $\gamma_0\vert_{[t'_0,t'_0+1)}$. \textcolor{black}{By working in the universal covering space and using arguments similar to what has been done in the proofs of (1) and (2) we can prove that there exist $d_0\in (a_0,c_0)$ and $d'_0\in (c'_0,b_0)$ such that $\gamma(d_0)=\gamma(d'_0)$ and such} that $\gamma\vert_{[a_0,c_0]}$ and $\gamma\vert_{[c'_0,b_0]}$ have a $\mathcal F$-transverse intersection at $\gamma(d_0)=\gamma(d'_0)$. This contradicts the hypothesis. 
\medskip

It remains to prove (4). Suppose that $\gamma$ draws two non equivalent transverse simple loops $\Gamma_0$, $\Gamma_1$. Denote $J_0$, $J_1$ the drawing components of $J_{\Gamma_0}$, $J_{\Gamma_1}$ respectively. The loops are not equivalent, so there is a leaf that is met by $\gamma\vert _{J_{\Gamma_0}}$ and not by $\gamma\vert _{J_{\Gamma_1}}$ and a leaf that is met by $\gamma\vert _{J_{\Gamma_1}}$ and not by $\gamma\vert _{J_{\Gamma_0}}$. Consequently, none of the inclusions $J_0\subset J_1$ or $J_1\subset J_0$ occurs and the conclusions follows.
\end{proof}

\begin{corollary}
\label{cor:decomposition}
Suppose that $\gamma:[a,b]\to\mathrm {dom} ({\mathcal F})$ is  a transverse path with no $\mathcal F$-transverse self-intersection such that $\phi_{\gamma(a)}=\phi_{\gamma(b)}$. Then, there exists a subdivision $(c_i)_{0\leq i\leq r}$ of $[a,b]$ such that:

\begin{enumerate}
\item $\phi_{\gamma(c_i)}=\phi_{\gamma(a)}$ for every $i\in\{0,\dots,r\}$;

\item $\phi_{\gamma(t)}\not=\phi_{\gamma(t')}$ if $c_i\leq t<t'<c_{i+1}$ and $i\in\{0,\dots,r-1\}$.

\end{enumerate}

\end{corollary}

\textcolor{black}{\begin{proof} The path $\gamma$ defines naturally a transverse loop $\Gamma$ met by $\phi_{\gamma(a)}$. As reminded in 
the beginning of \textcolor{black}{Subsection \ref{subsect:rotnumb}}, this loop meets every leaf finitely many times. Consequently, there exists a subdivision $(c_i)_{0\leq i\leq r}$ of $[a,b]$ such that:
\begin{itemize}
\item $\phi_{\gamma(c_i)}=\phi_{\gamma(a)}$ for every $i\in\{0,\dots,r\}$;
\item $\phi_{\gamma(t)}\not=\phi_{\gamma(a)}$ if $c_i\leq t<c_{i+1}$ and $i\in\{0,\dots,r-1\}$.
\end{itemize}
We would like to prove that $\phi_{\gamma(t)}\not=\phi_{\gamma(t')}$ if $c_i\leq t<t'<c_{i+1}$. It is sufficient to prove that if $a'<b'$ in $[a,b]$ satisfy
$\phi_{\gamma(a')}=\phi_{\gamma(b')}$, then there exists $c'\in[a',b']$ such that $
\phi_{\gamma(c')}=\phi_{\gamma(a)}$. As explained in the proof of Proposition 2 in  \cite{LeCalvezTal}, there exist $a''<b''$ in $[a',b']$ such that $
\phi_{\gamma(a'')}=\phi_{\gamma(b'')}$ and such that $\phi_{\gamma(t)}\not=\phi_{\gamma(t')}$ if $a''\leq t<t'<b''$. Indeed, every leaf that meets $\gamma$ is wandering and consequently, if $t$ and $t'$ are sufficiently close, one has $\phi_{\gamma(t)}\not=\phi_{\gamma(t')}$.  Moreover, because $\gamma$ is positively transverse to $\mathcal F$, one cannot find an increasing sequence $(a_n)_{n\geq 0}$ and a decreasing sequence $(b_n)_{n\geq 0}$ such that $\phi_{\gamma(a_n)}=\phi_{\gamma(b_n)}$. This implies the existence of $a''$ and $b''$ mentioned above. The path $\gamma\vert_{[a'',b'']} $ defines naturally a simple loop $\Gamma_0$ and $\gamma$ draws $\Gamma_0$. Let us explain, arguing by contradiction,  why there exists $c''\in[a'',b'']$ such that $
\phi_{\gamma(c'')}=\phi_{\gamma(a)}$.
Suppose that $\phi_{\gamma(a)}\not\subset U_{\Gamma_0}$. Then, $\gamma$ crosses $\Gamma_0$ by assertion $(3)$ of Proposition \ref{prop:drawing-crossing}. But it must have at least two crossing components because $\phi_{\gamma(a)}=\phi_{\gamma(b)}$
 and this contradicts assertion $(2)$ of Proposition \ref{prop:drawing-crossing}.\end{proof} }

\begin{remark*}Note that the conclusion of Corollary \ref{cor:decomposition} does not say that the paths $\gamma\vert_{[c_i,c_{i+1}]}$ must be equivalent, \textcolor{black}{as this is not always the case}.
\end{remark*}

Fix a transverse path $\gamma:J\to\mathrm {dom} ({\mathcal F})$  with no $\mathcal F$-transverse self-intersection. Suppose that $\gamma$ draws a transverse simple loop $\Gamma_0$ but is not drawn on it. Denote $J_0$ the drawing component of $J_{\Gamma_0}$. If the lower end $a_0$ of $J_0$ is finite and belongs to $J$, then $\gamma(a_0)$ does not belong to $U_{\Gamma_0}$. We denote $K^{\alpha}_{\Gamma_0}$ the connected component of $\S^2\setminus U_{\Gamma_0}$ that contains $\gamma(a_0)$ and $K^{\omega}_{\Gamma_0}$ the other one. Similarly, if the upper end $b_0$ of $J_0$ is finite and belongs to $J$, we denote $K^{\omega}_{\Gamma_0}$ the connected component of $\S^2\setminus U_{\Gamma_0}$ that contains $\gamma(b_0)$  and $K^{\alpha}_{\Gamma_0}$ the other one. Note that if the two ends are finite, the notations agree because $J_0$ must be a crossing component  by assertion $(3)$ of Proposition \ref{prop:drawing-crossing}. 
Note that in all cases, $$\begin{cases}\gamma(t)\in U_{\Gamma_0}\cup K^{\alpha}_{\Gamma_0},& \mathrm{if}\enskip t<a_0,\\ \gamma(t)\in U_{\Gamma_0}\cup K^{\omega}_{\Gamma_0},& \mathrm{if}\enskip t>b_0,\\ \end{cases}$$ otherwise $\gamma$ would have at least two crossing components.

\begin{proposition}
\label{prop:order-on-loops}Suppose that $\gamma:[a,b]\to\mathrm {dom} ({\mathcal F})$ is a transverse path with no $\mathcal F$-transverse self-intersection that draws two transverse simple loops $\Gamma_0$ and $\Gamma_1$ and denote $J_0$, $J_1$ the drawing components of $J_{\Gamma_0}$, $J_{\Gamma_1}$ respectively. If $J_1$ is on the right of $J_0$, then $K^{\alpha}_{\Gamma_0}\subset K^{\alpha}_{\Gamma_1}$ and  $K^{\omega}_{\Gamma_1}\subset K^{\omega}_{\Gamma_0}$.
\end{proposition}

\begin{proof} We can suppose that $\Gamma_0$ and $\Gamma_1$ are not equivalent, otherwise the result is trivial as: 
$$J_0=J_1, \enskip K^{\alpha}_{\Gamma_0}=K^{\alpha}_{\Gamma_1}, \enskip K^{\omega}_{\Gamma_0}=K^{\omega}_{\Gamma_1}.$$ 
Note that, by item (4) of Proposition \ref{prop:drawing-crossing}, either $J_0$ is on the right of $J_1$ or the converse is true, and we are assuming the latter. By the comments made just before the statement of Proposition \ref{prop:order-on-loops}, we have:
$$U_{\Gamma_1}\cap K^{\alpha}_{\Gamma_0}= U_{\Gamma_0}\cap K^{\omega}_{\Gamma_1}=\emptyset.$$ The set $K^{\alpha}_{\Gamma_0}$ is included in $\S^2\setminus U_{\Gamma_1}$. Being connected, it is contained in $K^{\alpha}_{\Gamma_1}$ or in $K^{\omega}_{\Gamma_1}$. Suppose that it is contained in $K^{\omega}_{\Gamma_1}$. It is a closed and open subset of $K^{\omega}_{\Gamma_1}$ because $U_{\Gamma_0}\cup  K^{\alpha}_{\Gamma_0}$ is a neighborhood of  $K^{\alpha}_{\Gamma_0}$ and $U_{\Gamma_0}\cap K^{\omega}_{\Gamma_1}=\emptyset$. So, $K^{\alpha}_{\Gamma_0}$ and $K^{\omega}_{\Gamma_1}$ are equal.
Since the loop $\Gamma_1$ is simple, it separates the two ends of $U_{\Gamma_1}$ and intersects every leaf in $U_{\Gamma_1}$ exactly once. Therefore, every leaf $\phi\in U_{\Gamma_1}$ is adjacent to both connected components of the complement of $U_{\Gamma_1}$, and we deduce that every leaf $\phi\in U_{\Gamma_1}$ is adjacent to $K^{\alpha}_{\Gamma_0}$ but disjoint from this set. So $\phi$ meets $U_{\Gamma_0}$, because $U_{\Gamma_0}\cup  K^{\alpha}_{\Gamma_0}$ is a neighborhood of  $K^{\alpha}_{\Gamma_0}$, and this implies that $\phi$ is contained in $U_{\Gamma_0}$. One deduces that $U_{\Gamma_1}$ is included in $U_{\Gamma_0}$, which is impossible   because $\Gamma_0$ and $\Gamma_1$ are not equivalent. In conclusion, one gets the inclusion $K^{\alpha}_{\Gamma_0}\subset K^{\alpha}_{\Gamma_1}$. The proof of the other inclusion is similar.\end{proof}

\subsection{Applications}

\begin{proposition}
\label{prop:non-wandering} Let $f$ be an orientation preserving homeomorphism of $\S^2$ with no topological horseshoe. Let $I$ be a maximal isotopy of $f$ and $\mathcal F$ a foliation transverse to $I$. 
If $z\in\mathrm{dom}(I)$ is a non-wandering point of $f$, then:

\begin{itemize}
\item  either $I^{\Z}_{\mathcal F} (z)$ never meets a leaf twice,
\item or $I^{\Z}_{\mathcal F} (z)$ draws a unique transverse simple loop $\Gamma_0$ (up to equivalence) and exactly draws it.
\end{itemize}
\end{proposition}

\begin{proof}Suppose that $I^{\Z}_{\mathcal F} (z)$ meets a leaf twice. In that case, it draws a transverse loop $\Gamma_0$, that can be supposed simple by Corollary \ref{cor:decomposition}.  If $I^{\Z}_{\mathcal F} (z)$ is not drawn on $\Gamma_0$, one can find integers $m<n$ such that $\gamma=I_{\mathcal F}^{\Z} (z)\vert_{[m,n]}=I_{\mathcal F}^{n-m}(f^m(z))$ draws $\Gamma_0$ and is not drawn on $\Gamma_0$: there exists $t\in[m,n]$ such that $\phi_{\gamma(t)}\not\subset U_{\Gamma_0}$. By Proposition \ref{prop:stability}, there exists a neighborhood $W$ of $z$ such that for every $z'\in W$,  $\gamma$ is a sub-path of $I_{\mathcal F}^{\Z} (z)\vert_{[m-1,n+1]}$. Consequently, $I_{\mathcal F}^{\Z} (z')\vert_{[m-1,n+1]}$ draws $\Gamma_0$ and is not drawn on $\Gamma_0$. The point $z$ being non-wandering, one can find $z'\in W$ and $l>n-m+2$ such that $f^{l}(z')\in W$. So $I_{\mathcal F}^{\Z} (z')\vert_{[m-1+l,n+1+l]}$ draws $\Gamma_0$ and is not drawn on $\Gamma_0$. More precisely, there exists an interval $J_0\subset [m-1,n+1]$ and an interval $J_1\subset [m-1+l,n+1+l]$ such that $I_{\mathcal F}^{\Z} (z')\vert_{J_0}$ and $I_{\mathcal F}^{\Z} (z')\vert_{J_1}$ are equivalent to $\gamma$. There exists at least two drawing components,  separated by the point $t_0\in J_0$ corresponding to $t$, if $t$ is at the right of the drawing component of $\gamma$,  separated  by the point $t_1\in J_1$ corresponding to $t$, if $t$ is at the left of the drawing component of $\gamma$. By assertion $(1)$ of Proposition \ref{prop:drawing-crossing}, we deduce that $I^{\Z}_{\mathcal F} (z)$ has a $\mathcal F$-transverse self intersection. This contradicts Theorem \ref{th: horseshoe}.\end{proof}

\begin{remarks*} The first case only appears in case $\alpha(z)$ and $\omega(z)$ are included in $\mathrm{fix}(I)$. We will define
$$\mathrm{ne}(I)=\{z\in \mathrm{dom}(I) \, \vert \, \alpha(z)\cup\omega(z)\not\subset\mathrm{fix}(I)\}.$$ 
Note that if $z\in  \mathrm{ne}(I)$, there exists at least one leaf that is met infinitely many times by $I^{\Z}_{\mathcal F} (z)$. By Corollary \ref{cor:decomposition}, $I^{\Z}_{\mathcal F} (z)$ 
 draws at least one transverse simple loop. Moreover, this loop is unique and drawn infinitely if $z\in\Omega(f)$.
\end{remarks*}
We can generalize the previous result.

\begin{proposition} 
\label{prop:Birkhoff-cycles}Let $f$ be an orientation preserving homeomorphism of $\S^2$ with no topological horseshoe. Let $I$ be a maximal isotopy of $f$ and $\mathcal F$ a foliation transverse to $I$. 
Let $(z_i)_{i\in\Z/r\Z}$ be a Birkhoff cycle in $\mathrm{ne}(I)$. Then there exists a transverse simple loop $\Gamma$ such that, for every $i\in \Z/r\Z$, $\Gamma$ is the unique simple loop (up to equivalence) \textcolor{black}{drawn} by $I^{\Z}_{\mathcal F} (z_i)$.  \textcolor{black}{Moreover $I^{\Z}_{\mathcal F} (z_i)$ exactly draws infinitely $\Gamma$.} 

\end{proposition}

\begin{proof} In the case where there exists a periodic orbit that contains all the $z_i$, we can apply Proposition \ref{prop:non-wandering}. Otherwise,  \textcolor{black}{using the first assertion of Proposition  \ref{prop: birkhoff connexions}}, and taking a subsequence if necessary, one can suppose that for every $i\in\Z/r\Z$, the point $z_{i+1}$ is not the image of $z_i$ by a positive iterate of $f$. Likewise, we may assume that $r\ge 2$, otherwise the unique point $z_0$ must be non-wandering and the result follows again by Proposition \ref{prop:non-wandering}. Note that, since each $z_i\in \mathrm{ne}(I)$, then each $I^{\Z}_{\mathcal{F}}(z_i)$ draws at least a simple transverse loop  \textcolor{black}{$\Gamma_i$. It suffices to show that $\Gamma_i$ and $\Gamma_j$ are equivalent if $i \not=j$ and that $I^{\Z}_{\mathcal F} (z_i)$ exactly draws $\Gamma_i$. Indeed $z_i$ belongs to $\mathrm{ne}(f)$ and so $I^{\Z}_{\mathcal F} (z_i)$ exactly draws infinitely $\Gamma_i$}. 

 The fact that $(z_i)_{i\in\Z/r\Z}$ is a Birkhoff cycle \textcolor{black}{and that $z_{i+1}$ does not belong to the positive orbit of $z_i$} implies that, \textcolor{black}{for every $N\geq 1$, if we choose for every $i$ a neighborhood $W_i$ of $z_i$, there exists $z'_i\in W_i$ and $n_i\geq N$ such that $f^{n_i}(z'_i)\in W_{i+1}$. In particular, if the $W_i$ are chosen small enough,} $I^{\Z}_{\mathcal F} (z'_i)$ draws $\Gamma_i$ and $\Gamma_{i+1}$  and moreover, the drawing component of $\Gamma_{i+1}$ is on the right of the drawing component of $\Gamma_{i}$. 
 For every \textcolor{black}{$i\in\Z/r\Z$ we denote:
 \begin{itemize} 
 \item $(a_i,b_i)$ the drawing component of $I^{\Z}_{\mathcal F} (z_i)$ with $\Gamma_i$,
 \item $(a'_i,b'_i)$ the drawing component of $I^{\Z}_{\mathcal F} (z'_i)$ with $\Gamma_i$, 
\item $(a''_i,b''_i)$ the drawing component of $I^{\Z}_{\mathcal F} (z'_{i-1})$ with $\Gamma_i$.
\end{itemize}}

\textcolor{black}{
Let us begin with the following preliminary result.
\begin{lemma} \label{lemma:equal loops} Suppose that $\Gamma_i=\Gamma_{i+1}$. Then $b_i=+\infty$ and  $a_{i+1}=-\infty$.
\end{lemma} 
\begin{proof}  Suppose that $b_i<+\infty$. Fix $c_i$ in $(a_i, b_i)$  and $e_{i+1}$, $f_{i+1}$ in $(a_{i+1}, b_{i+1})$ satisfying $e_{i+1}<f_{i+1}$, such that $I^{\Z}_{\mathcal F} (z_i)\vert_{[c_i, b_i]}$  and  $I^{\Z}_{\mathcal F} (z_{i+1})\vert_{[e_{i+1}, f_{i+1}]}$ draw $\Gamma$.  If the neighborhoods $W_{i}$ and $W_{i+1}$ are chosen sufficiently small and $n_i$ sufficiently large, there exist $c'_i< d'_i<e'_i<f'_i$  such that $I^{\Z}_{\mathcal F} (z'_i)\vert_{[c'_i, d'_i]}$ is equivalent to  $I^{\Z}_{\mathcal F} (z_i)\vert_{[c_i, b_i]}$ and $I^{\Z}_{\mathcal F} (z'_i)\vert_{[e'_i, f'_i]}$ is equivalent to $I^{\Z}_{\mathcal F} (z_{i+1})\vert_{[e_{i+1}, f_{i+1}]}$. Therefore both $[c'_i, d'_i]$ and $[e'_i, f'_i]$ intersect a drawing component of $I^{\Z}_{\mathcal F} (z'_i)$ with $\Gamma$. But since $I^{\Z}_{\mathcal F} (z_i)(b_i)$ does not belong to $U_{\Gamma}$,  $I^{\Z}_{\mathcal F} (z'_i)(d'_i)$ also does not belong to $U_{\Gamma}$, and therefore there must be at least two drawing components of $I^{\Z}_{\mathcal F} (z'_i)$ with $\Gamma$, a contradiction with the assertion $(1)$ of Proposition \ref{prop:drawing-crossing}. One proves in the same way that $a_{i+1}=-\infty$. \end{proof}}

Consequently, if we prove that the $\Gamma_i$ are all equal, we will deduce that $I^{\Z}_{\mathcal F} (z_i)$ exactly draws $\Gamma_i$.

Suppose for a contradiction that there exists \textcolor{black}{$i$ such that $\Gamma_{i}\not=\Gamma_{i+1}$}. As explained in Proposition \ref{prop:order-on-loops}, for each $i$ such that $\Gamma_{i}\not=\Gamma_{i+1}$, exactly one of the two connected components of the complement of $U_{\Gamma_i}$ meets $\Gamma_{i+1}$, which we denote it $K^{\omega}_{i\to i+1}$, and exactly one of the two connected components of the complement of $U_{\Gamma_{i+1}}$ meets $\Gamma_{i}$, which we denote it $K^{\alpha}_{i\to i+1}$. We can extend the above sets in a unique way to obtain two global families $(K^{\alpha}_{i\to i+1})_{i\in\Z/r\Z}$ and $(K^{\omega}_{i\to i+1})_{i\in\Z/r\Z}$ by imposing that, whenever
$$\Gamma_i=\Gamma_{i+1}\Rightarrow K^{\alpha}_{i\to i+1}=K^{\alpha}_{i-1\to i} \enskip\mathrm{and} \enskip K^{\omega}_{i\to i+1}=K^{\omega}_{i+1\to i+2}.$$

\begin{lemma}  There exists $i$ such that $K^{\alpha}_{i-1\to i}=K^{\omega}_{i\to i+1}$.
\end{lemma}
\begin{proof} If $K^{\alpha}_{i-1\to i}\not=K^{\omega}_{i\to i+1}$ for every $i\in\Z/r\Z$, then, the sequence $(K^{\alpha}_{i\to i+1})_{i\in\Z/r\Z}$ is non-decreasing while the sequence $(K^{\omega}_{i\to i+1})_{i\in\Z/r\Z}$ is non-increasing by Proposition \ref{prop:order-on-loops}. So the two sequences are constant, as is the sequence $(U_{\Gamma_i})_{i\in \Z/r\Z}$ because $U_{\Gamma_i}=\S^2\setminus \left(K^{\omega}_{i\to i+1}\cup K^{\alpha}_{i\to i+1}\right)$. This contradicts our assumption that there exists $i$ such that $\Gamma_{i}\not=\Gamma_{i+1}$.
\end{proof}

We deduce that there exists a transverse simple loop $\Gamma$, a connected component $K$ of $\S^2\setminus U_{\Gamma}$
and \textcolor{black}{$i_0\leq i_1$} such that 
\begin{itemize}
\item $\Gamma_{i}=\Gamma$ for every $i\in\{i_0,\dots,i_1\}$;
\item $\Gamma_{i_0-1}\not=\Gamma$;
\item $\Gamma_{i_1+1}\not=\Gamma$; 
\item the point $z'_{i_0-1}$ draws $\Gamma_{i_0-1}$ and $\Gamma$ and the drawing component of the first loop is on the left of the drawing component of the second one;
\item  the point $z'_{i_1}$ draws $\Gamma$ and $\Gamma_{i_1+1}$ and the drawing component of the first loop is on the left of the drawing component of the second one;
\item $K^{\alpha}_{i_0-1\to i_0}=K^{\omega}_{i_1\to i_1+1}=K$.
\end{itemize}

 We know that 
 $$-\infty<a''_{i_0}, \enskip b'_{i_1}<+\infty, \enskip I^{\Z}_{\mathcal F} (z'_{i_0-1})(a''_{i_0})\in K, \enskip I^{\Z}_{\mathcal F} (z'_{i_1})(b'_{i_1})\in K.$$ 

\begin{lemma} \label{lemma:different loops} There exists $i\in\{i_0,\dots, i_1\}$ such that: 
\begin{itemize}
\item $-\infty<a''_i$;
\item $ I^{\Z}_{\mathcal F} (z'_{i-1})(a''_{i})\in K$;
\item $I^{\Z}_{\mathcal F} (z'_{i-1})\vert_{(a''_i, b''_i)}$ meets every leaf of $U_{\Gamma}$ at least twice. 
\end{itemize}
\end{lemma} 
\begin{proof} Suppose first that $i_0=i_1$. We know that  $a''_{i_0}$ is finite and that  $I^{\Z}_{\mathcal F} (z'_{i_0-1})(a''_{i_0})\in K$. We will prove that $I^{\Z}_{\mathcal F} (z_{i_0})\vert_{(a_{i_0},b_{i_0})}$ meets every leaf of $U_{\Gamma}$ at least twice. It will also be the case for $I^{\Z}_{\mathcal F} (z'_{i_0-1})\vert_{(a''_{i_0}, b''_{i_0})}$  if the neighborhood $W_{i_0}$ is chosen sufficiently small and $n_{i_0-1}$ sufficiently large. So, the lemma will be proved. In fact we will prove that $I^{\Z}_{\mathcal F} (z_{i_0})$ is drawn on $\Gamma$. The point  $z_{i_0}$ belonging to $\mathrm{ne}(f)$, $I^{\Z}_{\mathcal F} (z_{i_0})$ will draw infinitely $\Gamma$.

Suppose that  $a_{i_0}$ is finite. In that case, $a'_{i_0}$ is also finite and  $I^{\Z}_{\mathcal F} (z'_{i_0-1})(a''_{i_0})$, $I^{\Z}_{\mathcal F} (z_{i_0})(a_{i_0})$ and $I^{\Z}_{\mathcal F} (z'_{i_0})(a' _{i_0})$ are on the same leaf. One deduces that  $I^{\Z}_{\mathcal F} (z'_{i_0})(a' _{i_0})\in K$. By hypothesis, one knows that $I^{\Z}_{\mathcal F} (z'_{i_0})(b' _{i_0})\in K$. This contradicts the assertion $(3)$ of Proposition \ref{prop:drawing-crossing}. Suppose that  $b_{i_0}$ is finite. In that case, $b''_{i_0}$ is also finite and  $I^{\Z}_{\mathcal F} (z'_{i_0-1})(b''_{i_0})$, $I^{\Z}_{\mathcal F} (z_{i_0})(b_{i_0})$ and $I^{\Z}_{\mathcal F} (z'_{i_0})(b' _{i_0})$ are on the same leaf. One deduces that  $I^{\Z}_{\mathcal F} (z'_{i_0})(b'' _{i_0})\in K$. By hypothesis, one knows that $I^{\Z}_{\mathcal F} (z'_{i_0})(a'' _{i_0})\in K$. We get the same contradiction.
In conclusion, $I^{\Z}_{\mathcal F} (z_{i_0})$ is drawn on $\Gamma$. 

 Suppose now that $i_0<i_1$. \textcolor{black}{Observe that $(a'_i,b'_i)=(a''_{i+1}, b''_{i+1})$ if $i_0\leq i<i_1$.} We have to study the case where $I^{\Z}_{\mathcal F} (z_{i_0})\vert_{(a_{i_0},b_{i_0})}$ does not meet every leaf of $U_{\Gamma}$ at least twice.
We have seen in \textcolor{black}{Lemma \ref{lemma:equal loops}} that $b_{i_0}$ is infinite. So $a_{i_0}$ is finite. In that case, $a'_{i_0}$ is also finite and  $I^{\Z}_{\mathcal F} (z'_{i_0-1})(a''_{i_0})$, $I^{\Z}_{\mathcal F} (z_{i_0})(a_{i_0})$ and $I^{\Z}_{\mathcal F} (z'_{i_0})(a' _{i_0})$ are on the same leaf and so $  \textcolor{black}{ I^{\Z}_{\mathcal F} (z'_{i_0})(a'' _{i_0+1})}=I^{\Z}_{\mathcal F} (z'_{i_0})(a' _{i_0})\in K$. Replacing $i_{0}$ by $i_0+1$, all what was done above is still valid.  If $I^{\Z}_{\mathcal F} (z_{i_0+1})$ is drawn on $\Gamma$, the lemma will be valid with $i_0+1$. If not, we go to the following index. The process will stop until reaching $i_1$. 
\end{proof}

By a similar reasoning,  there exists $j\in\{i_0,\dots, i_1\}$ such that 
\begin{itemize}
\item $b'_j<+\infty$,
\item $ I^{\Z}_{\mathcal F} (z'_{j})(b'_{j})\in K$,
\item $I^{\Z}_{\mathcal F} (z'_{j})\vert_{(a'_j, b'_j)}$ meets every leaf of $U_{\Gamma}$ at least twice. 
\end{itemize}

\textcolor{black}{
\begin{lemma}\label{lem:newmeettwo}
Let $\beta_0:[\overline a_0, \overline b_0)\to \mathrm{dom}(\mathcal{F})$ and $\beta_1:(\overline a_1, \overline b_1]\to \mathrm{dom}(\mathcal{F})$ be two admissible transverse paths such that:
\begin{itemize}
\item $\beta_0(\overline a_0)$ and $\beta_1(\overline b_1)$ are in the same connected component of $\S^2\setminus U_{\Gamma}$;
\item Both $\beta_0\vert_{(\overline a_0, \overline b_0)}$ and  $\beta_1\vert_{(\overline a_1, \overline b_1)}$ are in $U_{\Gamma}$, and both paths meet every leaf of $U_{\Gamma}$ at least twice.
\end{itemize}
Then there exists an admissible path with a $\mathcal F$-transverse self-intersection
\end{lemma}
\begin{proof}
Denote $\gamma$ the natural lift of $\Gamma$. Here again, for every $\overline c_0\in (\overline a_0, \overline b_0)$ and $\overline c_1\in (\overline a_1, \overline b_1)$, the paths $\beta_0\vert_{(\overline a_0,  \overline c_0]}$ and $\beta_1\vert_{[\overline c_1,  \overline b_1)}$ meet every leaf of $U_{\Gamma}$ finitely many times and one can choose $\overline c_0$ and $\overline c_1$ such that every leaf is met exactly twice by $\beta_0\vert_{(\overline a_0,  \overline c_0]}$ and by $\beta_1\vert_{[\overline c_1,  \overline b_1)}$. Consequently, there exist $t_0$ and $t_1$,  with $t_0-1<t_1\leq t_0$, such that $\beta_0\vert_{(\overline a_0,  \overline c_0]}$ is equivalent to $\gamma\vert_{(t_0,t_0+2]}$ and $\beta_1\vert_{[\overline c_1,  \overline b_1)}$ is equivalent to $\gamma\vert_{[t_1,t_{1}+2)}$.
Lets consider $\overline d_0\in (\overline a_0, \overline c_0)$ and $d_1\in (\overline c_1, \overline b_1)$ uniquely defined by the fact that  $\beta_0\vert_{(\overline a_0,  \overline c_0]}$ is equivalent to $\gamma\vert_{(t''_i,t''_i+1]}$ and $\beta_1\vert_{[\overline d_1,  \overline b_1)}$ is equivalent to $\gamma\vert_{[t''_i+1,t'_j+2)}$. One can suppose, by eventually changing the paths in their equivalence class, that $\beta_0(\overline d_0)= \beta_1(\overline d_1)$. Note that the paths $\beta_0\vert_{[\overline a_0,  \overline c_0]}$ and $\beta_1\vert_{[\overline c_1,  \overline b_1]}$ have a $\mathcal F$-transverse intersection at $\beta_0(\overline d_0)= \beta_1(\overline d_1)$. By Proposition \ref{pr: fundamental}, the path $\beta_0\vert_{[\overline a_0, \overline d_0]}\beta_1\vert_{[\overline d_1, \overline b_1]}$ is admissible, and contains a subpath that is equivalent to $\gamma\vert_{(t_0, t_1+2)}$. In particular, since $t_1>t_0-1$, this path draws $\Gamma$, does not cross it and its ends are not in $U_{\Gamma}$. By Proposition \ref{prop:drawing-crossing}, it has a $\mathcal F$-transverse self intersection.
\end{proof}
}
\textcolor{black}{We can therefore apply the previous lemma to the paths $\beta_0=I^{\Z}_{\mathcal F} (z''_{i-1})\vert_{[a''_i, b''_i)}$ and $\beta_1= I^{\Z}_{\mathcal F} (z'_{j})\vert_{(a'_j,b'_j]}$,
 then one deduces the existence of an admissible path with a $\mathcal F$-transverse self-intersection, and Theorem \ref{th: horseshoe} implies that $f$ has a topological horseshoe, a contradiction proving Proposition \ref{prop:Birkhoff-cycles}.}
 \end{proof}

\section{Paths on the annulus with no $\mathcal F$ transverse self-intersection}

\subsection{Definitions, positive and negative loops} A singular oriented foliation $\mathcal F$ on the annulus $\A=\T^1\times\R$ can be extended to a singular foliation on $\S^2$ by adding two singular points, the upper end $N$ and the lower end $S$. So one can apply the results of the previous section.  Let us study now this particular class of foliations. Write $[\Gamma]\in H_1(\A,\Z)$ the homology class of a loop $\Gamma$ of $\A$. Consider the loop $\Gamma_*:x\mapsto(x,0)$, which means the circle $\T^1\times\{0\}$ oriented with an increasing angular coordinate, so that the upper end is on its left side and  the lower end on its right side. One gets  a generator of $H_1(\A,\Z)$ by considering $[\Gamma_*]$. Say that a loop $\Gamma$ is essential if $[\Gamma]\not=0$. Say that $\Gamma$ is a {\it positive} loop if there exists $p>0$ such that $[\Gamma]=p[\Gamma_*]$ and define in the same way {\it negative, non positive and non negative} loops. Note that if $\Gamma$ is a simple loop, there exists $p\in\{-1,0,1\}$ such that $[\Gamma]=p[\Gamma_*]$.  Let us remind that if  a multi-loop $\Gamma=\sum_{1\leq i\leq p}\Gamma_i$  in $\A$ is homologous to zero in $\A$, one can choose the dual function (defined on $\S^2)$ to vanish in a neighborhood of $N$ and $S$.  If $\check \A=\R\times\R$ is the universal covering space of $\A$, we will denote $T:(x,y)\mapsto(x+1, y)$ the covering automorphism naturally associated with the loop $\Gamma_*$.

Suppose now that $\gamma:[a,b]\to\mathrm{dom}(I)$ is a transverse path such that $\phi_{\gamma(a)}=\phi_{\gamma(b)}=\phi$. One can find a transverse path $\gamma':[a,b]\to\mathrm{dom}(I)$ equivalent to $\gamma$ such that $\gamma'(a)=\gamma'(b)$. The leaf $\phi_{\gamma(a)}$ cannot be a closed leaf and so, the equivalence class of the transverse loop naturally defined by $\gamma'$ depends only on $\gamma$: we will called such a transverse loop, the {\it natural transverse loop} defined by $\gamma$.

\begin{proposition}
\label{prop:positive-negative} Suppose that $\gamma:J\to\mathrm {dom} ({\mathcal F})$ is a transverse path with no $\mathcal F$-transverse self-intersection that draws a positive transverse simple loop $\Gamma_0$ and a negative transverse simple loop $\Gamma_1$. Then $U_{\Gamma_0}\cap U_{\Gamma_1}=\emptyset$. 
\end{proposition}

\begin{proof} Of course $\Gamma_0$ and $\Gamma_1$ are not equivalent. Denote $J_0$, $J_1$ the drawing components of $J_{\Gamma_0}$, $J_{\Gamma_1}$ respectively. There is no loss of generality by supposing that  $J_1$ is on the right of $J_0$.  The path $\gamma\vert_{J_1}$ does not cross $\Gamma_0$. Otherwise $J_1$ would contain a crossing component of $J_{\Gamma_0}$, but the only eventual crossing component of $J_{\Gamma_0}$ 
is $J_0$, the drawing component,  by assertion $(2)$ of Proposition \ref{prop:drawing-crossing}. So, there exists a transverse path equivalent to $\gamma\vert_{J_1}$ that is disjoint from $\Gamma_0$ and consequently a transverse loop $\Gamma'_1$ equivalent to $\Gamma_1$ that is disjoint from $\Gamma_0$. The multi loop $\Gamma_{0}+\Gamma'_{1}$ is homologous to zero and its dual function can be chosen to take its values in $\{0,1\}$ because the  loops are disjoint. This implies that a leaf cannot meet both loops.\end{proof}

\begin{corollary}
\label{cor:positive-negative1} Suppose that $\gamma:J\to\mathrm {dom} ({\mathcal F})$ is a transverse path with no $\mathcal F$-transverse self-intersection such that $\phi_{\gamma(t_0)}=\phi_{\gamma(t_1)}=\phi_{\gamma(t_2)}$ where $t_0<t_1<t_2$. If the transverse loop naturally defined by $\gamma\vert_{[t_0,t_1]}$ is positive, then the transverse loop naturally defined by $\gamma\vert_{[t_1,t_2]}$ is non negative. Similarly if the first loop  is negative, the second one is non positive.
\end{corollary}

\begin{proof} Suppose that $\gamma\vert_{[t_0,t_1]}$ is positive and $\gamma\vert_{[t_1,t_2]}$ negative. Using Corollary \ref{cor:decomposition}, one can find $t_0\leq t'_0<t'_1\leq t_1$ and $t_1\leq t''_1<t''_2\leq {t_2}$ such that 
\begin{itemize}
\item $\phi_{\gamma(t'_0)}=\phi_{\gamma(t'_1)}=\phi_{\gamma(t''_1)}=\phi_{\gamma(t''_2)}=\phi_{\gamma(t_0)}$;
\item the transverse loop naturally defined by $\gamma\vert_{[t'_0,t'_1]}$ is a positive transverse simple loop; 
\item the transverse loop naturally defined by $\gamma\vert_{[t''_1,t''_2]}$ is a negative transverse simple loop; 
\item these two last loops are drawn by $\gamma$.
\end{itemize}
It remains to note that $\phi_{\gamma(t_0)}$ meets both loops, which contradicts Proposition \ref{prop:positive-negative}.\end{proof}

\begin{corollary}
\label{cor:positive-negative2} Suppose that $\gamma:J\to\mathrm {dom} ({\mathcal F})$ is a transverse path with no $\mathcal F$-transverse self-intersection such that $\phi_{\gamma(t_0)}=\phi_{\gamma(t_1)}$ and $\phi_{\gamma(t_2)}=\phi_{\gamma(t_3)}$, where $t_0<t_2<t_3<t_1$. If the transverse loop naturally defined by $\gamma\vert_{[t_0,t_1]}$ is positive, then the transverse loop naturally defined by $\gamma\vert_{[t_2,t_3]}$ is non negative. Similarly if the first loop  is negative, the second one is non positive.
\end{corollary}

\begin{proof} Suppose that $\gamma\vert_{[t_0,t_1]}$ is positive and $\gamma\vert_{[t_2,t_3]}$ negative. Using Corollary \ref{cor:decomposition}, one can find $t_0\leq t'_0<t'_1\leq t_1$ and $t_2\leq t'_2<t'_3\leq {t_3}$ such that 
\begin{itemize}
\item $\phi_{\gamma(t'_0)}=\phi_{\gamma(t'_1)}=\phi_{\gamma(t_0)}$;
\item $\phi_{\gamma(t'_2)}=\phi_{\gamma(t'_3)}=\phi_{\gamma(t_2)}$;
\item the transverse loop naturally defined by $\gamma\vert_{[t'_0,t'_1]}$ is a positive transverse simple loop; 
\item the transverse loop naturally defined by $\gamma\vert_{[t'_2,t'_3]}$ is a negative transverse simple loop; 
\item these two last loops are drawn by $\gamma$.
\end{itemize}
As explained in the proof of Corollary \ref{cor:decomposition}, the leaf $\phi_{\gamma(t_0)}$ meets both loops, which contradicts Proposition \ref{prop:positive-negative}.\end{proof}

\bigskip
Let us give an application.

\begin{corollary}
\label{cor:free-disks}
Let $f$ be a homeomorphism of $\A$ isotopic to the identity and $\check f$ a lift of $f$ to $\check {\A}$. Let $D$ be a topological open disk of $\A$ and $\check D$ a lift of $D$. We suppose that $\check D$ is a free disk of $\check f$ and that there exists a point $\check z\in \check D$ and positive integers $n_0,n_1,m_0,m_1$ such that 
$$\check z\in \check D, \,\check f^{n_0}(\check z)\in T^{m_0}(\check D),\, \check f^{n_0+n_1}(\check z)\in T^{m_0-m_1}(\check D).$$Then $f$ has a topological horseshoe.
\end{corollary}

\begin{proof} Choose an identity isotopy $I^*$ of $f$ that is lifted on $\check \A$ to an identity isotopy of $\check f$. Then choose a maximal isotopy $I$ such that $I^*\preceq I$. Set $z=\pi(\check z)$. As explained in \textcolor{black}{Proposition 59 of \cite{LeCalvezTal}}, one can find a singular foliation ${\mathcal F}$ transverse to $I$ such that $z$, $f^{n_0}(z)$ and $f^{n_0+n_1}(z)$ are on the same leaf of the restricted foliation ${\mathcal F}\vert_{D}$. Setting $\gamma=I^{n_0+n_1}_{{\mathcal F}}(z)$, one gets

\begin{itemize}
\item $\phi_{\gamma(0)}=\phi_{\gamma(n_0)}=\phi_{\gamma(n_0+n_1)}$;

\item the transverse loop naturally defined by $\gamma\vert_{[0,n_0]}$ is positive;

\item  the transverse loop naturally defined by $\gamma\vert_{[n_0,n_0+n_1]}$ is negative.
\end{itemize} 

By Corollary \ref{cor:positive-negative1}, one deduces that $I^{n_0+n_1}_{{\mathcal F}}(z)$ has a $\mathcal F$-transverse self-intersection, and by Theorem  \ref{th: horseshoe2}, that $f$ has a topological horseshoe.\end{proof}

\subsection{Rotation numbers}

Let us remind the theorem we want to prove, Theorem \ref{thmain:rotation-number} of the introduction.

\begin{theorem}
\label{th:rotation-number}
Let $f$ be a homeomorphism of $\A$ isotopic to the identity and $\check f$ a lift of $f$ to $\R^2$. We suppose that $f$ has no topological horseshoe. Then
\begin{enumerate}
\item each point $z\in\mathrm{ne}^+(f)$  has a {well defined} rotation number $\mathrm{rot}_{\check f}(z)$;
\item for every $z,z'\in\mathrm{ne}^+(f)$ such that $z'\in\omega(z)$ we have $\mathrm{rot}_{\check f}(z')=\mathrm{rot}_{\check f}(z)$;
\item if $z\in \mathrm{ne}^+(f)\cap \mathrm{ne}^-({f})$ is non-wandering, then $\mathrm{rot}_{\check f^{-1}}(z)=-\mathrm{rot}_{\check f}(z)$;
\item the map $\mathrm{rot}_{\check f^{\pm}} : \Omega(f) \cap \mathrm{ne}(f)\to\R$ is continuous, where
$$ \mathrm{rot}_{\check f^{\pm}}(z) = \begin{cases}\mathrm{rot}_{\check f}(z) &\mathrm{if}\enskip  z\in \Omega(f) \cap  \mathrm{ne}^+(f),\\
-\mathrm{rot}_{\check f^{-1}}(z) &\mathrm{if}\enskip  z\in \Omega(f) \cap  \mathrm{ne}^-(f).
\end{cases}$$

\end{enumerate} 
\end{theorem}

\begin{remark*} \textcolor{black}{Before starting the proof of Theorem \ref{th:rotation-number}, let us state a fact that will be often used in the article. Let $f$ be a homeomorphism of $\A$ isotopic to the identity and $\check f$ a lift of $f$ to $\R^2$. Let $I$ be a maximal isotopy of $f$ that is lifted to an identity isotopy $\check I$ of $\check f$ and  ${\mathcal F}$ a singular foliation transverse to $I$ lifted to $\check{\mathcal F}$. Suppose that $z$ belongs to the domain of $I$ and that $I_{\mathcal F}^{\N}(z)$ exactly draws infinitely a positive loop $\Gamma$. Then if $z\in\mathrm{ne}_+(I)$ and if $\mathrm{rot}_{\check f}(z)$ is well defined,
 then  $\mathrm{rot}_{\check f}(z)\geq 0$. More precisely, if there exists $z'\in\mathrm{dom}(I)$,  a sub-sequence $(f^{n_k}(z))_{k\geq 0}$ of the forward orbit of $z$ that converges to a point $z'\in\mathrm{dom}(I)$ and if there exists $\rho\in \R\cup\{-\infty, +\infty\}$ such that 
$$\lim_{k\to+\infty} {1\over n_k} (\pi_1(\check f^{n_k} (\check z) -\pi_1(\check z))= \rho, $$ \textcolor{black}{where $\check z$ is a lift of $z$.} Then it holds that $\rho\geq 0$. Indeed, \textcolor{black}{there exists some ${\check z}'$ a lift of $z'$ and $k_0$ sufficiently large such that $\check I_{\check {\mathcal F}}^{n_{k_0}}(\check z)$ meets $\phi_{\check z'}$.} If $\rho<0$, then there would exist $m$ arbitrarily large such that $\check I_{\check {\mathcal F}}^{\N}(\check z)$ \textcolor{black}{also meets $T^{-m}(\phi_{\check z'})$ and so $I_{\mathcal F}^{\N}(z)$ will draw a negative transverse loop, which is impossible.}}
\end{remark*}

\begin{proof} We will begin by proving $(1)$ by contradiction. Suppose that the $\omega$-limit set of $z$ is not empty and that $z$ has no rotation number.
There are two possibilities. 

\medskip
\noindent {Case 1:}\enskip there exists a compact set $K\subset\A$, two numbers $\rho_{-}<\rho_+$ in $\R\cup\{-\infty,\infty\}$ and two increasing sequences of integers $(n_k)_{k\geq 0}$  and $(m_k)_{k\geq 0}$ such that $f^{n_k}(z)\in K$, $f^{m_k}(z)\in K$ and such that for any lift $\check z$ of $z$, one has
$$\lim_{k\to+\infty}\frac{1}{n_k}\left( \pi_1(\check f^{n_k}(\check z))- \pi_1(\check z)\right)=\rho_{-}$$ 
and
$$\lim_{k\to+\infty}\frac{1}{m_k}\left( \pi_1(\check f^{m_k}(\check z))- \pi_1(\check z)\right)=\rho_{+}.$$

\medskip
\noindent {Case 2:}\enskip there exists a compact set $K\subset\A$, a number $\rho\in \{-\infty,\infty\}$ and an increasing sequence of integers $(n_k)_{k\geq 0}$ such that $f^{n_k}(z)\in K$ and such that for any lift $\check z$ of $z$, one has
$$\lim_{k\to\infty}\frac{1}{n_k}\left( \pi_1(\check f^{n_k}(\check z))- \pi_1(\check z)\right)=\rho.$$ 

\medskip
Let us begin by the first case. Taking subsequences if necessary, one can suppose that there exist $z_-$ and $z_+$ in $K$ such that $\lim_{k\to+\infty} f^{n_k}(z)=z_-$ and $\lim_{k\to+\infty} f^{m_k}(z)=z_+$. Let us choose two integers $p$ and $q>0$ relatively prime such that $\rho_-<p/q<\rho_+$ and such that both $z_-$ and $z_+$ are not fixed points of $f^q$ with rotation number $p/q$. Write $n_k=qn'_k+r_k$ and $m_k=qm'_k+s_k$, where $r_k$ and $s_k$ belong to $\{0,\dots, q-1\}$. Taking subsequences if necessary, one can suppose that there exist $r$ and $s$ in $\{0,\dots, q-1\}$ such that  $r_k=r$ and $s_k=s$ for every $k\in\Z$. Setting
$$z'_-=f^{-r}(z_-),\,\,z'_+=f^{-s}(z_+),\,\,f'=f^q,\,\check f'=\check f^q\circ T^{-p},$$
one knows that
$$\lim_{k\to+\infty} f'^{n'_k}(z)=z'_-,\,\,\lim_{k\to+\infty} f'^{m'_k}(z)=z'_+.$$Moreover
$$\lim_{k\to+\infty}\frac{1}{n'_k}\left( \pi_1(\check f'^{n'_k}(\check z))- \pi_1(\check z)\right)=q\rho_{-}-p<0$$ 
and
$$ \lim_{k\to+\infty}\frac{1}{m'_k}\left( \pi_1(\check f'^{m'_k}(\check z))- \pi_1(\check z)\right)=q\rho_{+}-p>0.$$

Choose an identity isotopy $I'^*$ of $f'$ that is lifted to an identity isotopy of $\check f'$, then a maximal isotopy $I'$ such that $I'^*\preceq I'$, and finally a singular foliation ${\mathcal F}'$ transverse to $I'$. The singular points of $I'$ are periodic points of $f$ of period $q$ and rotation number $p/q$. So $z'_-$ and $z'_+$ are non singular points. The trajectory $\gamma'=I'^{\N}_{{\mathcal F}'}(z)$ meets the leaves $\phi_{z'_-}$ and $\phi_{z'_+}$ infinitely many times. Denote $(t_{k}^-)_{k\geq 0}$ and $(t_{k}^+)_{k\geq 0}$ the (increasing) sequences of meeting times. One can find  $t_{k_0}^+< t_{k_2}^-<t_{k_3}^-<t_{k_1}^+$ such that $\gamma\vert_{[t_{k_0}^+,t_{k_1}^+]}$ is positive and $\gamma\vert_{[t_{k_2}^-,t_{k_3}^-]}$ negative. 
By Corollary \ref{cor:positive-negative2}, one deduces that $\gamma'$ has a ${\mathcal F}'$-transverse self-intersection, and by Theorem  \ref{th: horseshoe2}, that $f$ has a topological horseshoe.

\medskip
 Let us study now the second case. We will suppose that $\rho=-\infty$, the case where $\rho=+\infty$ being analogous. Taking a subsequence if necessary, one can suppose that there exists $z'\in K$ such that $\lim_{k\to+\infty} f^{n_k}(z)=z'$. 
 
 \medskip
 Suppose first that $z'$ is a fixed point and denote $p\in\Z$ its rotation number. Choose an identity isotopy $I'^*$ of $f$ that is lifted to an identity isotopy of $\check f'=\check f\circ T^{-p+1}$, then a maximal isotopy $I'$ such that $I'^*\preceq I'$, and finally a singular foliation ${\mathcal F}'$ transverse to $I'$. One knows that $\rho_{\check f'} (z')=1$. \textcolor{black}{By Proposition \ref{prop:non-wandering},} this means that the transverse loop $\Gamma$ associated with $z'$ is a positive simple loop. The fact that $z'\in\omega(z)$ implies that for every $k\geq0$, the path $I'^{\N}_{{\mathcal F}'}(f^k(z))$ draws $\Gamma$. One deduces from Proposition \ref{prop:drawing-crossing} that $I'^{\N}_{{\mathcal F}'}(f^k(z))$ \textcolor{black}{exactly} draws infinitely $\Gamma$, if $k$ is sufficiently large. This contradicts the fact that $\rho=-\infty$.  
 
 \medskip
 Suppose now that $z'$ is not a fixed point and consider a free disk $D$ containing $z'$. The orbit $(f^n(z))_{n\geq 0}$ meets the disk $D$ infinitely many times. Denote $(n_{k})_{k\geq 0}$  the (increasing) sequence of meeting times. Fix a lift $\check z$ of $z$. For every $k\geq 0$, there exists $p_k\in\Z$ such that $\check f^{n_{k}}(\check z)\in T^{p_k}(\check D)$. Moreover, $\lim_{k\to+\infty}  p_k/n_k=-\infty$.  Choose an integer $r$ such that $(p_1-p_0)+r(n_1-n_0)>0$. Consider the point $\check z''=\check f^{n_0}(\check z)$, the disk $\check D''=T^{p_0}(\check D)$ and the lift  $\check f''=\check f\circ T^{r}$.  We have
 $$\check z''\in \check D'', \,\check f''^{n_1-n_0}(\check z'')\in T^{p_1-p_0+r(n_1-n_0)}(\check D''),\, f''^{n_k-n_0}(\check z'')\in T^{p_k-p_0+r(n_k-n_0)}(\check D'').$$
 So $p_k-p_0+r(n_k-n_0)<0$, if $k$ is large enough. By Corollary \ref{cor:free-disks}, one deduces that $f$ has a horseshoe.

\bigskip
We will also prove $(2)$ by contradiction. Suppose that  $z$ and $z'$ belong to $\mathrm{ne}^+(f)$ and that $z'\in\omega(z)$ and suppose for instance that $\mathrm{rot}_{\check f}(z)<\mathrm{rot}_{\check f}(z')$. Choose two integers $p$ and $q>0$ relatively prime such that $\mathrm{rot}_{\check f}(z)<p/q<\mathrm{rot}_{\check f}(z')$. There exists $r\in\{0,\dots, q-1\}$ such that $f^r(z')$ is in the $\omega$-limit set of $z$ for $f^q$. Choose a maximal isotopy $I'$ of $f^q$ that is lifted to an identity isotopy of $\check f'=\check f^q\circ T^{-p}$ and a singular foliation ${\mathcal F}'$ transverse to $I'$. The point $f^r(z')$ being non-wandering and having a positive rotation number for $\check f'$, one deduces, by Proposition \ref{prop:non-wandering} that there exists a positive simple transverse loop $\Gamma$ such that $I'^{\Z}_{{\mathcal F}'}(f^r(z'))$ is drawn on $\Gamma$ and that $I'^{\N}_{{\mathcal F}'}(f^r(z')))$ draws infinitely $\Gamma$. The fact that $z'\in\omega_{f'}(z)$ implies that for every $k\geq 0$, $I'^{\N}_{{\mathcal F}'}(f'{}^k(z))$ draws $\Gamma$. By Proposition \ref{prop:drawing-crossing}, there exists $k\geq0$ such that $I'^{\N}_{{\mathcal F}'}(f'{}^k(z))$ is drawn on  $\Gamma$. This contradicts the fact that $\mathrm{rot}_{\check f'}(z)<0$.

\bigskip
The proof of $(3)$ is very close. Let $z$ be in $\Omega(f)\cap \mathrm{ne}^+(f)\cap \mathrm{ne}^-(f)$.  Suppose for instance that $-\mathrm{rot}_{\check f^{-1}}(z)<\mathrm{rot}_{\check f}(z)$. \textcolor{black}{Choose $z_-\in \alpha(z)$ and $z_+\in \omega(z)$ and then choose two integers $p$ and $q>0$ relatively prime such that $-\mathrm{rot}_{\check f^{-1}}(z)<p/q<\mathrm{rot}_{\check f}(z)$ and such that $\check f^q(z_ -)\not=T^p(z_-)$ and $\check f^q(z_ +)\not=T^p(z_+)$.} Choose a maximal isotopy $I'$ of $f^q$ that is lifted to an identity isotopy of $\check f'=\check f^q\circ T^{-p}$ and a singular foliation ${\mathcal F}'$ transverse to $I'$. The point $z$ being non-wandering for $f$ is \textcolor{black}{Birkhoff recurrent for $f$. By assertion (4) of Proposition \ref{prop: birkhoff connexions}, it is also Birkhoff recurrent for $f^q$. More precisely, there exists a Birkhoff circle $(z'_j)_{j\in\Z/p\Z}$ of $f^q$ such that every $z'_j$ is in the $f$-orbit of $z$ (and of course conjugating by power of $f$ one can suppose that $z'_0=z$). Moreover, \textcolor{black}{ it holds that  $z\in \mathrm{ne}^+(I')\cap \mathrm{ne}^-(I')$, because its $\alpha$-limit for $f^q$ contains an iterate of $z_{-}$ and because its omega-limit for $f^q$ contains an iterate of $z_+$.} One deduces by Proposition \ref{prop:Birkhoff-cycles}} that there exists a simple transverse loop $\Gamma$ such that $I'^{\Z}_{{\mathcal F}'}(z')$ is equivalent to the natural lift of $\Gamma$. This contradicts the fact that$-\mathrm{rot}_{\check f'^{-1}}(z)<0<\mathrm{rot}_{\check f'}(z)$.

It remains to prove $(4)$. Fix $z\in\Omega(f)\cap \mathrm{ne}(f)$. \textcolor{black}{Choose $z_{\pm}\in \alpha(z)\cup  \omega(z)$. If $p/q<\mathrm{rot}_{f^{\pm}}(z)$, written in an irreducible way, is sufficiently close to $\mathrm{rot}_{f^{\pm}}(z)$, then $\check f^q(z_ {\pm})\not=T^p(z_{\pm})$. Fix such a rational number $p/q$, }then choose a maximal isotopy $I'$ of $f^q$  that is lifted to an identity isotopy of $\check f^q\circ T^{-p}$ and a singular foliation ${\mathcal F}'$ transverse to $I'$. As explained \textcolor{black}{above, one knows that $z\in \mathrm{ne}(I')$ and so one can use Proposition \ref{prop:Birkhoff-cycles} and deduce that there exists a positive simple transverse loop $\Gamma$ such that $I'^{\Z}_{{\mathcal F}'}(z)$ exactly draws infinitely $\Gamma$. There exists a neighborhood $W$ of $z$ such that $I'{}^{\Z}_{{\mathcal F}'} (z')$ draws $\Gamma$ for every $z'\in W$. Suppose that $z'\in \textcolor{black}{\Omega(f)}\cap \mathrm{ne}(f)\cap W$. If $z'\not\in \mathrm{ne}(I')$, then  $\alpha(z')\cup  \omega(z')$ contains a fixed point of $\check f^q\circ T^{-p}$ and so $\mathrm{rot}_{\check f}(z')= p/q$.  If $z'\in \mathrm{ne}(I')$, it shares the same properties as $z$ and so $I'^{\Z}_{{\mathcal F}'}(z')$ exactly draws infinitely $\Gamma$. This implies that $\mathrm{rot}_{\check f}(z')\geq p/q$.} We have a similar result supposing  $p/q>\mathrm{rot}(z)$. So the function $\mathrm{rot}_{f^{\pm}}$ is continuous on $\Omega(f)\cap \mathrm{ne}(f)$.\end{proof}

\subsection{Birkhoff recurrent points and Birkhoff recurrent classes}  

Let us remind the theorem we want to prove, Theorem \ref{thmain:birkhoffcycles} of the introduction.
\begin{theorem}\label{th:birkhoffcycles}
Let $f$ be a homeomorphism of $\A$ isotopic to the identity and $\check f$ a lift of $f$ to $\R^2$.  We suppose that $f$ has no topological horseshoe. If $\mathcal B$ is a Birkhoff recurrence class of $f_{\mathrm{sp}}$, there exists $\rho\in\R$ such that,
$$\begin{cases}  \mathrm{rot}_{\check f}(z) = \rho&\mathrm{if}\enskip  z\in {\mathcal B}\cap\mathrm{ne}^+(f),\\
 \mathrm{rot}_{\check f^{-1}}(z) = -\rho&\mathrm{if}\enskip  z\in {\mathcal B}\cap\mathrm{ne}^-(f).
\end{cases}$$
\end{theorem}

Of course one can suppose that ${\mathcal B}\cap\mathrm{ne}(f)$ is not empty. Otherwise there is nothing to prove, every $\rho\in\R$ is suitable.  We will divide the proof into three cases, the case where $\mathcal B$ contains neither $N$ nor $S$, the case where $\mathcal B$ contains both $N$ and $S$ and the remaining case,  the most difficult one, where $\mathcal B$ contains exactly one end. Let us begin with the following lemma, that is nothing but Corollary \ref{corollarymain:birkhoffclasses} of the introduction.

\begin{lemma} \label{lemma:Birkhoff-cycles} Let $f$ be a homeomorphism of $\A$ isotopic to the identity with no topological horseshoe and $\check f$ a lift of $f$ to the universal covering space.  We suppose that $(z_i)_{i\in\Z/r\Z}$ is a Birkhoff cycle in $\mathrm{ne}(f)$ and $(\rho_i)_{i\in\Z/r\Z}$ a family of real numbers such that:
\begin{itemize} 
\item either $z_i\in\mathrm{ne}^+(f)$ and $\rho_i=\mathrm{rot}_{\check f}(z_i)$,
\item or $z_i\in\mathrm{ne}^-(f)$ and $\rho_i=-\mathrm{rot}_{\check f^{-1}}(z_i)$.
\end{itemize}
Then all $\rho_i$ are equal.
\end{lemma}

\begin{proof}  We will argue by contradiction.  Suppose that the $\rho_i$ are not equal. One can find $p\in\Z$ and $q>0$ relatively prime
such that
\begin{itemize}
\item  $\rho_i\not=p/q$, for every $i\in \Z$;
\item there exist $i_0,i_1$ in $\Z$ satisfying $\rho_{i_0}<p/q<\rho_{i_1}$;

\end{itemize}

Choose a maximal isotopy $I'$ of $f^q$ that is lifted to an identity isotopy of $\check f^q\circ T^{-p}$ and a singular foliation ${\mathcal F}'$ transverse to $I'$. The orbit of  each $z_i$ is included in the domain of $I'$ because the singular points of $I'$ are periodic points of period $q$ and rotation number $p/q$. Moreover, for every $i\in\Z/r\Z$ and every $k\in\Z$, it holds that $\check f^k(z_i)\in \mathrm{ne}(I')$. Using assertion (4) of Proposition \ref{prop: birkhoff connexions}, one can construct a Birkhoff cycle $(z'_j)_{j\in\Z/mr\Z}$ of $f^q$ such that $z'_j$ belongs to the orbit of $z_i$ if $j \equiv i \pmod r$. By Proposition  \ref{prop: rotation numbers powers}, one knows that 
\begin{itemize}
\item $z'_j\in\mathrm{ne}^+(f^q)$ if $z_i\in\mathrm{ne}^+(f)$ and $\mathrm{rot}_{\check f^q\circ T^{-p}}(z'_j)= q\rho_i-p$;
\item $z'_j\in\mathrm{ne}^-(f^q)$ if $z_i\in\mathrm{ne}^-(f)$ and $-\mathrm{rot}_{\check f^{-q}\circ T^p}(z'_j)= q\rho_i-p$.
\end{itemize}
\textcolor{black}{In particular $z'_j\in \mathrm{ne}^+(I')$ in the first case and $z'_j\in \mathrm{ne}^+(I')$ in the second case.} 	Consequently, $I'^{\Z}_{{\mathcal F}'}(z'_j)$ draws a negative transverse simple loop if $j \equiv  i_0 \pmod r$ and $I'^{\Z}_{{\mathcal F}'}(z'_j)$ draws a positive transverse simple loop if $j \equiv  i_1 \pmod r$.  This contradicts Proposition \ref{prop:Birkhoff-cycles}, \textcolor{black}{where it is said that there exists a transverse simple loop $\Gamma$ such that each trajectory $I'^{\Z}_{{\mathcal F}'}(z'_j)$ exacly draws infinitely $\Gamma$. }\hfill$\Box$

We will also use the following result
\begin{lemma}\label{lemma:periodic-orbits}
Let $f$ be a homeomorphism of $\A$ isotopic to the identity and $\check f$ a lift of $f$ to $\R^2$. Suppose that there exist $z_0$ and $z_1$ in $ \mathrm{ne}(f)$ and $\rho_0<\rho_1$ such that such for every $i\in\{0,1\}$:
\begin{itemize} 
\item either $z_i\in\mathrm{ne}^+(f)$ and $\rho_i=\mathrm{rot}_{\check f}(z_i)$,
\item or $z_i\in\mathrm{ne}^-(f)$ and $\rho_i=-\mathrm{rot}_{\check f^{-1}}(z_i)$.
\end{itemize}
Suppose moreover that $z_0$ and $z_1$ are in the same Birkhoff recurrence class of $f_{\mathrm{sp}}$, or that $N$ and $S$ are in the same Birkhoff recurrence class of $f_{\mathrm{sp}}$. Then, for every rational number $p/q\in(\rho_0,\rho_1)$, written in an irreducible way, there exists a periodic point $z$ of period $q$ and rotation number $\mathrm{rot}_{\check f}(z)=p/q$.
\end{lemma}

\begin{proof} Fix a rational number $p/q\in(\rho_0,\rho_1)$, written in an irreducible way, then choose a maximal isotopy $I'$ of $f^q$ that is lifted to an identity isotopy of $\check f^q\circ T^{-p}$ and a foliation ${\mathcal F}'$ transverse to $I'$.  It is sufficient to prove that if ${\mathcal F}'$ is non singular, then $z_0$ and $z_1$ are not in the same Birkhoff recurrence class of $f_{\mathrm{sp}}$, and $N$ and $S$ are not in the same Birkhoff recurrence class of $f_{\mathrm{sp}}$. By Proposition  \ref{prop: rotation numbers powers}, if $z'_i$ belongs to the orbit of $z_i$, then
\begin{itemize}
\item $z'_i\in\mathrm{ne}^+(f^q)$ if $z_i\in\mathrm{ne}^+(f)$ and $\mathrm{rot}_{\check f^q\circ T^{-p}}(z'_i)= q\rho_i-p$;
\item $z'_i\in\mathrm{ne}^-(f^q)$ if $z_i\in\mathrm{ne}^-(f)$ and $-\mathrm{rot}_{\check f^{-q}\circ T^p}(z'_i)= q\rho_i-p$.
\end{itemize}  
\textcolor{black}{In particular $z'_i\in \mathrm{ne}^+(I')$ in the first case and $z'_i\in \mathrm{ne}^+(I')$ in the second case. Consequently, there exists a  negative transverse loop $\Gamma_0$ that contains a point $f^{k_0q}(z'_0)$, $k_0\in\Z$,  and a positive transverse loop $\Gamma_1$ that contains a point $f^{k_1q}(z'_1)$, $k_1\in\Z$.}  One can find positive integers $n_0$ and $n_1$ such that the multi-loop $n_0\Gamma_0+n_1\Gamma_1$ is homologous to zero. Write $\delta$ for its dual function that vanishes in a neighborhood of the ends of $\A$. Either the minimal value $m_-$ of $\delta$ is negative or the maximal value $m_+$ is positive. We will look at the first case, the second one being analogous. The fact that $\delta$ decreases along the leaves of ${\mathcal F}'$ implies that the closure of a connected component $U$ of $\A\setminus (\Gamma_0\cup\Gamma_1)$ where $\delta$ is equal to $m_-$ is a topological manifold whose boundary is transverse to ${\mathcal F}'$, the leaves going inside $U$. By connectedness of $\Gamma_0$ and $\Gamma_1$, we know that $\overline U$ is a topological disk or an essential annulus. The first situation is impossible because we should have a singular point of ${\mathcal F}'$ inside $U$. So $\overline U$ is an annulus and by Poincar\'e-Bendixson Theorem, it contains an essential closed leaf. This leaf \textcolor{black}{must separate $\Gamma_0$ and $\Gamma_1$} and is disjoint from its image by $f^q$. So  $f^{k_0q}(z'_0)$ and $f^{k_1q}(z'_1)$, {which belong to $\Gamma_0$ and $\Gamma_1$ respectively} are not in the same Birkhoff recurrence class of $f_{\mathrm{sp}}^q$. \textcolor{black}{By assertions (2) and (4) of Proposition \ref{prop: birkhoff connexions}, one deduces first that $z'_0$ and $z'_1$ are not in the same Birkhoff recurrence class of $f^q_{\mathrm{sp}}$ and then that $z_0$ and $z_1$ are not in the same Birkhoff recurrence class of $f_{\mathrm{sp}}$. }Note also that $N$ and $S$ are not in the same Birkhoff recurrence class of $f_{\mathrm{sp}}$. \end{proof}

Let us go now to the proof of Theorem \ref{th:birkhoffcycles}. We suppose that there exist $z^*_0$ and $z^*_1$ in ${\mathcal  B}\cap \mathrm{ne}(f)$ and $\rho_0<\rho_{1}$ such that such for every $i\in\{0,1\}$:
\begin{itemize} 
\item either $z^*_i\in\mathrm{ne}^+(f)$ and $\rho_i=\mathrm{rot}_{\check f}(z^*_i)$,
\item or $z^*_i\in\mathrm{ne}^-(f)$ and $\rho_i=-\mathrm{rot}_{\check f^{-1}}(z^*_i)$.
\end{itemize}
We want to find a contradiction.

\bigskip

{\it First case: the Birkhoff class $\mathcal B$ does not contain $N$ or $S$.}

There exists a Birkhoff cycle  $(z_i)_{i\in\Z/r\Z}$ of $f_{\mathrm{sp}}$ that contains $z^*_0$ and $z^*_1$ and does not contain $N$ and $S$. So, it is a  Birkhoff cycle of $f$. Using item (1) of Proposition  \ref{prop: birkhoff connexions}, one knows that for every $i\in\Z/r\Z$, the closure of $O_{f_{\mathrm{sp}}}(z_i)$ is contained in $\mathcal B$ and so does not contain either $N$ or $S$. In particular $z\in \mathrm{ne}^+(f)\cap\mathrm{ne}^-(f)$. We can apply  Lemma \ref{lemma:Birkhoff-cycles} and obtain a contradiction.

\bigskip
{\it Second case: the Birkhoff class $\mathcal B$ contains $N$ and $S$.}

 In that case, we will get the contradiction from the following result that is nothing but Proposition \ref{propositionmain_regionofinstability} and does not use the fact that $z^*_0$ and $z^*_1$ are in the same Birkhoff recurrence class as $N$ and $S$.
\begin{proposition}\label{prop:PB}
Let $f$ be a homeomorphism of $\A$ isotopic to the identity, $\check f$ a lift of $f$ to the universal covering space and $f_{\mathrm{sp}}$ be the continuous extension of $f$ to $\S^2=\A\cup\{N,S\}$.   We suppose that $\rho_0<\rho_1$ are two real numbers and that:
\begin{itemize} 
\item there exists $z_0^*\in\mathrm{ne}^+(f)$ such that $\rho_0=\mathrm{rot}_{\check f}(z_0^*)$ or there exists $z_0^*\in\mathrm{ne}^{-}(f)$ such that $\rho_0=-\mathrm{rot}_{\check f^{-1}}(z_0^*)$;
\item there exists $z_1^*\in\mathrm{ne}^+(f)$ such that $\rho_1=\mathrm{rot}_{\check f}(z_1^*)$ or there exists $z_1^*\in\mathrm{ne}^-(f)$ such that $\rho_1=-\mathrm{rot}_{\check f^{-1}}(z_1^*)$;
\item $N$ and $S$ are in the same Birkhoff recurrence class of $f_{\mathrm{sp}}$.
\end{itemize}
Then, $f$ has a topological horseshoe.
\end{proposition}

\begin{proof} We will argue by contradiction, supposing that $f$ has no topological horseshoe. In that case, there exists a Birkhoff cycle  $(z_i)_{i\in\Z/r\Z}$ of $f_{\mathrm{sp}} $ that contains $N$ and $S$. Let $p/q<p'/q'$ be rational numbers in $(\rho_0,\rho_1)$ such that 
\begin{itemize}
\item there is no $i\in\Z/r\Z$ such that $z_i\in\mathrm{ne}^+(f)$ and $\mathrm{rot}_{\check f}(z_i)\in\{p/q,p'/q'\}$;
\item there is no $i\in\Z/r\Z$ such that $z_i\in\mathrm{ne}^-(f)$ and $\mathrm{rot}_{\check f^{-1}}(z_i)\in\{-p/q,-p'/q'\}$.
\end{itemize}
Since $N$ and $S$ belong to the same Birkhoff recurrence class of $f_{\mathrm{sp}}$, we can apply Lemma \ref{lemma:periodic-orbits} and find a periodic {point} $z\in \A$ of period $q$ and rotation number $\mathrm{rot}_{\check f}(z)=p/q$, a periodic point $z'\in\A$ of period $q'$ and rotation number $\mathrm{rot}_{\check f}({z'})=p'/q'$. \textcolor{black}{Using Theorem \ref{th:birkhoffcycles}  and the assumptions above, one know that
for every $i\in\Z/r\Z$ and every $k\in\Z$, it holds that}
$$\omega_{f_{\mathrm{sp}}^{qq'}}(f^{k}(z_i))\cap\{z,z'\}=\alpha_{f_{\mathrm{sp}}^{qq'}}(f^{k}(z_i))\cap\{z,z'\}=\emptyset.$$
By Proposition \ref{prop: birkhoff connexions}, one can construct a Birkhoff cycle $(z'_j)_{j\in\Z/mr\Z}$ of $f_{\mathrm{sp}}^{qq'}$ such that $z'_j$ belongs to the orbit of $z_i$ if $j \equiv i \pmod r$, and we know that $z'_j$ belong to $\mathrm{ne}(f_{\mathrm{sp}}^{qq'}\vert_{\S^2\setminus\{z, z'\}})$.
Applying Lemma \ref{lemma:Birkhoff-cycles}, we deduce that if $\check g$ is a lift of $f_{\mathrm{sp}}^{qq'}\vert_{\S^2\setminus\{z, z'\}}$ to the universal covering space, if $\kappa\in H_1(\S^2\setminus\{z, z'\},\Z)$ is the generator given by the oriented boundary of a small disk containing $z$ in its interior, then  $\mathrm{rot}_{\check g,\kappa}(N)=\mathrm{rot}_{\check g,\kappa}(S)$. The contradiction will come from the following equality
$$\mathrm{rot}_{\check g,\kappa}(S)-\mathrm{rot}_{\check g,\kappa}(N)= \mathrm{rot}_{\check f^{qq'}}(z)-\mathrm{rot}_{\check f^{qq'}}(z')= q'p-qp'\not=0.$$
There are different explanations of the previous equality (see Franks [F] for example). Let us mention the following one. By conjugacy, one can always suppose that $\S^2$ is the Riemann sphere. So one can define the cross ratio $[z_1,z_2,z_3,z_4]$ of four distinct points. Let us choose  an identity isotopy $(g_t)_{t\in[0,1]}$ of \textcolor{black}{$f^{qq'}_{\mathrm{sp}}$}. The closed path
$$t\mapsto[g_t(z), g_t(z'), g_t(S),g_t(N)]$$ defines a loop $\Gamma$ on $\S^2\setminus \{0,1,\infty\}$ whose homotopy class is independent of the chosen isotopy, two isotopies being homotopic. The integer $r=\delta_{\Gamma}(0)-\delta_{\Gamma}(\infty)$, where  $\delta_{\Gamma}$ is a dual function of $\Gamma$, also does not depend on the isotopy. The integer $\mathrm{rot}_{\check f^{qq'}}(z)-\mathrm{rot}_{\check f^{qq'}}(z')$ is independent of the lift $\check f$. In fact it is equal to $r$. Indeed, one can suppose that $(z',S,N)=(1,0,\infty)$ and that $(g_t)_{t\in[0,1]}$ fixes $z'$, $S$ and $N$. So it defines a lift $\check g^*$ of $f^{qq'}$ to the universal covering space of $\A$ such that $\mathrm{rot}_{\check g^*}(z')=0$. But one has $[g_t(z), g_t(z'), g_t(S),g_t(N)]=g_t(z)$. So $\mathrm{rot}_{\check g^*}(z)$ is nothing but $r$. As we know that  $[g_t(z), g_t(z'), g_t(S),g_t(N)]=[g_t(S), g_t(N), g_t(z),g_t(z')]$, we deduce for the same reasons that $\mathrm{rot}_{\check g,\kappa}(S)-\mathrm{rot}_{\check g,\kappa}(N)=r$. \end{proof}

Let us state now a corollary that will useful while studying the last case and that says that $z_0^*$ and $z_1^*$ cannot be periodic.

\begin{corollary}\label{cor:periodic}
Let $f$ be a homeomorphism of $\A$ isotopic to the identity, $\check f$ a lift of $f$ to the universal covering space and $f_{\mathrm{sp}}$ be the continuous extension of $f$ to $\S^2=\A\cup\{N,S\}$.   We suppose that there exist two periodic points $z_0^*$ and $z_1^*$ of $f$ such that
\begin{itemize}
\item $\mathrm{rot}_{\check f}(z_0^*)<\mathrm{rot}_{\check f}(z_1^*)$;
\item $z_0^*$ and $z_1^*$ are in the same Birkhoff recurrence class of $f_{\mathrm{sp}}$.
\end{itemize}
Then, $f$ has a topological horseshoe.
\end{corollary}
\begin{proof} Taking a power of $f$ instead of $f$, one can suppose that $z_0^*$ and $z_1^*$ are fixed points. As explained in the proof of the Proposition \ref{prop:PB}, if $\check g$ is a lift of $f_{\mathrm{sp}}\vert_{\S^2\setminus\{z_0^*, z_1^*\}}$ to the universal covering space, if $\kappa\in H_1(\S^2\setminus\{z_0^*, z_1^*\},\Z)$ is the generator given by the oriented boundary of a small disk containing $z_0^*$ in its interior, then  
$$\mathrm{rot}_{\check g,\kappa}(S)-\mathrm{rot}_{\check g,\kappa}(N)= \mathrm{rot}_{\check f}(z_0^*)-\mathrm{rot}_{\check f}(z_1^*)\not=0.$$
So, one can apply Proposition \ref{prop:PB} to deduce that $f$ has a topological horseshoe.
\end{proof}

\bigskip
{\it Third case: the Birkhoff class $\mathcal B$ contains $S$ but does not contain $N$.}

Applying Lemma \ref{lemma:periodic-orbits} and Corollary \ref{cor:periodic}, one can find a rational number $p/q\in(\rho_0,\rho_1)$ such that if $z\in \A$ is a periodic point of period $q$ and rotation number $\mathrm{rot}_{\check f}(z)=p/q$, then: 
\begin{itemize}
\item for every $k\in\Z$, one has $\omega_{f_{\mathrm{sp}}^q}(f^k(z_0^*))\not=\{z\}$ and $\alpha_{f_{\mathrm{sp}}^q}(f^k(z_0^*))\not =\{z\}$;
\item for every $k\in\Z$, one has $\omega_{f_{\mathrm{sp}}^q}(f^k(z_1^*))\not=\{z\}$ and $\alpha_{f_{\mathrm{sp}}^q}(f^k(z_1^*))\not =\{z\}$;
\item $z$ does not belong to $\mathcal B$.
\end{itemize}
By Proposition \ref{prop: birkhoff connexions}, $z^*_0$ and $z_1^*$  belong to the same Birkhoff recurrence class of $f^q_{\mathrm{sp}}$ {that $S$ belongs to}. Choose a maximal isotopy $I'$ of $f^q$ that is lifted to an identity isotopy of $\check f^q\circ T^{-p}$ and a singular foliation ${\mathcal F}'$ transverse to $I'$. The singular points of  ${\mathcal F}'$ are periodic points of period $q$ and rotation number $\mathrm{rot}_{\check f}(z)=p/q$, plus the two ends $N$ and $S$ if we consider the foliation as defined on the sphere. If $z$ and $z'$ are two singular points, the isotopy $I'\vert_{\S^2\setminus\{z,z'\}}$ can be lifted to an identity isotopy of  a lift of $f^q\vert_{\S^2\setminus\{z,z'\}}$. We write $\check g_{z,z'}$ for this lift. We will consider 
a small disk containing $z$ in its interior to define rotation number. Of course $\mathrm{rot}_{\check g_{z,z'}}(z'')=0$ if $z''$ is a third singular point. The first case studied above tells us that if neither $z$ nor $z'$ is equal to $S$, then $z^*_0$, $z^*_1$ and $S$ belong to $\mathrm{ne}(f^q\vert_{\S^2\setminus\{z,z'\}})$
and that $$\mathrm{rot}_{\check g_{z,z'}}(z^*_0)=\mathrm{rot}_{\check g_{z,z'}}(z^*_1)=\mathrm{rot}_{\check g_{z,z'}}(S)=0.$$
By Proposition \ref{prop: rotation numbers}, one deduces that if $z'$ is a singular point distinct from $S$, one has $$\mathrm{rot}_{\check g_{S,z'}}(z^*_0)=\mathrm{rot}_{\check g_{S,N}}(z^*_0)-\mathrm{rot}_{\check g_{z',N}}(z^*_0)=q\rho_0-p.$$
Similarly, one gets
$$\mathrm{rot}_{\check g_{S,z'}}(z^*_1)=q\rho_1-p,$$ and consequently 
$$\mathrm{rot}_{\check g_{S,z'}}(z^*_0)<\mathrm{rot}_{\check g_{S,z'}}(z^*_1).$$ 

By assumption, $\mathrm{dom}(I')\cap (\alpha_{f_{\mathrm{sp}}^q}(z^*_0)\cup \omega_{f_{\mathrm{sp}}^q}(z^*_0))$ is non empty. Choose a point $z_0$ in this set and denote $\phi_0$ the leaf of ${\mathcal F}'$ that contains this point. Similarly, choose a point $z_1$ in $\mathrm{dom}(I')\cap (\alpha_{f_{\mathrm{sp}}^q}(z^*_1)\cup \omega_{f_{\mathrm{sp}}^q}(z^*_1))$ and denote $\phi_1$ the leaf  that contains this point. Note that both $z_0$ and $z_1$ belong to the same Birkhoff recurrence class of $f_{\mathrm{sp}}^q$ as $S$, $z_0^*$ and $z_1^*$. The fact that $\mathrm{rot}_{\check g_{S,z'}}(z^*_0)<0<\mathrm{rot}_{\check g_{S,z'}}(z^*_1)$ implies that for every singular point $z'\not=S$ there exists a simple transverse loop $\Gamma_{0}$ crossing $\phi_0$ and such that $z'\in L(\Gamma_0)$ and $S\in R(\Gamma_0)$ and a simple transverse loop $\Gamma_{1}$ crossing $\phi_1$ and such that $z'\in R(\Gamma_1)$ and $S\in L(\Gamma_1)$. Denote $\widehat\omega(\phi_0)$ the complement of the connected component of $\S^2\setminus\omega(\phi_0)$ that contains $\phi_0$. It is a saturated cellular closed set that contains $\omega(\phi_0)$ and that contains at least one singular point. Define in similar way $\widehat\alpha(\phi_0)$, $\widehat\omega(\phi_1)$ and $\widehat\alpha(\phi_1)$. The fact that for every singular point there exists a simple transverse loop $\Gamma_{0}$ crossing $\phi_0$ and such that $z'\in L(\Gamma_0)$ and $S\in R(\Gamma_0)$ implies that $S$ is the unique singular point in $\widehat\omega(\phi_0)$. Indeed $\Gamma_0$ is contained in $\S^2\setminus\omega(\phi_0)$, as is $L(\Gamma_0)$.  Similarly $S$ is the unique singular point contained in $\widehat\alpha(\phi_1)$. We deduce that $\phi_1$ does not meet $\Gamma_0$, otherwise $\Gamma_0$ is contained in $\S^2\setminus\alpha(\phi_1)$, as is $L(\Gamma_0)$. Similarly $\phi_0$ does not meet $\Gamma_1$.

Suppose that $S\not\in\omega(\phi_0)$. In that case there is no singular point in  $\omega(\phi_0)$, which implies that $\omega(\phi_0)$ is a closed leaf $\phi'_0$. But $\phi_0'$ separates $z_0$ and $S$, and this is impossible since both belong to the same Birkhoff recurrence class. Therefore $S\in{\omega}(\phi_0)$ and one can likewise show that $S\in\alpha(\phi_1)$. Let us prove now that $\alpha(\phi_0)$ does not contain any singular point. Indeed if $z'\in\alpha(\phi_0)\cap \mathrm{Fix}(I')$, there exists a transverse simple loop $\Gamma_1$ that separates $S$ and $z'$ and that does not meet $\phi_0$. So $\alpha(\phi_0)$ is a closed leaf $\phi''_0$. Similarly, $\omega(\phi_1)$ is a closed leaf $\phi''_1$. Note that there is no closed leaf that separates $\phi_0$ and $\phi_1$ because both $z_0$ and $z_1$ are in the same Birkhoff recurrence class of $f^q$. So the leaves $\phi''_0$ and $\phi''_1$ are equal. But this is impossible because this leaf cannot be both attracting and repelling from the side that contains $S$. \textcolor{black}{This contradiction completes the proof of Theorem \ref{th:birkhoffcycles}.}\end{proof}

The following Corollary will be later used to prove a broader result in Corollary \ref{cr:onlytwoperiods}
\begin{corollary}\label{cr:divideperiods} Let $f$ be an orientation preserving homeomorphism of $\S^2$. We suppose that there exists a Birkhoff recurrence class $\mathcal{B}$ which contains two periodic points, $z_0, z_1$, with respective smallest periods $q_0<q_1$, and that 
$q_{0}$ does not divide $q_{1}$. Then $f$ has a topological horseshoe.
\end{corollary}

\begin{proof}We suppose for a contradiction that $f$ has no topological horseshoe. By Proposition \ref{prop: birkhoff connexions}, one can find a point in the orbit of $z_0$ and a point in the orbit of $z_1$ that belong to the same Birkhoff recurrence class of $f^{q_1}$ and there is no loss of generality by supposing that this is the case for $z_0$ and $z_1$. Note that for every $k\in\Z$, the points $f^k(z_0)$ and $f^k(z_1)$ belong to the same Birkhoff recurrence class of $f^{q_1}$. Suppose that $q_0$ does not divide $q_1$ and denote $s$ the remainder of the Euclidean division of $q_{1}$ by $q_{0}$. It implies that $f^{q_1}(z_{0})=f^s(z_{0})\not=z_0$. Since $q_{1}$ is not a multiple of $q_{0}$, it is larger than two and $f^{q_{1}}$ must have at least three distinct fixed points. Choose a maximal isotopy $I'$ of $f^{q_{1}}$ whose singular set contains contains $z_1$, $f^{s}(z_1)$ and at least a third fixed point of $f^{q_{1}}$, then consider a singular foliation ${\mathcal F}'$ transverse to $I'$.  The point $z_{0}$ belongs to the domain of $I'$ and $I'_{{\mathcal F}'}{}^{\Z}(z_{0})$ is equivalent to the natural lift of a simple transverse loop $\Gamma$. Otherwise, one knows by Proposition \ref{prop:non-wandering} that $f$ has a horseshoe. Moreover $I'_{{\mathcal F}'}{}^{\Z}(z_{0})$ coincides with $I'_{{\mathcal F}'}{}^{\Z}(f^s(z_{0}))$ because $f^s(z_{0})=f^{q_{1}}(z_{0})$. One can find two fixed points $z$ and $z'$ of $I'$ separated by $\Gamma$, and since $I'$ has at least three fixed points, one can assume that $\{z, z'\}\not=\{z_{1},f^s(z_{1})\}$. The isotopy $I'$ can be lifted to an identity isotopy of a certain lift $\check g$ of $f^{q_{1}}\vert_{\S^2\setminus\{z, z'\}}$ and $\mathrm{rot}_{\check g}(z_0)=\mathrm{rot}_{\check g}(f^s(z_0))\not=0.$
Moreover, $\mathrm{rot}_{\check g}(z'')=0,$ for every fixed point $z''$ of $I'$ that is different from $z$ and $z'$.
So, if $z_{1}\not\in\{z, z'\}$, we have $\mathrm{rot}_{\check g}(z_0)\not=\mathrm{rot}_{\check g}(z_1)$ and if $f^s(z_{1})\not\in\{z, z'\}$, we have $\mathrm{rot}_{\check g}(f^s(z_0))\not=\mathrm{rot}_{\check g}(f^s(z_1))$.  By Theorem \ref{th:birkhoffcycles}, this contradicts the fact that $f$ has no topological horseshoe. \end{proof}

\subsection{Circloids}
Let us remind the statement of Theorem \ref{thmain:circloids} of the introduction.

\begin{theorem}\label{th:circloids} Let $f$ be a homeomorphism of $\A$ isotopic to the identity and $\check f$ a lift of $f$ to $\R^2$. We suppose that $f$ has no topological horseshoe. Let $X$ be an invariant circloid. If $f$ is dissipative or if $X$ is locally stable, then the function \textcolor{black}{$\mathrm{rot}_{\check f}$}, which is well defined on $X$, is constant on this set. Moreover, the
sequence of maps
$$\check z\mapsto \frac{ \pi_1(\check f^{n}(\check z))- \pi_1(\check z)}{n}$$ 
converges uniformly to this constant on $\pi^{-1}(X)$.
\end{theorem}

\begin{proof} 
We suppose that $f$ and $X$ satisfy the hypothesis of the theorem. In particular $f$ is dissipative or $X$ is locally stable.  We denote $U_S$ and $U_N$ the non relatively compact connected components of $\A\setminus X$,  neighborhoods of $S$ and $N$ respectively. By Theorem \ref{th:rotation-number}, the function \textcolor{black}{$\mathrm{rot}_{\check f}$} is well defined on $X$, and it is continuous when restricted to $\Omega(f)\cap X$. To prove that it is constant, it is sufficient to prove that the
sequence of maps
$$\check z\mapsto \frac{ \pi_1(\check f^{n}(\check z))- \pi_1(\check z)}{n}$$ 
converges uniformly to  a constant on $\pi^{-1}(X)$.  Assume for a contradiction that it is not the case. Koropecki proved in \cite{Koropecki} that there exist real numbers $\rho<\rho'$ such that, for every $p/q\in (\rho,\rho')$ written in an irreducible way, there exists a periodic point of period $q$ and rotation number $p/q$, extending the method of Barge-Gillette used in the special case of cofrontiers (see  \cite{BargeGillette}). Let us consider  in $(\rho,\rho')$ a decreasing sequence $(p_n/q_n)_{n\geq 0}$ converging to $\rho$ and an increasing sequence $(p'_n/q'_n)_{n\geq 0}$ converging to $\rho'$. For every $n$, choose a periodic point $z_n$ of period $q_n$ and rotation number $p_n/q_n$ and denote $O_n$ its orbit. Similarly, choose a periodic point $z'_n$ of period $q'_n$ and rotation number $p'_n/q'_n$ and denote $O'_n$ its orbit. Taking subsequences if necessary, one can suppose that the sequences $(z_n)_{n\geq 0}$ and $ (z'_n)_{n\geq 0}$ converge to points $z$ and $z'$ in $X$, respectively. Since the rotation number is continuous on the nonwandering set of $f$ restricted to $X$, we get that $\rho_{\check f}(z)=\rho$ and $\rho_{\check f}(z)=\rho'$.  We will prove now that there is a Birkhoff connection from $z$ to $z'$. Of course, one could prove in the same way that there is a Birkhoff connection from $z'$ to $z$. So, by Theorem \ref{th:birkhoffcycles}, it implies that $f$ has a topological horseshoe, in contradiction with the hypothesis.

Say an open set $U\subset\A$ is {\it essential} if it contains an essential loop. Then
\begin{lemma}
\label{lemma:essential} If $W$ is a forward (respectively backward) invariant open set that contains $z$, then the connected component of $z$ in $W$ is essential and also forward (respectively backward) invariant.  
\end{lemma}

\begin{proof}
Let $W$ be a forward invariant open set that contains $z$. The connected component $V$ of $W$ that contains $z$ contains the periodic points $z_n$ if $n$ is large, so it is forward invariant by a power $f^q$ of $f$ and contains the $f^q$-orbit of $z_n$. Let $\check V$ be a connected component of $\check \pi^{-1}(V)$. If $n$ is large, $\check V$ contains a point $\check z_n$ such that $\check f^{qq_n}(\check z_n)=T^{qp_n}(\check z_n)$ and so $\check f^{qq_n}(\check V)\subset T^{qp_n}(\check V)$. Let us fix $n_0<n_1$ large enough. One deduces that $\check f^{qq_{n_0}q_{n_1}}(\check V)$ is included both in $T^{qp_{n_0}q_{n_1}}(\check V)$ and in $T^{qq_{n_0}p_{n_1}}(\check V)$, which implies that $T^{qp_{n_0}q_{n_1}}(\check V)=T^{qq_{n_0}p_{n_1}}(\check V)$. As $qp_{n_0}q_{n_1}\not=qq_{n_0}p_{n_1}$, it implies that $V$ is essential. Furthermore, since $V$ is a connected component of $W$ and $W$ is forward invariant, then either $f(V)\subset V$, or $f(V)$ is disjoint from $V$. But any open essential set of $\A$ whose image is disjoint from itself must be wandering, and since $V$ has periodic points, this implies that $V$ is forward invariant. We have a similar proof in case $W$ is backward invariant.\end{proof}

Let $U$ be a neighborhood of $z$ and $U'$ be a neighborhood of $z'$ such that $U\cap U'=\emptyset$. Set $V$ to be the connected component of $\bigcup_{n\geq 0} f^n(U)$ that contains $z$ and $V'$ to be the connected component of $\bigcup_{n\geq 0} f^{-n}(U')$ that contains $z'$. We will prove that $V\cap V'\not=\emptyset$. By Lemma \ref{lemma:essential}, $V$ is a connected essential and forward invariant open set that contains $z$. Similarly, $V'$ is an essential connected backward invariant open set that contains $z'$.

We will suppose from now that $V\cap V'=\emptyset$ and will find a contradiction. The connected component $K'$ of $\A\setminus V$ that contains $V'$ is a neighborhood of one end of the annulus, $N$ for instance, and is backward invariant. Its complement $W$ is an open neighborhood of $S$ that contains $V$ and is disjoint from $V'$. More precisely, it is a forward invariant open annulus. Similarly, the connected component $K$ of $\A\setminus V'$ that contains $V$ is a neighborhood of $S$ and is forward invariant. Moreover it contains $W$ because $V'\subset K'$. Its complement $W'$ is a backward invariant annulus, neighborhood of $N$ that contains $V'$. Moreover, one has $W\cap W'=\emptyset$ because $W\subset K$.

\begin{figure}[ht!]
\hfill
\includegraphics [height=48mm]{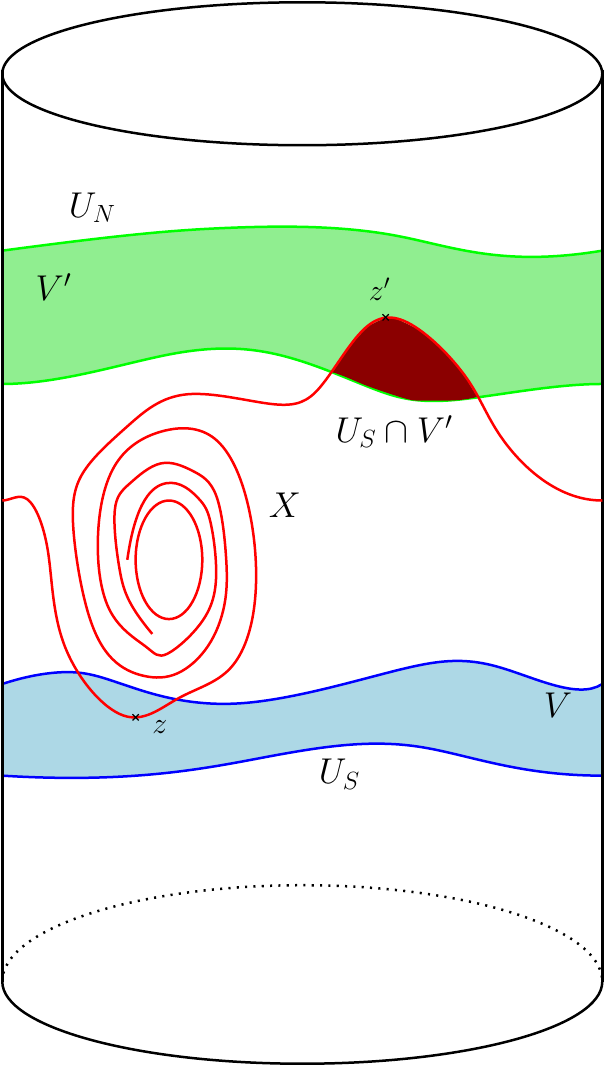}
\hfill{}
\caption{\small Theorem \ref{th:circloids}, the case where $f$ is dissipative.}
\label{figure_circloid}
\end{figure}

\medskip

The fact that $X$ is a circloid implies that $\overline U_S$ intersects $V'$. So $V'\cap U_S$ is a non empty open set, relatively compact and backward invariant (see Figure \ref{figure_circloid}). Consequently, $f$ is not dissipative.  The open set $W'\cap U_S$ is not empty because $V'\subset W'$. It implies that $W'\cup U_S$ is connected and contains a topological line $\lambda$ that joins $N$ to $S$. If $U$ is a sufficiently small neighborhood of $X$, then $U\cap\lambda\subset W'$. So, $W\cap U$ does not meet $\lambda$ because $W\cap U\cap\lambda\subset W\cap W'=\emptyset$, and is not essential. If $X$ is locally stable, one can choose such  a neighborhood $U$ to be forward invariant. This contradicts Lemma \ref{lemma:essential} because $W\cap U$ contains $z$ and is not essential. \end{proof}

\subsection{Annular homeomorphisms with fixed point free lifts} 
Let us finish this section with the proof of Proposition \ref{prop:zero-rotation} that we remind:

\begin{proposition}\label{prop:recurrentzero-rotation}
Let $f$ be a homeomorphism \textcolor{black}{of $\A$} isotopic to the identity and $\check f$ a lift of $f$ to the universal covering space.  We suppose that $\check f$ is fixed point free and that there exists  $z_*\in\Omega(f)\cap \mathrm{ne}^+(f)$ such that  $\mathrm{rot}_{\check f}(z_*)$ is well defined and equal to $0$.

\begin{itemize}
\item[i)]{Then, at least one of the following situation occurs:
\begin{enumerate}
\item  there exists $q$ arbitrarily large such that $\check f^q\circ T^{-1}$ has a fixed point;  
\item there exists $q$ arbitrarily large such that $\check f^q\circ T$ has a fixed point.
\end{enumerate}}
\item[ii)]{Moreover, if $z_*$ is recurrent, then $f$ has a horseshoe.}
\end{itemize}
\end{proposition}

\begin{proof} We begin with the proof of $i)$. Let $I$ be a maximal isotopy of $f$ that is lifted to an identity isotopy $\check I$ of $\check f$. Choose a transverse foliation $\mathcal F$ of $I$ and write $\check {\mathcal F}$ for the lifted foliation. Note that $\mathcal F$ and $\check {\mathcal F}$ are not singular because $\check f$ has no fixed point. By hypothesis $\omega_f(z_*)$ is non empty and $I_{\mathcal F}^{\N}(z_*)$ meets infinitely often every leaf that contains a point of $\omega_f(z_*)$. Fix such a leaf $\phi$ and a lift $\check z_*$ of $z_*$. There exist two integers $p_0\in\Z$  and $q_0\geq 2$ and a lift $\check \phi$ of $\phi$ such that  $\check I_{\check{\mathcal F}}^{q_0-1}(\check f(\check z_*))$ contains as a sub-path a transverse path that goes from $\check \phi$ to $T^{p_0}(\check \phi)$. The foliation $\check {\mathcal F}$ being non singular,  every transverse path meets a leaf at most once and so $p_0\not=0$. For the same reason, $\check z_*$ is not periodic and consequently $z_*$ itself is not periodic (because  $\mathrm{rot}_{\check f}(z_*)=0$). We will prove the first assertion in case $p_0>0$ (and would obtain the second one in case $p_0<0$).

\medskip
One can find an open disk $D$ containing $z_*$ such that:

\begin{itemize}
\item$f^k(D)\cap D=\emptyset$ for every $k\in\{1,\dots, q_0\}$;

\item if $\check D$ is the lift of $D$ that contains $\check z_*$, then for every $\check z\in\check D$, the path $\check I_{\check{\mathcal F}}^{q_0-1}(\check f(\check z_*))$ is a sub-path of $\check I_{\check{\mathcal F}}^{q_0+1}( \check z)$ (up to equivalence) . 
\end{itemize}

The second item tells us that there exists a transverse path from every point of $\check D$ to $\check \phi$ and a transverse path from $T^{p_0}(\check \phi)$ to every point of $\check f^q(D)$ if $q>q_0$.  

The point $z_*$ being non-wandering, there exists $q>q_0$ arbitrarily large such that $f^{q}(D)\cap D\not=\emptyset$. We will prove that $\check f^q\circ T^{-1}$ has a fixed point. There exists $p\in\Z$ such that $\check f^q(\check D)\cap T^p(\check D)\not=\emptyset$. Let us prove that $p>p_0$.  There exists a transverse path from $\check \phi$ to $T^{p_0}(\check \phi)$ and so a transverse path from $\check \phi$ to $T^{np_0}(\check \phi)$ for every $n\geq 1$. There exists a transverse path from every point of $\check D$ to $\check \phi$ and a transverse path from $T^{p_0}(\check \phi)$ to every point of $\check f^q(D)$.  So, if $\check f^q(\check D)\cap T^p(\check D)\not=\emptyset$, there exists a transverse path from $T^{p_0}(\check\phi)$ to $T^{p}(\check\phi)$, which implies that there is a transverse path from  $\check\phi$ to $T^{p-p_0}(\check\phi)$ and for every $m\geq 1$ a transverse path from $\check\phi$ to $T^{m(p-p_0)}(\check\phi)$. Consequently, there exists a transverse path from $\check\phi$ to $T^{np_0+m(p-p_0)}(\check\phi)$. If $p\leq p_0$, one can choose suitably \textcolor{black}{$n\geq 0$ and $m\geq 1$}, such that $np_0+m(p-p_0)=0$, but this is impossible, as otherwise there would be a transverse path beginning and ending at $\check \phi$, which contradicts the fact that $\check{\mathcal{F}}$ is nonsingular. In particular, \textcolor{black}{$\check f^q(\check D)\cap T^{p'}(\check D)=\emptyset$ if $p'\leq 1$}.
 
 \medskip

Let $I'$ be a maximal isotopy of $f'=f^q$ that is lifted to an identity isotopy $\check I'$  of $\check f'=\check f^q\circ T^{-1}$. By definition of $q$, the set $D\cap f'{}^{-1}(D)$ is not empty. Fix a point $z$ in this set. \textcolor{black}{By Proposition 59 of \cite{LeCalvezTal}}, \textcolor{black}{the fact that $\check D$ is a free disk of $\check f'$ implies that there exists} a foliation ${\mathcal F}'$ transverse to $I'$ such that $z_*$, $z$ and $f'(z)$ belong to the same leaf $\phi'_*$ of $\mathcal F'\vert_{D}$.  We want to prove that ${\mathcal F}'$ has singularities. Write $\check z$ for the lift of $z$ that belongs to $\check D$ and $\check \phi'_*$ for the lift of $\phi'_*$ that contains $\check z_*$ and $\check z$. Note that $\check f'(\check z')\in T^{p-1}(\check \phi_*)$, because $\check f'(\check z')\in T^{p-1}(\check D)$ by hypothesis. So there exists a transverse path from $\check \phi'_*$ to $T^{p-1}(\check \phi'_*)$ and consequently a positive loop $\Gamma_+$ transverse to ${\mathcal F}'$  \textcolor{black}{ that contains $z_*$. The fact that $z_*\in\mathrm{ne}^+(f)$ and that $\mathrm{rot}_{\check f} (z_*)=0$ implies that $z_*\in\mathrm{ne}^+(I')$ and that $\mathrm{rot}_{\check f'} (z_*)=-1$.  So every leaf of ${\mathcal F}'$ that contains a point of the $\omega$ limit set of $z_*$ for $f'$ is met infinitely often by ${I'_{{\mathcal F}'}}^{\N}(z_*)$. Consequently, there exists a negative transverse loop $\Gamma_-$ that contains a point $f'{}^n(z_*)$, where $n\geq 1$. There exists $n_-\geq 1$ and $n_+\geq 1$ such that the multi-loop $n_{+}\Gamma_{+}+n_{-}\Gamma_{-}$ is homologous to zero.}
  As explained in the proof of Proposition \ref{prop:PB}, we can suppose that its dual function $\delta$ that vanishes in a neighborhood of the ends of $\A$ has a negative minimal value $m_-$. We have seen in the proof of Proposition \ref{prop:PB} that a connected component $U$ of $\A\setminus (\Gamma_-\cup\Gamma_+)$ where $\delta$ is equal to $m_-$ contains a singular point of ${\mathcal F}'$ if it is a topological disk. If it is not disk, it is an open annulus separating $\Gamma_-$ and $\Gamma_+$ \textcolor{black}{and this annulus contains a closed leaf $\phi'$. This leaf is disjoint from its image by $f'=q$ and separates $z_*$ and $f'{}^{n}(z_*)$.  The point $z_*$ being a non wandering point of $f$ is a Birkhoff recurrent point of $f'$ and its Birkhoff class is invariant by $f'$. \textcolor{black}{But} the existence of $\phi'$ implies that there is no Birkhoff connection from $f'{}^{n}(z_*)$ to $z_*$. We have got our contradiction.}

To get the second assertion of the proposition, note that the point $z$ can be chosen equal to $z_*$ in case $z_*$ is recurrent. \textcolor{black}{This implies that $I'{}_{\mathcal F'}^{\Z}(z_*)$ draws a positive loop $\Gamma_+$.  It can be easily proven that $z_*$ is a recurrent point of $f'$. Consequently, if $f$ has no horsehoe, then ${I'_{{\mathcal F}'}}^{\N}(z_*)$ exactly draws infinitely a simple transverse loop $\Gamma$. Moreover this loop must be negative because   $\mathrm{rot}_{\check f'} (z_*)=-1$. We have a contradiction.} Note that we can also give a proof by using Corollary \ref{cor:free-disks} applied to $\check D$ and $\check f'$. \end{proof}

\begin{remark*} Note that the first item is not true if we suppose that $z\in\mathrm{ne}^+(f)\setminus\Omega(f)$. Consider on $\S^2=\A\cup\{N,S\}$ a flow satisfying:
 \begin{itemize}
 \item the end $N$ and $S$ are the only fixed points,
 \item there exists an orbit $O$ homoclinic to  $S$,
 \item the disk bounded by $O\cup \{S\}$ that does not contain $N$ is foliated by orbits homoclinic to $S$,
 \item the $\alpha$-limit set of a point disjoint from this disk is equal to $N$ and its $\omega$- limit set to $O\cup \{S\}$ if the orbit is not reduced to $N$,\end{itemize}
Now, choose $f$ to be the time one map of this flow, reduced to $\A$, and $\check f$ the time one map of the lifted flow.
 
 \end{remark*}

\bigskip

\section{ Homeomorphisms of the sphere with no topological horseshoe}

The goal of this section is to prove Theorem \ref{thmain:global-structure} of the introduction, the fundamental result that permits to describe the structure of homeomorphisms of the sphere with no topological horseshoe. We will introduce the notion of positive an squeezed annuli in Sub-section \ref{subsection: annuli},  which will be fundamental in our study, then will state and prove a ``local version'' of the theorem in Sub-sections \ref{subsection: prooflocal} and {\ref{subsection:auxilliary}},  which means a version related to a maximal isotopy. The proof of  Theorem \ref{thmain:global-structure} will be given in Sub-section \ref{subsection: proofglobal}

\subsection{Positive and squeezed annuli} \label{subsection: annuli}

\begin{definition*}Let $f$ be a homeomorphism of $\A$ isotopic to the identity and $\check f$ a lift of $f$ to the universal covering space.   A  {\it $\check f$-positive disk}  is a topological open disk $D\subset\A$ such that: 
\begin{itemize}
\item every lift $\check D$ of $D$ is $\check f$-free;
\item  there exist $p>0$ and $q>0$ satisfying $\textcolor{black}{\check f^q(\check D)}\cap T^p(\check D)\not=\emptyset$. 
\end{itemize}
Replacing the condition $p>0$ by $p<0$ in the definition
one defines similarly  a {\it $\check f$-negative disk}. 
\end{definition*}

Of course a disk $D\subset\A$ that contains a $\check f$-positive disk is a $\check f$-positive disk, if it is lifted to $\check f$-free disks.  Franks' lemma (see \cite{Franks}) asserts that if $\check f$ is fixed point free, then a disk $D\subset\A$ that is lifted to $\check f$-free disks cannot be both $\check f$-positive and $\check f$-negative. Moreover, every lift of $D$ is wandering. This implies, again assuming $\check f$ is fixed point free, that every point $z\in\Omega(f)$ admits a fundamental system of neighborhoods by $\check f$-positive disks or a fundamental system of neighborhoods by $\check f$-negative disks.  In the first situation, every disk containing $z$ and lifted to $\check f$-free disks is positive, and we will say that $z$ is {\it $\check f$-positive}. We will use the following easy result:
\begin{proposition}\label{prop: positive} Let $f$ be a homeomorphism of $\A$ isotopic to the identity and $\check f$ a lift of $f$ to the universal covering space. We suppose that $f$ has no horseshoe and that $\check f$ is fixed point free. Let $I$ be a maximal isotopy of $f$ that is lifted to an identity isotopy of $\check f$ and $\mathcal F$ a transverse foliation of $I$. Then if $z\in\Omega(f)$ is $\check f$-positive, there exists a positive transverse loop that meets the leaf $\phi_z$.
 \end{proposition}
 \begin{proof} Write $\check{\mathcal F}$ for the lifted foliation to $\R^2$. Suppose that $z\in\Omega(f)$ is $\check f$-positive. If $D$ is a topological open disk containing $z$ sufficiently small, if $\check z\in\R^2$ is a lift of $z$ and $\check D$ the lift of $D$ containing $z$, then $I^{2}_{\check {\mathcal F}} ( \check f^{-1}(\check z'))$ meets $\phi_{\check z}$ for every $\check z'\in\check D$. By hypothesis, one can find $\check z'\in \check D$, $n\geq 2$ and $p>0$ such $\check f^n(\check z')\in T^p(\check D)$. Consequently $I^{n+2}_{\check {\mathcal F}} ( \check f^{-1}(\check z'))$ contains as a sub-path a transverse path joining $\phi_{\check z}$ to $T^p(\phi_{\check z})$.\end{proof}

\begin{definition*}Let $f$ be a homeomorphism of $\A$ isotopic to the identity such that $\Omega(f)\not=\emptyset$. We will say that a lift $\check f$ of $f$ to the universal covering space,  is a {\it positive lift} of $f$ if:
\begin{itemize}
\item$\check f$ is fixed point free,
\item  every point $z\in\Omega(f) $ is positive.
\end{itemize}
\end{definition*}
 We can define similar objects replacing ``positive'' by ``negative''. Note that if $\check f$ is a  positive lift of $f$, then the rotation number of a non-wandering point is non negative when it is defined.

\begin{definition*}Let $f$ be a homeomorphism of $\A$ isotopic to the identity. We will say that $f$ is {\it squeezed} if there exists a lift $\check f$ such that
\begin{itemize}
\item$\check f$ is positive,
\item $\check f\circ T^{-1}$ is negative.
\end{itemize}
\end{definition*}
 
Note that a squeezed homeomorphism is fixed point free.

   \begin{definition*}Let $f$ be an orientation preserving homeomorphism of $\S^2$. A {\it squeezed annulus } of $f$ is an invariant open annulus $A\subset \S^2$ such that $f\vert _A$ is conjugate to a squeezed homeomorphism of $\A$.  \end{definition*} 
 
Note that if $A$ is a fixed point free invariant annulus, then the two connected components of $\S^2\setminus A$ are fixed. Otherwise they are permuted by $f$, which implies that $f$ is fixed point free, in contradiction with the fact that it preserves the orientation. Consequently, $f\vert_A$ is isotopic to the identity. Note that if $A\subset A' $ are squeezed annuli, then 
$A$ is essential in $A'$ (which means that it contains a loop non homotopic to zero in $A'$). Otherwise, if we add to $A$ the compact connected component $K$ of $A'\setminus A$, we obtain an invariant disk included in $A'$ that contains a non-wandering point (because $K$ is compact and invariant) and such a disk must contain a fixed point. So, if $(A_j)_{j\in J}$ is a totally ordered family of squeezed annuli, then $\bigcup_{j\in J} A_j$ is a squeezed annulus. Consequently, by Zorn's lemma, every squeezed annulus is included in a maximal squeezed annulus. It follows from a classical fact that if $f$ has no wandering point, the squeezed annuli are the fixed point free invariant annuli. Let us explain why. As explained above, if $A$ is a fixed point free invariant annulus, then $f\vert_A$ is isotopic to the identity. So to prove that $A$ is squeezed, it is sufficient to prove that a homeomorphism $f$ of $\A$ isotopic to the identity that has no wandering point and is fixed point free, is squeezed. Fix a lift $\check f$ and a free disk $D$ and consider the first integer $q>1$ such that $f^q(D)\cap D\not=\emptyset$. If $\check D$ is a lift of $D$, there exists $p\in\Z$ such that $\check f^q(\check D)\cap T^p(\check D)\not=\emptyset$. Replacing $\check f$ by another lift if necessary, one can suppose that $0\leq p<q$. As explained above, $p$ is different from $0$ and $D$ is a $\check f$-positive disk and a $\check f\circ T^{-1}$-negative disk. The set of positive points of $\widetilde f$ is open, being the union of $\widetilde f$-positive disks. Similarly, the set of negative points of $\widetilde f$ is open. So, by connectedness of $\A$, one deduces that every point of $\A$ is a positive point of $\widetilde f$ and for the same reason is a negative point of $\widetilde f\circ T^{-1}$.

Let us define now positive invariant annuli of an orientation preserving homeomorphism $f$ of $\S^2$. They will be defined relative to a maximal isotopy $I$ of $f$. It has been explained in the introduction (just before the statement of Theorem \ref{thmain:rotation-number}) that if $A\subset \S^2$ is a topological open annulus invariant by $f$, such that $f$ fixes the two connected components of $\S^2\setminus A$, one can define rotation numbers of $f\vert_A$ as soon a generator $\kappa$ of $H_1(A,\Z)$ is chosen. Remind that $\kappa_*$ is the generator of $H_1(\A,\Z)$ induced by the loop $\Gamma_* :t\mapsto (t,0)$.
If $U\subset V$ are two open sets of $\S^2$ we will  denote $\iota_*: H_1(U,\Z)\to H_1(V,\Z)$ the morphism induced by the inclusion map $\iota: U\to V$. Let $I$ be a maximal isotopy of $f$. Remind that $\widetilde \pi :\widetilde{\mathrm{dom}}(I)\to\mathrm{dom}(I)$ is the universal covering projection and that $I$ can be lifted to $\widetilde{\mathrm{dom}}(I)$ into an identity isotopy of a certain lift $\widetilde f$ of $f\vert_{\mathrm{dom}(I)}$. Let $A\subset \mathrm{dom}(I)$ be a topological open annulus invariant by $f$ such that each connected component of $\S^2\setminus A$ contains at least \textcolor{black}{one} fixed point of $I$. This implies that the two connected components of $\S^2\setminus A$ are fixed.  Let $\widetilde A$ be a connected component of $\widetilde\pi^{-1}(A)$. Consider a generator $T$ of the stabilizer of $\widetilde A$ in the group of covering automorphisms. There exists a covering automorphism $S$ such that $\widetilde f(\widetilde A)=S(\widetilde A)$ and one must have \textcolor{black}{$S\circ T\circ S^{-1}=T^k$ for some integer $k$  because $\widetilde f$ commutes with the covering automorphisms. }
Of course, one can suppose that $\mathrm{dom}(I)$ is connected, which implies that the group of covering automorphisms is a free group. Consequently $S$ is a power of $T$ and $\widetilde A$ is invariant by $\widetilde f$. Moreover $\widetilde A$ is the universal covering space of $A$ and $\widetilde f\vert_{\widetilde A}$ is a lift of $f\vert_A$.

\begin{definitions*}
Let $f$ be an orientation preserving homeomorphism of $\S^2$ and $I$ a maximal isotopy of $f$. A {\it positive annulus } of $I$ is an open annulus $A\subset \S^2$ invariant by $f$ verifying:
\begin{itemize}
\item each connected component of $\S^2\setminus A$ contains at least a fixed point of $I$;
\item there exists a homeomorphism $h:A\to \A$ such that $\widetilde f\vert_{\widetilde A}$ is a positive lift of  $\widetilde h\circ \textcolor{black}{\widetilde f\vert_{\widetilde A}}\circ \widetilde h^{-1}$, where $\widetilde h:\widetilde A\to \R^2$ is a lift of $h$ between the two universal covering spaces (we keep the notations introduced above to define $\widetilde f$ and $\widetilde A$). 
\end{itemize}
In that case the {\it positive class} of $ H_1(A,\Z)$ is the generator $\kappa$ such that $h_*(\kappa)=\kappa_*$, its inverse being the {\it negative class}. 
\end{definitions*} 

\begin{remarks*}
If $A\subset A' $ are two positive annuli, then $A$ is essential in $A'$. So, if $(A_j)_{j\in J}$ is a totally ordered family of positive annuli, then $\bigcup_{i\in I} A_i$ is a positive annulus, which implies that  every positive annulus of $I$ is included in a maximal positive annulus. Note also that if $A\subset A' $ are two positive annuli, the positive class of $H_1(A,\Z)$ is sent onto the positive class of $H_1(A',\Z)$ by the morphism $\iota_*: H_1(A,\Z)\to H_1(A',\Z)$.
\end{remarks*}

\subsection{Local version of Theorem \ref{thmain:global-structure}} \label{subsection: prooflocal}If $f$ is an orientation preserving homeomorphism of $\S^2$ we set
$$\Omega'(f)=\{z\in\Omega(f)\,\vert\, \alpha(z)\cup\omega(z)\not\subset\mathrm{fix}(f)\}.$$Moreover, if $I$ is a maximal isotopy of $f$, we set
$$\Omega'(I)=\{z\in\Omega(f)\,\vert\, \alpha(z)\cup\omega(z)\not\subset\mathrm{fix}(I)\}.$$
In this long sub-section, we will prove the following result:
\begin{theorem}
\label{th:local-structure}
Let $f:\S^2\to\S^2$ be an orientation preserving homeomorphism that has no topological horseshoe, and $I$ a maximal isotopy of $f$. Then $\Omega'(I)$ is covered by  positive annuli.\end{theorem}
We suppose in \ref{subsection: prooflocal} that $f$ is an orientation preserving homeomorphism of $\S^2$ with no topological horseshoe and that $I$ is a maximal isotopy of $f$. We want to cover $\Omega'(I)$ by positive annuli. To construct these annuli, we need to use a foliation $\mathcal F$ transverse to $I$. Recall that, \textcolor{black}{by Proposition \ref{prop:non-wandering} and the remark following it,} the whole transverse trajectory of a point $z\in\Omega'(I)$ draws a unique transverse simple loop \textcolor{black}{$\Gamma$} and exactly draws it infinitely. In particular, it is contained in the open annulus $U_{\Gamma}$, union of leaves that meet $\Gamma$. We want to construct a positive annulus that contains $z$. The simplest case where it is possible to do so is the case where $U_{\Gamma}$ itself is invariant. This is a very special case because it happens if and only if $U_{\Gamma}$ is a connected component of $\mathrm{dom}(I)$. Another case where it is not difficult to construct such an annulus is the case where $U_{\Gamma}$ is bordered by two closed leaves $\phi$ and $\phi'$ . The reader should convince {himself} that the set of points whose whole transverse trajectory meets $U_{\Gamma}$, which is nothing but
 the set 
 $$U_{\Gamma}\cup \left(\bigsqcup_{k\in\Z} R(f^{k+1}(\phi))\setminus R(f^k(\phi))\right)\cup\left(\bigsqcup_{k\in\Z} R(f^{k+1}(\phi'))\setminus R(f^k(\phi'))\right)$$ is a positive annulus containing $z$. The general case is much more complicated. Roughly speaking{, one} will consider{, for a given connected component of the complement of $U_{\Gamma}$, either} the set of points that ``enter in the annulus $U_{\Gamma}$'' {from that connected component or the set that ``leave it'' to that component, but not both, and similarly the points that enter or leave from the other connected component of the complement of $U_{\Gamma}$}. We will need many of the results previously stated in this article, like the result about general regions of instability (Proposition \ref{prop:PB})for instance, to succeed  in constructing our annulus. 
 
\textcolor{black}{Let us now define, for an arbitrary transverse simple closed loop $\Gamma$}, the following sets
\begin{itemize}
\item $r_{\Gamma}$ is the connected component of $\S^2\setminus
U_{\Gamma}$ that lies on the right of $\Gamma$;
\item $l_{\Gamma}$ is the connected component $\S^2\setminus
U_{\Gamma}$ that lies on the left of $\Gamma$;
\item $\Omega'(I)_{\Gamma}$ is the set of points of $z\in\Omega'(I)$ such that $I_{\mathcal F}^{\Z}(z)$ draws $\Gamma$;
\item $V_{\Gamma}$ is the set of points $z$ such that $I_{\mathcal F}^{\Z}(z)$ draws $\Gamma$;
\item $W_{\Gamma}$ is the set of points $z$ such that $I_{\mathcal F}^{\Z}(z)$ meets $U_{\Gamma}$;
\item  $W_{\Gamma}^{\rightarrow r}$ is the set of points $z$ such that $I_{\mathcal F}^{\Z}(z)$ meets a leaf $\phi\subset r_{\Gamma}\cap \overline U_{\Gamma}$ where $U_{\Gamma}$ is locally on the right of $\phi$;
\item $W_{\Gamma}^{r\rightarrow }$ is the set of points $z$ such that $I_{\mathcal F}^{\Z}(z)$ meets a leaf $\phi\subset r_{\Gamma}\cap \overline U_{\Gamma}$ where $U_{\Gamma}$ is locally on the left of $\phi$;
\item  $W_{\Gamma}^{\rightarrow l}$ is the set of points $z$ such that $I_{\mathcal F}^{\Z}(z)$ meets a leaf $\phi\subset l_{\Gamma}\cap \overline U_{\Gamma}$ where $U_{\Gamma}$ 
is locally on the right of $\phi$;
\item $W_{\Gamma}^{l\rightarrow }$ is the set of points $z$ such that $I_{\mathcal F}^{\Z}(z)$ meets a leaf $\phi\subset l_{\Gamma}\cap \overline U_{\Gamma}$ where $U_{\Gamma}$ is locally on the left of $\phi$.
\end{itemize}
Equivalently, $W_{\Gamma}^{\rightarrow r}$ is the set of points $z$ such that there exist $s<t$ verifying 
$$I_{\mathcal F}^{\Z}(z)(s)\in U_{\Gamma}, \enskip I_{\mathcal F}^{\Z}(z)(t)\in r_{\Gamma}$$
 and $W_{\Gamma}^{r\rightarrow }$ is the set of points $z$ such that there exist $s<t$ verifying $$I_{\mathcal F}^{\Z}(z)(s)\in r_{\Gamma},\enskip I_{\mathcal F}^{\Z}(z)(t)\in  U_{\Gamma}.$$ Note that the six last sets are  open and invariant, and that:
\begin{itemize}
\item the whole transverse trajectory of $z\in W_{\Gamma}\setminus (W_{\Gamma}^{r\rightarrow}\cup W_{\Gamma}^{\rightarrow l})$ meets $U_{\Gamma}$ in a unique real interval $(a,b)$, moreover $I_{\mathcal F}^{\Z}(z)(a)\in l_{\Gamma}$ if $a>-\infty$ and $I_{\mathcal F}^{\Z}(z)(b)\in r_{\Gamma}$ if $b<+\infty$; 

\item the whole transverse trajectory of $z\in W_{\Gamma}\setminus (W_{\Gamma}^{l\rightarrow}\cup W_{\Gamma}^{\rightarrow r})$ meets $U_{\Gamma}$ in a unique real interval $(a,b)$, moreover $I_{\mathcal F}^{\Z}(z)(a)\in r_{\Gamma}$ if $a>-\infty$ and $I_{\mathcal F}^{\Z}(z)(b)\in l_{\Gamma}$ if $b<+\infty$; 

\item the whole transverse trajectory of $z\in W_{\Gamma}\setminus (W_{\Gamma}^{\rightarrow r}\cup W_{\Gamma}^{\rightarrow l})$ meets $U_{\Gamma}$ in a unique real interval $(a,+\infty)$ where $a\geq -\infty$; 

\item the whole transverse trajectory of $z\in W_{\Gamma}\setminus (W_{\Gamma}^{r\rightarrow}\cup W_{\Gamma}^{l\rightarrow })$ meets $U_{\Gamma}$ in a unique real interval $(-\infty,b)$ where $b\leq +\infty$.

\end{itemize}
 
 We will define the four following invariant open sets:
 
 $$W_{\Gamma}^{l\rightarrow r}=W_{\Gamma}\setminus (\overline{W_{\Gamma}^{r\rightarrow}}\cup \overline{W_{\Gamma}^{\rightarrow l}})$$
 $$W_{\Gamma}^{r\rightarrow l}=W_{\Gamma}\setminus (\overline{W_{\Gamma}^{l\rightarrow}}\cup \overline{W_{\Gamma}^{\rightarrow r}}),$$
 $$W_{\Gamma}^{r,l\rightarrow }=W_{\Gamma}\setminus (\overline{W_{\Gamma}^{\rightarrow r}}\cup \overline{W_{\Gamma}^{\rightarrow l}}),$$
 $$W_{\Gamma}^{\rightarrow r,l}=W_{\Gamma}\setminus (\overline{W_{\Gamma}^{r\rightarrow}}\cup \overline{W_{\Gamma}^{l\rightarrow }}).$$

\bigskip
We will prove the following result, which immediately implies Theorem \ref{th:local-structure}:

\begin{proposition}
\label{prop:local-structure}
There exists a positive annulus $A_{\Gamma}$ that contains $\Omega'(I)_{\Gamma}$ and whose positive generator is the homology class of a simple loop freely homotopic to $\Gamma$ in $\mathrm{dom}(I)$. More precisely, at least one of the set $W_{\Gamma}^{l\rightarrow r}$, $W_{\Gamma}^{r\rightarrow l}$, $W_{\Gamma}^{r,l\rightarrow }$ or $W_{\Gamma}^{\rightarrow r,l }$ has a connected component that contains $\Omega'(I)_{\Gamma}$, and if we add to this component the connected components of its complement that have no singular points, we obtain an open annulus $A_{\Gamma}$ verifying the properties above.\end{proposition}

Before proving Proposition \ref{prop:local-structure}, we will study more carefully the sets related to $\Gamma$ that we have introduced. Replacing $\mathrm{dom}(I)$ by a connected component, we can always suppose that $\mathrm{dom}(I)$ is connected. Denote by $\widetilde\pi: \widetilde{\mathrm{dom}}(I)\to\mathrm{dom}(I)$ the universal covering projection, by $\widetilde I$ the lifted identity isotopy, by $\widetilde f$ the induced lift of $f$, by $\widetilde{\mathcal{F}}$ the lifted foliation.

Let $\gamma$ be the natural lift of $\Gamma$. For every lift $\widetilde\gamma$ of $\gamma$ in $\widetilde{\mathrm{dom}}(I)$,  define the following objects:
 
 \begin{itemize}
 \item $T_{\widetilde \gamma}$ is the covering automorphism such that $\widetilde\gamma(t+1)=T_{\widetilde \gamma}(\widetilde\gamma(t))$ for every $t\in\R$;
 
 \item $U_{\widetilde\gamma}$ is the union of leaves that meet $\widetilde\gamma$;
 \item $r_{\widetilde\gamma}$ is the union of leaves that are not in $U_{\widetilde\gamma}$ and are on the right of $\widetilde \gamma$;
  \item $l_{\widetilde\gamma}$ is the union of leaves that are not in $U_{\widetilde\gamma}$ and are on the left of $\widetilde \gamma$;
 \item $\Omega'(I)_{\widetilde\gamma}$ is the set of points $\widetilde z$ that lift a point of $\Omega'(I)$ and such that $\widetilde I_{\widetilde{\mathcal F}}^{\Z}(\widetilde z)$ is included in $U_{\widetilde\gamma}$;
 \item $V_{\widetilde\gamma}$ is the set of points $\widetilde z$ such that $\widetilde I_{\widetilde{\mathcal F}}^{\Z}(\widetilde z)$ contains a sub-path of $\widetilde\gamma$ that projects onto $\Gamma$;
\item $W_{\widetilde\gamma}$ is the set of points $\widetilde z$ such that $\widetilde I_{\widetilde{\mathcal F}}^{\Z}(\widetilde z)$ meets $U_{\widetilde\gamma}$;
\item $W_{\widetilde\gamma}^{\rightarrow r}$ is the set of points $\widetilde z\in W_{\widetilde\gamma}$ such that $\widetilde I_{\widetilde{\mathcal F}}^{\Z}(\widetilde z)$ meets a leaf $\widetilde\phi\subset r_{\widetilde\gamma}\cap \overline U_{\widetilde\gamma}$ where $U_{\widetilde\gamma}$ is on the right of $\widetilde \phi$;
\item   $W_{\widetilde\gamma}^{r\rightarrow }$, $W_{\widetilde\gamma}^{\rightarrow l}$ and $W_{\widetilde\gamma}^{l\rightarrow}$ are defined in a similar way;
\item $W_{\widetilde\gamma}^{l\rightarrow r}$ is the set of points $\widetilde z\in W_{\widetilde\gamma}$ that lift a point of $W_{\Gamma}^{l\rightarrow r}$;
\item $W_{\widetilde\gamma}^{r\rightarrow l}$, $W_{\widetilde\gamma}^{r,l\rightarrow }$ and  $W_{\widetilde\gamma}^{\rightarrow r,l }$ are defined in a similar way.

\end{itemize}  
Note that, since $U_{\Gamma}$ is an annulus, $U_{\widetilde \gamma}$ is a connected component of $\widetilde \pi^{-1}(U_{\Gamma})$ and therefore if $T$ is a covering transformation, then $T(U_{\widetilde \gamma})\cap U_{\widetilde \gamma}$ is not empty if and only if $T$ is a power of $T_{\widetilde\gamma}$. Furthermore,  as told above, if $z$ belongs to $W_{\Gamma}^{l\rightarrow r}$, it meets $U_{\Gamma}$ in a unique real interval. This implies that, given a lift $\widetilde z$ of a point $z$ in $ W_{\Gamma}^{l\rightarrow r}$ that belongs to $W_{\widetilde\gamma}$, the whole transverse trajectory $\widetilde I_{\widetilde{\mathcal F}}^{\Z}(\widetilde z)$ can only meet a single connected component of $\widetilde\pi^{-1}(U_{\Gamma})$. Therefore, the only possible lifts of $z$ that are in $W_{\widetilde\gamma}$ are the $T_{\widetilde\gamma}^k(\widetilde z)$, $k\in\Z$.  Note that for every covering transformation $T$,  we have  $T(W_{\widetilde\gamma}^{l\rightarrow r})=W_{T(\widetilde\gamma)}^{l\rightarrow r}$ and that $W_{\widetilde\gamma}^{l\rightarrow r}\cap T(W_{\widetilde\gamma}^{l\rightarrow r})=\emptyset$ if $T$ is not a power of $T_{\widetilde \gamma}$.  We have similar properties for the three other sets
  $W_{\widetilde\gamma}^{r\rightarrow l}$, $W_{\widetilde\gamma}^{r,l\rightarrow }$ and $W_{\widetilde\gamma}^{\rightarrow r,l }$.

\begin{lemma} Suppose that $V_{\Gamma}\cap U_{\Gamma}\not=\emptyset$ and fix $z\in V_{\Gamma}\cap U_{\Gamma}$. Then the image of the morphism $i_*:\pi_1(z,V_{\Gamma})\to \pi_1(z,\mathrm{dom}(I))$ induced by the inclusion $i:V_{\Gamma}\to \mathrm{dom}(I)$ is included in the image of the morphism $j_*:\pi_1(z,U_{\Gamma})\to \pi_1(z,\mathrm{dom}(I))$ induced by the inclusion $j:U_{\Gamma}\to \mathrm{dom}(I)$.
\end{lemma}

\begin{proof} Note that $V_{\widetilde \gamma}$ is invariant by $T_{\widetilde\gamma}$ and projects onto $V_{\Gamma}$ and conversely that $\widetilde \pi^{-1}(V_{\widehat\Gamma})$ is the union of the $V_{\widetilde \gamma}$, indexed by the set of lifts. For every point $ z\in V_{\Gamma}$, the path $I_{{\mathcal F}}^{\Z}(z)$ has no ${\mathcal F}$-transverse self-intersection and draws $\Gamma$. By Proposition \ref{prop:drawing-crossing}, there is a unique drawing component and so, for every lift $\widetilde z$ of $z\in V_{\Gamma}$, there exists a unique lift $\widetilde \gamma$ of $\gamma$ (up to a composition by a power of $T_{\widetilde\gamma}$) such that $\widetilde z$
 belongs to $V_{\widetilde \gamma}$. In particular $T(V_{\widetilde \gamma})\cap V_{\widetilde \gamma}=\emptyset$ if $T$ is not a power of $T_{\widetilde\gamma}$. This implies that $V_{\widetilde \gamma}$ is open and closed in $\widetilde \pi^{-1}(V_{\Gamma})$. Consequently, the stabilizer in the group of covering automorphisms of a connected component of $\widetilde \pi^{-1}(V_{\Gamma})$ is a sub-group of the stabilizer of a lift $\widetilde\gamma$. \end{proof}

\bigskip 
To prove Proposition \ref{prop:local-structure} we need to work on the annular covering $\widehat{\mathrm{dom}}(I)=\widetilde{\mathrm{dom}}(I)/T_{\widetilde \gamma}$. Denote by $\pi: \widetilde{\mathrm{dom}}(I)\to \widehat{\mathrm{dom}}(I)$ and $\widehat\pi: \widehat{\mathrm{dom}}(I)\to\mathrm{dom}(I)$ the covering projections, by $\widehat I$ the induced identity isotopy, by $\widehat f$ the induced lift of $f$, by $\widehat{\mathcal{F}}$ the induced foliation. The line $\widetilde\gamma$  projects onto the natural lift of a transverse simple loop $\widehat\Gamma$.
Write $$U_{\widehat\Gamma}, \enskip r_{\widehat\Gamma}, \enskip l_{\widehat\Gamma}, \enskip \Omega'(I)_{\widehat\Gamma}, \enskip V_{\widehat\Gamma},\enskip W_{\widehat\Gamma},\enskip W_{\widehat\Gamma}^{\rightarrow r}, \enskip W_{\widehat\Gamma}^{r\rightarrow }, \enskip W_{\widehat\Gamma}^{\rightarrow l}, \enskip W_{\widehat\Gamma}^{l\rightarrow}, \enskip W_{\widehat\Gamma}^{l\rightarrow r},\enskip W_{\widehat\Gamma}^{l\rightarrow r}, \enskip W_{\widehat\Gamma}^{r,l\rightarrow }, \enskip W_{\widehat\Gamma}^{\rightarrow r,l }$$ for the respective projections of $$U_{\widetilde\gamma}, \enskip r_{\widetilde\gamma}, \enskip l_{\widetilde\gamma}, \enskip \Omega'(I)_{\widetilde\gamma},\enskip W_{\widetilde\gamma}, \enskip W_{\widetilde\gamma},\enskip W_{\widetilde\gamma}^{\rightarrow r},\enskip W_{\widetilde\gamma}^{r\rightarrow },\enskip W_{\widetilde\gamma}^{\rightarrow l},\enskip W_{\widetilde\gamma}^{l\rightarrow},\enskip W_{\widetilde\gamma}^{l\rightarrow r},\enskip W_{\widetilde\gamma}^{l\rightarrow r},\enskip W_{\widetilde\gamma}^{r,l\rightarrow },\enskip W_{\widetilde\gamma}^{\rightarrow r,l }.$$ 

We have the following:
\begin{itemize}
\item The set $U_{\widehat\Gamma}$ is the union of leaves
 that meet $\widehat\Gamma$ and is the unique annular connected component of $\widehat\pi^{-1}(U_{\Gamma})$. 
 \item The two connected components of $\widehat{\mathrm{dom}}(I)_{\textcolor{black}{\mathrm {sp}}}\setminus U_{\widehat \Gamma}$ are $r_{\widehat\Gamma}\sqcup\{R\}$ and $l_{\widehat\Gamma}\sqcup\{L\}$, where $\widehat{\mathrm{dom}}(I)_{\textcolor{black}{\mathrm {sp}}}$ is the sphere obtained by adding the end $R$ of $\widehat{\mathrm{dom}}(I)$ on the right of $\widehat\Gamma$ and the end $L$ on its left.

 \end{itemize}
 
 We will often refer to the the following classification, for the right side of $U_{\Gamma}$:
\begin{enumerate}
\item the set $\mathrm{Fr}(U_{\widehat \Gamma})\cap r_{\widehat\Gamma}$ is empty;
\item the set $\mathrm{Fr}(U_{\widehat \Gamma})\cap r_{\widehat\Gamma}$ is reduced to a closed leaf $\widehat\phi$ such that $U_{\widehat\Gamma}\subset R(\widehat\phi)$;
\item the set $\mathrm{Fr}(U_{\widehat \Gamma})\cap r_{\widehat\Gamma}$ is reduced to a closed leaf $\widehat\phi$ such that $U_{\widehat\Gamma}\subset L(\widehat\phi)$;
\item the set $\mathrm{Fr}(U_{\widehat \Gamma})\cap r_{\widehat\Gamma}$ is a non empty union of leaves \textcolor{black}{$(\widehat\phi_{i})_{i\in I}$, all of them} homoclinic to $R$. 
 \end{enumerate}
 In situation (1), one has
 $$W_{\widehat\Gamma}^{r\to}=W_{\widehat\Gamma}^{\to r}=\emptyset.$$
 In situation (2), one has
 $$W_{\widehat\Gamma}^{r\to}=\emptyset,\enskip W_{\widehat\Gamma}^{\to r} =\bigsqcup_{k\in\Z} R(\widehat f^{k+1}(\widehat \phi))\setminus R(\widehat f^k(\widehat\phi)).$$ In situation (3), one has 
 $$W_{\widehat\Gamma}^{\to r}=\emptyset,\enskip W_{\widehat\Gamma}^{r\to } =\bigsqcup_{k\in\Z} R(\widehat f^{k+1}(\widehat \phi))\setminus R(\widehat f^k(\widehat\phi)).$$ 
 In situation (4), one has $$\textcolor{black}{W_{\widehat\Gamma}^{\rightarrow r}\cup W_{\widehat\Gamma}^{r\rightarrow }=\bigcup_{i\in I}\bigsqcup_{k\in\Z} R(\widehat f^{k+1}(\widehat \phi_{i}))\setminus R(\widehat f^k(\widehat\phi_{i})).}$$
 
 Note that in all situations, every whole transverse trajectory meets $\mathrm{Fr}(U_{\widehat \Gamma})\cap r_{\widehat\Gamma}$ at most twice. Of course one has a similar classification for $\mathrm{Fr}(U_{\widehat \Gamma})\cap l_{\widehat\Gamma}$. Note also that every transverse trajectory that meets $U_{\widehat \Gamma}$ has a unique interval of intersection with $U_{\widehat \Gamma}$. 
 
Now define the sets $X^{\rightarrow R}_{\widehat\Gamma}$, $X^{R\rightarrow }_{\widehat\Gamma}$, $X^{\rightarrow L}_{\widehat\Gamma}$, $X^{L\rightarrow }_{\widehat\Gamma}$ as follows:
 $$\widehat z\in X^{\rightarrow R}_{\widehat\Gamma}\Leftrightarrow \lim_{n\to+\infty} \widehat f^{n}(\widehat z)=R,$$
  $$\widehat z\in X^{R\rightarrow }_{\widehat\Gamma}\Leftrightarrow \lim_{n\to+\infty} \widehat f^{-n}(\widehat z)=R,$$
   $$\widehat z\in X^{\rightarrow L}_{\widehat\Gamma}\Leftrightarrow \lim_{n\to+\infty} \widehat f^{n}(\widehat z)=L,$$
    $$\widehat z\in X^{L\rightarrow }_{\widehat\Gamma}\Leftrightarrow \lim_{n\to+\infty} \widehat f^{-n}(\widehat z)=L.$$
So we have
$$\widehat{\mathrm{dom}}(I)\setminus \left(X^{\rightarrow R}_{\widehat\Gamma}\cup X^{\rightarrow L}_{\widehat\Gamma}\right)=\mathrm{ne}^{+}(\widehat f), \enskip \widehat{\mathrm{dom}}(I)\setminus \left(X^{R\rightarrow }_{\widehat\Gamma}\cup X^{L\rightarrow }_{\widehat\Gamma}\right)=\mathrm{ne}^{-}(\widehat f).$$

\begin{lemma}\label{lem:omegalimits}
\textcolor{black}{If situation (4) holds then 
$$W_{\widehat\Gamma}^{\rightarrow r}\subset X^{\rightarrow R}_{\widehat\Gamma}, \enskip W_{\widehat\Gamma}^{r\rightarrow }\subset X^{R\rightarrow }_{\widehat\Gamma}, \textcolor{black}{\enskip r_{\widehat \Gamma}\setminus W_{\widehat\Gamma}\subset X^{\rightarrow R}_{\widehat\Gamma}\cap X^{R\rightarrow }_{\widehat\Gamma}, \enskip r_{\widehat \Gamma}\subset X^{\rightarrow R}_{\widehat\Gamma}\cup X^{R\rightarrow }_{\widehat\Gamma}}.$$}
\end{lemma}
\begin{proof}
Indeed, if $z\in W_{\widehat\Gamma}^{\rightarrow r}$, then the whole transverse trajectory of $z$ meets a leaf $\widehat\phi\subset \mathrm{Fr}(U_{\widehat \Gamma})\cap r_{\widehat\Gamma}$, uniquely defined, homoclinic to $R$ and such that $L(\widehat \phi)$ is simply connected. There exists $n$ such that $\widehat I^{\N}_{\widehat{\mathcal F}}(\widehat{f}^n(z))\subset L(\phi)$. In particular $\widehat I^{\N}_{\widehat{\mathcal F}}(\widehat{f}^n(z))$ never meets a leaf more than once because $L(\widehat\phi)$ is simply connected and $\widehat{\mathcal F}$ non singular. So $\omega(z)=\emptyset$. But this means that $\lim_{n\to+\infty} \widehat f^{n}(\widehat z)=R$. \textcolor{black}{That $W_{\widehat\Gamma}^{r\rightarrow }\subset X^{R\rightarrow }_{\widehat\Gamma}$ follows from a similar argument. \textcolor{black}{If $z$ is in $r_{\widehat \Gamma}\setminus W_{\widehat\Gamma}$, then its whole transverse trajectory is contained in a single connected component of $r_{\widehat \Gamma}$ and as such it never beets the boundary of $r_{\widehat \Gamma}$ . Since the interior of $r_{\widehat \Gamma}$ is simply connected, the third assertion follows again from similar arguments.} The final assertion comes from $r_{\widehat \Gamma}\subset \left(r_{\widehat \Gamma}\setminus W_{\widehat\Gamma}\right) \cup 
W_{\widehat\Gamma}^{\rightarrow r} \cup W_{\widehat\Gamma}^{r\rightarrow }$.}
\end{proof}

 In the rest of the section, when $\widehat\phi$ is a closed leaf of $\widehat{\mathcal F}$ we will write
$$A_{\widehat\phi}= \bigsqcup_{k\in\Z} R(\widehat f^{k+1}(\widehat \phi))\setminus R(\widehat f^k(\widehat\phi)),$$
noting that we get an invariant open annulus. Similarly, if  $\phi$ is a closed leaf of $\mathcal F$ we will write
$$A_{\phi}= \bigsqcup_{k\in\Z} R(f^{k+1}(\phi))\setminus R(f^k(\phi)).$$
Let us state three preliminary results:
\begin{lemma}The set $\Omega'(I)_{\widehat \Gamma}$ is included in \textcolor{black}{$\Omega'(\widehat f_{\mathrm{sp}})$}.
\end{lemma}

\begin{proof} First note that, as $\Omega'(I)_{\widetilde\gamma}\subset U_{\widetilde\gamma}$, we have that $\Omega'(I)_{\widehat\Gamma}\subset U_{\widehat\Gamma}$. Let $\widehat z$ be in $\Omega'(I)_{\widehat \Gamma}$. It projects by $\widehat \pi$  onto a point $z \in \Omega'(I)_{\Gamma}$. One can suppose that $z\in\Omega(I)$ and that $\omega(z)\not\in\mathrm{fix}(I)$ for instance. This implies  that there exists a point $z'\in\ U_{\Gamma}$ and a subsequence $(f^{n_k}(z))_{k\geq 0}$ of $(f^n(z))_{n\geq 0}$ that converges to $z'$. The trajectory $I^{n_k}(z)$ is homotopic to $I^{n_k}_{\mathcal F}(z)$ and this last path is contained in $U_{\Gamma}$. Lifting our paths to $\widehat{\mathrm{dom}}(I)$ we deduce that $(\widehat f^{n_k}(\widehat z))_{k\geq 0}$ converges to $\widehat z'$, the lift of $z'$ that belongs to $U_{\widehat \Gamma}$. So $\widehat z$ belongs to $\mathrm{ne}^{+}(\widehat f)$. Moreover, if $W$ is a neighborhood of $z$ sufficiently small, there exists $m\geq 1$ such that  $I_{\mathcal F}^{m}(z')$ is included in $U_{\Gamma}$ and draws $\Gamma$, for every $z'\in W$. We deduce that if $\textcolor{black}{z''}\in W\cap f^{-n}(W)$, both $I_{\mathcal F}^m(\textcolor{black}{z''})$ and $I_{\mathcal F}^m(f^n(\textcolor{black}{z''}))$ draw $\Gamma$, and since, \textcolor{black}{by Proposition \ref{prop:drawing-crossing},} $I_{\mathcal F}^{\Z}(\textcolor{black}{z''})$ has a single drawing component for $\Gamma$, then $I_{\mathcal F}^n(\textcolor{black}{z''})$ is included in $U_{\Gamma}$. Lifting our paths to to $\widehat{\mathrm{dom}}(I)$ and using the fact that $I^{n}(\textcolor{black}{z''})$ is homotopic to $I^{n}_{\mathcal F}(\textcolor{black}{z''})$, we deduce that $z\in\Omega(\widehat f)$. Consequently, $\Omega'(I)_{\widehat \Gamma}$ is included in \textcolor{black}{$\Omega'(\widehat f_{\mathrm{sp}})$}.\end{proof}

\begin{lemma}
\label{lemma:existence} If $\Omega'(I)_{\widehat \Gamma}\not=\emptyset$, there is a point $\widehat z\in \Omega'(\widehat f_{\mathrm{sp}})$ such that $\mathrm{rot}_{\widetilde f, [\widehat\Gamma]}(\widehat z)>0$, where $[\widehat\Gamma]\in H_1(\widehat{\mathrm{dom}}(I),\Z)$ is the homology class of $\widehat\Gamma$.
\end{lemma}

\begin{proof} Observe first that the lift $\widetilde f$ of $\widehat f$ is fixed point free and that $\widehat f$ has no topological horseshoe because it is the case for $f$. So, every point in $\mathrm{ne}(\widehat f)$ has a rotation number by Theorem \ref{th:rotation-number}. Choose a point $\widehat z\in\Omega'(I)_{\widehat \Gamma}$. Its rotation number is well defined because $\widehat z$ belongs to $\mathrm{ne}(\widehat f)$. The rotation number cannot be negative because $\widehat I_{\widehat{\mathcal F}}^{\Z}(\widehat z)$ is drawn on $\widehat\Gamma$. If it is positive, the conclusion of the lemma holds. If if it zero, one can apply Proposition \ref{prop:recurrentzero-rotation} to ensure the existence of a periodic point whose rotation number is non zero, noting that the proof of Proposition \ref{prop:recurrentzero-rotation} tells us that this rotation number is positive.\end{proof}

\begin{lemma} 
\label{lemma:no-Birkhoff-cycle}If $\Omega'(I)_{\widehat \Gamma}\not=\emptyset$ and $U_{\widehat\Gamma}\not= \widehat{\mathrm{dom}}(I)$, then either there is no Birkhoff connection from $R$ to $L$ or there is no Birkhoff connection from $L$ to $R$.
\end{lemma}

\begin{proof} By Lemma \ref{lemma:existence}, there is a point $\widehat z\in \Omega'(\widehat f)$ such that $\mathrm{rot}_{\widetilde f, [\widehat\Gamma]}(\widehat z)>0$. The assumption $U_{\widehat\Gamma}\not= \widehat{\mathrm{dom}}(I)$ means that the sets $r_{\widehat\Gamma}$ and  $l_{\widehat\Gamma}$ are not both empty. Let us suppose for instance that $r_{\widehat\Gamma}\not=\emptyset$. This means that we are not in situation (1). In case we are in situation (2) or (3), there exists an essential closed leaf $\widehat \phi$ and this leaf is disjoint from its image by $\widehat f$. In that case the conclusion is immediate. It remains to study the case where situation (4) occurs. In that case, there exists a leaf $\widehat \phi$ homoclinic to $R$. We deduce that $\widehat\phi$ is disjoint from  all its images by the iterates of $\widehat f$. 
By a result of B\'eguin-Crovisier-Le Roux (see\cite{LeRoux}, \textcolor{black}{Proposition 2.3.3}) one can blow up the end $R$ by adding a circle $\widehat \Sigma_R$ to $\widehat {\mathrm {dom}}(I)$ in a neighborhood of $R$ to get a semi-open annulus $\widehat {\mathrm {dom}}(I)_{\mathrm {ann}}$ such that $\widehat f$ extends to a homeomorphism $\widehat f_{\mathrm {ann}}$ having a fixed  point $\widehat z_*$ on $\widehat \Sigma_R$ and also a non-fixed point. \textcolor{black}{One can also assume that $\widehat z_*$ is the unique fixed point on $\widehat \Sigma_R$, or else one can just collapse the complement of a maximal interval without fixed points in $\widehat \Sigma_R$ to a single point, again obtaining an extension of $\widehat f$ to a semi-open annulus, now with a single fixed point in the boundary.} Write $\widetilde{\mathrm {dom}}(I)_{\mathrm {ann}}=\widetilde{\mathrm {dom}}(I)\sqcup\widetilde\Sigma_S$ for the universal covering space of $\widehat {\mathrm {dom}}(I)_{\mathrm {ann}}$, where $\widetilde\Sigma_R$ is the universal covering space of $\widehat\Sigma_R$  and keep the notation $T_{\widetilde \gamma}$ for the natural covering automorphism. There is a unique  lift $\widetilde f_{\mathrm {ann}}$ of $\widehat f_{\mathrm {ann}}$ that extends  $\widetilde f$ and that fixes the preimages of $\widehat z_*$. One can construct an open annulus  $\widehat {\mathrm {dom}}(I)_{\mathrm {double}}$ by pasting two copies of $\widehat {\mathrm {dom}}(I)_{\mathrm {ann}}$ on $\widehat \Sigma_{R}$ and naturally define a homeomorphism $\widehat f_{\mathrm {double}}$ on  $\widehat {\mathrm {dom}}(I)_{\mathrm {double}}$. Its universal covering space $\widetilde {\mathrm {dom}}(I)_{\mathrm {double}}$ can be constructed by pasting two copies of $\widetilde {\mathrm {dom}}(I)_{\mathrm {ann}}$ on $\widetilde \Sigma_{R}$ and one gets a natural lift $\widetilde f_{\mathrm {double}}$. We keep the notation $T_{\widetilde\gamma}$ for the natural covering automorphism. Suppose that there exists a Birkhoff connection from $R$ to $L$ and a Birkhoff connection from $L$ to $R$. In that case $\widehat f$ satisfies the \textcolor{black}{curve intersection property, that is, the image by $\widehat f$ of any essential simple loop $\Gamma'$ must also intersect $\Gamma'$}. It is straightforward to prove that $\widehat f_{\mathrm {double}}$ itself has the curve intersection property. Consequently, \textcolor{black}{by  a classical result (see, for instance, Theorem 9.1 of \cite{LeCalvez1}}, for every integers $p$ and $q>0$ relatively prime, such that $p/q\in(0,\mathrm{rot}_{\widetilde f} (\widehat z))$, there exists a periodic point of $\widehat f_{\mathrm {double}}$ of period $q$ and rotation number $p/q$ and consequently there exists a periodic point of $\widehat f$ of period $q$ and rotation number $p/q$. Applying 
Proposition \ref{prop:PB}, one deduces that $\widehat f$ has a horseshoe. We have got a contradiction. \end{proof}

 We will suppose from now on that $U_{\widehat\Gamma}\not= \widehat{\mathrm{dom}}(I)$, or equivalently that the boundary of $U_{\widehat\Gamma}$ is not empty and contains at least one leaf of $\widehat{\mathcal F}$. In the other case, $U_{\widehat\Gamma}$ coincides with $U_{\Gamma}$, it is a connected component of $ {\mathrm{dom}}(I)$. Note that it is a positive annulus.

\bigskip
Let us state now the most important lemma.

 \bigskip

\begin{lemma}
\label{lemma:homoclinic}
\textcolor{black}{If $\Omega'(I)_{\widehat \Gamma}\not=\emptyset$ then at} least one of the sets  $\overline{W_{\widehat\Gamma}^{r\rightarrow }}$, $\overline {W_{\widehat\Gamma}^{\rightarrow r}}$ is contained in $X^{R\rightarrow }_{\widehat\Gamma}\cap X^{\rightarrow R}_{\widehat\Gamma}$, and one of the sets  $\overline{W_{\widehat\Gamma}^{l\rightarrow }}$, $\overline {W_{\widehat\Gamma}^{\rightarrow l}}$ is contained in $X^{L\rightarrow }_{\widehat\Gamma}\cap X^{\rightarrow L}_{\widehat\Gamma}$.
\end{lemma} 

\begin{proof} 
 We can suppose that both sets $W_{\widehat\Gamma}^{r\rightarrow }$ and $W_{\widehat\Gamma}^{\rightarrow r}$ are non empty, otherwise the statement is obvious. Consequently we are in situation (4): the set $U_{\widehat\Gamma}$ is adherent to $R$ and each leaf on the right of $U_{\widehat\Gamma}$ is a leaf homoclinic to $R$. Moreover, one has 
$W_{\widehat\Gamma}^{r\rightarrow }\subset X^{R\rightarrow }_{\widehat\Gamma}$ and $W_{\widehat\Gamma}^{\rightarrow r}\subset X^{\rightarrow R}_{\widehat\Gamma}$. We will consider two cases, depending whether there exists a whole transverse trajectory  that meets four times a leaf of $U_{\widehat \Gamma}$ or not.

\bigskip
{\it First case: there is no whole transverse trajectory that meets four times a leaf of $U_{\widehat \Gamma}$.}

By Lemma \ref{lemma:no-Birkhoff-cycle}, we know that at least one of the sets $X^{R\rightarrow }_{\widehat\Gamma}$ or $X^{\rightarrow R}_{\widehat\Gamma}$ is not adherent to $L$. Let us suppose for example that this is the case for $X^{R\rightarrow }_{\widehat\Gamma}$. The set 
$W_{\widehat\Gamma}^{r\rightarrow }$ shares the same property because $W_{\widehat\Gamma}^{r\rightarrow }\subset X^{R\rightarrow }_{\widehat\Gamma}$. Of course its closure $\overline{W_{\widehat\Gamma}^{r\rightarrow }}$ is not adherent to $L$. Let us suppose that $\overline{W_{\widehat\Gamma}^{r\rightarrow }}\not\subset (X^{R\rightarrow }_{\widehat\Gamma}\cap X^{\rightarrow R}_{\widehat\Gamma})$ and choose $\widehat z\in \overline{W_{\widehat\Gamma}^{r\rightarrow }}\setminus (X^{R\rightarrow }_{\widehat\Gamma}\cap X^{\rightarrow R}_{\widehat\Gamma})$. The fact that the orbit of $\widehat z$ is non adherent to $L$ and that $\widehat z\not\in X^{R\rightarrow }_{\widehat\Gamma}\cap X^{\rightarrow R}_{\widehat \Gamma}$ implies that $\widehat z\in \mathrm{ne}(\widehat f)$. Moreover, $\omega(\widehat z)$ and $\alpha(\widehat z)$ are not included in $\widehat U_{\Gamma}$ if not empty, otherwise\textcolor{black}{, by Proposition \ref{prop:non-wandering} and the following remark,} $\widehat I_{\widehat{\mathcal F}}(\widehat z)$ would draw $\widehat\Gamma$ infinitely, in contradiction with the hypothesis. The limit sets are included in $l_{\widehat\Gamma}$, more precisely, they are contained in $l_{\widehat\Gamma}\setminus \mathrm{Fr}(U_{\widehat\Gamma})$ and so the whole transverse trajectory of $\widehat z$ meets $l_{\widehat\Gamma}$. Of course $\widehat z\not\in W_{\widehat\Gamma}^{l \rightarrow }$ because $\widehat z\in \overline{W_{\widehat\Gamma}^{r\rightarrow }}$ and because $W_{\widehat\Gamma}^{r\rightarrow }$ and $ W_{\widehat\Gamma}^{l\rightarrow }$ are disjoint open sets. So $\widehat z$ belongs to $W_{\widehat\Gamma}^{\rightarrow l}$ and $\omega(\widehat z)\not=\emptyset$. \textcolor{black}{By Lemma \ref{lem:omegalimits}, situation (4) cannot occur and so situation (3) must hold on the left side:} there exists a closed essential leaf $\widehat\phi\subset \mathrm{Fr}(U_{\widehat \Gamma})\cap l_{\widehat \Gamma}$ such that $U_{\widehat \Gamma}\subset R(\widehat\phi)$. But this implies that $X^{\rightarrow R}_{\widehat\Gamma}$ is not adherent to $L$. Let us explain now why $\overline{W_{\widehat\Gamma}^{\rightarrow r}}\subset X^{R\rightarrow }_{\widehat\Gamma}\cap X^{\rightarrow R}_{\widehat\Gamma}$. Indeed, replacing $\widehat z$ by $\widehat z'\in \overline{W_{\widehat\Gamma}^{\rightarrow r}}\setminus (X^{R\rightarrow }_{\widehat\Gamma}\cap X^{\rightarrow R}_{\widehat\Gamma})$ in the previous argument will lead to a contradiction because $W_{\widehat\Gamma}^{l\rightarrow }=\emptyset$ and $\overline{W_{\widehat\Gamma}^{\rightarrow r}}\cap W_{\widehat\Gamma}^{\rightarrow l}=\emptyset$.

\bigskip
{\it Second case: there exists a whole transverse trajectory that meets four times a leaf of $U_{\widehat \Gamma}$. }

Let suppose that $\widehat\phi\subset U_{\widehat \Gamma}$ is met four times by a whole transverse trajectory. Fix a lift $\widetilde\phi\subset U_{\widetilde \gamma}$. We know that $$\bigcup_{n\geq 1} \widetilde f^{-n}(L(T^3(\widetilde \phi))\cap R(\widetilde \phi)\not=\emptyset.$$

\begin{sub-lemma} If $\overline{W_{\widehat\Gamma}^{r\rightarrow }}\not\subset (X^{R\rightarrow }_{\widehat\Gamma}\cap X^{\rightarrow R}_{\widehat\Gamma})$, there exists an admissible path $\widetilde\gamma_0: [a_0,b_0]\to \widetilde{\mathrm{dom}}(I)$ such that 
$$\widetilde\gamma_0(a_0)\in r(\widetilde\gamma), \enskip
\widetilde\gamma_0(a_0)\in L(\widetilde\phi)\cap R(T(\widetilde\phi)), \enskip
\widetilde\gamma_{(a_0,b_0]}\subset U_{\widetilde\gamma}, \enskip \widetilde\gamma(b_0)\in T^3(\widetilde\phi).$$

 \end{sub-lemma}
 \begin{proof} One can construct a line $\widetilde\lambda\subset\bigcup_{n\geq 1} \widetilde f^{-n}(L(T^3(\widetilde \phi))\cup R(\phi)$
that coincides  with $\widetilde \gamma$ outside a compact set, then one can find a neighborhood $V$ of $L$ in $\widehat {\mathrm{dom}}(I)$ whose preimage in $\widetilde {\mathrm{dom}}(I)$ belongs to $L(\widetilde\lambda)$.
Suppose that $\overline{W_{\widehat\Gamma}^{r\rightarrow }}\not\subset (X^{R\rightarrow }_{\widehat\Gamma}\cap X^{\rightarrow R}_{\widehat\Gamma})$,
and choose $\widehat z\in \overline{W_{\widehat\Gamma}^{r\rightarrow }}\setminus (X^{R\rightarrow }_{\widehat\Gamma}\cap X^{\rightarrow R}_{\widehat\Gamma})$. As explained in the proof given in the first case, there are three possibilities

\begin{itemize}

\item the orbit of $\widehat z$ is adherent to $L$;
\item there exists a closed essential leaf $\widehat\phi'\subset \mathrm{Fr}(U_{\widehat \Gamma})\cap l_{\widehat \Gamma}$ such that $U_{\widehat \Gamma}\subset R(\widehat\phi')$;
\item $\widehat I_{\widehat{\mathcal F}}(\widehat z)$ draws $\widehat\Gamma$ infinitely.
\end{itemize}

In the last case, one can find $\widehat z'\in W_{\widehat\Gamma}^{r\rightarrow }$ such that $\widehat I_{\widehat{\mathcal F}}(\widehat z')$ meets $\widehat\phi$ at least three times. Denote $\widehat\phi''$ the unique leaf on the boundary of $U_{\widehat\Gamma}$ met by  $\widehat I_{\widehat{\mathcal F}}(\widehat z')$  such that $U_{\widehat \Gamma}$ is on its left and let $\widetilde\phi''$  be the unique lift of $\widehat\phi''$ between $\widetilde\phi$ and $T(\widetilde\phi)$. The lift of  $\widehat I_{\widehat{\mathcal F}}(\widehat z')$ that intersects $\widetilde\phi''$ contains, as a sub-path, a path satisfying the conclusion of the sub-lemma.

\textcolor{black}{Assume then that the last case does not hold.} In the first case the orbit of $\widehat z$ meets $V$ and so, one can find $\widehat z'\in W_{\widehat\Gamma}^{r\rightarrow }$ whose orbit meets $V$. In the second case, one can find $\widehat z'\in W_{\widehat\Gamma}^{r\rightarrow }$ whose orbit meets $L(\widehat \phi')$ and one can observe that the lift $\widetilde \phi'$ of $\widehat \phi'$ is included in $L(\widetilde\lambda)$. In both cases, denote $\widehat\phi''$ the unique leaf on the boundary of $U_{\widehat\Gamma}$ met by  $\widehat I_{\widehat F}(\widehat z')$  such that $U_{\widehat \Gamma}$ is on its left and let $\widetilde\phi''$ be the unique lift of $\widehat\phi''$ lying between $\widetilde\phi$ and $T(\widetilde\phi)$. The set $\bigcup_{n\geq 0} \widetilde f^n(R(\widetilde\phi'')) $ meets $\widetilde \lambda$ but does not meet $R(\widetilde\phi)$, so it meets $\bigcup_{n\geq 1} \widetilde f^{-n}(L(T^3(\widetilde \phi)))$. It implies that the conclusion of the sub-lemma is true in both cases. \end{proof}

 Similarly, if  $\overline{W_{\widehat\Gamma}^{\rightarrow r}}\not\subset (X^{R\rightarrow }_{\widehat\Gamma}\cap X^{\rightarrow R}_{\widehat\Gamma})$, there exists an admissible path $\widetilde\gamma_1: [a_1,b_1]\to \widetilde{\mathrm{dom}}(I)$ such that 
$$\widetilde\gamma_1(b_1)\in r(\widetilde\gamma), \enskip
\widetilde\gamma_1(b_1)\in L(T^2(\widetilde\phi))\cap R(T^3(\widetilde\phi)), \enskip
\widetilde\gamma_{[a_1,b_1)}\subset U_{\widetilde\gamma}, \enskip \widetilde\gamma(a_1)\in \widetilde\phi.$$
The paths $\widetilde\gamma_0$ and $\widetilde\gamma_1$ project onto transverse loops that meet every leaf of $U_{\widehat\gamma}$ at least twice. Such a situation has been studied in \textcolor{black}{Lemma \ref{lem:newmeettwo}}
We can construct an admissible path with a $\widehat{\mathcal F}$-transverse self-intersection, in contradiction with the fact that $f$ has no horseshoe. \end{proof}

We are ready now to construct the positive annulus we are looking for. Let us begin with a simple result.

\begin{lemma}
\label{lemma:fundamental-group}There exists at most one fixed essential connected component of $W_{\widehat\Gamma}^{l\rightarrow r}$. We have a similar result for $W_{\widehat\Gamma}^{r\rightarrow l}$, $W_{\widehat\Gamma}^{\rightarrow r, l}$ and $W_{\widehat\Gamma}^{r,l \rightarrow }$.

\end{lemma}

\begin{proof} 
Let us prove by contradiction that there exists at most one essential connected component of $W_{\widehat\Gamma}^{l\rightarrow r}$ that is fixed. Indeed, if $V_1$ and $V_2$ are two such components, the complement of $V_1\cup V_2$ in $\widehat{\mathrm{dom}}(I)$ has a compact connected component $K$ and this component is fixed. This implies that for every $\widehat z\in K$, the sequences $(\widehat f^n(\widehat z))_{n\geq 0}$ and $(\widehat f^{-n}(\widehat z))_{n\geq 0}$ do not converge to an end of $\widehat{\mathrm{dom}}(I)$.

 Let us prove that $K\cap r_{\widehat\Gamma}=\emptyset$. This is clear in situation (1). It is also true in situation (4) because $K\cap X^{R\rightarrow }_{\widehat\Gamma}= K\cap X^{\rightarrow R }_{\widehat\Gamma} =\emptyset$ and $r_{\widehat\Gamma}\subset X^{R\rightarrow }_{\widehat\Gamma}\cup X^{\rightarrow R }_{\widehat\Gamma}$. Suppose now that we are in situation (2), write $\widehat\phi$ for the closed leaf included in $\mathrm{Fr}(U_{\widehat\Gamma})\cap r_{\widehat\Gamma}$, \textcolor{black}{and recall that, in this case, $U_{\widehat \Gamma}$ is contained in $R(\widehat \phi)$}. Choose an essential loop $\Gamma_1$ in $V_1$ and an essential loop in $\Gamma_2$ in $V_2$. Using the fact that $\Gamma_1$ and $\Gamma_2$ are compact and included in $W_{\widehat\Gamma}$, one knows that there exists $n\geq 0$ such that $\widehat f^{-n}(\Gamma_1)$ and $\widehat f^{-n}(\Gamma_2)$ are both contained in $R(\widehat\phi)$. This implies that $K$ is contained in $R(\widehat\phi)$. In case situation (3) occurs, we conclude with a similar argument. 
 
A similar proof tells us that $K\cap l_{\widehat\Gamma}=\emptyset$ and consequently that $K$ is included in $U_{\widehat \Gamma}$. This implies that $K$ is contained in $W_{\widehat\Gamma}$ and disjoint from $W_{\widehat\Gamma}^{r\rightarrow }$ and $W_{\widehat\Gamma}^{\rightarrow l}$. But $V_1\cup V_2\cup K$ is a neighborhood of $K$ that is disjoint from $W_{\widehat\Gamma}^{r\rightarrow }$ and $W_{\widehat\Gamma}^{\rightarrow l}$, and so $K$ is disjoint from $\overline{W_{\widehat\Gamma}^{r\rightarrow }}$ and $\overline{W_{\widehat\Gamma}^{\rightarrow l}}$. Consequently, $V_1\cup V_2\cup K$ is included in $W_{\widehat\Gamma}^{l\rightarrow r}$. We have got our final contradiction.
 
 The proof for the other sets is similar.
 \end{proof}

\begin{proposition} \label{prop: existence} If not empty, the set $\Omega'(I)_{\widehat \Gamma}$ is contained in an essential connected component of $W_{\widehat\Gamma}^{l\rightarrow r}$, $W_{\widehat\Gamma}^{r\rightarrow l}$, $W_{\widehat\Gamma}^{r,l\rightarrow }$ or $W_{\widehat\Gamma}^{\rightarrow r,l }$. Moreover this component will contain $A_{\widehat\phi}$, if $\widehat\phi$ is a closed leaf in $\mathrm{Fr}(U_{\widehat \Gamma})$.
\end{proposition}

\noindent{\it Proof.}\enskip\enskip 
 We will suppose that $\Omega'(I)_{\widehat \Gamma}\not=\emptyset$. \textcolor{black}{By Lemma \ref{lemma:homoclinic}, there are four cases to consider:
\begin{itemize}
\item[(a)]{$\overline{W_{\widehat\Gamma}^{r\rightarrow }}\subset X^{R\rightarrow }_{\widehat\Gamma}\cap X^{\rightarrow R}_{\widehat\Gamma}, \enskip \overline {W_{\widehat\Gamma}^{\rightarrow l}}\subset X^{L\rightarrow }_{\widehat\Gamma}\cap X^{\rightarrow L}_{\widehat\Gamma}$}
\item[(b)]{$\overline{W_{\widehat\Gamma}^{r\rightarrow }}\subset X^{R\rightarrow }_{\widehat\Gamma}\cap X^{\rightarrow R}_{\widehat\Gamma}, \enskip \overline {W_{\widehat\Gamma}^{l\rightarrow}}\subset X^{L\rightarrow }_{\widehat\Gamma}\cap X^{\rightarrow L}_{\widehat\Gamma}$}
\item[(c)]{$\overline{W_{\widehat\Gamma}^{\rightarrow r }}\subset X^{R\rightarrow }_{\widehat\Gamma}\cap X^{\rightarrow R}_{\widehat\Gamma}, \enskip \overline {W_{\widehat\Gamma}^{\rightarrow l}}\subset X^{L\rightarrow }_{\widehat\Gamma}\cap X^{\rightarrow L}_{\widehat\Gamma}$}
\item[(d)]{$\overline{W_{\widehat\Gamma}^{\rightarrow r }}\subset X^{R\rightarrow }_{\widehat\Gamma}\cap X^{\rightarrow R}_{\widehat\Gamma}, \enskip \overline {W_{\widehat\Gamma}^{l\rightarrow}}\subset X^{L\rightarrow }_{\widehat\Gamma}\cap X^{\rightarrow L}_{\widehat\Gamma}$}
\end{itemize}
Assume that we are in the case (a).}  We will prove that 
\begin{itemize}
\item there exists a fixed essential component $W$ of $W_{\widehat\Gamma}^{l\rightarrow r}$ (unique by Lemma \ref{lemma:fundamental-group}); 
\item $W$ contains $\Omega'(I)_{\widehat \Gamma}$;
\item $W$ contains $A_{\widehat\phi}$, if $\widehat\phi$ is a closed leaf in $\mathrm{Fr}(U_{\widehat \Gamma})$.
\end{itemize}
If $\widehat\phi$ is a closed leaf in $\mathrm{Fr}(U_{\widehat \Gamma})\cap r_{\widehat \Gamma}$, the assumption $\overline{W_{\widehat\Gamma}^{r\rightarrow }}\subset X^{R\rightarrow }_{\widehat\Gamma}\cap X^{\rightarrow R}_{\widehat\Gamma}$ \textcolor{black}{tells us that $\overline{W_{\widehat\Gamma}^{r\rightarrow }}$ has to be empty, and so} we are in situation (2) on the right side and we know that $A_{\widehat\phi}\subset W_{\widehat\Gamma}^{l\rightarrow r}$. So  there exists a fixed essential component of $W_{\widehat\Gamma}^{l\rightarrow r}$  and this component contains $A_{\widehat\phi}$. Similarly, if $\widehat\phi$ is a closed leaf in $\mathrm{Fr}(U_{\widehat \Gamma})\cap l_{\widehat \Gamma}$, we are in situation (3) on the left side and there exists a fixed essential component of $W_{\widehat\Gamma}^{l\rightarrow r}$  that contains $A_{\widehat\phi}$.

Note now that every essential connected component of $W_{\widehat\Gamma}^{l\rightarrow r}$ that meets $\Omega'(I)_{\widehat \Gamma}$ is fixed. Indeed, remember that $\Omega'(I)_{\widehat \Gamma}\subset \Omega(\widehat f)$. Moreover the set $W_{\widehat\Gamma}^{l\rightarrow r}$ being invariant by $\widehat f$, every connected component of $W_{\widehat\Gamma}^{l\rightarrow r}$ that meets $\Omega(\widehat f)$ is periodic.  Consequently it is fixed if supposed essential. We know that \textcolor{black}{$\Omega'(I)_{\widehat \Gamma} \subset W_{\widehat\Gamma}\cap \mathrm{ne}(\widehat f)$. Since we are in case (a), the assumptions on $\overline{W_{\widehat\Gamma}^{r\rightarrow }}$ and $\overline {W_{\widehat\Gamma}^{\rightarrow l}}$ imply these sets are disjoint from $\mathrm{ne}(\widehat f)$. We also know that $\mathrm{ne}(\widehat f)$ is disjoint from $(X^{\rightarrow R}_{\widehat\Gamma}\cap X^{L\rightarrow }_{\widehat\Gamma})$. One deduces that:}
$$\Omega'(I)_{\widehat \Gamma}\subset W_{\widehat\Gamma}\setminus (X^{\rightarrow R}_{\widehat\Gamma}\cap X^{L\rightarrow }_{\widehat\Gamma})\subset W_{\widehat\Gamma}^{l\rightarrow r}.$$
Choose $\widehat z\in \Omega'(I)_{\widehat \Gamma}$ and denote $\widehat W$ the connected component of $ W_{\widehat\Gamma}^{l\rightarrow r}$ that contains $\widehat z$. We will prove that $\widehat W$ is essential. Choosing another point  of $\Omega'(I)_{\widehat \Gamma}$ would give us the same set $\widehat W$ by Lemma \ref{lemma:fundamental-group}. So, $\widehat W$ contains $\Omega'(I)_{\widehat \Gamma}$. We will argue by contradiction and suppose that $\widehat W$ is not essential. The point $\widehat z$ being a non-wandering point of $\widehat f$,  we know that $\widehat W$ is periodic. We denote $q$ its minimal period. One can fill $\widehat W$, adding the compact connected components of its complement, to get a disk $\widehat W'$. The fact that $\widehat z$ is a non-wandering point of $\widehat f$ \textcolor{black}{contained in the open set $\widehat W$, and that $\widehat f^{i}(\widehat W)$ is disjoint from $\widehat W$ if $0\le i\le q$, implies that $\widehat z$ is also a non-wandering point of $\widehat f^{q}_{\widehat W'}$}. So, $\widehat f^q\vert _{\widehat W'}$ has a fixed point $\widehat z'$ by Brouwer's theory. There exists an integer $p\in\Z$ such that every lift $\widetilde W'$ of $\widehat W'$ satisfies $\widetilde f^{q}(\widetilde W')=T_{\widetilde \gamma}^p(\widetilde W')$. Moreover, the lift of $\widetilde z'$ that is in $\widetilde W'$ verifies $\widetilde f^q(\widetilde z')=T^p(\widetilde z')$. In other terms, there exists a periodic point $\widehat z'\in \widehat W'$ of period $q$ and rotation number $p/q$. 
Note that $p$ is positive. Indeed it is not negative because $\widehat z\in \Omega'(I)_{\widehat \Gamma}$ and does not vanish because $\widehat z$ cannot be lifted to a non wandering point of $\widetilde f$, this map being fixed point free. 

A point of the frontier of $W_{\widehat\Gamma}^{l\rightarrow r}$ in $\widehat{\mathrm{dom}}(I)$ belongs to $\overline{W_{\widehat\Gamma}^{r\rightarrow }}$, to $\overline{W_{\widehat\Gamma}^{\rightarrow l}}$, or does not belong to $W_{\widehat \Gamma}$. In the last case it belongs to $\bigcap_{k\in\Z} \widehat f^k(r_{\widehat\Gamma})$ or to $\bigcap_{k\in\Z} \widehat f^k(l_{\widehat\Gamma})$. By assumption, $\overline{W_{\widehat\Gamma}^{r\rightarrow }}$ is included in $X^{R\rightarrow }_{\widehat\Gamma}\cap X^{\rightarrow R}_{\widehat\Gamma}$ and $\overline{W_{\widehat\Gamma}^{\rightarrow l}}$ is included in $X^{L\rightarrow }_{\widehat\Gamma}\cap X^{\rightarrow L}_{\widehat\Gamma}$. So the frontier of $\widehat W'$ in $\widehat{\mathrm{dom}}(I)_{\textcolor{black}{\mathrm {sp}}}$ is contained in the frontier of $W_{\widehat\Gamma}^{l\rightarrow r}$, which means in the union of the two disjoint closed sets $\{R\}\cup\overline{W_{\widehat\Gamma}^{r\rightarrow }}\cup\left(\bigcap_{k\in\Z} \widehat f^k(r_{\widehat\Gamma})\right)$ and $\{L\}\cup\overline{W_{\widehat\Gamma}^{\rightarrow l}}\cup\left(\bigcap_{k\in\Z} \widehat f^k(l_{\widehat\Gamma})\right)$. Moreover, one knows that this frontier is connected because $\widehat W'$ is a topological disk. So it is contained in one of the two sets and there is no loss of generality by supposing that it is contained in $\{R\}\cup\overline{W_{\widehat\Gamma}^{r\rightarrow }}\cup\left(\bigcap_{k\in\Z} \widehat f^k(r_{\widehat\Gamma})\right)$. Of course the frontier is not reduced to $R$, because  $\widehat W'$ does not contain $L$ and so $r_{\widehat\Gamma}$ is not empty, we are not in situation (1). The set $\mathrm{Fr}(U_{\widehat \Gamma})\cap r_{\widehat\Gamma}$ cannot be reduced to a closed leaf $\widehat\phi$. Indeed, as explained above, we would be in situation (2) and $W$ should be disjoint from $A_{\widehat\phi}$ because it is not essential. Its closure should be disjoint from $r_{\widehat\Gamma}$ because $\widehat W$ meets $U_{\widehat \Gamma}$ (recall that $\widehat z$ belongs to $U_{\widehat \Gamma}$) and the same occurs for $\widehat W'$. We conclude that we are in situation (4), the set $\mathrm{Fr}(U_{\widehat \Gamma})\cap r_{\widehat\Gamma}$ is a non empty union of leaves homoclinic to $R$. Recall that in this situation, we have $\bigcap_{k\in\Z} \widehat f^k(r_{\widehat\Gamma})\subset X^{R\rightarrow }_{\widehat\Gamma}\cap X^{\rightarrow R}_{\widehat\Gamma}$ and so 
 the frontier of $\widehat W'$ in $\widehat{\mathrm{dom}}(I)$ is contained  in $X^{R\rightarrow }_{\widehat\Gamma}\cap X^{\rightarrow R}_{\widehat\Gamma}$.
 
 As in the proof of Lemma \ref{lemma:no-Birkhoff-cycle} we blow up the end $R$ by adding a circle $\widehat \Sigma_R$ to $\widehat {\mathrm {dom}}(I)$, we construct $\widehat f_{\mathrm {ann}}$ on the semi-open annulus $\widehat {\mathrm {dom}}(I)_{\mathrm {ann}}$ and $\widehat f_{\mathrm {ann}}$ on the annulus $\widehat {\mathrm {dom}}(I)_{\mathrm{double}}$. We remind that  $\widehat f_{\mathrm {double}}$ has a unique fixed point $\widehat z_*$ on $\widehat \Sigma_R$ and denote $\widetilde f_{\mathrm {double}}$ the lift that fixes the preimages of $\widehat z_*$. Consider a maximal isotopy $\widehat I'$ of $\widehat f^q_{\mathrm {double}}$ that is lifted to a maximal isotopy $\widetilde I'$ of $T_{\widetilde\gamma}^{-p}\circ\widetilde f^q_{\mathrm {double}}$ such that $\widehat z'\in  \mathrm{sing}(\widehat I')$ and then a foliation $\widehat{\mathcal{F}'}$ transverse to $\widehat I'$. 
The point $\widehat z_*$ belongs to the domain of $\widehat I'$. Being a fixed point, \textcolor{black}{its whole trajectory is equivalent to the natural lift of a loop $\widehat \Gamma '$, which we can assume to be simple by Proposition \ref{prop:non-wandering},} as  $\widehat f_{\mathrm double}$ has no horseshoe. The $\alpha$-limit set and the $\omega$-limit set of a point of $\widehat \Sigma_R$ being reduced to $\widehat z^*$ and $\widehat f_{\mathrm double}$ having no horseshoe, we know that the whole transverse trajectory of  any point of  $\widehat \Sigma_R$  is equivalent to the natural lift of $\widehat \Gamma'$. The same occurs for a point on $\mathrm{Fr}(\widehat W')$ because its $\alpha$-limit set and its $\omega$-limit set are contained in $\widehat \Sigma_R$. Consequently, the set $V_{\widehat\Gamma'}$ of points whose whole transverse trajectory draws $\widehat\Gamma'$ is an open set that contains $\widehat \Sigma_R$ and $\mathrm{Fr}(\widehat W')$. As $V_{\widehat \Gamma'}$ is contained in the open annulus $U_{\widehat \gamma'}$ of leaves met by $\Gamma'$ and as $\widehat W'$ is a topological disk with $\mathrm{Fr}(\widehat W') \subset U_{\widehat \gamma'}$, it follows that $\widehat W'$ itself must be contained in this annulus, a contradiction since $\widehat z'$ is a singular point. The proof is finished in this case.

Suppose now that \textcolor{black} {case (b) holds, that is,}$$\overline{W_{\widehat\Gamma}^{r\rightarrow }}\subset X^{R\rightarrow }_{\widehat\Gamma}\cap X^{\rightarrow R}_{\widehat\Gamma}, \enskip \overline {W_{\widehat\Gamma}^{l\rightarrow }}\subset X^{L\rightarrow }_{\widehat\Gamma}\cap X^{\rightarrow L}_{\widehat\Gamma}.$$  
 We will \textcolor{black}{again} prove that 
\begin{itemize}
\item[\textcolor{black}{(i)}] there exists a fixed essential component $W$ of $W_{\widehat\Gamma}^{\rightarrow r,l}$ (unique by Lemma \ref{lemma:fundamental-group}); 
\item[\textcolor{black}{(ii)}] $W$ contains $\Omega'(I)_{\widehat \Gamma}$;
\item[\textcolor{black}{(iii)}] $W$ contains $A_{\widehat\phi}$, if $\widehat\phi$ is a closed leaf in $\mathrm{Fr}(U_{\widehat \Gamma})$.
\end{itemize}
If $\widehat\phi$ is a closed leaf in $\mathrm{Fr}(U_{\widehat \Gamma})\cap r_{\widehat \Gamma}$, the assumption $\overline{W_{\widehat\Gamma}^{r\rightarrow }}\subset X^{R\rightarrow }_{\widehat\Gamma}\cap X^{\rightarrow R}_{\widehat\Gamma}$ tells us that we are in situation (2) on the right side and the assumption $\overline {W_{\widehat\Gamma}^{l\rightarrow }}\subset X^{L\rightarrow }_{\widehat\Gamma}\cap X^{\rightarrow L}_{\widehat\Gamma}$ tells us that $A_{\widehat\phi}$ does not meet $\overline {W_{\widehat\Gamma}^{l\rightarrow }}$ and so that $A_{\widehat\phi}\subset W_{\widehat\Gamma}^{\rightarrow r,l}$. So  there exists a fixed essential component of $W_{\widehat\Gamma}^{\rightarrow r, l}$  and this component contains $A_{\widehat\phi}$. Similarly, if $\widehat\phi$ is a closed leaf in $\mathrm{Fr}(U_{\widehat \Gamma})\cap l_{\widehat \Gamma}$, we are in situation (2) on the left side and there exists a fixed essential component of $W_{\widehat\Gamma}^{\rightarrow r,l}$  that contains $A_{\widehat\phi}$. \textcolor{black}{
This show that (iii) holds provided (i) holds. The arguments above showing that items (i) and (ii) hold for case (a) can be applied to case (b) with minor modifications.
As before, since we are in case (b), the assumptions on $\overline{W_{\widehat\Gamma}^{r\rightarrow }}$ and $\overline {W_{\widehat\Gamma}^{l\rightarrow}}$ imply these sets are disjoint from $\mathrm{ne}(\widehat f)$, so that $\Omega'(I)_{\widehat \Gamma}$ is contained in $W_{\widehat\Gamma}^{\rightarrow r,l}$. The proof of the invariance of any essential connected component of $W_{\widehat\Gamma}^{\rightarrow r,l}$ that meets $\Omega'(I)_{\widehat \Gamma}$ is identical.
The final step, showing that any connected component $\widehat W$ of $W_{\widehat\Gamma}^{\rightarrow r,l}$ that meets $\Omega'(I)_{\widehat \Gamma}$ must be essential is also similar, just noting that $\widehat W$ will again be $f^q$-invariant for some integer $q\ge 1$, and that if we assume by contradiction that it is not essential, then $\widehat W$ may be filled to a set $\widehat W'$ having a periodic point of $f$,
 whose boundary is either entirely contained in $\{R\}\cup\overline{W_{\widehat\Gamma}^{r\rightarrow }}\cup\left(\bigcap_{k\in\Z} \widehat f^k(r_{\widehat\Gamma})\right)$ or it is entirely contained in $\{L\}\cup\overline{W_{\widehat\Gamma}^{rl\rightarrow }}\cup\left(\bigcap_{k\in\Z} \widehat f^k(l_{\widehat\Gamma})\right)$. The finals arguments then apply in the same way. Case (c) and (d) are dealt similarly}.
\end{proof}

 \begin{remark*} Note that if $U_{\widehat\Gamma}$ is bordered by two closed leaves $\widehat\phi$ and $\widehat\phi'$, then the annulus $A_{\widehat\phi}\cup U_{\widehat\Gamma}\cup A_{\widehat\phi'}$ coincides with one the sets $W_{\widehat\Gamma}^{l\rightarrow r}$, $W_{\widehat\Gamma}^{r\rightarrow l}$, $W_{\widehat\Gamma}^{r,l\rightarrow }$ or $W_{\widehat\Gamma}^{\rightarrow r,l }$ and is equal to $\widehat W$. In case $\mathrm{Fr}(U_{\widehat \Gamma})$ contains a closed leaf $\widehat\phi$,  then $A_{\widehat\phi}$ does not separate $\widehat W\setminus A_{\widehat\phi} $.  
\end{remark*}

\bigskip

\noindent{\it Proof of Proposition \ref{prop:local-structure}}\enskip\enskip Here again we suppose that $\overline{W_{\widehat\Gamma}^{r\rightarrow }}$ is contained in $X^{R\rightarrow }_{\widehat\Gamma}\cap X^{\rightarrow R}_{\widehat\Gamma}$ and $\overline {W_{\widehat\Gamma}^{\rightarrow l}}$ contained in $X^{L\rightarrow }_{\widehat\Gamma}\cap X^{\rightarrow L}_{\widehat\Gamma}$. Write $\widehat W$ for the essential component of $W_{\widehat\Gamma}^{l\rightarrow r }$ that contains $\Omega'(I)_{\widehat \Gamma}$. One gets an invariant annulus $\widehat A$ by filling $\widehat W$, which means by adding the compact connected components of its complement. The choice of the class $[\widehat \Gamma]$ as a generator of $H_1(\widetilde{\mathrm{dom}}(I),\Z)$ permits us to give a partition of $\Omega(\widehat f)$ in $\widetilde f$-positive or $\widetilde f$-negative points. Let us prove that $\widehat A$ is a positive annulus of $\widehat I$ and $\kappa\in H_{1}(\widehat A,\Z)$ its positive generator, where $\iota_*(\kappa)=[\widehat \Gamma]$. Suppose that $\widehat z\in\Omega(\widehat f)$ is $\widetilde f$-negative. As stated in Proposition \ref{prop: positive}, the transverse loop associated with $\widehat z$ is negative. This loop cannot meet $U_{\widehat \Gamma}$, so it is contained  in $r_{\widehat\Gamma}$ or in $l_{\widehat \Gamma}$. Suppose for instance that it is contained in $r_{\widehat\Gamma}$. This is possible only if $\mathrm{Fr}(U_{\widehat \Gamma})\cap r_{\widehat\Gamma}$ is reduced to a closed leaf $\widehat\phi$, more precisely when situation (2) occurs. As explained in Proposition  \ref{prop: existence}, the annulus $A_{\widehat\phi}$ is included in $\widehat W$. So all compact components you need to add to $\widehat W$ to get $\widehat A$ are included in $R(\widehat\phi)$ and so $\widehat z$ does not belong to $\widehat A$. The fact that every point of $\widehat A\cap \Omega(\widehat f)$ is $\widetilde f$-positive, implies that $\widehat A$ is positive and $\kappa\in H_{1}(\widehat A,\Z)$ its positive generator, where $\iota_*(\kappa)=[\widehat \Gamma]$. 

Write $\widetilde W=\pi^{-1}(\widehat W)$ and $\widetilde A=\pi^{-1}(\widehat A)$. As a connected component of $\widetilde\pi^{-1}({W_{\Gamma}^{l\rightarrow r }})$ one knows that the stabilizer of $\widetilde W$  in the group of covering automorphisms is generated by $T_{\widetilde \gamma}$ and, as remarked after the definition of $W_{\widetilde\gamma}^{l\rightarrow r}$, that $T(\widetilde W)\cap \widetilde W=\emptyset$ if $T$ is not power of $T_{\widetilde\gamma}$. One deduces similar results for $\widetilde A$. This implies that $\widetilde A$ projects onto a positive annulus $A$ of $I$ that contains $\Omega'(I)_{\Gamma}$ and whose positive generator is given by the homology class of a simple loop freely homotopic to $\Gamma$ in $\mathrm{dom}(I)$. In fact $\widetilde W$ projects onto a connected component of $W_{\widehat\Gamma}^{l\rightarrow r }$ and $A$ is obtained from this component by adding the connected components of the complement that have no singular points. \hfill$\Box$

\subsection{Additional results} \label{subsection:auxilliary}

  We would like to state a result relied only on the isotopy $I$ and not on the transverse foliation $\mathcal F$. In this subsection, $f$ is an orientation preserving homeomorphism of $\S^2$ and $I$ a maximal isotopy of $f$. We keep the usual notations such as $\mathrm{dom}(I)$, $\widetilde{\mathrm{dom}}(I)$, $\widetilde \pi$, $\widetilde f$, .... \textcolor{black}{We say that a topological disk $D\subset \mathrm{dom}(I)$ is $I$-free if it holds that $\widetilde f(\widetilde D)$ is disjoint from $\widetilde D$ whenever $\widetilde D$ is a lift of $D$.} Let us begin with an useful lemma.

\begin{lemma}\label{lemma:homotopic annuli}
 If $A$ and $A'$ are open annuli in $\mathrm{dom}(I)$ invariant by $f$ such that:
 \begin{itemize}
 \item the two connected components of  $\S^2\setminus A$ contain a fixed point of $I$;
 \item the two connected components of  $\S^2\setminus A'$ contain a fixed point of $I$;
 \item $A\cap A'\cap \Omega(f)\not=\emptyset$.
 \end{itemize}
  Then $A$ and $A'$ are homotopic in $\mathrm{dom}(I)$ (meaning that an essential simple loop $\Gamma$ of $A$ and an essential loop $\Gamma'$ of $A'$ are freely homotopic in $\mathrm{dom}(I)$, up to a change of orientation).
\end{lemma}
\begin{proof}Let us fix $z\in A\cap A'\cap\Omega(f)$ and a lift $\widetilde z\in \widetilde{\mathrm{dom}}(I)$. Write $\widetilde A$ and $\widetilde A'$ for the connected components of $\widetilde\pi^{-1}(A)$ and $\widetilde\pi^{-1}(A')$ respectively, that contain $\widetilde z$.  Fix a generator $T$ of the stabilizer of $\widetilde A$ in the group of covering automorphisms and a generator $T'$ of the stabilizer of $\widetilde A'$. We have seen at the beginning of this section that $\widetilde A$ and $\widetilde A'$ are invariant by $\widetilde f$. Let $D\subset A\cap A'$ be a topological $I$-free disk containing $z$. One can find an integer $n\geq 1$ and a point $z'\in D\cap f^{-n}(D)$. Write $\widetilde D$ for the lift of $D$ that contains $\widetilde z$, and $\widetilde z'$ the lift of $z'$ that lies in $\widetilde D$. There exists a non trivial covering transformation $T''$ such that $\widetilde f^n(\widetilde z)\in T''(\widetilde D)$ and so $T''(\widetilde D)\subset \widetilde A\cap\widetilde A'$. This implies that $T''$ is a non trivial power of $T$ and a non trivial power of $T'$. In other terms, the stabilizers of $\widetilde A$ and $\widetilde A'$ are equal. But this means that $A$ and $A'$ are homotopic. \end{proof}

 Let us explain now how to associate simple loops to points of $\Omega'(I)$ independently of a transverse foliation. Note first that for every point $z\in \Omega'(I)$, there exist $I$-free disks $D$ such that $f^k(z)\in D$ for infinitely many $k\in\Z$. Indeed, consider a point $z'\in\alpha(z)\cup\omega(z)$ that is not fixed by $I$. Any topological open disk containing $z'$ is $I$-free if sufficiently small, and meets the orbit of $z$ infinitely often.

\bigskip
\begin{proposition} \label{prop: free disks} \textcolor{black}{If $f$ has no topological horseshoe}, for every  $z\in\Omega'(I)$, there exists a simple loop $\Gamma \subset\mathrm{dom}(I)$, uniquely defined up to homotopy, such that: 
\begin{enumerate}
\item there exists a positive annulus $A$ containing $z$ such that an essential simple loop of $A$ defining the positive class of $A$ is homotopic to $\Gamma$;
\item every invariant annulus \textcolor{black}{$A'$ in $\mathrm{dom}(I)$} containing $z$ and such that the two connected components of  its complement contain a fixed point of $I$ is homotopic to $A$ in $\mathrm{dom}(I)$;
\item for every $I$-free disk $D$ that contains two points $f^n(z)$ and $f^m(z)$, $n<m$, and every path $\delta\subset D$ joining $f^m(z)$ to $f^n(z)$,  the loop $I^{m-n}(f^n(z))\delta$ is homotopic to a positive power of $\Gamma$.
\end{enumerate}
\end{proposition}
\begin{proof} We have seen in Theorem \ref{th:local-structure} that there exists a positive annulus $A$ that contains $z$. More precisely, if $\mathcal F$ is a transverse foliation of $I$, one can choose for $\Gamma$ the transverse simple loop associated with $z$ (or any simple loop homotopic in $\mathrm{dom}(I)$). The assertion (1) is proved and the assertion (2) is an immediate consequence of Lemma \ref{lemma:homotopic annuli}. It remains to prove (3). Suppose that $D$ is a $I$-free disk that contains two points $f^n(z)$ and $f^m(z)$, $n<m$. \textcolor{black}{By Proposition 59 of \cite{LeCalvezTal},} there exists a transverse foliation ${\mathcal F}'$ of $I$ such that $f^n(z)$ and $f^m(z)$ belong to the same leaf of the restricted foliation ${\mathcal F}'\vert_D$. Let $\phi'$ be the leaf of ${\mathcal F}'$ that contains $f^n(z)$ and $f^m(z)$ and $\Gamma'$ the transverse loop associated with $z$. The homotopy class of $I^{m-n}(f^n(z))\delta$ does not depend on the choice of the path $\delta\subset D$ joining $f^m(z)$ to $f^n(z)$. Moreover, choosing for $\delta$ the segment of leaf of $\phi'\cap D$ that joins $f^m(z)$ to $f^n(z)$, and noting that $I^{m-n}(f^n(z))\delta$ is homotopic to $I^{m-n}_{{\mathcal F}'}(f^n(z))\delta$, one deduces that $I^{m-n}(f^n(z))\delta$ is homotopic to a positive power of $\Gamma'$. As seen in the proof of Theorem \ref{th:local-structure}, there exists a positive annulus $A'$ that contains $z$, such that $\Gamma'$ is homotopic in $\mathrm{dom}(I)$ to an essential simple loop of $A'$ that defines the positive class of $A'$. By $(2)$, one deduces that $\Gamma'$ is homotopic to $\Gamma$ or to its inverse. If we want to prove that $\Gamma'$ is homotopic to $\Gamma$, we have to look carefully at the proof of Lemma \ref{lemma:homotopic annuli}. If the disk $D$ is chosen very small, then the path $I^{n+2}_{\widetilde{\mathcal F}} (\widetilde f^{-1}(\widetilde z'))$ contains as a sub-path a transverse path joining the leaf $\widetilde\phi$ of $\widetilde{\mathcal F}$ that contains $\widetilde z$ to $T''(\widetilde\phi)$. Similarly, the path $I^{n+2}_{\widetilde{\mathcal F}'} (\widetilde f^{-1}(\widetilde z'))$ contains as a sub-path a transverse path joining the leaf $\widetilde\phi'$ of $\widetilde{\mathcal F}'$ that contains $\widetilde z$ to $T''(\widetilde\phi')$. We deduce that $T''$ is a positive power of $T$ and a positive power of $T'$. So $T$ and $T'$ are equal, which means that $\Gamma$ and $\Gamma'$ are homotopic in $\mathrm{dom}(I)$.\end{proof}

\bigskip

A simple loop $\Gamma\subset\mathrm{dom}(I)$ that satisfies the properties stated in Proposition \ref{prop: free disks} is said to be {\it associated with $z\in\Omega'(I)$}. 

\begin{theorem}\label{th:familyofannulilocal}
\textcolor{black}{If $f$ has no topological horseshoe, then} one can cover $\Omega'(I)$ by invariant annuli $A\subset\mathrm{dom}(I)$ satisfying the following:
a point $z\in\Omega'(I)$ belongs to $A$ if and only if one can choose as a loop  associated with $z\in\Omega'(I)$ an essential simple loop of $A$.
\end{theorem}

\begin{proof}
Let $\mathcal F$ be a foliation transverse to $I$. \textcolor{black}{If $\Gamma\subset \mathrm{dom}(I)$ is a simple loop, write $U(\Gamma)$ for the union of the leafs of $\mathcal F$ that intersect $\Gamma$, a topological annulus.} For every $z\in\Omega'(I)$, write $\Gamma_z$ for the transverse simple loop associated to $z$ and set $z\sim z'$ if $\Gamma_z$ and $\Gamma_{z'}$ are homotopic, up to the sign. If $\widehat{\mathrm{dom}}(I)$ is an annular covering space associated to $\Gamma_z$, as defined in the previous sub-section, then $z'\sim z$ if and only if $\Gamma_{z'}$ can be lifted as a simple loop of  $\widehat{\mathrm{dom}}(I)$.  In case, $z\sim z'\Leftrightarrow \Gamma_z=\Gamma_{z'}$, there is nothing to prove, one can choose for $A$ the annulus given by Proposition \ref{prop:local-structure}. We will suppose from now that this does not happen. This implies that for every $z'\sim z$, the annulus $U(\Gamma_{z'})$ has a closed leaf on its boundary. Moreover, there exist at most two different \textcolor{black}{points $z'\sim z$ such that $U(\Gamma_{z'})$ is not bordered by two closed leaves}.  Write $\mathcal C$ for the set of closed annulus bordered by two closed leaves homotopic to $\Gamma_z$, up to the sign, and eventually reduced to a single loop. For every $C\in{\mathcal C}$, define
$$A_C=A_{\phi}\cup C\cup A_{\phi'},$$ where $C$ is bordered by $\phi$ and $\phi'$. One gets an invariant open annulus homotopic to $U(\Gamma_{z})$. Now define
$$C^*=\bigcup_{C\in\mathcal C} C, \, A^*=\bigcup_{C\in\mathcal C} A_C.$$
The first set is an annulus, that can be open, semi open and bordered with a closed leaf, compact and bordered by two closed leaves or degenerate and reduced to unique closed leaf. The second set is an an invariant open annulus.

In case $C^*$ is open, it contains all sets \textcolor{black}{$U({\Gamma_{z'}})$}, $z'\sim z$, and consequently all points $z'\sim z$. We can choose $A= A^*$.

In case $C^*$ is semi open and contains all sets \textcolor{black}{$U({\Gamma_{z'}})$}, $z'\sim z$, we can choose $A=A^*$.

In case $C^*$ is semi open and does not contain all sets \textcolor{black}{$U({\Gamma_{z'}})$}, $z'\sim z$, there exists a unique \textcolor{black}{$U({\Gamma_{z_*'}})$}{, with $z_*\sim z$, such that \textcolor{black}{$U({\Gamma_{z_*'}})$}} is not contained in $C^*$ and the closed leaf $\phi$ that borders $C^*$ also borders  \textcolor{black}{$U({\Gamma_{z_*'}})$}. Using Proposition \ref{prop: existence} and the remark that follows, we know that the positive annulus $A_{\Gamma_{z_*}}$ constructed in the proof of Proposition \ref{prop:local-structure} contains $A_{\widehat\phi}$, more precisely
$A_{\Gamma_{z_*}} \cap A^*= A_{\widehat\phi}$. So we have two homotopic annuli whose intersection is a annulus homotopic to both of them, so the union is an annulus and we can choose $A=A_{\Gamma_{z_*}} \cup A^*$.

In case $C^*$ is compact, we similarly construct $A$ as the union of $A^*$ and two annuli of type $A_{\Gamma}$. \end{proof}

\subsection{Proof of Theorem \ref{thmain:global-structure}} \label{subsection: proofglobal}

Using the fact that every squeezed annulus is contained in a maximal squeezed annulus, Theorem \ref{thmain:global-structure} can be written as below:

\begin{theorem}\label{th:global-structure2}
Let $f:\S^2\to\S^2$ be an orientation preserving homeomorphism that has no topological horseshoe. Then $\Omega'(f)$ is covered by squeezed annuli.\end{theorem}

\begin{proof}
Let $z$ be a point of $\Omega'(f)$. Suppose first that $z$ belongs to an invariant open annulus $A_0$ such that:
\begin{itemize}
\item the two connected components of $\S^2\setminus A_0$ are fixed by $f$;
\item if $\kappa_0$ is a generator of $H_1(A_0,\Z)$, then  $\mathrm{rot}_{ f\vert_{A_0},\kappa_0}(z)\not=0+\Z$.
\end{itemize}
We write $f_0=f\vert_{A_0}$ and denote $\check f_0 $ the lift of $f_0$ to the universal covering space $\check A_0$,  such that $\mathrm{rot}_{\check f_0,\kappa_0}(z)\in(0,1)$. Let  $I_0$ be a maximal  isotopy of $f_0$ that is lifted to an identity isotopy of $\check f_0$. By Theorem  \ref{th:local-structure},  there exists a positive annulus $A_1$ of $I_0$ that contains $z$. The fact that $\mathrm{rot}_{\check f_0,\kappa_0}(z)>0$ implies that $A_1$ is essential in $A_0$ and that the positive generator $\kappa_1$ of $H_1(A_1,\Z)$ is sent onto $\kappa_0$ by the morphism $i_*: H_1(A_1,\Z)\to H_1(A_0,\Z)$. Let $\check f_1$ be a lift of $f_1=f\vert_{A_1}$ to the universal covering space of $A_1$ such that $\mathrm{rot}_{\check f_1,\kappa_1} (z)=\mathrm{rot}_{\check f_0, \kappa_0}(z)-1$ and $I_1$ be a maximal isotopy of $f_1$ lifted to an identity isotopy of $\check f_1$. By Theorem  \ref{th:local-structure}, there exists a positive annulus $A_2$ of $I_1$ containing $z$. The fact that $\mathrm{rot}_{\check f_1,\kappa_1} (z)<0$ implies that $A_2$ is essential in $A_1$ and that the positive generator $\kappa_2$ of $H_1(A_2,\Z)$ is sent onto $-\kappa_1$ by the morphism $i_*: H_1(A_2,\Z)\to H_1(A_1,\Z)$. The annulus $A_2$ is a squeezed annulus containing $z$.

\medskip

Suppose now that such an annulus $A_0$ does not exist. We will slightly change the proof. By assumption,  $\alpha(z)\cup\omega(z)$ is not included in $\mathrm{fix}(f)$. Choose a point $z'$ of $\alpha(z)\cup\omega(z)$ that is not fixed. Any topological open disk containing $z'$ is free if sufficiently small, and meets the orbit of $z$ infinitely often. So, there exists a free open disk $D$ containing $z$ and $f^q(z)$, where $q\geq 2$. One can find a simple path $\delta$ in $D$ that joins $f^q(z)$ to $z$. There exists an identity isotopy $I_*=(h_t)_{t\in[0,1]}$ supported on $D$, such that $h=h_1$ sends  $f^{q}(z)$ on $z$. Moreover, one can suppose that the trajectory  $t\mapsto h_t(z)$ is the subpath of $\delta$ that joins $f^q(z)$ to $z$. The point $z$ is a periodic point of $h\circ f$, of period $q$. Moreover $f$ and $h\circ f$ have the same fixed points because $D$ is free. Fix $z_*\in\mathrm{fix}(f)$. The main result of \cite{LeCalvez3} tells us that $f$ has a fixed point $z_0\not=z_*$ such that $\mathrm{rot}_{h\circ f\vert_{A_0}, \kappa_0}(z)\not =0+\Z$, where $A_{0}=\S^2\setminus\{z_*,z_0\}$ and $\kappa_0$ is a generator of $H_1(A_0,\Z)$. Write $f_0=f\vert_{A_0}$, $h_0=h\vert _{A_0}$, denote $\check h_0$ the lift of $h_0$ to the universal covering space $\check A_0$ of $A_0$ naturally defined by $I_*$ and $\check f_0$ the lift of $f_0$ to $\check A_0$  such that $\mathrm{rot}_{\check f_0, \kappa_0}(z)=0$, which exists since by assumption $A_0$ does not satisfy the properties of the previous case. Replacing $\kappa_0$ with $-\kappa_0$ if necessary, one can suppose that $\mathrm{rot}_{\check h_0\circ\check f_0, \kappa_0}(z)>0$. Let  $I_0$ be a maximal isotopy of $f_0$ lifted to an identity isotopy of $\check f_0$. As explained in \textcolor{black}{in the proof of Proposition 59 of \cite{LeCalvezTal}}, there exists a foliation $\mathcal F_0$ transverse to $I_0$ such that $\delta$ is included in a leaf of $\mathcal F_0$. 

\textcolor{black}{By Proposition \ref{prop:non-wandering} and the remark following it, one has that $I_0{}_{{\mathcal F}_0}^{\Z}(z)$ exactly draws infinitely a transverse simple loop $\Gamma_1$, and therefore $z$ belongs to $\Omega'(I_0)_{\Gamma_1}$. Proposition \ref{prop:local-structure} tells us that:} 
\begin{itemize}
\item $z$ belongs to a positive annulus $A_1$;
\item the image of the positive class $\kappa_1$ of $A_1$  is sent onto the class of $\Gamma_1$ by the morphism $i_*: H_1(A_1,\Z)\to H_1(\mathrm{dom}(I),\Z)$. 
\end{itemize}

The homology class in $A_0$ of the loop naturally defined by $ I_0{}_{\mathcal F_0}^q(z)\delta$ is a positive multiple of $\kappa_0$ because  $\mathrm{rot}_{\check h_0\circ\check f_0, \kappa_0}(z)>0$ but it is a positive multiple of  the class of $\Gamma_1$. So $A_1$ is essential in $A_0$ and the positive generator $\kappa_1$ of $H_1(A_1,\Z)$ is sent onto $\kappa_0$ by the morphism $i_*: H_1(A_1,\Z)\to H_1(A_0,\Z)$. Set $f_1=f\vert_{A_1}$. The rotation number of $z$ for the lift of $f_1$ to the universal covering space of $A_1$ naturally defined by $I_0$ (as explained at the beginning of this section) is equal to zero because $\mathrm{rot}_{\check f_0, \kappa_0}(z)=0$. Let $\check f_1$ be the lift of $f_1$  such that $\mathrm{rot}_{\check f_1,\kappa_1} (z)=-1$ and $I_1$ be a maximal isotopy of $f_1$ that is lifted to an identity isotopy of $\check f_1$. By Theorem  \ref{th:local-structure},  there exists a positive annulus $A_2$ of $I_1$ that contains $z$. To ensure that $A_2$ is a squeezed annulus, one must prove that the positive generator $\kappa_2$ of $H_1(A_2,\Z)$ is sent onto $-\kappa_1$ by the morphism $i_*: H_1(A_2,\Z)\to H_1(A_1,\Z)$. It cannot be sent on $\kappa_1$ because $\mathrm{rot}_{\check f_1,\kappa_1} (z)<0$ so it remains to prove that $A_2$ is essential in $A_1$. If it is not essential, the compact connected component $K'$ of $A_1\setminus A_2$ contains a fixed point $z''$ of $I_1$ and we have $\mathrm{rot}_{\check f_1,\kappa_1} (z'')=0$. Let $\check{f}_0'$ be the lift of $f_0$ such that $\mathrm{rot}_{\check{f}_0',\kappa_0} (z)=-1$, which implies that $\mathrm{rot}_{\check{f}_0',\kappa_0}  (z'')=0$. The set $A_2\cup K'$ being an invariant open disk in $A_1\subset A_0$ containing $z$ and $z''$, one deduces that the pre-images of $A_2\cup K'$ by the projection from $\check{A}_0$ to $A_0$ are all invariant by $\check{f}_0'$, and as $z$ is nonwandering, any point $\check{z}$ projecting onto $z$ is also nonwandering for $\check{f_0}'$.  Since  $\mathrm{rot}_{\check{f}_0',\kappa_0} (z)=-1$, one can find a disk $\check D$ that is free for $\check{f}_0'$ and that contains $\check z$, and positive integers $n_0,m_0$ such that $(\check{f}_0')^{n_0}(\check z)\in T^{-m_0}(\check D)$. Since $\check z$ is nonwandering, there exists some point $\check z_1$ sufficiently close to $\check z$ and a positive integer $n_1$ such that both $\check z_1, (\check{f}_0')^{n_0+n_1}(\check z_1)$ belong to $\check D$ and such that $(\check{f}_0')^{n_0}(\check z_1)\in T^{-m_0}(\check D)$. We have got a contradiction with Corollary \ref{cor:free-disks}.\end{proof}

\bigskip
The following result can be easily deduced from the proof above.

\begin{proposition} 
\label{prop:linking-recurrent} Let $f:\S^2\to\S^2$ be an orientation preserving homeomorphism that has no topological horseshoe. Then for every recurrent point $z\not\in\mathrm{fix}(f)$ and every fixed point $z_*$, there exists a fixed point $z_0\not=z_*$ such that $\mathrm{rot}_{f\vert_{\S^2\setminus\{z_*,z_0\}},\kappa}(z)\not=0+\Z$, if $\kappa$ is a generator of $H_1(\S^2\setminus\{z_*,z_0\},\Z)$.\end{proposition}

\begin{proof}Indeed, if the conclusion fails, the second part of the proof of Theorem \ref{th:global-structure2} tells us that there exists a fixed point $z_0\not=z_*$, a maximal isotopy $I_0$ of $f\vert_{\S^2\setminus\{z_*,z_0\}}$ and a positive annulus $A_1$ of $I_0$ such that $\mathrm{rot}_{\check f_1,\kappa_1} (z)=0$, where $\check f_1$ is the lift of $f_1\vert_{A_1}$ to the universal covering space of $A_1$ defined by $I_0$ (as explained at the beginning of the section) and $\kappa_1$ the positive generator of $H_1(A_1,\Z)$. We have seen in Proposition \ref{prop:zero-rotation} that it implies that $f$ has a topological horseshoe.\end{proof}

\section{Birkhoff Recurrence classes and recurrent points of homeomorphisms of $\S^2$ with no topological horseshoe}

In this section we examine the possible dynamical behaviour of Birkhoff recurrence classes and transitive sets for homeomorphisms and diffeomorphisms of $\S^2$ with no topological horseshoe.

\subsection{Birkhoff recurrence classes}

We start by restating and proving Proposition \ref{propmain:birkoffclasses3fixedpoints}.

\begin{proposition}\label{prop_birkoffclasses3fixedpoints}
Suppose that $f$ is an orientation preserving homeomorphism of $\S^2$ with no topological horseshoe. \textcolor{black}{Assume $f$ has at least three distinct recurrent points.} If $\mathcal{B}$ is a Birkhoff recurrence class containing two fixed points $z_0$ and $z_1$, then 
\begin{itemize}
\item[(1)] either there exists $q\geq1$ such that \textcolor{black}{every recurrent point in $\mathcal B$ is a periodic point of period $q$ and} $\mathrm{fix}(f^q)\cap\mathcal{B}$ is an unlinked set of $f^q$;
\item[(2)] or $f$ is an irrational pseudo-rotation.
\end{itemize}
\end{proposition}
\begin{proof}
\textcolor{black}{Since $f$ has at least three distinct recurrent points, note that $\mathrm{ne}(f\vert_{\S^2\setminus\{z_0,z_1\}})$ is not empty.} For a given generator of $H_1(\S^2\setminus\{z_0,z_1\},\Z)$ and a given lift of $f\vert_{\S^2\setminus\{z_0,z_1\}}$, there exists a unique rotation number by Proposition \ref{prop:PB}. If this number is irrational, $f$ is an irrational pseudo-rotation. Suppose that it is rational and can be written $p/q$ in an irreducible way. \textcolor{black}{Then all periodic points have rotation number $p/q$  and their period is a multiple of $q$.} If there is no recurrent point \textcolor{black}{in $\mathcal{B}$} but $z_0$ and $z_1$, item (1) is true. Suppose now that there exists \textcolor{black}{a recurrent point $z\in\mathcal{B}$}. By Proposition \ref{prop:recurrentzero-rotation}, one knows that the annulus $\S^2\setminus\{z_0,z_1\}$ contains fixed points of $f^q$. Let $z_2$ be such a point. The class $\mathcal B$, containing fixed points, is \textcolor{black}{also a} Birkhoff recurrence class of $f^q$. Let us prove by contradiction, that  $z$ is a fixed point of $f^q$.  By Proposition \ref{prop:linking-recurrent}, there exists a fixed point $z_3$ such that $\mathrm{rot}_{f\vert_{\S^2\setminus\{z_2,z_3\}},\kappa}(z)\not=0+\Z$, if $\kappa$ is a generator of $H_1(\S^2\setminus\{z_2,z_3\},\Z)$. Of course either $z_3\not=z_0$ or $z_3\not=z_1$.The assumption $z_3\not=z_0$ contradicts Theorem \ref{th:birkhoffcycles} because we would have 
$$\mathrm{rot}_{f^q\vert_{\S^2\setminus\{z_2,z_3\}},\kappa}(z)\not=0+\Z, \enskip\mathrm{rot}_{f^q\vert_{\S^2\setminus\{z_2,z_3\}},\kappa}(z_0)=0+\Z.$$
The assumption $z_3\not=z_1$ contradicts Theorem \ref{th:birkhoffcycles}, because we would have
 $$\mathrm{rot}_{f^q\vert_{\S^2\setminus\{z_2,z_3\}},\kappa}(z)\not=0+\Z, \enskip\mathrm{rot}_{f^q\vert_{\S^2\setminus\{z_2,z_3\}},\kappa}(z_1)=0+\Z.$$  

\textcolor{black}{We already know that if $f$ is not an irrational pseudo-rotation, then every recurrent point in $\mathcal{B}$ is fixed, so all that remains to be proved is that $\mathrm{fix}(f^q)\cap\mathcal{B}$ is an unlinked set of $f^q$. Consider then} a maximal isotopy $I$ of $f^q$ that fixes $z_0$, $z_1$ and $z_2$. If $z'_0$ and $z'_1$ are two different fixed points of $I$ and $z'_2$ is a third fixed point of $f^q$, denote $\mathrm{rot}_{I, z'_0,z'_1}(z'_2)$ the rotation number of $z'_2$ defined for the lift of $f^q\vert_{\S^2\setminus\{z'_0,z'_1\}}$ to the universal covering space that fixes the lifts of the remaining points of $\mathrm{fix}(I)$, a generator of $H_1(\S^2\setminus\{z'_0,z'_1\})$ being given by the oriented boundary of a small disk containing $z'_0$ in its interior.  In particular, one has $\mathrm{rot}_{I, z'_0,z'_1}(z'_2)=0$ if $z'_2\in \mathrm{fix}(I)$. By Proposition \ref{prop:PB}, one knows that $\mathrm{rot}_{I, z_0,z_1}(z'_2)=0$ for every fixed point $z'_2\in \S^2\setminus\{z_0,z_1\}$ of $f^q$  because $\mathrm{rot}_{I, z_0,z_1}(z_2)=0$. \textcolor{black}{We claim that, if exactly one of the points $z'_0$, $z'_1$ is equal to $z_0$ or $z_1$, then $\mathrm{rot}_{I, z'_0,z'_1}(z'_2)=0$ for every  $z'_2\in \left(\mathcal B\cap \mathrm{fix}(f^q)\right)\setminus\{z'_0,z'_1\}$. Indeed if, for instance,  one of the points $z'_0$, $z'_1$ is equal to $z_0$ and $z_1\notin\{z'_0, z'_1\}$, then  $\mathrm{rot}_{I, z'_0,z'_1}(z_1)=0$, and the claim follows from Theorem \ref{th:birkhoffcycles}.}

\textcolor{black}{Likewise, if neither $z_0$ nor $z_1$ belongs to $\{z'_0, z'_1\}$, then again $\mathrm{rot}_{I, z'_0,z'_1}(z'_2)=0$ for every  $z'_2\in \left(\mathcal B\cap \mathrm{fix}(f^q)\right)\setminus\{z'_0,z'_1\}$. Indeed, if $z'_2\not=z_0$, then  by Proposition \ref{prop: rotation numbers}},
$$\mathrm{rot}_{I, z'_0,z'_1}(z'_2)= \mathrm{rot}_{I, z'_0,z_0}(z'_2)-\mathrm{rot}_{I, z'_1,z_0}(z'_2)=0$$ and a similar argument, \textcolor{black}{using $z_1$ instead of $z_0$, can be done if $z'_2=z_0$}.

Assume now that $\mathcal{B}\cap\mathrm{fix}(f^q)$ is not unlinked. Then, one can find a fixed point $z'_2\in{\mathcal B}$ of $f^q$ which belongs to the domain of $I$. Let $\mathcal F$ be a foliation transverse to $I$. The closed curve $I_{\mathcal F}(z'_2)$ defines naturally a loop $\Gamma$. Let $\delta_{\Gamma}$ be a dual function. As seen before, one can find a point $z'_0\in\mathrm{fix}(I)$ where the minimal value $k_-$ of $\delta_{\Gamma}$ is reached and a point $z'_1\in\mathrm{fix}(I)$ where its maximal value $k_+$ is reached. Note now that
 $$\mathrm{rot}_{I, z'_0,z'_1}(z'_2)=k_+-k_->0.$$We have found a contradiction. 
\end{proof}

As a consequence, we obtain Corollary \ref{crmain:periodsofbirkhoffclasses}  from the introduction, which we restate here:

\begin{corollary}\label{cr:onlytwoperiods} Let $f$ be an orientation preserving homeomorphism of $\S^2$ with  no topological horseshoe. Let $\mathcal{B}$ be a Birkhoff recurrence class containing periodic points of different periods, then:
\begin{itemize}
\item there exist integers $q_1$ and $q_2$, with $q_1$ dividing $q_2$, such that every periodic point in $\mathcal{B}$ has a period either equal to $q_1$ or to $q_2$
\item if $q_1\geq 2$, there exists a unique periodic orbit of period $q_1$ in $\mathcal B$;
\item if $q_1=1$, there exist at most two fixed points in $\mathcal B$.
\end{itemize}
\end{corollary}
\begin{proof}
By Corollary \ref{cr:divideperiods}, we already know that, if $\mathcal{B}$ has periodic points of periods $q_1<q_2$, then $q_1$ divides $q_2$. Assume, for a contradiction, that $\mathcal{B}
$ contains a periodic point $z_1$ of period $q_1$, a periodic point $z_2$ of period $q_2$ and a periodic point $z_3$ of period $q_3$, where $q_1<q_2< q_3$. The homeomorphism $f^{q_2}$ has at least three fixed points ($z_1$, $z_2$ and $f(z_2)$). Moreover, by the item (4) of Proposition \ref{prop: birkhoff connexions}, there exists a Birkhoff  recurrence class of $f^{q_2}$ that contains a point $z'_1$ in the orbit of $z_1$, a point $z'_2$ in the orbit of $z_2$ and a point $z'_3$ in the orbit of $z_3$. Since $z'_1$ and $z'_2$ are fixed by $f^{q_2}$, and $z'_3$ is periodic but not fixed, $f^{q_2}$ must have a topological horseshoe by Proposition \ref{prop_birkoffclasses3fixedpoints}, as does $f$. 

Suppose that $q_1\geq 2$ and that $\mathcal B$ contains at least two periodic orbits of period $q_1$. One can find a Birkhoff recurrence class of $f^{q_1}$ that contains a point in each of these two orbits plus a point in a periodic orbit of period $q_2$. This last point being not fixed by $f^{q_1}$, this contradicts Proposition \ref{prop_birkoffclasses3fixedpoints}. The last item is an immediate consequence of Proposition \ref{prop_birkoffclasses3fixedpoints}.
\end{proof}

\subsection{Recurrent points and transitive sets.}
Let $f:\S^2\to\S^2$ be an orientation preserving homeomorphism with no horseshoe. If $z_0$ and $z_1$ are periodic points of $f$ and $z$ is recurrent and not periodic, we define $\mathrm{rot}_{f,z_0,z_1}(z)\in\T^1$ in the following way: we denote $q$ the smallest common period of $z_0$ and $z_1$, we choose the class $\kappa$ of the boundary of a small topological disk containing $z_0$ as a generator of $H_1(\S^2\setminus \{z_0,z_1\},\Z)$, and we set
$$\mathrm{rot}_{f,z_0,z_1}(z)=\mathrm{rot}_{f^{q}\vert_{\S^2\setminus \{z_0,z_1\}},\kappa}(z),$$
 which is well defined by Theorem \ref{th:rotation-number}. We will say that a recurrent  and non periodic point $z$ has a {\it rational type} if $\mathrm{rot}_{f,z_0,z_1}(z)\in\Q/\Z$ for every choice of $z_0$ and $z_1$ and has an {\it irrational type} otherwise.

\subsubsection{Recurrent points of rational type}

Let us show that, if $z$ is recurrent and of rational type, then $f$ is topologically infinitely renormalizable over $\Lambda=\omega(z)$.

\begin{proposition} 
\label{prop:rational}Let $f:\S^2\to\S^2$ be an orientation preserving homeomorphism with no topological horseshoe, $z$ a recurrent point of $f$ of rational type, and $\Lambda=\omega(z)$. There exists an increasing sequence $(q_n)_{n\geq 0}$ of positive integers and a decreasing sequence  $(D_n)_{n\geq 0}$ of open disks containing $z$ such that:
\begin{itemize}
\item $q_n$ divides $q_{n+1}$; 
\item $D_n$ is $f^{q_n}$ periodic;
\item the disks $f^k(D_n)$, $0\leq k<q_n$ are pairwise disjoint;
\item $\Lambda\subset \bigcup_{k=0}^{q_n-1}f^{k}(D_n)$.
\end{itemize}
Furthermore, $f$ has periodic points of arbitrarily large period.
\end{proposition}

\begin{proof}We will construct the sequence by induction. The key point is the fact that $f$ is ``abstractly'' renormalizable over $\overline{O_f(z)}$, as explained in the remark following Proposition \ref{prop: birkhoff classes iterate}.
\textcolor{black}{Note that, as $f$ has a recurrent point which is not fixed, then it must also have at least two fixed points. Since $\Lambda$ is contained in a Birkhoff recurrence class and $z$ is of rational type, Proposition \ref{prop_birkoffclasses3fixedpoints} tells us that at least one of these points does not belong to $\Lambda$. We fix $z_*$ in $\mathrm{fix}(f)\setminus \Lambda$}
and set 
$$q_0=1, \enskip D_0=\S^2\setminus\{z_*\}.$$ Let us construct $q_1$ and $D_1$.  By Proposition \ref{prop:linking-recurrent}, there exists $z_0\in \mathrm{fix}(f)\cap D_0$ such that $\mathrm{rot}_{f,z_*,z_0}(z)\not=0+\Z$. By hypothesis, this number is rational and can be written $p/q+\Z$ where $p$ and $q>0$ are relatively prime. Let us state the key lemma, that will use Theorem \ref{th:global-structure2}
\begin{lemma} 
\label{lemma:differentclasses}The Birkhoff recurrence classes of $f^i(z)$, $0\leq i<q$,  for $f^q$, are all distinct.
\end{lemma}
\begin{proof} Set $A_0=\S^2\setminus\{z_*,z_0\}$ and write $\kappa_0$ for the generator of $H_1(A_0,\Z)$ defined by the boundary of a small topological disk containing $z_0$ in its interior. By Theorem \ref{th:global-structure2}, there exists a squeezed annulus $A$ of $f^q$ that contains $z$. Using Proposition \ref{prop:recurrentzero-rotation}, one deduces that $\mathrm{rot}_{f^q\vert A,\kappa}(z)\not=0+\Z$, if $\kappa$ is a generator of $H_1(A,\Z)$. We know that $\mathrm{rot}_{f^q\vert A_0,\kappa_0}(z)=0+\Z$ and so $A$ is not essential in $A_0$. The compact connected component $K$ of $A_0\setminus A$ contains at least a fixed point of $f^q$, because $A\cup K$ is a $f^q$-invariant open disk that contains a $f^q$-invariant compact set. Moreover, there exists $\rho\in\R/\Z$, equal to $\mathrm{rot}_{f^q\vert A,\kappa}(z)$ up to the sign, such that 
$$\begin{cases} 
\mathrm{rot}_{f^q,z',z^*} (z)=\rho&\mathrm{if }\,\, z'\in\mathrm{fix}(f^q)\cap K,\\
\mathrm{rot}_{f^q,z',z^*} (z)=0+\Z&\mathrm{if }\,\, z'\in\mathrm{fix}(f^q)\setminus K,\\
\end{cases}$$
 the last equality being due to the fact that $A\cup K$ is a $f^q$-invariant disk included in the annulus $\S^2\setminus\{z',z^*\}$ if $z'\in\mathrm{fix}(f^q)\setminus K$. Let us consider $0<i<q$. We immediately deduce that 
$$\begin{cases} 
\mathrm{rot}_{f^q,f^i(z'),z^*} (f^i(z))=\rho&\mathrm{if }\,\, z'\in\mathrm{fix}(f^q)\cap K,\\
\mathrm{rot}_{f^q,f^i(z'),z^*} (f^i(z))=0+\Z&\mathrm{if }\,\, z'\in\mathrm{fix}(f^q)\setminus K.\\
\end{cases}$$ 
To prove the lemma, it is sufficient to prove that the Birkhoff recurrence classes of $z$ and $f^i(z)$ for $f^q$ are distinct if $1\leq i<q$. We will argue by contradiction and suppose that there exists $i\in\{1,\dots, q-1\}$ such that the Birkhoff recurrence classes of $z$ and $f^i(z)$ for $f^q$ are equal. 
Let  $\mathring A_0$ be the $q$-tuple cover of $A_0$, let $\mathring \pi: \mathring A_0\to A_0$ be the covering projection and $\mathring T$ the generator of the group of covering automorphisms naturally defined by $\kappa_0$. Denote $\mathring z^* $ the end of $\mathring A_0$ corresponding to $z^*$ and $\mathring z_0 $ the end of $\mathring A_0$ corresponding to $z_0$. Fix a lift $\mathring{f}$ of $f\vert_{A_0}$ to $\mathring A_0$, set $\mathring g=\mathring f^q\circ \mathring T^{-p}$, and \textcolor{black}{ denote $\mathring g_{\mathrm{sp}}$ the extension of $g$ to $\mathring A_0\cup\{\mathring z^*,\mathring z_0\}$.} A simple finiteness argument permits us to say that if there exists a Birkhoff connection for $f^q$ from a point $z'\in A_0$ to a point $z''\in A_0$, there exists a Birkhoff connection for \textcolor{black}{$\mathring g_{\mathrm{sp}}$} from a given lift $\mathring z'\in \mathring{\pi}^{-1}(\{z'\})$ to a certain lift $\mathring z''\in \mathring{\pi}^{-1}(\{z''\})$. Fix $\mathring z\in\mathring{\pi}^{-1}(\{z\})$.
One deduces that there exists $0 \le j \le q-1$ such that $\mathring z\underset{\mathring f^q}\preceq \mathring T^j(\mathring f^i(\mathring z))$. The maps $\mathring f$ and $\mathring T$ commuting, one gets 
$$\mathring z\underset{\mathring g_{\mathrm{sp}}}\preceq \mathring T^j(\mathring f^i(\mathring z))\underset{\mathring g_{\mathrm{sp}}}\preceq \mathring T^{2j}(\mathring f^{2i}(\mathring z))\underset{\mathring g_{\mathrm{sp}}}\preceq \dots \underset{\mathring g_{\mathrm{sp}}}\preceq \mathring T^{(q-1)j}(\mathring f^{(q-1)i}(\mathring z))\underset{\mathring g_{\mathrm{sp}}}\preceq \mathring T^{qj}(\mathring f^{qi}(\mathring z))= \mathring f^{qi}(\mathring z)=\mathring g^i\circ \mathring T^{pi}(\mathring z).$$ Similarly, there exists $0 \le j' \le q-1$ such that $ \mathring T^{j'}(\mathring f^i(\mathring z))\underset{\mathring g_{\mathrm{sp}}}\preceq\mathring z$ and consequently
$$\mathring g^i\circ \mathring T^{pi}(\mathring z)=\mathring f^{qi}(\mathring z)=\mathring T^{qj'}(\mathring f^{qi}(\mathring z))\underset{\mathring g_{\mathrm{sp}}}\preceq \mathring T^{(q-1)j'}(\mathring f^{(q-1)i}(\mathring z))\underset{\mathring g_{\mathrm{sp}}}\preceq\dots \underset{\mathring g_{\mathrm{sp}}}\preceq \mathring T^{2j'}(\mathring f^{2i}(\mathring z))\underset{\mathring g_{\mathrm{sp}}}\preceq \mathring T^{j'}(\mathring f^i(\mathring z))\underset{\mathring g_{\mathrm{sp}}}\preceq \mathring z.$$

Consequently $\mathring z$ is a Birkhoff recurrent point of $\mathring g_{\mathrm{sp}}$, and so is every point $\mathring f^{m}\circ \mathring T^n(\mathring z)$, where $m$ and $n$ are integers. Moreover the Birkhoff recurrence class of $\mathring z$ is equal to the Birkhoff recurrence class of $\mathring g^i\circ \mathring T^{pi}(\mathring z)$ and so is equal to the Birkhoff recurrence class of  $\mathring T^{pi}(\mathring z)$.

 For every $l\in\{0,\dots,q\}$ denote $\mathring A^l$ the connected component of $\mathring\pi^{-1}(A)$ that contains $\mathring T^l(\mathring z)$ and $\mathring K^l$ the compact connected component of $\mathring A_0\setminus \mathring  A^l$ (which is itself a connected component of $\mathring\pi^{-1}(K)$). 
Every annulus $\mathring A^l$ is a squeezed annulus of $\mathring g_{\mathrm{sp}}$. Moreover every set $\mathring K^l$ contains at least one fixed point of $\mathring g$ and like in the annulus $A_0$, one has
$$\begin{cases} 
\mathrm{rot}_{\mathring g,\mathring z',\mathring z^*} (\mathring T^l(\mathring z))=\rho&\mathrm{if }\,\, \mathring z'\in\mathrm{fix}(\mathring g)\cap \mathring K^l,\\
\mathrm{rot}_{\mathring g,\mathring z',\mathring z^*} (\mathring T^l(\mathring z))=0+\Z&\mathrm{if }\,\, \mathring z'\in\mathrm{fix}(\mathring g)\setminus \mathring K^l.\\
\end{cases}$$
Note now that $pi$ is not a multiple of $q$ because $p$ and $q$ are relatively prime, and so the Birkhoff recurrence class of $\mathring z$ for \textcolor{black}{$\mathring g_{\mathrm{sp}}$} contains a translate $\mathring T^l(\mathring z)$, where $l\in\{1,\dots ,q-1\} $. This, and the the equalities above contradict Theorem \ref{th:birkhoffcycles}, \textcolor{black}{since if $ \mathring z'\in\mathrm{fix}(\mathring g)\cap \mathring K^0$, then $\mathrm{rot}_{\mathring g,\mathring z',\mathring z^*} (\mathring T^l(\mathring z))=0+\Z$ while $\mathrm{rot}_{\mathring g,\mathring z',\mathring z^*} (\mathring z)=\rho$}. 
\end{proof}

Now one can apply Proposition \ref{prop: birkhoff classes iterate} and the remarks that follow this proposition: there exists an integer $r\geq 1$ and a covering of $\Lambda$ by a family $(W^j)_{j\in \Z/rq\Z}$ of pairwise disjoint open sets of $\S^2$ that satisfies $f(W^j)=W^{j+1}$ for every $j\in\Z/rq\Z$. These open sets are included in $A_0$ and are not essential, because $rq\geq 2$ and $f(W^j)=W^{j+1}$ for every $j\in\Z/rq\Z$. For the same reason, if $j'\not=j$, the compact connected components of $A_0\setminus W^i$ do not contain any $W^{j'}$, $j'\not=j$. Adding the compact connected components of $A_0\setminus W^j$ to $W^j$, one gets a disk $D^j$. We have got a covering of $\Lambda$ by a family $(D^j)_{j\in \Z/rq\Z}$ of pairwise disjoint open disks of $\S^2$ that satisfies $f(D^j)=D^{j+1}$ for every $j\in\Z/rq\Z$. One can suppose that $z$ belongs to $D^0$. We set $D_1=D^0$ and $q_1=rq$. To construct $D_2$ and $q_2$ we do exactly the same replacing $D_0$ by $D_1$ and $f$ by $f^{q_1}\vert_{D_1}$ and continue the process.

Note that, since $D_n$ is homeomorphic to the plane and invariant by $f^{q_n}$, and since $f^{q_n}$ has a recurrent point in $D_n$, then by standard Brouwer Theory arguments $f^{q_n}$ must have a fixed point in $D_n$. Since the disks $f^k(D_n)$, $0\leq k<q_n$, are pairwise disjoint, this implies that $f$ has a periodic point of prime period $q_n$.

 \end{proof}
 
 As a direct consequence, we have:
 
 \begin{corollary}
 Let $f:\S^2\to\S^2$ be an orientation preserving homeomorphism with no horseshoe, $z$ a recurrent point of $f$ of rational type, and $\Lambda=\omega(z)$. Then the restriction of $f$ to $\Lambda$ is semiconjugated to an odometer.
 \end{corollary}

\subsubsection{Recurrent points of irrational type}

\textcolor{black}{ To help with the understanding of this subsection, we begin first with a look at some particular recurrent points} of irrational type.

\begin{proposition} 
\label{prop:irrational} Let $f:\S^2\to\S^2$ be an orientation preserving homeomorphism with no horseshoe and  $z$ a recurrent point of $f$ of irrational type. We suppose that there exist two fixed points $z_0$ and $z_1$ such that $\mathrm{rot}_{f,z_0,z_1}(z)=\alpha+\Z$, where $\alpha\not\in\Q$. We denote $\check f$ the lift of $f\vert_{\S^2\setminus\{z_0,z_1\}}$ such that $\mathrm{rot}_{\check f}(z)=\alpha$ and such that $\kappa$ is the generator of $H_1(\S^2\setminus\{z_0,z_1\},\Z)$ defined by the boundary of a small topological disk containing $z_0$ in its interior. Let $(p_n/q_n)_{n\geq 0}$ and $(p'_n/q'_n)_{n\geq 0}$ be two sequences of rational numbers converging to $\alpha$, the first one increasing, the second one decreasing. There exists a decreasing sequence  $(A_n)_{n\geq 0}$ of invariant open annuli containing $z$ such that 

\begin{itemize}
\item $A_{n}$ is essential in $\S^2\setminus\{z_0,z_1\}$,
\item $A_n$ is a positive annulus of $\check f^{q_n}\circ T^{-p_n}$, with a positive generator sent onto $\kappa$  by the morphism $\iota_*: H_1(A_n,\Z)\to H_1(\S^2\setminus\{z_0,z_1\},\Z)$;
\item $A_n$ is negative annulus of  $\check f^{q'_n}\circ T^{-p'_n}$, with a negative generator sent onto $\kappa$  by the morphism $\iota_*: H_1(A_n,\Z)\to H_1(\S^2\setminus\{z_0,z_1\},\Z)$.
\end{itemize}
\end{proposition}
\begin{proof}Let $I_0$ be a maximal isotopy of $f^{q'_0}\vert_{\S^2\setminus\{z_0,z_1\}}$ that is lifted to an identity isotopy of $\check f^{q'_0}\circ T^{-p'_0}$.  By Theorem \ref{th:local-structure}, there exists a positive annulus $A'_0$ containing $z_0$.  It must be essential and its positive class is sent onto $-\kappa$ by the morphism $\iota_*: H_1(A'_0,\Z)\to H_1(\S^2\setminus\{z_0,z_1\},\Z)$. Let $I'_0$ be a maximal isotopy of $f^{q_0}\vert_{A'_0}$ that is lifted to an identity isotopy of the restriction of $\check f^{q_0}\circ T^{-p_0}$ to the lift of $A'_0$.  By Theorem \ref{th:local-structure}, there exists a positive annulus \textcolor{black}{$A_0\subset A'_0$} containing $z_0$. $A_0$ must be essential, \textcolor{black}{because otherwise $A_0$ would be contained in a $\check f^{q_0}$ invariant open topological disk $D_0\subset A'_0$ containing $z_0$, and this contradicts $\mathrm{rot}_{\check f}(z)=\alpha$. Also} its positive class is sent onto $\kappa$ by the morphism $\iota_*: H_1(A_0,\Z)\to H_1(\S^2\setminus\{z_0,z_1\},\Z)$. We continue this process and construct alternately two sequences $(A'_n)_{n\geq 0}$ and $(A_n)_{n\geq 0}$. The sequence $(A_n)_{n\geq 0}$ satisfies the conclusion of the proposition. \end{proof}

\begin{remarks*} Let $X=\bigcap_{n\geq 0}A_n$. Note that the complement of $A_n$ has exactly two connected components, $F_n$ that contains $z_0$ and $F_n'$ that contains $z_1$. Let us define $F=\bigcup_{n\geq 0}F_n$ and $F'=\bigcup_{n\geq 0}F'_n$,  which are connected, invariant and disjoint. This implies that the complement of $X$ is either connected (if $F\cup F'$ is connected) or has two connected components (if $F\cup F'$ is not connected). Note also that if $z'_0$ and $z'_1$ are periodic points that belong to $F$, then $\mathrm{rot}_{f,z'_0,z'_1}(z)=0+\Z$, because there exists $n$ such that $z$ belongs to an invariant open disk $A_n\cup F_n'$ of $\S^2\setminus\{z'_0,z'_1\}$. In case $z'_0$ belongs to $F$ while $z'_1$ belongs to $F'$,  then $\mathrm{rot}_{f,z'_0,z'_1}(z)=q\mathrm{rot}_{f,z_0,z_1}(z)$, where $q$ is the smallest common period of $z'_0$ and $z'_1$. Finally, one knows that every point of \textcolor{black}{$X$} has a rotation number equal to $\alpha$ when this rotation number is defined.  
 This is not necessarily the case for every point in $\overline X$. Indeed let us suppose that $f$ is a homeomorphism of $\S^2$ that contains an invariant closed annulus $A$ satisfying the following: 
\begin{itemize}
\item every point of the boundary of $A$ is fixed;
\item there exists an essential invariant loop $\Gamma$ in $A$ of irrational rotation number;
\item the dynamics is north-south in the open annuli bounded by $\Gamma$ and the boundary circles.
\end{itemize}

If $z_0$ and $z_1$ are chosen on different boundary components of $A$ and $z$ chosen on $\Gamma$, all sets $A_n$ coincide with the interior of $A$, for $n$ large enough, and consequently, there exist infinitely many fixed points in $\overline X$.

\medskip

In the previous example, there exists a compact invariant annular set of irrational rotation number. In the following situation, this will not be the case. Starting with a Denjoy counterexample, one can construct a homeomorphism $f$ of $\S^2$ that contains an invariant closed annulus $A$ satisfying the following: 

\begin{itemize}
\item every point of the boundary is fixed;
\item there exists an invariant Cantor set $K$ of irrational rotation number;
\item the orbit of a point $z$ in the interior of $A$ that does not belong to $K$ has either its $\alpha$ limit or its $\omega$ limit contained in one of the boundary circles.
\end{itemize}

\end{remarks*}

Let us explain now what happens for a general recurrent point of irrational type. 

\textcolor{black}{
\begin{proposition}
 Let $f:\S^2\to\S^2$ be an orientation preserving homeomorphism with no topological horseshoe and  $z$ a recurrent point of $f$ of irrational type. There exists some integer $q\ge 1$, such that, if $g=f^q$, then there exists two points $z_0, z_1$ that are fixed by $g$ and such that $\mathrm{rot}_{g,z_0,z_1}(z)=\alpha+\Z$, where $\alpha\not\in\Q$. We denote $\check g$ the lift of $f\vert_{\S^2\setminus\{z_0,z_1\}}$ such that $\mathrm{rot}_{\check g}(z)=\alpha$ and such that $\kappa$ is the generator of $H_1(\S^2\setminus\{z_0,z_1\},\Z)$ defined by the boundary of a small topological disk containing $z_0$ in its interior. Let $(p_n/q_n)_{n\geq 0}$ and $(p'_n/q'_n)_{n\geq 0}$ be two sequences of rational numbers converging to $\alpha$, the first one increasing, the second one decreasing. There exists a decreasing sequence  $(A_n)_{n\geq 0}$ of open annuli containing $z$ such that 
\begin{itemize}
\item $A_n$ is invariant by $g$, and $f^k(A_n)\cap A_n=\emptyset$ if $0<k<q$,
\item $A_{n}$ is essential in $\S^2\setminus\{z_0,z_1\}$,
\item $A_n$ is a positive annulus of $\check g^{q_n}\circ T^{-p_n}$, with a positive generator sent onto $\kappa$  by the morphism $\iota_*: H_1(A_n,\Z)\to H_1(\S^2\setminus\{z_0,z_1\},\Z)$;
\item $A_n$ is negative annulus of  $\check g^{q'_n}\circ T^{-p'_n}$, with a negative generator sent onto $\kappa$  by the morphism $\iota_*: H_1(A_n,\Z)\to H_1(\S^2\setminus\{z_0,z_1\},\Z)$.
\end{itemize}
\end{proposition}}

\begin{proof}
If there exist two fixed points $z_0$ and $z_1$ such that $\mathrm{rot}_{f,z_0,z_1}(z)\not\in\Q/\Z$ we can apply Proposition \ref{prop:irrational}. Suppose now that $\mathrm{rot}_{f,z_0,z_1}(z)\in \Q/\Z$ for all fixed points $z_0$ and $z_1$. Then starting from any fixed point $z_*$, one can construct a disk $D_1$ containing $z$ and an integer $q_1$  such that $f^{q_1}(D_1)=D_1$ and $f^k(D_1)\cap f^{k'}(D_1)=\emptyset$ if $0\leq k<k'<q_1$, like in the proof of Proposition \ref{prop:rational}.  Moreover we get that $\omega_{f^{q_1}}(z)$ is also contained in $D_1$. One deduces that $\mathrm{rot}_{f,z'_0,z'_1}(z)\in \Q/\Z$ for all periodic points $z'_0$ and $z'_1$ that do not belong to $D_1$. In particular $\mathrm{rot}_{f,z'_0,z'_1}(z)\in \Q/\Z$ if the periods of $z'_0$ and $z'_1$ are both smaller than $q_1$. Let us consider the map $f^{q_1}\vert_{D_1}$ or more precisely its extension $f_1$ to the Alexandroff compactification of $D_1$ writing $z_{1,*}$ the point at infinity.  Suppose that $\mathrm{rot}_{f_1,z_{1,*},z_1}(z)\in \Q/\Z$ for every fixed point $z_1\in D$ of $f_1$. Then we can continue the process and construct a disk $D_2$ containing $z$ and an integer $q_2>q_1$, multiple of $q_1$, such that $f^{q_2}(D_2)=D_2$ and $f^k(D_1)\cap f^{k'}(D_1)=\emptyset$ if $0\leq k<k'<q_2$. We know that $\mathrm{rot}_{f,z'_0,z'_1}(z)\in \Q/\Z$ if the periods of $z'_0$ and $z'_1$ are smaller than $q_2$. This implies the process must end at finite time, since $z$ is of irrational type. Consequently, one can apply Proposition \ref{prop:irrational} to a power $f^q$ of $f$ and obtain the result.
\end{proof}

\textcolor{black}{Before ending this section, there} is still some more information one can derive about the Birkhoff recurrence class of a recurrent point of irrational type:
\begin{corollary}\label{co: B.classes}
Let $f:\S^2\to\S^2$ be an orientation preserving homeomorphism with no topological horseshoe and  $z$ a recurrent point of $f$ of irrational type. We suppose that there exist two $q$-periodic fixed points $z_0$ and $z_1$ such that $\mathrm{rot}_{f,z_0,z_1}(z)=\alpha+\Z$, where $\alpha\not\in\Q$. Let $\mathcal{B}$ be the Birkhoff recurrence class of $z$ for $f$. Then:
\begin{itemize}
\item $\mathcal{B}$ does not contain any periodic point that is neither in the orbit of $z_0$ nor in the orbit of $z_1$;
\item if $\mathcal{B}$ contains the orbit of $z_0$, then for any pair of periodic points $z'_0, z'_1$ not containing $z_0$, $\mathrm{rot}_{f,z'_0,z'_1}(z)=0+\Z$;
\item if $\mathcal{B}$ contains both $z_0$ and $z_1$, then \textcolor{black} {$q=1$} and $f$ is a irrational pseudo-rotation. 
\end{itemize}
\end{corollary}

\begin{proof}
\textcolor{black}{Note that $\mathcal{B}= \bigcup_{i=0}^{q-1} f^{i}(\mathcal{B}_0)$, where $\mathcal{B}_0$ is the Birkhoff recurrence class of $z$ for $g=f^q$. From Theorem  \ref{th:birkhoffcycles} we know that $\mathcal{B}_0$ does not contain any periodic point that is neither $z_0$ or $z_1$, since $\mathrm{rot}_{g,z_0,z_1}(z')\in\Q/\Z$ if $z'$ is periodic. This implies the first item.  The second item is also a consequence of Theorem  \ref{th:birkhoffcycles}, since $\mathrm{rot}_{f,z'_0,z'_1}(z_0)=0+\Z$. The fact that $f^q$ is a pseudo-rotation in the third item is a direct consequence of Proposition \ref{prop:PB}, where $z_0$ and $z_1$ play the role of $N$ and $S$. But if $f^q$ is a pseudo-rotation, it follows that directly that $q=1$.}
\end{proof}

\subsubsection{Transitive Sets}
We say that a transitive set $\Lambda$ is {\it of irrational type} if $\Lambda=\omega(z)$ for some recurrent point $z$ of irrational type.  In particular $\Lambda$  is contained in the Birkhoff recurrence class of $z$.
Let us now restate Proposition \ref{prmain:transitivesetsgeneralcase} from the introduction.

\begin{proposition}\label{pr:transitivesetsgeneralcase} Let $f:\S^2\to\S^2$ be an orientation preserving homeomorphism with no horseshoe and $\Lambda$ a closed and invariant transitive set. Then:
\begin{enumerate}
\item either $\Lambda$ is a periodic orbit;
\item or $f$ is topologically infinitely renormalizable over $\Lambda$;
\item or $\Lambda$ is of irrational type.
\end{enumerate}
Furthermore, if the second possibility holds, then $f$ has periodic points of arbitrarily large prime periods.
\end{proposition}
\begin{proof}
\textcolor{black}{Pick $z$ in $\Lambda$ such that $\omega(z)=\Lambda$, which exists since $\Lambda$ is transitive. If $z$ is periodic, we are in case 1. If $z$ is a point of rational type then, by Proposition \ref{prop:rational}, we are in case 2. If $z$ is a point of irrational type then by Proposition \ref{prop:irrational} we are in case 3.}
\end{proof}

 We also obtain directly a proof of Corollary \ref{crmain:transitivewithperiodicorbit} of the introduction as, if the second possibility in Proposition \ref{pr:transitivesetsgeneralcase} holds, then $\Lambda$ does not contain any periodic point, and if $\Lambda$ is of irrational type and contains two distinct periodic orbits, then applying Corollary \ref{co: B.classes} one may deduce that $f$ is a irrational pseudo-rotation.
 


Let us conclude this section with an application that concerns dissipative $\mathcal{C}^1$ diffeomorphisms of the plane $\R^2$. Let us recall Proposition \ref{prmain:dissipativediffeomorfisms}

\begin{proposition}\label{pr:dissipativediffeomorfisms}
Let $f:\R^2\to\R^2$ be an orientation preserving diffeomorphism with no topological horseshoe, such that $0<\mathrm{det} Df(z)<1$ for all $z\in\R^2$. Let $\Lambda$ be a compact, invariant, and transitive subset that is locally stable. Then either $\Lambda$ is a periodic orbit, or $f$ is topologically infinitely renormalizable over $\Lambda$.
 \end{proposition}
 \begin{proof}
Let $f_{\mathrm{sp}}$ be the natural extension of $f$ to the Alexandrov compactification of $\R^2$ by adding a point $\infty$ at infinity.  By Proposition \ref{pr:transitivesetsgeneralcase}, it suffices to show that $\Lambda$ cannot have a recurrent point of irrational type.

Assume, for a contradiction, this is false and let $z$ be a recurrent point of irrational type such that $\Lambda=\omega(z)$. There exist $q\geq 1$ and  $z_0, z_1$ in $\mathrm{fix}(f^q_{\mathrm{sp}})$ such that $\mathrm{rot}_{f^q_{\mathrm{sp}}, z_0, z_1}(z)\not \in\Q/\Z$.  By Theorem \ref{th:global-structure2}, there exists a squeezed annulus $A$ of $f^q_{\mathrm{sp}}$ that contains $z$. Moreover $z_0$ and  $z_1$ are not in the same component of the complement of $A$. Denote $z^*$ the point $z_0$ or $z_1$ that is not in the same component as $\infty$. As explained before, one has $\mathrm{rot}_{f^q_{\mathrm{sp}},  z^*,\infty}=\mathrm{rot}_{f^q\vert_{\R^2\setminus\{z^*\}}}(z)\not \in\Q/\Z$. Write $\Lambda'=\omega_{f^q_{\mathrm{sp}}}(z)$. Since $\Lambda'$ is a compact subset of $\R^2$ and $\Lambda'\subset \Lambda$, one knows that $\infty$ does not belong to $\Lambda'$. We claim that $z^*$ also does not belong to $\Lambda'$. To see this, we argue by contradiction and suppose that $z^*\in\Lambda'$. Note first that $z^*$ cannot be neither a sink nor a source because $\Lambda'$ is a transitive set of $f^q$. We deduce that the eigenvalues of $Df^q(z^*)$ are distinct real numbers, otherwise $z$ would be a sink because $0<\mathrm{det} Df^q(z^*)<1$. Therefore one may blowup the point $z^{*}$ to a circle $\Sigma$ and extend $f^q\vert_{\R^2\setminus\{z^*\}}$ to a homeomorphism $g_{\mathrm{ann}}$ of $\R^2_{\mathrm{ann}}=(\R^2\setminus\{z^*\})\sqcup\Sigma$ admitting on $\Sigma$ either a fixed point (if the eigenvalues are both positive), or a periodic point with period two (if the eigenvalues are both negative). We paste two copies of $\R^2_{\mathrm{ann}}$ on $\Sigma$ to get an open annulus $\R^2_{\mathrm{double}}$ and construct by reflection a homeomorphim $g_{\mathrm{double}}$ on $\R^2_{\mathrm{double}}$. The set $\omega_{g_{\mathrm{double}}}(z)$ contains the point $z$, with $\mathrm{rot}_{g_{\mathrm{double}}}(z)\not \in\Q/\Z$, and a periodic point $z'\in\Sigma$, with $\mathrm{rot}_{g_{\mathrm{double}}}(z')=0+\Z$ or $\mathrm{rot}_{g_{\mathrm{double}}}(z')=1/2+\Z$. By Theorem \ref{th:rotation-number}, we deduce that $g_{\mathrm{double}}$ has a topological horseshoe, which is also the case for $f$.

Since $\Lambda$ is locally stable for $f$, so is $\Lambda'$ for $f^q$ and one can find a forward invariant neighborhood $U$ of $\Lambda'$  such that $z^{*}$ does not belong to $\overline U$. The connected component of $U$ that contains $z$, denoted $U'$, is forward invariant by a power of $f^q$, because $z$ is recurrent. It is essential in the annulus $\R^2\setminus\{z^*\}$, otherwise $\mathrm{rot}_{f^q\vert_{\R^2\setminus\{z^*\}}}(z)$ would be rational. This implies that $U'$ is forward invariant by $g$ and that the connected component of $\R^2\setminus\overline{U'}$ that contains $z^*$ is backward invariant and relatively compact. This contradicts the fact that $f^q$ decreases the area.
 \end{proof}


\begin{thebibliography}{99}


\bibitem{Angenent}
Angenent S. B.: A remark on the topological entropy and invariant circles of an area preserving twistmap. Twist mappings and their applications, 1--5
\textit{IMA Vol. Math. Appl.} \textbf{44} (1992), Springer, New York.

\bibitem{Arnaud} Arnaud M.C.: Cr\'eation de connexions en topologie C1.
\textit{Ergodic Theory Dynam. Systems} \textbf{21} (2001), 339--381.

\bibitem{BargeGillette}
Barge M., Gillette R. M.: Rotation and periodicity in plane separating continua.
\textit{Ergodic Theory Dynam. Systems} \textbf{11} (1991), no. 4, 619--631.



\bibitem{BeguinCrovisierLeRoux1}
B\' eguin F., Crovisier S.,  Le Roux F.:Construction of curious minimal uniquely ergodic homeomorphisms on
manifolds: the Denjoy-Rees technique. 
\textit{Ann. Sci. \'Ecole Norm. Sup. (4)} \textbf{40} (2007), no. 2, 251--308.


\bibitem{BeguinCrovisierLeRoux2}
B\' eguin F., Crovisier S.,  Le Roux F.: Fixed point sets of isotopies on surfaces.
\textit{Jour. Europ. Math. Soc.} \textbf{22} (2020), no. 6, 1971--2046.

\bibitem{Besicovitch}
Besicovitch A.: A problem on topological transformations of the plane. II.
\textit{ Proc. Cambridge Philos. Soc.} \textbf{47} (1951), 38--45.

\bibitem{Birkhoff}
Birkhoff G.D.: Sur quelques courbes ferm\'ees remarquables
\textit{Bull. Soc. math. France} \textbf{60} (1932), no. 2, 1-26.


\bibitem{BGLT}
Bonatti C., Gambaudo J.M., Lion J.M., Tresser C.: Wandering domains for infinitely renormalizable diffeomorphisms of the disk.
\textit{Proc. Am. Math. Soc.} \textbf{122} (1994), no.4, 1273--1278.


\bibitem{Boyland}
Boyland P.:Rotation sets and monotone periodic orbits for annulus homeomorphisms.
\textit{Comment. Math. Helv.} \textbf{67} (1992), no. 2, 203--213.

\bibitem{BoylandHall}
Boyland P., Hall G.R.: Invariant circles and the order structure of periodic orbits in monotone
twist maps,
\textit{Topology,} \textbf{26} (1987), no 1,  21--35.



\bibitem{Brown}
Brown, R. F.: The  Lefschetz fixed  point  theorem. \textit{Scott  Foreman  Co. } Glenview
Illinois, London, (1971).




\bibitem{Conejeros} Conejeros, J.: The local rotation set is an interval. \textit{Ergodic Theory and Dynam. Systems}, \textbf{38} (2018), no. 7, 2571--2617.

\bibitem{crovisier_birkhoff} Crovisier, S.: Periodic orbits and chain-transitive sets of $\mathcal{C}^1$-diffeomorphisms 
\textit{Publ. Math. Inst. Hautes \'etudes Sci.} \textbf{104} (2006), 87--141. 

\bibitem{CKKP} Crovisier, S.; Kocsard, A.; Koropecki A.; Pujals, E.: Structure of attractors for strongly dissipative
surface diffeomorphisms. In preparation. 

\bibitem{CPT} Crovisier, S.;  Pujals, E.; Tresser, C.: Boundary of chaos and period doubling for dissipative embeddings
of the 2-dimensional disk. In preparation.

\bibitem{CLM} de Carvalho, A.; Lyubich M.; Martens M.: Renormalization in the H\'enon family, I: Universality but non-rigidity.
\textit{J. Stat. Phys.,} \textbf{121} (2005), no. 5--6, 611--669. 

\bibitem{LM} Lyubich M.; Martens M.: Renormalization in the H\'enon family, II: the heteroclinic web.
\textit{Invent. Math.,} \textbf{186} (2011), no. 1, 115--189.


\bibitem{Dold}
Dold A.:
 Fixed point indices of iterated maps.
\textit{Invent. Math.,} \textbf{74}(1983), 419--435.
,


\bibitem{Franks}
 Franks J.:
 Generalizations of the Poincar\'e-Birkhoff theorem. 
\textit{Ann. of Math. (2),} \textbf{128}(1988), no.1, 139--151.


\bibitem{FranksHandel}
 Franks J., Handel M.:
 Entropy zero area preserving diffeomorphisms of $\S^2$.
\textit{Geom. Topol.,} \textbf{16}(2013), no.4, 2187--2284.


\bibitem{FranksRicheson}
 Franks J., Richeson D.:
Shift equivalence and the Conley index. 
\textit{Trans. Amer. Math. Soc.,} \textbf{352}(2000), no.7, 3305--3322.


\bibitem{Handel}
Handel M.: Zero entropy surface diffeomorphisms (preprint)


\bibitem{HernandezLeCalvezRuizdelPortal}
Hern\'andez-Corbato L., Le Calvez P., Ruiz del Portal, F. R.:
About the homological discrete Conley index of isolated invariant acyclic continua.
\textit{Geom. Topol.,} \textbf{17}(2013), no.5, 2977--3026.

\bibitem{HMT} Hazard, P.; Martens, M.; Tresser, C.: Infinitely many moduli of stability at the dissipative boundary of chaos.
To appear in \textit{Trans. Amer. Math. Soc.} (2017). 


\bibitem{Homma}
Homma, T.:
 An extension of the Jordan curve theorem.
\textit{Yokohama Math. J.,} \textbf{1}(1953),125--129.



\bibitem{Jaulent}
Jaulent O.: Existence d'un feuilletage positivement transverse \`a un hom\'eomorphisme de surface.  
\textit{Ann. Inst. Fourier (Grenoble),} \textbf{ 64} (2014), no. 4, 1441--1476.



\bibitem{Katok}
Katok, A.. Lyapounov exponents, entropy and periodic orbits for diffeomorphisms,
\textit{Publications Math\'ematiques de l'I.H.\'E.S.}\textbf{51} (1980), no. 4, 131--173.



\bibitem{KennedyYorke}
Kennedy, J., Yorke, J. A.: Topological horseshoes.
\textit{Trans. Amer. Math. Soc.} \textbf{ 353} (2001), no. 6, 2513--2530.


\bibitem{Koropecki} Koropecki A.: Realizing rotation numbers on annular continua.
\textit{Math. Z.} \textbf{285} (2007), 549--564. 

\bibitem{LeCalvez1}
Le Calvez P.: 
Une version feuillet\'ee \'equivariante du th\'eor\`eme de translation de Brouwer. \textit{Publ. Math. Inst. Hautes \'etudes Sci.} \textbf{102} (2005), 1--98.




\bibitem{LeCalvez2}  Le Calvez P.: Periodic orbits of Hamiltonian homeomorphisms of surfaces.
\textit{Duke Math. J. } {\bf 133} (2006), no. 1, 125--184.



\bibitem{LeCalvez3}  Le Calvez P.: Pourquoi les points p\'eriodiques des hom\'eomorphismes du plan tournent-ils autour de certains points fixes?
\textit{Ann. Sci. \'Ec. Norm. Sup\'er. (4)} {\bf 41} (2008), no. 1, 141--176.

\bibitem{LeCalvezTal}Le Calvez P., Tal F. A.: Forcing theory for transverse trajectories of surface homeomorphisms. To appear in {\it Invent. Math.}

\bibitem{LeCalvezYoccoz1}\textcolor{black}{ Le Calvez P. Yoccoz J.-C : Un th\'eor\`eme d'indice pour les hom\'eomorphismes du plan au voisinage d'un point fixe. {\it Ann. of Math. (2)} {\bf 146} (1997), no. 2, 241--293. }

\bibitem{LeCalvezYoccoz2} \textcolor{black}{Le Calvez P. Yoccoz J.-C : Suite des indices de Lefschetz des it\'er\'ees pour un domaine de Jordan qui est un bloc isolant. {\it Preprint} (1997).}
\bibitem{LeRoux}  Le Roux F.: L'ensemble de rotation autour d'un point fixe.
\textit{Ast\'erisque } {\bf 350} (2013), 1--109.


\bibitem{Nussbaum}  Nussbaum R.D.: The fixed point index and some applications,
\textit{S\'eminaire  de
Math\'ematiques  sup\'erieures,
 } Les  Presses  de  L'Universit\'e  de  Montr\'eal,  Montr\'eal,
(1985).


\bibitem{PassegiPotrieSambarino} Passegi A., Potrie R., Sambarino M.: Rotation intervals and entropy on attracting annular continua. 
 To appear in {\it Geom. and Top.} (2017).

\bibitem{Perez-Marco}
\textcolor{black}{Perez-Marco R.: Fixed points and circle maps.
\textit{Acta Math} {\bf{179}}(1997), No. 2, 243--294. }



\bibitem{Rees}  Rees M.: A minimal positive entropy homeomorphism of the $2$-torus.
{\it J.  London
Math. Soc.
 } {\bf 23} (1981), 537--550.
 
\bibitem{Szymczak} \textcolor{black}{ Szymczak A.: The Conley index and symbolic dynamics. {\it Topology} {\bf 35} (1996), no. 2, 28--299. }
  
 \bibitem{smale} Smale S.: Differentiable dynamical systems. \textit{Bull. of the Am. Math. Soc.} {\bf{73}}(1967), 747--817.

\end{thebibliography}
\end{document}